\documentclass[10pt]{article}
\usepackage{amsthm}
\usepackage{thmtools}
\usepackage[colorlinks=true,linkcolor=blue]{hyperref}
\usepackage{enumitem}
\usepackage{mathrsfs}
\usepackage[all]{xy}
\usepackage{caption}
\usepackage{graphics}
\usepackage{graphicx}
\usepackage{tikz}
\usepackage{url}
\usepackage{amsmath}
\usepackage{amsxtra}
\usepackage{amssymb}
\usepackage{turnstile}
\usepackage{verbatim}
\usepackage{enumitem}
\usepackage[retainorgcmds]{IEEEtrantools}
\usepackage{accents}
\usepackage[all]{xy}
\usetikzlibrary{matrix}
\usetikzlibrary{calc,decorations.pathmorphing,shapes}

\declaretheoremstyle[bodyfont=\sl]{slanted}

\swapnumbers
\declaretheorem[name=Definition,style=definition,qed=$\dashv$,
numberwithin=section]{dfn}
\declaretheorem[name=Definition,style=definition,numbered=no,qed=$\dashv$]{dfn*}
\declaretheorem[name=Definition,style=definition,numbered=no]{dfnnoqed*}

\declaretheorem[name=Theorem,style=slanted,sibling=dfn]{tm}
\declaretheorem[name=Theorem,style=slanted,numbered=no]{tm*}
\declaretheorem[name=Lemma,style=slanted,sibling=dfn]{lem}

\declaretheorem[name=Corollary,style=slanted,numbered=no]{cor*}
\declaretheorem[name=Remark,style=definition,sibling=dfn]{rem}

\swapnumbers
\declaretheoremstyle[headfont=\scshape]{claimstyle}
\declaretheorem[name=Claim,style=claimstyle]{clm}

\declaretheorem[name=Claim,style=claimstyle]{clmtwo}
\declaretheorem[name=Claim,style=claimstyle]{clmthree}
\declaretheorem[name=Claim,style=claimstyle]{clmfour}
\declaretheorem[name=Claim,style=claimstyle]{clmfive}

\declaretheorem[name=Claim,style=claimstyle,numbered=no]{clm*}
\declaretheorem[name=Subclaim,style=claimstyle,numberwithin=clm]{sclm}

\declaretheorem[name=Subclaim,style=claimstyle,numberwithin=clmthree]{sclmthree}
\declaretheorem[name=Subclaim,style=claimstyle,numberwithin=clmfour]{sclmfour}
\declaretheorem[name=Subclaim,style=claimstyle,numberwithin=clmfive]{sclmfive}

\declaretheorem[name=Subclaim,style=claimstyle,numbered=no]{sclm*}

\declaretheorem[name=Subsubclaim,style=claimstyle,numbered=no]{ssclm*}

\declaretheoremstyle[headfont=\scshape]{casestyle}

\declaretheorem[name=Case,style=casestyle]{case}

\declaretheorem[name=Case,style=casestyle]{casetwo}
\declaretheorem[name=Case,style=casestyle]{casethree}
\declaretheorem[name=Case,style=casestyle]{casefour}
\declaretheorem[name=Case,style=casestyle]{casefive}
\declaretheorem[name=Case,style=casestyle]{casesix}

\newcommand{\lgcd}{\mathrm{lgcd}}

\newcommand{\compmode}{1}
\newcommand{\compopt}[2]{\ifthenelse{\equal{\compmode}{0}}{#1}{#2}}

\newcommand{\J}{\mathcal{J}}

\newcommand{\sub}{\subseteq}
\newcommand{\cross}{\times}

\newcommand{\om}{\omega}
\newcommand{\pow}{\mathcal{P}}
\newcommand{\OR}{\mathrm{OR}}
\newcommand{\Hull}{\mathrm{Hull}}
\newcommand{\cut}{\backslash}

\newcommand{\Tt}{\mathcal{T}}

\newcommand{\Uu}{\mathcal{U}}

\newcommand{\rg}{\mathrm{rg}}
\newcommand{\dom}{\mathrm{dom}}
\newcommand{\ins}{\trianglelefteq}

\newcommand{\pins}{\triangleleft}
\newcommand{\npins}{\ntriangleleft}
\newcommand{\crit}{\mathrm{cr}}
\newcommand{\rest}{\!\upharpoonright\!}
\newcommand{\com}{\circ}
\newcommand{\range}{\rg}
\newcommand{\lh}{\mathrm{lh}}
\newcommand{\Ult}{\mathrm{Ult}}
\newcommand{\sats}{\models}

\newcommand{\es}{\mathbb{E}}

\newcommand{\core}{\mathfrak{C}}
\newcommand{\pred}{\mathrm{pred}}

\newcommand{\id}{\mathrm{id}}

\newcommand{\conc}{\ \widehat{\ }\ }

\DeclareMathOperator{\Th}{Th}
\DeclareMathOperator{\card}{card}
\DeclareMathOperator{\cof}{cof}

\newcommand{\rSigma}{\mathrm{r}\Sigma}
\newcommand{\bfrSigma}{\undertilde{\rSigma}}

\newcommand{\bfSigma}{\undertilde{\Sigma}}

\newcommand{\psub}{\subsetneq}

\newcommand{\cHull}{\mathrm{cHull}}

\newcommand{\tu}{\textup}

\renewcommand{\cut}{\backslash}

\newcommand{\pvec}{\vec{p}}

\newcommand{\exit}{\mathrm{ex}}

\newcommand{\passive}{\mathrm{pv}}
\newcommand{\dropset}{\mathscr{D}}

\renewcommand{\cut}{\backslash}

\newcommand{\spc}{\mathrm{spc}}
\newcommand{\ph}{\mathfrak{P}}
\newcommand{\D}{\mathrm{D}}
\newcommand{\muvec}{\vec{\mu}}
\title{The initial segment condition for $\kappa^+$-supercompactness}
\author{Farmer Schlutzenberg\footnote{
afirstname dot alastname at tuwien dot ac dot at, afirstname dot alastname at gmail}}

\begin{document}

\maketitle

\begin{abstract}
We give a development of the fine structure of mice with long extenders,
to the level of $\kappa^+$-supercompact cardinals $\kappa$.
We do this using a hierarchy
with features more analogous to those familiar in the short extender context than the hierarchies  introduced by Woodin and by Neeman-Steel.
In particular, the mice we consider satisfy  stronger versions of the initial segment condition.
We establish a form of fine structural condensation involving embeddings
$\pi:H\to M$ which need not be the identity below the projectum of $H$ (under special assumptions). We also adapt the analysis of the Dodd structure of short extenders on the sequence to mice at this level.
\end{abstract}

\tableofcontents
\section{Introduction}

The theory of fine structural models containing cardinals $\kappa$ which
are $\kappa^{+}$-supercompact, and in fact, $\kappa^{+n}$-supercompact, where $n<\om$,
was introduced by Woodin in the 2010s,
in particular with the circulation of his monograph on the topic.
Around the same time, Neeman and Steel  released some notes \cite{nsp1}, \cite{nsp1fs}
on the key ideas involved which
go beyond the earlier theory
(the earlier theory  had reached the level of superstrong cardinals).
The most obvious difference between the two levels is that at the lower level,
all extenders $E$ used to build premice
are short, consisting of component
measures $E_a$ 
concentrating on the critical point $\crit(E)$ of $E$,
whereas at the higher level,
there are also long extenders,
with measures concentrating on ordinals $>\crit(E)$,
in particular on $\crit(E)^{+M}$, where $M$ is the structure over which 
$E$ is an extender, in the case of $\kappa^+$-supercompactness. Thus, we refer to the earlier theory as that of short extenders,
and the latter as that of long extenders.

Woodin developed the theory of long extenders
through the level of $\kappa^{+n}$-supercompactness. Neeman-Steel \cite{nsp1},
\cite{nsp1fs}, explained
the key new concepts, focusing
on the basic definitions and the proof of $1$-solidity at the $\kappa^+$-supercompact level. The hierarchy they define differs
from that of Woodin's, but they claim in their notes that it is equivalent. The segments of the models $M$ in the Neeman-Steel hierarchy need not be fully solid/sound, whereas in Woodin's hierarchy they are. In either case,
one must deal with failures of solidity for iterates; in the Neeman-Steel hierarchy, these failures already appear within the premice themselves, whereas in Woodin's hierarchy, they only appear after taking certain kinds of ultrapowers. This means that ultrapower maps arising from the Neeman-Steel hierarchy preserve standard parameters, whereas in Woodin's, they need not. But the failure of solidity and/or preservation of standard parameters only occurs in a certain specific manner, which can be tracked almost like in the short extender realm, so in the end it does not cause a major problem.
And while iteration maps in  the Neeman-Steel hierarchy preserve standard parameters,
this comes at the cost of complicating the basic definitions of premice, with rather explicit demands on the fine structure of the structures involved (that of ``projectum free spaces''). Moreover, the Neeman-Steel hierarchy involves failures of the initial segment condition (ISC),
in that, for example, there can be short extenders $E$ in the extender sequence, whose derived normal  measure is not in the sequence. It appears that Woodin's hierarchy also has such failures of the ISC (see the remarks further below). Both hierarchies also suffer from failures of condensation,  related to the failure of solidity and/or preservation of standard parameters.

The goals of this paper are as follows.
First, we will give a full account of the basic long extender theory at the level of $\kappa^+$-supercompactness. Our account will work with a new hierarchy, different from Woodin's and Neeman-Steel's, one whose features are much closer to the features in the short extender hierarchy. (It seems likely that this hierarchy should, however, be translatable into the Woodin or the Neeman-Steel hierarchy.) We will establish stronger, and more traditional, condensation and initial segment condition properties for this hierarchy. For example, normal measures of extenders on the sequence will themselves be on the sequence. We will also analyse Dodd-solidity, and our hierarchy will be designed with this built in. And although this is a coarse property (given a premouse $M$
with active extender $F$,
it demands that certain fragments of $F$
are already in $M$, but does not demand anything about how they are derived from the extender sequence of $M$), it fails for the Neeman-Steel hierarchy (and presumably for Woodin's also, based on remarks in \cite{nsp1}).
In this regard, we will prove the following results (all the premice below are those in the hierarchy introduced in this paper):

\begin{tm*}[\ref{tm:rho^M_D=0_implies_Dodd-absent-solid_and_universal}]
  Let $M$ be a $(0^-,\om_1,\om_1+1)^*$-iterable active short premouse
  with $\rho_{\mathrm{D}}^M=0$.
 Then $M$ is Dodd-absent-solid and Dodd-absent-universal.
\end{tm*}

\begin{tm*}[\ref{tm:Dodd-absent-soundness}]
 Let $M$ be a $(0^-,\om_1,\om_1+1)^*$-iterable active short $1$-sound
 premouse with $\rho_{\D}^M>0$.
 Then $M$ is Dodd-absent-sound.
\end{tm*}

Here the ``$0^-$'' in the notation is explained in \S\ref{sec:it_trees}, but it has essentially the same meaning as ``$0$'' would have in that role in the short extender context (it specifies a the soundness degree associated to $M$ when iterating $M$). The ordinal $\rho_{\D}^M$ is the \emph{Dodd-absent-projectum} of $M$, which is related to, but distinct from, the usual \emph{Dodd projectum}. And \emph{Dodd-absent-soundness} is the corresponding relative of the usual notion of \emph{Dodd-soundness}.

As a corollary of these two results
and Theorems \ref{tm:Dodd-absent-core_of_M_when_rho_D^M=0}
and \ref{tm:1-core_of_Dodd-abs-sound}, we have:
\begin{cor*}
Let $M$ be a $(0^-,\om_1,\om_1+1)^*$-iterable active short $1$-sound premouse. Then $M$ is Dodd-absent-sound.
\end{cor*}

Theorem \ref{tm:solidity} (for which we skip the formal statement of for the present) establishes, \emph{almost}, that mice  $M$ in our hierarchy are
 solid and universal (in the sense of the standard parameter). This needs the qualifier ``almost'', since it was shown by Woodin that full solidity does fail at this level,
 and this failure is a key new feature in this territory. But the results here are analogous to those of Woodin and Neeman-Steel.

Voellmer \cite{voellmer}
proved some condensation facts for the Neeman-Steel hierarchy. We will establish the analogues here, but also some more condensation facts without analogues there.
(Some of these are provably false for the Neeman-Steel hierarchy when
taken literally
when taken literally,
but should yield corresponding theorems when appropriately modified.) The basic such result is:

\begin{tm*}[\ref{tm:first_cond}]  Let $n<\om$ and let $M$ be an $(n,\om_1,\om_1+1)^*$-iterable  $n$-sound premouse with $\om<\rho_n^M$.
 Let $W$ be an $(n+1)$-sound premouse. Suppose that if $M$ is active short then $M,W$ are Dodd-absent-sound.
 Let $\pi:W\to M$ be an $n$-lifting $c$-preserving $\pvec_n$-preserving embedding
 such that either:
 \begin{enumerate}[label=(\alph*)]
  \item
$\eta=\rho_{n+1}^W\leq\crit(\pi)$, or
\item\label{item:cond_moving_below_proj_intro} $\crit(\pi)=\mu<\eta=\rho_{n+1}^W=\mu^{+W}=\mu^{+M}$ and $E\in\es^M$,
where $E$ is the short extender derived from $\pi$.
\end{enumerate}
 Then either:
 \begin{enumerate}
  \item $W\pins M$, or
  \item $M|\eta$ is active
  and $W\pins\Ult(M|\eta,F^{M|\eta})$, or
  \item\label{item:clause_3_cond_intro}
 $W=\core_{n+1}(M)$ and $\pi$ is the core map, or
  \item\label{item:clause_4_cond_intro}
 $M$ is stretched-$(n+1)$-solid,
  $\eta=\min(p_{n+1}^M)$, and letting $p=p_{n+1}^M\cut\{\eta\}$,
  \[W=\cHull_{n+1}^M(\eta\cup\{\pvec_n^M,p\}),\]
  or equivalently, $W=\Ult_n(\core_{n+1}(M),F^{M|\eta})$; moreover, $\pi$ is the uncollapse map.
 \end{enumerate}
\end{tm*}

Here a case of particular interest is clause \ref{item:cond_moving_below_proj_intro} above, in which $\crit(\pi)<\rho_{n+1}^W$. The standard condensation
results familiar from the short extender context assume $\rho_{n+1}^W\leq\crit(\pi)$. Clause \ref{item:clause_4_cond_intro}
of the conclusion relates to  failures of solidity; in that case $M$ is not $(n+1)$-solid.

Theorem \ref{tm:second_cond} is a variant of Theorem \ref{tm:first_cond}, in which the $(n+1)$st core is replaced by the Dodd-absent core.

Finally, we establish the following instance of the \emph{Mitchell-Steel initial segment condition} (MS-ISC,
Definition \ref{dfn:MS-ISC}),
which is just a direct translation of the usual initial segment condition of \cite{outline} to our context (for short extenders on the sequence),
assuming some soundness properties:

\begin{tm*}[\ref{tm:MS-ISC}]
 Let $M$ be a $(0,\om_1,\om_1+1)^*$-iterable
 Dodd-absent-sound $1$-sound
 active short premouse.
 Then $M$ satisfies the MS-ISC.
\end{tm*}

This theorem implies in particular that
the normal measure derived from $F^M$ is in the extender sequence $\es_+^M$ of $M$.
This contrasts heavily with the hierarchy of Neeman-Steel, in which this can fail.
The same appears to be the case for Woodin's hierarchy;
  see for example \cite[p.~2, following Question 1.1]{goldberg_linearity_Mitchell_order}, in particular the phrase ``\ldots one must \emph{prevent} certain normal measures from appearing on their extender sequences''.

The material in this paper is also  used
in  \cite{fullnorm_long}, which establishes full normalization
for the kinds of mice  considered here (cf.~\cite{fullnorm},
\cite{a_comparison_process_for_mouse_pairs}).

\subsection*{Acknowledgements}

Funded by the Deutsche Forschungsgemeinschaft (DFG, German Research Foundation) -- project number 445387776. Gef\"ordert durch die Deutsche Forschungsgemeinschaft (DFG) im Rahmen der Exzellenzstrategie des Bundes und der L\"ander EXC 2044--390685587, Mathematik M\"unster: Dynamik-Geometrie-Struktur. Editing
funded by the Austrian Science Fund (FWF) [10.55776/Y1498].\\

The author thanks the organizers of the \emph{From $\om$ to $\Omega$} conference, National University of Singapore, 2023,
and  the organizers of the \emph{2nd Irvine Conference on Inner Model Theory}, UC Irvine, 2023,
for the opportunity to present parts of this work.

\subsection{Notation}\label{subsec:notation}

Unexplained notation should be explained for example in
the introductions to \cite{iter_for_stacks}, \cite{premouse_inheriting}.
For the definition of $(z_{n+1},\zeta_{n+1})$ see \cite[\S2]{extmax}.  For a premouse $M$,
we write $F^M$ for the active extender of $M$, $F_J^M$ for the largest witness to the Jensen ISC for $F^M$, if it exists (otherwise $F_J^M=\emptyset$), $\nu(F^M)$ for the strict sup of generators of $F^M$, $\lh(F^M)=\OR^M$,
$\crit(F^M)$ for the critical point of $F^M$, $\spc(F^M)$ for the space of $F^M$ (which in our case will be either $\crit(F^M)$, if $F^M$ is short, and $\crit(F^M)^{+M}$, if $F^M$ is long), $\dom(F^M)$
for the domain of $F^M$
(which will be $M|\crit(F^M)^{+M}$, if $F^M$ is short, and $M|\crit(F^M)^{++M}$, if $F^M$ is long;
we sometimes also write $\dom(F^M)$ for the ordinal height of this structure, but hopefully no confusion will arise; $\lambda(F^M)=j(\crit(F^M))$ where $j=i^M_{F^M}:M\to\Ult(M,F^M)$
denotes the ultrapower map,
if $n<\om$ and $F^M$ is a $P$-extender with $\spc(F^M)<\rho_n^P$
then $\Ult_n(P,F^M)$
denotes the degree $n$ ultrapower of $P$ by $F^M$, and $i^{P,n}_{F^M}:P\to\Ult_n(P,F^M)$
the ultrapower map. And for $a\in[\nu(F^M)]^{<\om}$
and a function $f$ used in forming this ultrapower, $[a,f]^{P,n}_{F^M}$ denotes the object represented by $(a,f)$. If $n>0$ and $t$ is an $\rSigma_n$ Skolem term and $q\in P$ then $f^P_{q,t}:P^{<\om}\to P$ 
denotes the partial function defined by $x\mapsto t^P(q,x)$. If $t=0$
then $f^P_{q,t}$ just denotes $q$.
If $M$ is a premouse then $\es^M$
denotes the internal extender sequence of $M$, $F^M$ the active extender, and $\es_+^M=\es^M\conc\left<F^M\right>$.
We write
$M^{\passive}$ for the passivization of $M$; i.e. the premouse $N$ with
$\OR^N=\OR^M$ and $\es^N=\es^M$
but $F^N=\emptyset$.
For $\alpha\leq\OR^M$
we write $M|\alpha$
for the initial segment of $M$ of height $\alpha$, including the active extender $\es_\alpha^M$,
and $M||\alpha=(M|\alpha)^{\passive}$. We write $M\ins N$ if $M$ is an initial segment of $N$; that is, $M=N|\alpha$ for some $\alpha$.
We write $M\pins N$ iff $M\ins N$ but $M\neq N$.

\section{Mice for $\kappa^+$-supercompactness}\label{sec:long_mice}
We begin by describing the hierarchy of mice $M$ we will use, and their fine structure. The extenders in the extender sequence $\es_+^M$ of $M$ can be either short or long; the short extenders $E$ consist of measures over their critical point $\crit(E)$, whereas the long extenders have measures over $\crit(E)^{+M}$. We will index short extenders $E$ and long extenders $E$ with a largest generator at $\lambda^{+\Ult(M,E)}$, whereas long extenders with no largest generator (which have sup of generators $\lambda^{+\Ult(M,E)}$)
will be indexed at $\lambda^{++\Ult(M,E)}$.
Forms of the \emph{initial segment condition} will be crucial. As usual, \emph{potential premice}
will have the basic structure,
and \emph{premice} will be potential premice with refined fine structure.
As part of this refinement of fine structure,
we will want to make demands
on the Dodd structure of short extenders indexed on proper segments of $M$.

\subsection{Premice}
\begin{dfn}
 A \emph{potential premouse \tu{(}pot-pm\tu{)}}
 is a structure $M=(\J_\alpha^{\es},\es,F)$
 where $\es$ is a sequence of (partial) extenders,
 and if $F\neq\emptyset$ then $F$ is an $M$-extender such that
 $F$ coheres $\es$ and letting $\kappa=\crit(F)$
 and $\lambda=\lambda(F)=i^M_F(\kappa)$ and $\nu=\nu(F)$ (the strict sup of the generators of $F$) and $U=\Ult(M,F)$
 and $j:M\to U$ the ultrapower map, we have:
 \begin{enumerate}
  \item If $F$ is short (so $\dom(F)=M|\kappa^{+M}$
  and $\spc(F)=\kappa$)
  then $\lambda$ is the largest cardinal of $M$, $\OR^M=\lambda^{+U}$, and $F$ satisfies the Jensen ISC; that is, for all $\bar{\lambda}\in(\kappa,\lambda)$, if $G=F\rest\bar{\lambda}$
  is (a) \emph{whole  \tu{(}proper fragment of $F$\tu{)}} , meaning that $i^M_G(\kappa)=\bar{\lambda}$,
  then $G\in\es^M$.\footnote{Often the \emph{Jensen ISC} is formulated as only requiring that $G\in M$, instead of $G\in\es^M$. When we write ``$G\in\es^M$'',
  literally it should be the \emph{trivial completion} of $G$ which is in $\es^M$;
  that is, $G'\in\es^M$ where $G'$
  is the extender equivalent to $G$
  satisfying the requirements we are presently specifying.}
  \item If $F$ is long then:
  \begin{enumerate}
   \item 
$\dom(F)=M|\kappa^{++M}$
  (so $\spc(F)=\kappa^{+M}$),
  \item either:
  \begin{enumerate}
  \item $\nu=\nu^-+1$ for some $\nu^-$,
  $\OR^M=\lambda^{+U}$, 
   $\lambda<\nu^-<\OR^M$,
   and $F^M$ satisfies the $\{\nu^-\}$\emph{-Jensen ISC};
  that is, for each $\bar{\lambda}\in(\kappa,\lambda)$, letting
  \[ G=F^M\rest(\bar{\lambda}\cup\{\nu^-\}),\]
  if $G$ is (a) \emph{$\{\nu^-\}$-whole \tu{(}proper fragment of $F$\tu{)}}, meaning that $i^M_G(\kappa)=\bar{\lambda}$, then (the trivial completion of) $G$ is in $\es^M$, or
   \item $\nu$ is a limit ordinal (it follows that $\nu=\lambda^{+U}$)
   and $\OR^M=\lambda^{++U}$,
  \end{enumerate}
  \item the \emph{long ISC} holds: for each $\bar{\nu}\in[\lambda,\nu)$,
  the trivial completion of $F\rest\bar{\nu}$
  is in $\es^M$,
  \end{enumerate}
  \item $F$ is coded in the natural amenable fashion.\qedhere
 \end{enumerate}
\end{dfn}

\begin{dfn}
 The \emph{language $\mathscr{L}$ of pot-pms} has symbols $\dot{\es}$, $\dot{F}$, $\dot{F_J}$, $\dot{F_{\downarrow}}$.
 Let $M=(\J_\alpha^\es,\es,F)$ be a pot-pm. Then we define the \emph{$0$-core} of $M$ (which we will mostly identify with $M$) as the $\mathscr{L}$-structure $\core_0(M)=(\J_\alpha^\es,\es,F,F_J,F_{\downarrow})$
 (interpreting symbol $\dot{x}$ with $x$),
 where:
 \begin{enumerate}
  \item If $M$ is passive (i.e. $F=\emptyset$) then $F_J=F_\downarrow=\emptyset$,
  \item If $M$ is active (i.e. $F\neq\emptyset$) and either
  \begin{enumerate}[label=--]
   \item 
short (i.e.~$F$ is short), or
\item long (i.e.~$F$ is long) with a largest generator $\nu^-$,
\end{enumerate}then either:
  \begin{enumerate}
  \item (\emph{type A}) there is no ($\{\nu^-\}$-)whole proper fragment of $F$,
  and $F_J=\emptyset$, or
  \item (\emph{type B}) there is a largest ($\{\nu^-\}$-)whole proper fragment $G$ of $F$, and $F_J=G$, or
  \item (\emph{type C}) there is a ($\{\nu^-\}$-)whole proper fragment of $F$, but no largest such,
  and $F_J=\emptyset$ (note that in this case, by the ($\{\nu^-\}$-)Jensen ISC, either
  \begin{enumerate}[label=--]\item $F$ is short and for every $\xi<\lambda(F)$, there is $\bar{\lambda}\in(\xi,\lambda(F))$ such that $F\rest\bar{\lambda}$ is whole, or
  \item $F$ is long and for every $\xi<\lambda(F)$, there is $\bar{\lambda}\in(\xi,\lambda(F))$ such that  $F\rest\{\bar\lambda\}\cup\{\nu^-\}$ is $\{\nu^-\}$-whole).
  \end{enumerate}
  \end{enumerate}
  \item If $M$ is active long (i.e. $F\neq\emptyset$ is long) and $\nu(F)$ is a successor,
  then $F_\downarrow=F\rest\nu^-$;
  in all other cases, $F_\downarrow=\emptyset$.\qedhere
 \end{enumerate}
\end{dfn}

\begin{dfn}
 Given a pot-pm $M$ and a (limit) ordinal $\beta\leq\OR^M$, $M|\beta$ denotes $(\J_\beta^{\es},\es\rest\beta,\es^M_\beta)$,
 and $M||\beta$ denotes $(\J_\beta^{\es},\es\rest\beta,\emptyset)$. We also write $M^{\passive}=M||\OR^M$
 for the \emph{passivization} of $M$.
\end{dfn}

\begin{rem}
 Let $N$ be an active short pot-pm and $\kappa=\crit(F^N)$.
As usual, the \emph{Dodd projectum}
$\tau^N$ of $N$ is the least ordinal $\tau\geq\kappa^{+N}$ such that
$F^N$ is generated by $\tau\cup t$
with some $t\in[\OR]^{<\om}$.
(That is, for some such $t$,
every $x\in N$ is of the form $i^N_{F^N}(f)(a,t)$ for some $f\in N|\kappa^{+N}$ and $a\in[\tau]^{<\om}$.)
And the \emph{Dodd parameter} $t^N$ of $N$
is the least $t\in[\OR]^{<\om}$
such that $F^N$ is generated by $\tau^N\cup t$.\footnote{Here and elsewhere, given $a,b\in[\OR]^{<\om}$,
we set $a<b$ iff $a\neq b$ and $\max(a\Delta b)\in b$.} 

This is the traditional way of defining the Dodd projectum and parameter.
However, it is not entirely analogous
to the $(n+1)$st projectum $\rho_{n+1}$ and standard parameter $p_{n+1}$.
 For our purposes it turns out more convenient to focus on a better analogue, 
 specified in terms of the ``least missing fragment'' of $F^N$, instead of the ``least generating fragment''.
 
 Also, the requirement that $\tau^M\geq\kappa^{+M}$ is somewhat artificial. If $\tau^M=\kappa^{+M}$,
 although $F^M$ is generated by $t^M\cup\kappa^{+M}$, and hence by $t^M\cup\{\kappa\}$, it need not be that $F^M$ is generated by $t^M$ by itself.
 See \cite{fsfni_v5} for more on this.
 Here, it will be important to distinguish between these cases, which the following definition does.
\end{rem}

\begin{dfn}\label{dfn:rho_D,p_D}
 Let $N$ be an active short premouse.
 We define the \emph{Dodd-absent projectum} $\rho_{\D}^N$
 as the least $\rho\in\OR^N$ such that  \[F^N\rest(\rho\cup p)\notin N \]
 for some $p\in[\OR^N]^{<\om}$.
 We define the \emph{Dodd-absent parameter} $p_{\D}^N$
 as the least $p\in[\OR^N]^{<\om}$ such that
 \[F^N\rest(\rho_{\D}^N\cup p)\notin N.\]
 
 We say that $N$ is \emph{Dodd-absent solid}
 iff
 \[ F^N\rest(\alpha\cup (p_{\D}^N\cut(\alpha+1)))\in N\]
  for each $\alpha\in p_{\D}^N$.
 
 We say that $N$ is \emph{Dodd-absent sound} if
 $N$ is Dodd-absent solid and
  $F^N$ is generated by $\rho_{\D}^N\cup p_{\D}^N$.
  
  By analogy with $(z_{n+1},\zeta_{n+1})$
  (see \cite[\S2]{extmax}),
  we also define $(s^N,\sigma^N)$
  as the lexicographically least $(s,\sigma)\in[\OR]^{<\om}\cross\OR$
  such that $F^N\rest(s\cup\zeta)\notin N$.
\end{dfn}

\begin{dfn}
 Let $M$ be a pot-pm. We say that $M$ is \emph{$0$-sound} and define $\rho_{0}^M=\OR^M$ and $p_{0}^M=\emptyset$.
 We define $\rSigma_0^M$ and $\rSigma_1^M$ as $\Sigma_0^M$ and $\Sigma_1^M$
 in the language $\mathscr{L}$ of pot-pms
 (that is, over the structure $\core_0(M)$).

 We define $\rho_1^M,p_1^M$,
 \emph{$1$-universality}, \emph{$1$-solidity},
 \emph{$1$-soundness} as usual, using the language $\mathscr{L}$ and $\Sigma_1$-definability over $\core_0(M)$. (We define \emph{$1$-solidity} conventionally,
 like for short extender mice. We define \emph{$1$-soundness} as the conjunction of
  $1$-solidity and the assertion that $M=\Hull_1^M(\rho_1^M\cup\{p_1^M\})$.)
 
 Let $0<n<\om$, and suppose that $M$ is $n$-sound
 and $\om<\rho_n^M$.
 Then we define $\rSigma_{n+1}^M,\rho_{n+1}^M,p_{n+1}^M$, \emph{$(n+1)$-universality},
 \emph{$(n+1)$-solidity}, \emph{$(n+1)$-soundness} as usual. (In particular, we define the predicate $T_n^M$, by setting $T_n^M(\alpha,q,t)$
 iff $\alpha<\rho_n^M$ and $t=\Th_{\rSigma_n}^M(\alpha\cup\{q\})$. Then the $\rSigma_{n+1}^M$
 assertions $\varphi(\vec{x})$ are those of the form ``there are $\alpha,q,t$ such that $T_n(\alpha,q,t)$
 and $\psi(\alpha,q,t,\vec{x})$'', where $\psi$ is $\rSigma_n$. As in \cite[\S5]{V=HODX_pub}, we do not use any analogue of the parameters $u_n$ of \cite{fsit}. So for example, $p_{n+1}^M$
 is the lex-least $p\in[\OR^M]^{<\om}$
 such that $\Th_{\rSigma_{n+1}}^M(\rho_{n+1}^M\cup\{p,p_1^M,\ldots,p_n^M\})\notin M$.

 And \emph{$\om$-soundness} is the conjunction of $n$-soundness, for all $n<\om$.
\end{dfn}

\begin{dfn}
 Let $H,M$ be $n$-sound pot-pms.
 We say that $\pi:H\to M$ is an \emph{$n$-lifting}
 embedding iff $\pi$ is $\rSigma_0$-elementary
 and if $n>0$ then $\pi``T_n^H\sub T_n^M$.
 We say that $\pi$ is \emph{$n$-nice}
 iff $\pi$ is $n$-lifting and $\pi(\pvec_n^H)=\pvec_n^M$.
\end{dfn}

\begin{dfn}\label{dfn:pm}
 A \emph{premouse} is a pot-pm $M$
 such that for all $\beta<\OR^M$, we have:
 \begin{enumerate}
  \item $M|\beta$ is $\om$-sound,
  \item if $M|\beta$ is active short then $F^{M|\beta}$ is Dodd-absent-sound,
  \item For all $n<\om$
  and all $H,\pi\in M$
  such that $H$ is an $(n+1)$-sound
  pot-pm and $\pi:H\to M|\beta$
  is $n$-nice with $\rho=\rho_{n+1}^H\leq\crit(\pi)$,
  either:
  \begin{enumerate}[label=--]
  \item $H\ins M|\beta$, or
  \item 
   $M|\rho$ is active and $H\pins\Ult(M|\rho,F^{M|\rho})$.
   \end{enumerate}
   \item If $M|\beta$ is active short
   and $\kappa=\crit(F^{M|\beta})$
   then for all $H,\pi\in M$ such that
   $H$ is an active short Dodd-absent-sound pot-pm
   and $\kappa=\crit(F^H)$
   and $H|\kappa^{+H}=M|\kappa^{+(M|\beta)}$
   and $\pi:H\to M|\beta$
   is $0$-lifting with $\rho_{\mathrm{D}}^H\leq\crit(\pi)$
   and $\kappa^{+H}<\crit(\pi)$,
   either:
     \begin{enumerate}[label=--]
  \item $H\ins M|\beta$, or
  \item 
   $M|\rho$ is active and $H\pins\Ult(M|\rho,F^{M|\rho})$.\qedhere
   \end{enumerate}
  \end{enumerate}
\end{dfn}

\begin{rem}
Because of the traditional form of soundness
we require (in that it is not modified in the manner (not explicitly) defined by Neeman/Steel in \cite{nsp1}, \cite{nsp1fs}), and that we do not demand projectum free spaces,
the hierarchy clearly differs from \cite{nsp1}.
Although Woodin has similar soundness requirements,
his hierarchy also differs, because we make the Dodd-absent-soundness and condensation requirements for proper segments.\end{rem}

\subsection{Preservation of fine structure under ultrapowers}

We now want to work through preservation of fine structural features by ultrapower embeddings.

\begin{dfn}
 Let $M$ be a premouse and $E$ be an $M$-extender with $\crit(E)=\kappa<\OR^M$
and such that
 either $\spc(E)=\kappa$ (so $E$ is short) or $\spc(E)=\kappa^{+M}<\OR^M$ (so $E$ is long).
 We say that $E$ is \emph{weakly amenable \tu{(}to $M$\tu{)}} iff
 either:
  \begin{enumerate}[label=--]
  \item $E$ is short and $E_a\cap(N|\gamma)\in N$ for each $\gamma<\kappa^{+N}$ and  $a\in[\nu(E)]^{<\om}$, or
  \item $E$ is long and $E_a\cap(N|\gamma)\in N$  for each $\gamma<\kappa^{++N}$
  and  $a\in[\nu(E)]^{<\om}$.\qedhere
  \end{enumerate}
\end{dfn}

\begin{lem}
 Let $M$ be a premouse. Then $F^M$ is weakly amenable.
\end{lem}
\begin{proof}
 If $E=F^M$ is short, this is as usual.
 Suppose it is long. Let $\gamma,a$ be given,
 and let $\vec{X}=\left<X_\beta\right>_{\beta<\kappa^{+M}}$
 enumerate $\pow([\kappa^{+M}]^{|a|})\cap (M|\gamma)$ with $\vec{X}\in M$.
 Let $j:M\to U=\Ult(M,F^M)$ be the ultrapower map. Then letting $k=j\rest\kappa^{+M}$,
 we have $k\in U$ by the ISC,
 since $k$ is determined by the short part of $F^M$. But for $\beta<\kappa^{+M}$,
 we have $X_\beta\in E_a$
 iff $a\in j(X_\beta)$ iff $a\in j(\vec{X})_{k(\beta)}$, so from the parameters $j(\vec{X})$ and $k$,
 we get $E_a\cap (M|\gamma)\in U$,
 but $M|\kappa^{++M}=U|\kappa^{++U}$,
 and so in fact $E_a\cap (M|\gamma)\in U|\kappa^{++U}\sub M$.
\end{proof}

\begin{dfn}
 Let $N$ be a premouse and $E=F^M$
 for some active premouse $M$.

 We say that $E$ is an \emph{$N$-extender} iff letting $\kappa=\crit(E)$, either:
 \begin{enumerate}[label=--]
  \item $E$ is short and $N||\kappa^{+N}=M|\kappa^{+M}$, or
  \item $E$ is long and $N||\kappa^{++N}=M|\kappa^{++M}$. \end{enumerate}

  We say that $E$ is \emph{close to $N$} iff
 $E$ is an $N$-extender and
each component measure $E_a$ of $E$
  is $\bfrSigma_1^N$-definable.
\end{dfn}

\begin{rem}
 The definitions above might abuse terminology,
 since it does not seem immediately obvious
 that $E$ determines $M$,
 and in particular, that $E$ determines $M|\kappa$. (Given $M|\kappa$ and $E$,
 $M$ is determined.) But in context,
 we will know which $M$ we are dealing with, so this will not be a problem.
\end{rem}

\begin{dfn}
 We say that a tuple $(N,P,n)$ is \emph{proto-anomalous}
 or is \emph{a proto-anomaly} iff
  $N,P$ are premice, $n=0$,
 $N$ is active short, $P$ is active long,
 $\crit(F^N)=\crit(F^P)$ and 
 $F^P$ is an $N$-extender.
\end{dfn}

\begin{dfn}
Let $n<\om$ and $N$ be an $n$-sound premouse.
 As in \cite{extmax},
 $(z_{n+1}^N,\zeta_{n+1}^N)$
 denotes the lex-least pair
 $(z,\zeta)\in[\OR]^{<\om}\cross\OR$ such that
 \[ \Th_{\rSigma_{n+1}}^N(\zeta\cup\{z,\pvec_n^N\})\notin N.\qedhere\]
\end{dfn}

\begin{rem}
 Here and elsewhere, the ordering on $[\OR]^{<\om}\cross\OR$ is that where given $(a,\alpha),(b,\beta)\in[\OR]^{<\om}\cross\OR$,
 we set $(a,\alpha)<(b,\beta)$ iff letting $a':(n+1)\to\OR$ be the function such that $a'\rest n$ is strictly decreasing and $a'`` n=a$ and $a'(n)=\alpha$, and $b'$ likewise, then $a'<b'$ (lexicographically).
\end{rem}

\begin{lem}\label{lem:param_proj_pres}
Let $n<\om$ and $N$ be an $n$-sound  premouse.
Let $P$ be an active premouse.
Suppose $F^P$ is an $N$-extender and $\dom(F^P)\leq\rho_n^N$.
Suppose $(N,P,n)$ is not proto-anomalous.
Let $U=\Ult_n(N,F^P)$ and $j:N\to U$ be the ultrapower map. Let $\kappa=\crit(j)$.
Let $\eta=\sup j``\kappa^{+N}$.
Suppose $U$ is wellfounded.
Then:
\begin{enumerate}
 \item\label{item:z,zeta<=p,rho} $(z_{n+1}^N,\zeta_{n+1}^N)\leq(p_{n+1}^N,\rho_{n+1}^N)$.
 \item\label{item:solidity_charac_z,zeta} $N$ is $(n+1)$-solid $\iff z_{n+1}^N=p_{n+1}^N\iff\zeta_{n+1}^N=\rho_{n+1}^N$.
 \item\label{item:U_is_pm}  $U$ is an $n$-sound premouse
 and $j$ is an $n$-embedding,
 \item\label{item:z,zeta_pres} $(z_{n+1}^U,\zeta_{n+1}^U)=(j(z_{n+1}^N),\sup j``\zeta_{n+1}^U)$.
 \item\label{item:when_dom_<=rho} If $\dom(F^P)\leq\rho_{n+1}^N$
 then:
 \begin{enumerate}[label=\tu{(}\alph*\tu{)}]\item $\rho_{n+1}^U=\sup j``\rho_{n+1}^N$,
 \item if $N$ is $(n+1)$-solid then $j(p_{n+1}^N)=p_{n+1}^U$.
 \end{enumerate}
  \item\label{item:standard_pres} Suppose $\rho_{n+1}^N\leq\kappa$.
 Then:
 \begin{enumerate}[label=\tu{(}\alph*\tu{)}]
 \item $\rho_{n+1}^U\leq\rho_{n+1}^N$,
 \item if $F^P$ is close to $N$ then $\rho_{n+1}^U=\rho_{n+1}^N$,
 \item if $\rho_{n+1}^N=\rho_{n+1}^U$ and $N$ is $(n+1)$-solid then:
 \begin{enumerate}
 \item $p_{n+1}^U=j(p_{n+1}^N)$,
 \item $U$ is $(n+1)$-solid.
 \end{enumerate}
 \end{enumerate}
 \item\label{item:long_proj} Suppose $\rho_{n+1}^N=\kappa^{+N}<\kappa^{++N}=\dom(F^P)$ \tu{(}so $F^P$ is long\tu{)}.
 Then:
 \begin{enumerate}[label=\tu{(}\alph*\tu{)}]\item\label{item:rho^U<=rho^N} $\rho_{n+1}^U\leq\rho_{n+1}^N$,
 \item\label{item:closeness_preserves_rho} if $F^P$ is close to $N$ then $\rho_{n+1}^U=\rho_{n+1}^N$,
 \item \label{item:if_rhos_match_and_solid}if $\rho_{n+1}^N=\rho_{n+1}^U$ and $N$ is $(n+1)$-solid then:
 \begin{enumerate}
 \item $p_{n+1}^U=j(p_{n+1}^N)\cup\{\eta\}$,
 \item $U$ is not $(n+1)$-solid,
 \item letting $p=j(p_{n+1}^N)$, we have:
 \begin{enumerate}
 \item $(U,p)$ is $(n+1)$-solid,
 \item\label{item:weak_solidity} $(U,p\cup\{\alpha\})$ is $(n+1)$-solid for each $\alpha<\eta$.
 \end{enumerate}
 \end{enumerate}
 \end{enumerate}

 \item\label{item:shift_defsss} Suppose $F^P$ is close to $N$. Let $X\sub\spc(F^P)$ \tu{(}recall $\spc(F^P)=\kappa$ if $F^P$ is short, and $\spc(F^P)=\kappa^{+P}$ if $F^P$ is long\tu{)}.
 Then 
 \[ X\text{ is }\bfrSigma_{n+1}^N\text{-definable}\iff X\text{ is }\bfrSigma_{n+1}^U\text{-definable}.\]
\end{enumerate}
\end{lem}
\begin{proof}
Parts \ref{item:z,zeta<=p,rho},
\ref{item:solidity_charac_z,zeta}, \ref{item:U_is_pm}, \ref{item:z,zeta_pres},
\ref{item:when_dom_<=rho},
 \ref{item:standard_pres}
are straightforward or like in the short extender hierarchy (see \cite[\S2]{extmax} for details on $(z_{n+1},\zeta_{n+1})$,
and in particular part \ref{item:z,zeta_pres}).

Part \ref{item:long_proj}\ref{item:rho^U<=rho^N}: Let $X\sub\rho_{n+1}^N=\kappa^{+N}$ and $x\in N$ and $\varphi$ be $\rSigma_{n+1}$, with
\[ \alpha\in X\iff N\sats\varphi(x,\alpha).\]
Then
\[ \alpha\in X\iff U\sats\varphi(j(x),j(\alpha)).\]
But $k=j\rest\kappa^{+N}$ is (the map of) 
the short part of $F^P$, so $k\in U$.
So the above equivalence gives an $\rSigma_{n+1}^U(\{k,j(x)\})$-definition of $X$. Since also $N|\kappa^{++N}=U|\kappa^{++U}$, it follows that $\rho_{n+1}^U\leq\kappa^{+N}=\rho_{n+1}^N$.

Part \ref{item:long_proj}\ref{item:closeness_preserves_rho}: It remains to see that $\rho_{n+1}^U\geq\rho_{n+1}^N$.
This follows from part \ref{item:shift_defsss} below.

Part \ref{item:long_proj}\ref{item:if_rhos_match_and_solid}: This follows from earlier parts and their proofs, especially the preservation of $(z_{n+1},\zeta_{n+1})$ (as in part \ref{item:z,zeta_pres}),
and the proof of part \ref{item:long_proj}\ref{item:rho^U<=rho^N}, and the fact that the short part $E$ of $F^U$ is just $F^{U|\eta}$,
and for part \ref{item:weak_solidity},
using that for each $\alpha<\rho_{n+1}^N$,
if $t=\Th_{\rSigma_{n+1}}^N(\alpha\cup\{x\})$
then we can recover
$t'=\Th_{\rSigma_{n+1}}^N(j(\alpha)\cup\{j(x)\})$
from $j(t)$, and therefore $t'\in U$,
as in the usual proof of the preservation of the standard parameter under solidity from the short extender context.

Part \ref{item:shift_defsss}:
If $E=F^P$ is short this is as usual (and anyway is a simplification of the long case), so consider
the case that
 $E$ is long.
 Let $\kappa=\crit(E)$. For simplicity assume $n=0$; the general case is very analogous.
 
 We have $\spc(E)=\kappa^{+N}=\kappa^{+U}$.
 We first show that if $X\sub\kappa^{+U}$
 is $\bfSigma_1^U$ then $X$ is $\bfSigma_1^N$.
  Fix an $\rSigma_0$ formula $\varphi$
 and $x\in U$
 and consider the set\[ X=\Big\{\xi<\kappa^{+U}\Bigm|
 U\sats\exists y\ \varphi(y,x,\xi)\Big\}.\]
 Fix $f_x\in N$ and $b\in[\nu(E)]^{<\om}$ such that $f_x:(N|\kappa^{+N})\to N$ 
 and 
 $x=[b,f_x]^{N,0}_{E}$. We may assume $\kappa=\min(b)$.
 For $\xi<\kappa^{+N}$
 let $f_\xi:\kappa\to\kappa$
 be the $\xi$th canonical function.
 So $i_{E}(f_\xi)(\kappa)=\xi$
 for each such $\xi$.
 Then $\xi\in X$ iff
 there is $A\in E_b$
 and a function $f\in N$
 such that for all $u\in A$,
 \[ N\sats\varphi(f(u),f_x(u),f_\xi(\min(u))).\]
 But we know that $E$ is close to $N$,
 so this shows that $X$ is $\bfrSigma_1^{N}$, as desired.
 
 Now consider the converse.
 Let $\varphi$ be an $\rSigma_0$ formula,  $x\in N$ and consider the set
 \[ X=\Big\{\xi<\kappa^{+N}\Bigm|N\sats\exists y\ \varphi(y,x,\xi)\Big\}.\]
 Then 
  $\xi\in X$ iff
 \[ U\sats\exists y\ \varphi(y,j(x),j(\xi)),\]
 where $j:N\to U$
 is the (degree $0$) ultrapower map via $E$. But since $E$ is long,
 the short segment $E\rest\lambda(E)$ is in $U$, and hence so is the map $k=j\rest\kappa^{+N}$.
 So from the parameter $k$,
 $U$ can compute $\xi\mapsto j(\xi)$, so the above equivalence gives a definition of $X$ over $U$
 which is $\Sigma_1^{U}(\{j(x),k\})$, which suffices. 
\end{proof}

\begin{dfn}
 Let $N$ be an $n$-sound premouse.
 We say that $N$ is \emph{$(n+1)$-stretched-solid}
 iff $p_{n+1}^N\neq\emptyset$ and letting $\eta=\min(p_{n+1}^N)$ and $p=p_{n+1}^N\cut\{\eta\}$, then we have:
 \begin{enumerate}
  \item $(N,p)$
 is $(n+1)$-solid, 
 \item for each $\alpha<\eta$,
 $(N,p\cup\{\alpha\})$ is $(n+1)$-solid,
 \item $N|\eta$ is active with a short extender $E$
 which is $N$-total,
 \item $\rho_{n+1}^N=\kappa^{+N}$ where $\kappa=\crit(E)$,
 \item there is an $(n+1)$-sound premouse $H$ and an $n$-embedding $\pi:H\to N$
 such that $\kappa=\crit(\pi)$
 (where $\kappa$ is as above)
 and $\rho_{n+1}^H=\kappa^{+N}=\kappa^{+H}$
 and $\pi(p_{n+1}^H)=p=p_{n+1}^N\cut\{\eta\}$,
 \item $E$ is the short extender derived from $\pi$,
 \item $H||\kappa^{++H}=N||\kappa^{++N}$.\qedhere
 \end{enumerate}
\end{dfn}

 Note that the witnessing $H,\pi$
 are uniquely determined by $N$, if they exist
 (indeed, 
 \[ H=\cHull_{n+1}^N\Big(\rg(E)\cup\Big\{\pvec_{n}^N,p\Big\}\Big)\]
 and $\pi$ is the uncollapse map).

\subsection{Protomice and their conversions}
\begin{dfn}\label{dfn:Ult_n}
 Let $N$ be an $n$-sound premouse.
 Let $M$ be an active premouse.
 Let $\xi=\OR^{\dom(F^M)}$ (so if $F^M$ is short then $\xi=\kappa^{+M}$ and if $F^M$ is long then $\xi=\kappa^{++M}$, where $\kappa=\crit(F^M)$). Suppose  $\dom(F^M)=N||\xi$,
 $\xi\leq\rho_n^N$, and either $\xi=\OR^N$ or $\xi$ is an $N$-cardinal. Then we define $\Ult_n(N,F^M)$
 in the conventional manner, except in case $n=0$, $F^M$ is long, $N$ is active short and $\crit(F^N)=\kappa=\crit(F^M)$. In the latter case,
 first define $U^*=\Ult_0'(N,F^M)$
 in the conventional manner. Then $U^*$ is  a proper protomouse. Let $E$ be the short part of $F^M$. By the long-ISC for $M$ and coherence, $E\in\es^{U^*}$. Now define
 \[\Ult_0(N,F^M)=((U^*)^{\passive},\es^{U^*},F^{U^*}\com E).\qedhere\]
\end{dfn}

\begin{lem}
 Continue with the setup of Definition \ref{dfn:Ult_n}. Let $U=\Ult_0(N,F^M)$.
 Then $U$ is an active short premouse.
\end{lem}
\begin{proof}
 This is mostly clear, except possibly for the Jensen ISC. But letting $F^*=F^{U^*}$
 and $j:N\to U$ be the ultrapower map, which is continuous at $\lambda(F^N)$, just note that either:
 \begin{enumerate}
  \item $F^N$ has no whole proper fragment
  and $E\in\es^U=\es^{U^*}$ is the largest whole proper fragment of $F^*\com E$, or
  \item $F^N$ has unbounded whole proper fragments, as does $F^*\com E$, and in fact, for each whole proper fragment $G$ of $F^N$,
  $j(G)\com E$ is a whole proper fragment of $F^*\com E$, and $i^{U^*}_{j(G)}(E)\in\es^{U^*}$, or
  \item $F^N$ has a largest whole proper fragment $G$, and then much as in the previous case,
  $j(G)\com E$ is the largest whole proper fragment of $F^*\com E$, and $i^{U^*}_{j(G)}(E)\in\es^{U^*}$.\qedhere
 \end{enumerate}
\end{proof}

So we avoid the protomouse $U^*$ by replacing it with $U$. There seem to be some subtleties in understanding the fine structural relationship between $U$ and $U^*$, however, particularly
because the language $\mathscr{L}$ involves the constant $\dot{F}_J$, which can have different interpretations in $U$ and $U^*$, and they don't seem to be easily inter-computable. Consider
in particular the case that $N$ has a largest whole proper segment, so this is $G=F^N_J$.
Then $j(G)$ is compatible with $F^{U^*}$,
and there is no larger such whole extender in $U$.
We get that $j(G)\com E$ is the largest whole proper segment of $F^{U}$. The fine structure
of $U^*$, taking it to be determined by $N,j$,
uses $j(G)$ as a constant. It seems non-trivial to convert in either direction between $j(G)$ and $j(G)\com E$. If we have $E$ available as a parameter, we can of course compute $j(G)\com E$ from $j(G)$ (and $E$). But the converse,
computing $j(G)$ from $j(G)\com E$ (and $E$)
doesn't seem obvious; note that $\lambda(j(G))<\lh(j(G)\com E)<\lh(j(G))$. Since $\Sigma_1^{U^*}$ and $\Sigma_1^{U}$ are defined relative these parameters, this seems to make the determination of the relationship between $p_1^{U^*}$ and $p_1^U$ somewhat unclear. The
$1$-soundness or otherwise of $U$ of course also depends on these things. In the end we will partially (and enough for our overall purposes) resolve such issues by considering Dodd-absent-solidity and Dodd-absent-soundness.

\subsection{Preservation of Dodd structure under ultrapowers}

In this section we establish results analogous to those of the previous one, but for Dodd structure replacing standard fine structure.
Recall that the Dodd-absent projectum $\rho_{\mathrm{D}}$ and parameter $p_{\mathrm{D}}$, and associated notions, were introduced in \ref{dfn:rho_D,p_D}.
Although the (standard) Dodd parameter and projectum $\tau,t$ can in general differ from $\rho_{\mathrm{D}},p_{\mathrm{D}}$,
in the presence of either soundness notion, everything coincides. A straightforward calculation gives:
\begin{lem}\label{lem:rho_D,p_D_basic_facts}
 Let $N$ be an active short premouse
 and $\kappa=\crit(F^N)$.
 Then:
 \begin{enumerate}
 \item\label{item:rho_1=rho_D_mod_mu^+} $\rho_{\D}^N$ is an $N$-cardinal
 and $\max(\rho_1^N,\kappa^{+N})=\max(\rho_{\D}^N,\kappa^{+N})$,
  \item $\max(\kappa^{+N},\rho_{\D}^N)\leq\tau^N$,
  \item if $\max(\kappa^{+N},\rho_{\D}^N)=\tau^N$
  then $p_{\D}^N\cut\{\kappa\}\leq t^N$,
  \item if $\rho_{\D}^N>\kappa^{+N}$
  or $\kappa\in p_{\mathrm{D}}^N$ then
  \[ N\text{ is
 Dodd-absent-sound iff }N\text{ is Dodd-sound,}\]
  \item if $\rho_{\D}^N=0$ and $\kappa\notin p_{\mathrm{D}}^N$ then
  the following are equivalent:
  \begin{enumerate}[label=\tu{(}\alph*\tu{)}]
  \item $N$ is Dodd-absent-sound
  \item $N$ is Dodd-sound
  and $\{\kappa\}$ is generated by $p_{\mathrm{D}}^N$.
  \end{enumerate}
 
  \item if $N$ is Dodd-absent-sound then
  $t^N=p_{\D}^N\cut\{\kappa\}$ and $\tau^N=\max(\rho_{\D}^N,\kappa^{+N})$.
  \item $(s^N,\sigma^N)\leq(p_{\D}^N,\rho_{\D}^N)$,
 \item $N$ is Dodd-absent-solid $\iff$ $s^N=p_{\D}^N$ $\iff$ $\sigma^N=\rho_{\D}^N$.
 \end{enumerate}
\end{lem}

\begin{lem}\label{lem:Dodd_param_proj_pres}
 Let $N$ be a an active short premouse. Let $P$ be an active  premouse.
 Suppose $F^P$ is an $N$-extender and $(N,P,0)$ is not proto-anomalous. 
 Let $U=\Ult_0(N,F^P)$ and $j:N\to U$ be the ultrapower map. Let $\kappa=\crit(j)$. Let $\eta=\sup j``\kappa^{+N}$. Suppose $U$ is wellfounded. Let $\mu=\crit(F^N)$. Then:
 \begin{enumerate}
 
 \item $U$ is an active short premouse and $j$ is a $0$-embedding.
 \item\label{item:pres_s,sigma}$(s^U,\sigma^U)=(j(s^N),\sup j``\sigma^N)$.
 \item\label{item:dom(F^P)_low_pres_Dodd-solid}If $\dom(F^P)\leq\max(\rho_{\D}^N,\mu^{+N})$  then:
 \begin{enumerate}[label=\tu{(}\alph*\tu{)}]
  \item $\rho_{\D}^U=\sup j``\rho_{\D}^N$.
  \item If $N$ is Dodd-absent-solid, then
$p_{\D}^U=j(p_{\D}^N)$
and $U$ is Dodd-absent-solid.
 \end{enumerate}
 \item\label{item:Dodd_make_simulated} Suppose $\mu^{+N}<\rho_{\D}^N=\kappa^{+N}<\kappa^{++N}=\dom(F^P)$,
 so $F^P$ is long.
 Then:
 \begin{enumerate}[label=\tu{(}\alph*\tu{)}]
  \item\label{item:rho^U_D_leq_rho^N_D_when_mu^+N<rho^N_D=kappa^+N<kappa^++N=dom(F^P)} $\rho_{\D}^U\leq\rho_{\D}^N$.
  \item\label{item:close_implies_rho_D_match} If $F^P$ is close to $N$ then $\rho_{\D}^U=\rho_{\D}^N$.
  \item\label{item:rho^U_D_leq_rho^N_D_when_mu^+N<rho^N_D=kappa^+N<kappa^++N=dom(F^P)_if_rho^N_D=rho^U_D_and_N_Dodd-abs-solid} If $\rho_{\D}^N=\rho_{\D}^U$ and
  $N$ is Dodd-absent-solid then:
  \begin{enumerate}
  \item $p_{\D}^U=j(p_{\D}^N)\cup\{\eta\}$,
  \item $U$ is not Dodd-absent-solid,
  \item letting $p=j(p_{\D}^N)$, we have:
  \begin{enumerate}
  \item $(U,p)$ is Dodd-absent-solid,
  \item $(U,p\cup\{\alpha\})$ is Dodd-absent-solid for each $\alpha<\eta$.
  \end{enumerate}
  \end{enumerate}
   \end{enumerate}
  \item\label{item:pres_Dodd_structure_when_rho^N_D_leq_kappa} Suppose $\rho_{\D}^N\leq\kappa$.
  Then:
  \begin{enumerate}[label=\tu{(}\alph*\tu{)}]
  \item\label{item:pres_Dodd_structure_when_rho^N_D_leq_kappa_rho^U_D_leq_rho^N_D} $\rho_{\D}^U\leq\rho_{\D}^N$.
  \item If $F^P$ is close to $N$ then $\rho_{\D}^U=\rho_{\D}^N$.
  \item If $\rho_{\D}^N=\rho_{\D}^U$ and $N$ is Dodd-absent-solid then:
  \begin{enumerate}
  \item $p_{\D}^U=j(p_{\D}^N)$,
  \item $U$ is Dodd-absent-solid.
  \end{enumerate}
  \end{enumerate}
  \item\label{item:pres_Dodd-soundness} $U$ is Dodd-absent-sound iff  $N$ is Dodd-absent-sound
  and $\dom(F^P)\leq\max(\rho_{\D}^N,\mu^{+N})$.
 \end{enumerate}

\end{lem}
\begin{proof}
 This is essentially completely analogous to Lemma \ref{lem:param_proj_pres}, but one should use Lemma \ref{lem:rho_D,p_D_basic_facts}(\ref{item:rho_1=rho_D_mod_mu^+}) to establish part \ref{item:Dodd_make_simulated}\ref{item:close_implies_rho_D_match} and other similar parts.
 
 Part \ref{item:pres_Dodd-soundness}:
 Suppose $N$ is Dodd-absent-sound and $\dom(F^P)\leq\max(\rho_{\D}^N,\mu^{+N})$.
 Then one can use part \ref{item:dom(F^P)_low_pres_Dodd-solid}
 together with some standard calculations
 showing that $F^U$ is generated by $p^U_{\D}\cup\rho_{\D}^U$. Now suppose $N$ is Dodd-absent-sound
 but $\dom(F^P)>\max(\rho_{\D}^N,\mu^{+N})$.
 Suppose $U$ is Dodd-absent-sound. Then by Lemma \ref{lem:rho_D,p_D_basic_facts} and parts \ref{item:pres_s,sigma}, \ref{item:Dodd_make_simulated}\ref{item:rho^U_D_leq_rho^N_D_when_mu^+N<rho^N_D=kappa^+N<kappa^++N=dom(F^P)} and \ref{item:pres_Dodd_structure_when_rho^N_D_leq_kappa}\ref{item:pres_Dodd_structure_when_rho^N_D_leq_kappa_rho^U_D_leq_rho^N_D},
 $\rho_{\D}^N=\sigma^N=\sigma^U=\rho_{\D}^U$
 and $j(p_{\D}^N)=j(s^N)=s^U=p_{\D}^U$.
So by part \ref{item:Dodd_make_simulated}\ref{item:rho^U_D_leq_rho^N_D_when_mu^+N<rho^N_D=kappa^+N<kappa^++N=dom(F^P)_if_rho^N_D=rho^U_D_and_N_Dodd-abs-solid},
 $\kappa\geq\max(\rho^N_{\D},\mu^{+N})$,
and clearly then $F^N\rest(\rho_{\D}^N\cup p_{\D}^N)$ is isomorphic to $F^U\rest(\rho_{\D}^U\cup p_{\D}^U)$, and by Dodd-absent-soundness, these are equivalent to $F^N,F^U$ respectively. But $F^N\neq F^U$, contradiction.
 Now  suppose that $N$ is not Dodd-absent-sound
 but $U$ is Dodd-absent-sound.
 If $N$ is Dodd-absent-solid then $F^N$ is not generated by $s^N\cup\sigma^N$, but then easily $F^U$
 is not generated by $s^U\cup\sigma^U$ (using part \ref{item:pres_s,sigma}), and therefore $F^U$ is not Dodd-absent-sound. So suppose $N$ is not Dodd-absent-solid. 
 Then $s^N<p_{\D}^N$ and $\sigma^N>\rho_{\D}^N$,
 and $p^U_{\D}=s^U=j(s^N)$ and $\rho_{\D}^U=\sigma^U=\sup j``\sigma^N>0$, so $\rho^U_{\D}>\theta^{+U}$ where $\theta=\crit(F^U)$.
 But in fact, $\rho_{\D}^U\leq\sup j``\rho_{\D}^N$, by the proof of part \ref{item:pres_s,sigma} (see \cite[\S2]{extmax}).
\end{proof}

\begin{lem}\label{lem:long_proto_Ult}
 Let $N$ be a Dodd-absent-sound active short premouse. Let $M$ be an active long premouse with $\kappa=\crit(F^N)=\crit(F^M)$.
 Suppose that $F^M$ is an $N$-extender which is close to $N$.
 Suppose $\rho_{\D}^N=0$, so $\rho_1^N\leq\kappa^{+N}$. Let $U=\Ult_0(N,F^M)$ \tu{(}defined as in \ref{dfn:Ult_n}, so $U$ is a premouse and $F^U=F^{U^*}\com E$, with notation as there\tu{)}.
Let $j:N\to U$ be the ultrapower map. Then:
 \begin{enumerate}
  \item\label{lem:long_proto_Ult_U_Dodd-abs-solid_not_Dodd-abs-sound}  $U$ is Dodd-absent-solid, but
 \underline{not} Dodd-absent-sound.
  \item\label{item:long_proto_Ult_rho^U_D=rho^N_D} $\rho_{\D}^U=0=\rho_{\D}^N$,
  \item\label{item:long_proto_Ult_p^U_D=j(p^N_D)} $p^U_{\D}=j(p^N_{\D})$,

  \item\label{item:rg(j)=X} $\rg(j)$ is the set of all elements of $U$ which are generated by $p^U_{\D}$; that is,
  letting $k:U\to\Ult(U,F^U)$ be the ultrapower map, then $\rg(j)=X$ where
  \[ X=\Big\{k(f)(p^U_{\D})\Bigm|f\in U|\kappa^{+U}\Big\}.\]
  In particular, $\{\kappa\}$
  is not generated by $p^U_{\D}$.
  \item\label{item:recover_N,j}
   $N^{\passive}$ is the transitive collapse
  of the substructure of $U^{\passive}$ whose universe is $X$, and $j$ is the transitive collapse map. Moreover, for all $a\in[\lambda(F^N)]^{<\om}$
  and $A\in\pow([\kappa]^{|a|})\cap N$, we have
  \[ A\in (F^N)_a\iff j(A)\cap[\kappa]^{|a|}\in (F^U)_{j(a)}\iff A\in (F^U)_{j(a)}.\]
 \item\label{item:proto_shift_defsss}  Let $X\sub\kappa^{+N}$.
 Then 
 \[ X\text{ is }\bfrSigma_{1}^N\text{-definable}\iff  X\text{ is }\bfrSigma_{1}^{U^*}\text{-definable}\iff X\text{ is }\bfrSigma_{1}^U\text{-definable}.\]
 \end{enumerate}
\end{lem}
\begin{proof}
Let $U^*=\Ult(N,F^M)$, formed directly, so $U^*$ is a proper protomouse, and $F^U=F^{U^*}\com E$,
where $E$ is the short part of $F^M$.

 Suppose $\alpha=\max(p_{\D}^N)>\kappa$.
 Then $F^N\rest\alpha\in N$,
 and $j(F^N\rest\alpha)$ is an extender compatible
 with $F^{U^*}$ (but of course $j(F^N\rest\alpha)$
 is $U^*$-total, whereas $F^{U^*}$ is not;
 so what we have is that $j(F^N\rest\alpha)\rest\dom(F^{U^*})\sub F^{U^*}$).
 But then $j(F^N\rest\alpha)\com E=F^U\rest j(\alpha)$. So this gives the Dodd-solidity witness needed for $j(\alpha)$. We similarly
 get Dodd-solidity witnesses at each $\alpha\in p_{\D}^N\cut\{\kappa\}$.
 
 Suppose that $\kappa\notin p_{\D}^N$.
 We have that $F^N\rest p_{\D}^N\notin N$,
 and since $F^N$ is short, this is a subset of $\kappa^{+N}$ which is missing from $N$.
 But $N|\kappa^{++N}=M|\kappa^{++M}=U|\kappa^{++U}$,
 so $F^N\rest p_{\D}^N\notin U$.
 But note that $F^U\rest j(p_{\D}^N)=F^N\rest p_{\D}^N$, so we have $F^U\rest j(p_{\D}^N)\notin U$. 
  
  Suppose instead that $\kappa\in p_{\D}^N$.
  Then  $F^N\rest q\in N$
  where $q=p_{\D}^N\cut\{\kappa\}$ (this doesn't even use Dodd-absent-solidity; it  holds anyway since $q<p^N_{\D}$).
  This is a measure over $\kappa$.
  So $j(F^N\rest q)$ is a measure over $j(\kappa)$ compatible with $F^{U^*}$.
  But then calculating as before,
  we get that $F^U\rest (j(q)\cup j(\kappa))\in U$. But also as before,
  since $F^N\rest p_{\D}^N\notin N$,
  we have $F^U\rest j(p_{\D}^N)\notin U$.
  
  In either case, 
  it follows that $\rho_{\D}^U=0$,
  $p_{\D}^U=j(p_{\D}^N)$ (giving parts \ref{item:long_proto_Ult_rho^U_D=rho^N_D} and \ref{item:long_proto_Ult_p^U_D=j(p^N_D)}) and $U$ is Dodd-absent-solid.
  
  Now let $X$ be defined as in part \ref{item:rg(j)=X};
  let us observe that $X=\rg(j)$.
  Well, given $f\in U|\kappa^{+U}=N|\kappa^{+N}$,
   \[ k(f)=F^{U^*}\com E(f)=F^{U^*}(j(f))=j(F^U(f))\in\rg(j).\]
  But $p_{\D}^U=j(p_{\D}^N)\in\rg(j)$,
  so $k(f)(p^U_{\D})\in\rg(j)$.
  Conversely, if $x\in N$ then there is $f\in N|\kappa^{+N}$ such that $x=\bar{k}(f)(p_{\D}^N)$,
  where $\bar{k}:N\to\Ult_0(N,F^N)$ is the ultrapower map. But $j$ preserves this
  to $U^*$, so 
  \[ j(x)=i_{F^{U^*}}(j(f))(j(p_{\D}^N))=k(f)(p_{\D}^U)\in X. \]
  This establishes part \ref{item:rg(j)=X},
and since $\rho^U_{\D}=0$, therefore $U$ is not Dodd-absent-sound, completing the proof of part \ref{lem:long_proto_Ult_U_Dodd-abs-solid_not_Dodd-abs-sound}.
  
  Part \ref{item:recover_N,j} follows easily from the foregoing considerations.
  
  Part \ref{item:proto_shift_defsss}: This equivalence between $N$ and $U^*$ (the protomouse) is as usual, so we just need to see
  that $X$ is $\bfrSigma_1^{U^*}$ iff $X$ is $\bfrSigma_1^{U}$. But in fact, it is easy to see that
  $\rSigma_1^{U^*}(\{E\})$ (where $j(F^N_J)$ is made automatically available as a constant) is inter-computable
  with $\rSigma_1^{U}(\{E,j(F^N_J)\})$
  (where $F^U_J$ is automatically available),
  which suffices.
\end{proof}

\begin{lem}\label{lem:Ult_0_avoiding_proto_when_rho_D>kappa^+}
 Let $N$ be a Dodd-absent-sound active short premouse. Let $M$ be an active long premouse with $\kappa=\crit(F^N)=\crit(F^M)$.
 Suppose that $F^M$ is an $N$-extender.
 Suppose $\rho_{\D}^N>0$, so $\rho_1^N=\rho_{\D}^N>\kappa^{+N}$. Let $U=\Ult_0(N,F^M)$ \tu{(}as in  \ref{dfn:Ult_n}, so $U$ is a premouse and $F^U=F^{U^*}\com E$, with notation as there\tu{)}.
Let $j:N\to U$ be the ultrapower map. Then:
 \begin{enumerate}
  \item  $U$ is Dodd-absent-sound,
  \item $\rho_{\D}^U=\sup j``\rho_{\D}^N$,
  \item $p^U_{\D}=j(p^N_{\D})$.
\end{enumerate}
\end{lem}
\begin{proof}
 Let $\rho=\sup j``\rho_{\D}^N$ and $p=j(p_{\D}^N)$. Then $F^U$ is generated by $\rho\cup p$,
 like in the preceding proof.
 Also like in that proof, we get the relevant
 Dodd-solidity witnesses below $(p,\rho)$.
 This suffices.
\end{proof}

\begin{rem}
In the context of Lemma \ref{lem:Ult_0_avoiding_proto_when_rho_D>kappa^+},
 we also know that $\rho_1^U=\sup j``\rho_1^N$ (since $\rho_1^N=\rho_{\D}^N$ and $\rho_1^U=\rho_{\D}^U$).
 Note, though, that we haven't verified that $U$ is $1$-sound, nor described $p_1^U$ in terms of $p_1^N$. This is because of the discrepancy between the constants $F_J^U$ and $F_J^{U^*}$,
 which might not be in $U|\rho_{\D}^U$.
\end{rem}

\begin{dfn}\label{dfn:stretched-Dodd-absent-solid}
 Let $N$ be an active short  premouse.
 We say that $N$ is \emph{stretched-Dodd-absent-solid}
 iff $p_{\D}^N\neq\emptyset$ and letting $\eta=\min(p_{\D}^N)$ and $p=p_{\D}^N\cut\{\eta\}$, then we have:
 \begin{enumerate}
  \item $(N,p)$
 is Dodd-absent-solid, 
 \item for each $\alpha<\eta$,
 $(N,p\cup\{\alpha\})$ is Dodd-absent-solid,
 \item $N|\eta$ is active with a short extender $E$
 which is $N$-total,
 \item $\crit(F^N)<\kappa<\kappa^{+N}=\rho_{\D}^N$ where $\kappa=\crit(E)$,
 \item there is a Dodd-absent-sound premouse $H$ and a $0$-embedding $\pi:H\to N$
 such that $\kappa=\crit(\pi)$
 (where $\kappa$ is as above)
 and $\rho_{\D}^H=\kappa^{+N}=\kappa^{+H}$
 and $\pi(p_{\D}^H)=p$,
 \item $E$ is the short extender derived from $\pi$,
 \item $H||\kappa^{++H}=N||\kappa^{++N}$.\qedhere
 \end{enumerate}
\end{dfn}

 The witnessing $H,\pi$
 are uniquely determined by $N$, if they exist
 ($\rg(\pi)$ is just the set of points generated (by $F^N$) with elements of $\rg(E)\cup\{p\}$).
 
 \subsection{Cores}
 
 In this section we define cores and core embeddings, and consider their relationship to (especially dropping) iteration maps.
 \begin{dfn}
  Let $N$ be an active short premouse.
  We define the \emph{Dodd-absent-core} $\core_{\D}(N)$,
  and the associated \emph{core map} $\sigma:\core_{\D}(N)\to N$, as follows:
  \begin{enumerate}
   \item\label{item:stretched_Dodd-absent-core} If $N$ is  stretched-Dodd-absent-solid,
   as witnessed by $\pi:H\to N$
   (see Definition \ref{dfn:stretched-Dodd-absent-solid})
   then $\core_{\D}(N)=H$ and $\sigma=\pi$.
   \item\label{item:direct_Dodd-absent-core} Otherwise, $C=\core_{\D}(N)$ 
   and $\sigma:C\to N$ are such that
    $\rg(\sigma)$
    is the set of elements of $N$ generated by elements of $\rho_{\D}^N\cup p_{\D}^N$,
    and $F^C$ is the natural extender this induces. That is, noting that $\lambda(F^N)\in\rg(\sigma)$,
    $\kappa=\crit(F^N)\sub\rg(\sigma)$, $\kappa\leq\sigma(\kappa)$
and $\pow([\kappa]^{|a|})\cap C=\pow([\kappa]^{|a|})\cap N$,
    we set $\sigma(\lambda(F^C))=\lambda(F^N)$,
    $\crit(F^C)=\kappa$,
    and 
    for $a\in[\lambda(F^C)]^{<\om}$
    and $X\in\pow([\kappa]^{|a|})\cap C$,
    we put $a\in F^C(X)$ iff $\sigma(a)\in F^N(X)$, and note  $X=\sigma(X)\cap[\kappa]^{|a|}$.
    \qedhere
  \end{enumerate}
 \end{dfn}

 \begin{rem}
 Although in case $N$ is stretched-Dodd-absent-solid, we directly define $\core_{\D}(N)$ via clause \ref{item:stretched_Dodd-absent-core}, we could have equivalently
 first used clause \ref{item:direct_Dodd-absent-core} and then clause \ref{item:stretched_Dodd-absent-core}. That is,
 suppose $N$ is stretched-Dodd-absent-solid,
 but define $C$ and $\sigma:C\to N$ as in clause \ref{item:direct_Dodd-absent-core} above
 (instead of using the definition in clause \ref{item:stretched_Dodd-absent-core}).
Then $C$ is also stretched-Dodd-absent-solid, $\bar{C}=\core_{\D}(C)=\core_{\D}(N)$
(with both of these defined via clause \ref{item:stretched_Dodd-absent-core}),
 and letting $\sigma_{\bar{C}C}:\bar{C}\to C$ and $\sigma_{\bar{C}N}:\bar{C}\to N$ be the core maps, then $\sigma\com\sigma_{\bar{C}C}=\sigma_{\bar{C}N}$. \end{rem}

 \begin{dfn}
 Let $N$ be an active short premouse. We say that $N$ is \emph{Dodd-absent-universal}
 iff either:
 \begin{enumerate}
 \item $N$ is stretched-Dodd-absent-solid,
 or
 \item otherwise, and letting $C=\core_{\D}(N)$,
 $\kappa=\crit(F^N)$
 and $\rho=\max(\rho^N_{\D},\kappa^{+N})$, we have
  $C||\rho^{+C}=N||\rho^{+N}$.\qedhere
\end{enumerate}
 \end{dfn}
\begin{lem}\label{lem:stretch_Dodd-abs-sound_short_ext_above_its_crit}
 Let $N,P$ be as in Lemma \ref{lem:Dodd_param_proj_pres}
 part \ref{item:Dodd_make_simulated}, and suppose further that $N$ is Dodd-absent-sound and $F^P$ is close to $N$. Then $U$ is stretched-Dodd-absent-solid,
 $N=\core_{\D}(U)$ and $j$ is the Dodd-absent-core map.
\end{lem}

\begin{lem}
 Let $N,M$ be as in Lemma \ref{lem:long_proto_Ult}. \tu{(}That lemma
 already assumed that $N$ is Dodd-absent-sound
 and that $F^M$ is close to $N$,
 and by that lemma, $U$ is Dodd-absent-solid.\tu{)} Then
 $N=\core_{\D}(U)$ and $j$ is the Dodd-absent-core map.
\end{lem}

\begin{lem}\label{lem:iterate_sim_Dodd-solid}
 Let  $N$ be an active short premouse which is either Dodd-absent-solid or stretched-Dodd-absent-solid. 
Let $P$ be an active premouse.
Suppose $F^P$ is an $N$-extender and $\max(\crit(F^N),\rho_{\D}^N)<\kappa=\crit(F^P)<\dom(F^P)\leq\OR^N$
\tu{(}so $(N,P,0)$ is not proto-anomalous\tu{)}.
Let $U=\Ult_0(N,F^P)$ and $j:N\to U$ be the ultrapower map.
Suppose $U$ is wellfounded. Then:
\begin{enumerate}
 \item $\rho_{\D}^U\leq\rho_{\D}^N$,
 \item if $F^P$ is close to $N$ then $\rho_{\D}^U=\rho_{\D}^N$,
 \item\label{item:iterate_sim_Dodd-solid_if_rho^U_D=rho^N_D} if $\rho_{\D}^U=\rho_{\D}^N$ then:
 \begin{enumerate}[label=\tu{(}\alph*\tu{)}]
 \item $p_{\D}^U=j(p_{\D}^N)$,
 \item if $N$ is Dodd-absent-solid then so is $U$,
 \item\label{item:iterate_sim_Dodd-solid_if_rho^U_D=rho^N_D_if_N_stretched_Dodd-solid} if $N$ is stretched-Dodd-absent-solid then so is $U$, and $j$ is continuous at $\eta=\min(p_{\D}^N)$.
 \item $C=\core_{\D}(U)=\core_{\D}(N)$
 and letting $\pi_N:C\to N$ and $\pi_U:C\to U$
 be the Dodd-absent-core maps, we have $\pi_U=j\com\pi_N$.
 \end{enumerate}
\end{enumerate}
\end{lem}
\begin{proof}
This is mostly standard; one key point is that in part \ref{item:iterate_sim_Dodd-solid_if_rho^U_D=rho^N_D}\ref{item:iterate_sim_Dodd-solid_if_rho^U_D=rho^N_D_if_N_stretched_Dodd-solid},
 $j$ is continuous at $\min(p_{\D}^N)=\eta=\lh(E)$ where $E=F^{N|\eta}$, because letting $\crit(E)=\mu$, we have $\cof^N(\eta)=\mu^{+N}=\rho^N_{\D}<\kappa=\crit(F^P)$.
\end{proof}

We will define the higher fine structural cores by first taking the Dodd-absent-core
(or demanding that the structure is already Dodd-absent-sound),
and then proceeding in the usual way.
 \begin{dfn}
 Let $N$ be an active short premouse.
 We (attempt to) define the \emph{1st core} $\core_1(N)$,
 and the associated \emph{core map} $\sigma:\core_1(N)\to N$, as follows.
 Let $D=\core_{\mathrm{D}}(N)$ and $\sigma_{DN}:D\to N$ be the Dodd-absent core embedding. In order to define $\core_1(N)$ and $\sigma$, we first require that either:
 \begin{enumerate}[label=--]
 \item $N$ is stretched-Dodd-absent-solid (hence $D$ is Dodd-absent-sound
 with $\rho_{\D}^D=\rho_{\D}^N$ etc), or
 \item
 $N$ is Dodd-absent-solid and Dodd-absent-universal (hence $\rho_{\D}^D=\rho_{\D}^N$ and $\sigma_{DN}(p_{\D}^D)=p_{\D}^N)$),
and $D$ is Dodd-absent-solid (hence Dodd-absent-sound).
 \end{enumerate}
Given this, then we define $\core_1(N)$
and $\sigma:\core_1(N)\to N$ as follows:
 \begin{enumerate}
  \item If $D$ is stretched-$1$-solid,
  as witnessed by $\pi:H\to D$,
  then $\core_1(N)=H$ and $\sigma=\sigma_{DN}\com\pi$.
  \item Otherwise, letting $C=\core_1(D)$ be the traditionally defined 1st core of $D$,
  and $\pi:C\to D$ the associated core map,
  then $\core_1(N)=C$ and $\sigma=\sigma_{DN}\com\pi$.
  (That is, $C=\cHull_1^D(\rho_1^D\cup\{p_1^D\})$ and $\pi$ is the uncollapse map.)
  \end{enumerate}

  Now let $n<\om$ and let $N$
  be an $n$-sound premouse,
  and suppose that\[ \text{ if }N\text{ is active short then [}n>0\text{
  and }N\text{ is Dodd-absent-sound].}\]
  We define the $(n+1)$th core $\core_{n+1}(N)$,
  and the associated \emph{core map} $\sigma:\core_{n+1}(N)\to N$, as follows:
  \begin{enumerate}
   \item If $N$ is  stretched-$(n+1)$-solid,
   as witnessed by $\pi:H\to N$,
   then $\core_{n+1}(N)=H$ and $\sigma=\pi$;
   \item Otherwise, $\core_{n+1}(N)$ and the core map are defined in the traditional manner.
   That is, $\core_{n+1}(N)=\cHull_{n+1}^N(\rho_{n+1}^N\cup\{\pvec_{n+1}^N\})$
   and $\sigma$ is the uncollapse map.\qedhere
  \end{enumerate}
 \end{dfn}

\begin{lem}\label{lem:stretch_n+1-sound_if_short_Dodd-abs-sound}
 Let $n,N,P$ be as in Lemma \ref{lem:param_proj_pres} part \ref{item:long_proj}, and suppose further that $N$ is $(n+1)$-sound, $F^P$ is close to $N$,
 and if $N$ is active short
 then $N$ is Dodd-absent-sound. Then $U$ is stretched-$(n+1)$-solid,
 $N=\core_{n+1}(U)$ and $j$ is the core map.

 Suppose further that $N$ is active short. If either $n>0$ or  $\crit(F^P)<\crit(F^N)$ then $U$ is Dodd-absent-sound,
 whereas if $n=0$ and $\crit(F^P)>\crit(F^N)$ then $U$ is not Dodd-absent-sound.\footnote{Suppose $N$ is active short and $n=0$. By the hypotheses of Lemma \ref{lem:param_proj_pres}
 part \ref{item:long_proj},
  then $F^P$ is long and letting $\kappa=\crit(F^P)$, we have
 $\crit(F^N)\neq\kappa$ and
 $\rho_{n+1}^N=\kappa^{+N}$. So if $\crit(F^N)<\kappa$, Lemma \ref{lem:stretch_Dodd-abs-sound_short_ext_above_its_crit} applies and yields
 Lemma \ref{lem:stretch_n+1-sound_if_short_Dodd-abs-sound} in this case;
 if $\kappa<\crit(F^N)$
 then Lemma \ref{lem:Dodd_param_proj_pres} gives that $U$ is Dodd-absent-sound.}
\end{lem}

\begin{lem}\label{lem:iterate_sim_solid}
 Let $n<\om$ and $N$ be a stretched-$(n+1)$-solid $n$-sound premouse,
 and suppose that if $N$ is active short
 then $N$ is
 Dodd-absent-sound. 
Let $P$ be an active premouse.
Suppose $F^P$ is an $N$-extender and $\rho_{n+1}^N<\kappa=\crit(F^P)<\dom(F^P)\leq\rho_n^N$.
Suppose $(N,P,n)$ is not proto-anomalous.
Let $U=\Ult_n(N,P)$ and $j:N\to U$ be the ultrapower map. 
Suppose $U$ is wellfounded. Then:
\begin{enumerate}
 \item $\rho_{n+1}^U\leq\rho_{n+1}^N$,
 \item if $F^P$ is close to $N$ then $\rho_{n+1}^U=\rho_{n+1}^N$,
 \item if $\rho_{n+1}^U=\rho_{n+1}^N$ then:
 \begin{enumerate}
 \item $p_{n+1}^U=j(p_{n+1}^N)$,
 \item $U$ is stretched-$(n+1)$-solid,
 \item $j$ is continuous at $\eta=\min(p_{n+1}^N)$.
 \item $C=\core_{n+1}(U)=\core_{n+1}(N)$
 and letting $\pi_N:C\to N$ and $\pi_U:C\to U$
 be the core maps, we have $\pi_U=j\com\pi_C$.
 \end{enumerate}
\end{enumerate}
\end{lem}
\begin{proof}
 This is like for Lemma \ref{lem:iterate_sim_Dodd-solid}. 
\end{proof}

\subsection{Iteration trees}\label{sec:it_trees}

\begin{rem}
We now define $n$-maximal iteration trees. The main difference here is in the rules for determining tree order.
We will set $\pred^\Tt(\alpha+1)$
to be the least $\beta$ such that $\spc(E^\Tt_\beta)<\lgcd(\exit^\Tt_\alpha)$.\footnote{Here ``$\lgcd(P)$'' denotes the largest cardinal of $P$, if it has one.
So $\lgcd(\exit^\Tt_\alpha)=\lambda(E^\Tt_\alpha)$ unless $E^\Tt_\alpha$
is long with no largest generator,
and then $\lgcd(\exit^\Tt_\alpha)=\lambda^{+\exit^\Tt_\alpha}=\lambda^{+M^\Tt_{\alpha+1}}=\nu(E^\Tt_\alpha)$. So in the latter case, generators of $E^\Tt_\alpha$ are not moved by maps $i^\Tt_{\alpha+1,\gamma}$ unless the first extender used along $(\alpha+1,\gamma]^\Tt$ is long.}

 Traditional $n$-maximal iteration trees $\Tt$  (where $n\leq\om$)
 come with an associated \emph{degree} function $\deg^\Tt$. It will be important for us to also keep track of the Dodd-absent-(un)soundness of the models $M^\Tt_\alpha$ of $\Tt$, analogously to how $\deg^\Tt_\alpha$ traditionally keeps track of the degree of (un)soundness of the models $M^\Tt_\alpha$. Toward this purpose,
 we will consider ``degrees'' $d\in(\om+1)\cup\{0^-\}$ (where $0^-=(0,0)$ say).
 In case $M^\Tt_\alpha$ is active short
 and either there is a drop in model along $[0,\alpha]^\Tt$
 or $\deg^\Tt_\alpha\neq n$ (where $\Tt$ is $n$-maximal), we will have:
 \begin{enumerate}[label=--]
  \item 
$F(M^\Tt_\alpha)$ is Dodd-absent-sound
 iff $\deg^\Tt_\alpha\in\om$  iff $\deg^\Tt_\alpha\neq 0^-$, and
 \item if $\deg^\Tt_\alpha=0^-$
 then letting $\beta+1\leq^\Tt\alpha$
 be least such that $(\beta+1,\alpha]^\Tt$ is non-model-dropping
 and $\deg^\Tt_{\beta+1}=0^-$,
 we have that $M^{*\Tt}_{\beta+1}=\core_{\D}(M^\Tt_\alpha)$ and  $i^{*\Tt}_{\beta+1,\alpha}:M^{*\Tt}_{\beta+1,\alpha}\to M^\Tt_\alpha$ is just the Dodd-absent core map.
 \end{enumerate}
\end{rem}
\begin{dfn}\label{dfn:0^-}
 Let $0^-=(0,0)$. (So $0^-\notin\om+1$.)
 Every  premouse is considered \emph{$0^-$-sound}.  We order $(\om+1)\cup\{0^-\}$
 with $0^-<0<1<2<\ldots<\om$.
 If $m\in(\om+1)\cup\{0^-\}$
 and $M$ is an $m$-sound premouse,
 then given any $(N,n)$ such that
$n\in(\om+1)\cup\{0^-\}$ and $N$ is an $n$-sound premouse,
 we write $(N,n)\ins (M,m)$
 iff $N\ins M$ and if $N=M$ then $n\leq m$.
 Given a premouse $M$,
 define $\rho_{0^-}^M=\OR^M$
 and $\rho_{0^-+1}^M=\rho_0^M$ unless $M$ is active short, in which case $\rho_{0^-+1}^M=\max(\kappa^{+M},\rho_{\D}^M)$, where $\kappa=\crit(F^M)$. Define $\Ult_{0^-}=\Ult_0$.
\end{dfn}

\begin{dfn}\label{dfn:n-max}
Let $n\in(\om+1)\cup\{0^-\}$.
 Let $M$ be an $n$-sound premouse
 such that if $M$ is active short
 and $0^-<n$ then $M$ is Dodd-absent-sound.
 An \emph{$n$-maximal iteration tree}
 on $M$ is a tuple \[\Tt=(\left<M_\alpha\right>_{\alpha<\theta},\left<M^*_{\alpha+1},E_\alpha\right>_{\alpha+1<\theta},{<}^{\Tt},\dropset,\deg) \] satisfying the usual requirements
 of $n$-maximality (or of $0$-maximality in case $n=0^-$), except that for each $\alpha+1<\lh(\Tt)$, we have:
 \begin{enumerate}
  \item $\pred(\alpha+1)$ is the least $\beta\leq\alpha$
  such that either:
  \begin{enumerate}
  \item $\crit(E_\alpha)<\lambda(E_\beta)$, or
  \item $E_\alpha$ is short, $E_\beta$ is long, $\nu(E_\beta)$ is a limit ordinal and $\crit(E_\alpha)=\lambda(E_\beta)$.
  \end{enumerate}
  \item Let $\beta=\pred(\alpha+1)$.
  Then $(M^{*}_{\alpha+1},\deg_{\alpha+1})$ is the largest $(N,n)$ such that \[(\exit^\Tt_\beta,0^-)\ins(N,n)\ins (M_\beta,\deg_\beta)\text{ and
  }\dom(E_\alpha)\leq\rho_n^N,\]
  \emph{unless} this results in an active short $N$
  and $n=0$ 
  and $\rho_{0^-+1}^N<\dom(E_\alpha)$,
  in which case $(M^*_{\alpha+1},\deg_{\alpha+1})=(N,0^-)$.
  \item $M_{\alpha+1}=\Ult_d(M^{*}_{\alpha+1},E_\alpha)$ where $d=\deg_{\alpha+1}$, noting that we are using the definition of $\Ult_n$ given by Definitions \ref{dfn:Ult_n} and \ref{dfn:0^-}.
 \end{enumerate}
 
 Suppose $\beta\leq^\Tt\alpha$.
 We say  $\Tt$ \emph{drops in model in $(\beta,\alpha]^\Tt$} if $(\beta,\alpha]^\Tt\cap\dropset^\Tt\neq\emptyset$.
 If $\Tt$ does not drop in model in $(\beta,\alpha]^\Tt$,
 we say $\Tt$ \emph{drops in degree in $(\beta,\alpha]^\Tt$} if $\deg^\Tt_\beta>0$ and $\deg^\Tt_\beta\neq\deg^\Tt_\alpha$.
 If $\Tt$ does not drop in model or degree
 in $(\beta,\alpha]^\Tt$,
 we say  $\Tt$ \emph{drops in Dodd-degree in $(\beta,\alpha]^\Tt$}
 if $\deg^\Tt_\beta=0$ and $\deg^\Tt_\alpha=0^-$. A \emph{drop of any kind} is a drop in model, degree or Dodd-degree.

 We define \emph{$(n,\alpha)$-} and \emph{$(n,\alpha,\beta)^*$-iteration strategies} for $M$ (where $n,M$ are as above) in the usual manner,
 based on the notion of $n$-maximal trees introduced above. (So an $(n,\alpha)$-strategy is a winning strategy for player II in the iteration game producing an $n$-maximal tree through length $\alpha$, and an $(n,\alpha,\beta)^*$-strategy is one for stacks of length $\beta$, consisting of such trees, with the usual conditions.
 The meaning of the superscript ``$*$''
 is as in the game $G^*_k(\om_1,\om_1+1)$, described in \cite[p.~1202, prior to Corollary 1.10]{cmwmwc}.
\end{dfn}

\begin{rem}
 Note that if $M$ is active short, then by definition, we only consider $n$-maximal trees on $M$
 for $n\in\om+1$ assuming that $M$ is Dodd-absent-sound and $n$-sound. (It would have been possible to define $n$-maximal trees on such $M$ without the Dodd-absent-soundness assumption, but we will have no use for them. If we had allowed it, $0$-maximal trees $\Tt$ would anyway be equivalent to $0^-$-maximal  $\Tt'$, except that $\deg^\Tt$ could differ from $\deg^{\Tt'}$.)
\end{rem}

\begin{lem}[Closeness]\label{lem:closeness}
Let $n\in(\om+1)\cup\{0^-\}$.
 Let $\Tt$ be an $n$-maximal tree on an $n$-sound premouse $N$ \tu{(}so if $n\neq 0^-$ and $N$ is active short then $N$ is Dodd-absent-sound\tu{)}.
 Then:
 \begin{enumerate}
  \item for each $\alpha+1<\lh(\Tt)$,
  $E^\Tt_\alpha$ is close to $M^{*\Tt}_{\alpha+1}$,
  \item\label{item:shift_definitions} for all $\alpha+1<^\Tt\beta<\lh(\Tt)$,
  if $(\alpha+1,\beta]^\Tt\cap\dropset^\Tt=\emptyset$ and $k=\max(\deg^\Tt_{\alpha+1},0)=\max(\deg^\Tt_\beta,0)$, then
 for all $X\sub\spc(E^\Tt_\alpha)$,
  \[ X\text{ is }\bfrSigma_{k+1}^{M^{*\Tt}_{\alpha+1}}\text{-definable iff }X\text{ is }\bfrSigma_{k+1}^{M^\Tt_\beta}\text{-definable}.\]
 \end{enumerate}
\end{lem}
\begin{proof}
 This is a direct generalization of the short extender proof  \cite[6.1.5]{fsit},
 and partly like in \cite{voellmer}, though also different to \cite{voellmer}, because we do not assume projectum free spaces, and 
 have the modified $\Ult_0$ notion.
 
 We proceed by induction on $\lh(\Tt)$.
 
 Let $\alpha+1<\lh(\Tt)$ be given,
 and suppose the lemma holds for $\Tt\rest(\alpha+1)$.
 
 Let us first see that $E^\Tt_\alpha$ is close to $M^{*\Tt}_{\alpha+1}$.
 Let $\beta=\pred^\Tt(\alpha+1)$. We may assume that $\beta<\alpha$.

 If $E^\Tt_\alpha\in M^\Tt_\alpha$
 or $\rho_1(M^\Tt_\alpha)>\tau=\OR(\dom(E^\Tt_\alpha))$
 then for each $a$, we have $(E^\Tt_\alpha)_a\in M^\Tt_\beta$, which suffices.
 So from now on assume
 \begin{equation}\label{eqn:E^T_alpha_is_active_ext} E^\Tt_\alpha=F(M^\Tt_\alpha)\text{ and } \rho_1(M^\Tt_\alpha)\leq\tau=\OR(\dom(E^\Tt_\alpha)).
\end{equation}
 
 Let $\gamma$ be least such that $\beta\leq\gamma$ and $\gamma+1\leq^\Tt\alpha$
 and $(\gamma+1,\alpha]^\Tt$ does not drop in model or degree. From our assumption in line (\ref{eqn:E^T_alpha_is_active_ext}),  $\deg^\Tt_{\gamma+1}\in\{0,0^-\}$ (since $\lambda(E^\Tt_\gamma)\geq\lambda(E^\Tt_\beta)$).

 Now note that for all $\gamma'$ with $\gamma+1\leq^\Tt\gamma'+1\leq^\Tt\alpha$,
 we have $\crit(E^\Tt_{\gamma'})\geq\crit(E^\Tt_\alpha)$, and if $\crit(E^\Tt_{\gamma'})=\crit(E^\Tt_\alpha)$
 then $\gamma'=\gamma$ and $E^\Tt_\gamma$
 is long and $E^\Tt_\alpha$ is short,
 and hence $F(M^{*\Tt}_{\gamma+1})$ is short
 and $\crit(F(M^{*\Tt}_{\gamma+1}))=\crit(E^\Tt_\alpha)$ (and we form $M^\Tt_{\gamma+1}=\Ult_0(M^{*\Tt}_{\gamma+1},E^\Tt_\gamma)$ as in the special case described in Definition \ref{dfn:Ult_n}, modifying the protomouse).
 But each  measure $(E^\Tt_\alpha)_a$ is $\bfrSigma_1^{M^\Tt_\alpha}$-definable, and so by
 clause \ref{item:shift_definitions}
 and induction, it follows that it is also $\bfrSigma_1^{M^{*\Tt}_{\gamma+1}}$-definable.

 Let $\delta=\pred^\Tt(\gamma+1)$.
 Suppose
  $\beta<\delta$.
 If $M^{*\Tt}_{\gamma+1}\pins M^\Tt_\delta$,
 then by the previous discussion,
 each $(E^\Tt_\alpha)_a$ is in $M^\Tt_\delta$,
 and hence in $\exit^\Tt_\beta$, and hence in $M^{*\Tt}_{\alpha+1}$, which suffices.
 So suppose $M^{*\Tt}_{\gamma+1}=M^\Tt_\delta$.
 Then $\deg^\Tt_\delta>0$ (by choice of $\gamma+1$)
 and since $\beta<\delta$, therefore $\OR(\exit^\Tt_\beta)\leq\rho_1(M^\Tt_\delta)$,
 so each $(E^\Tt_\alpha)_a$ is again in $M^\Tt_\delta$, which again suffices.
 
 So we may assume $\delta\leq\beta$.
 But  $\crit(E^\Tt_\alpha)\leq\crit(E^\Tt_\gamma)$, so
 considering the rules for $\pred^\Tt$,
 if $\delta<\beta$ then $\crit(E^\Tt_\gamma)=\crit(E^\Tt_\alpha)$,
 and $E^\Tt_\gamma$ is \emph{short} but $E^\Tt_\alpha$ long (and $E^\Tt_\delta$ is long with $\nu(E^\Tt_\delta)$ a limit).
 But
 since $(\gamma+1,\alpha]^\Tt$ does not drop in model
 and $E^\Tt_\alpha=F(M^\Tt_\alpha)$,
 if $\crit(E^\Tt_\gamma)=\crit(E^\Tt_\alpha)$
  then $E^\Tt_\gamma$ is \emph{long},
  and $M^\Tt_{\gamma+1}=\Ult_0(M^{*\Tt}_{\gamma+1},E^\Tt_\gamma)$ is formed
  according to the special case of  \ref{dfn:Ult_n}; contradiction.
  
  So $\pred^\Tt(\gamma+1)=\delta=\beta=\pred^\Tt(\alpha+1)$.
  So $M^{*\Tt}_{\gamma+1}\ins M^\Tt_\beta$,
  and we already saw that each $(E^\Tt_\alpha)_a$
  is $\bfrSigma_1^{M^{*\Tt}_{\gamma+1}}$.
  But either
  \begin{enumerate}[label=--]
   \item $\crit(E^\Tt_\alpha)<\crit(E^\Tt_\gamma)$,
   or
   \item $\crit(E^\Tt_\alpha)=\crit(E^\Tt_\gamma)$,
   $E^\Tt_\alpha$ is short and $E^\Tt_\gamma$ long,
   (and $M^\Tt_{\gamma+1}=\Ult_0(M^{*\Tt}_{\gamma+1},E^\Tt_\gamma)$ is formed according to the special case of  \ref{dfn:Ult_n}), so $\spc(E^\Tt_\alpha)<\spc(E^\Tt_\gamma)$,
  \end{enumerate}
  and so in either case, we have $M^{*\Tt}_{\gamma+1}\ins M^{*\Tt}_{\alpha+1}$, and hence $(E^\Tt_\alpha)_a$
   is $\bfrSigma_1^{M^{*\Tt}_{\alpha+1}}$, as desired. This completes the proof that $E^\Tt_\alpha$ is close to $M^{*\Tt}_{\alpha+1}$.
   
Clause  \ref{item:shift_definitions}
for $\Tt\rest(\alpha+2)$ follows by the same fact for $\Tt\rest(\alpha+1)$ and Lemmas \ref{lem:param_proj_pres} and \ref{lem:long_proto_Ult}.
 Propagating clause \ref{item:shift_definitions}
 through limit stages is easy. This completes the proof.
\end{proof}

We can now describe iteration maps
produced by $k$-maximal trees, above a drop of some kind along a branch, in terms of core embeddings:

\begin{lem}\label{lem:dropped_iterated_core_embedding}
Let $m\in(\om+1)\cup\{0^-\}$,
 let $M$ be an $m$-sound premouse
  and $\Tt$ be an $m$-maximal tree on $M$
 of length $\alpha+1$.
 Suppose that $[0,\alpha]^\Tt$ drops in model, degree or Dodd-degree. Let $\beta+1\leq^\Tt\alpha$ be
 least such that $(\beta+1,\alpha]^\Tt$
 does not drop in model, degree or Dodd-degree.
 Let $d=\deg^\Tt_\alpha$.
 Then:
 \begin{enumerate}
  \item If $d\neq 0^-$  then:
  \begin{enumerate}[label=\tu{(}\alph*\tu{)}]
   \item 
$M^{\Tt}_{\alpha}$ is $d$-sound but not $(d+1)$-sound,
\item if $M^\Tt_\alpha$ is active short then it is Dodd-absent-sound,
\item $M^{*\Tt}_{\beta+1}=\core_{d+1}(M^\Tt_\alpha)$ is $(d+1)$-sound and if active short is Dodd-absent-sound, and
\item $i^{*\Tt}_{\beta+1,\alpha}$ is the $(d+1)$-core embedding.
\end{enumerate}
  \item\label{item:d=0,active_short} If $d=0^-$ then:
  \begin{enumerate}[label=\tu{(}\alph*\tu{)}]
  \item $M^\Tt_\alpha$ is active short
  and non-Dodd-absent-sound,
  \item $M^{*\Tt}_{\beta+1}=\core_{\D}(M^\Tt_\alpha)$ is Dodd-absent-sound,
  and
  \item $i^{*\Tt}_{\beta+1,\alpha}$ is
  the Dodd-absent core map.
  \end{enumerate}
  \item\label{item:d=0^-_with_degree_drop} If $d=0^-$ and $[0,\alpha]^\Tt$ drops in model or degree,
  then letting $\beta_1+1\leq^\Tt\alpha$ be least such that $(\beta_1+1,\alpha]^\Tt$ does not drop in model or degree, we have:
  \begin{enumerate}
  \item $M^{*\Tt}_{\beta_1+1}=\core_1(M^\Tt_\alpha)$ is $1$-sound and Dodd-absent-sound, and
  \item $i^{*\Tt}_{\beta_1+1,\alpha}$ is the $1$-core embedding.
  \end{enumerate}
 \end{enumerate}
\end{lem}

The lemma follows readily from the preceding material on preservation of fine structure and Dodd structure.
(Recall that in part \ref{item:d=0^-_with_degree_drop}, the $1$-core is defined
as $\core_1(\core_{\D}(M^\Tt_\alpha))$,
and the $1$-core embedding $\core_1(M^\Tt_\alpha)\to M^\Tt_\alpha$
is the composition $f\com g$ of the two core embeddings $f:\core_{\D}(M^\Tt_\alpha)\to M^\Tt_\alpha$
and $g:\core_1(\core_{\D}(M^\Tt_\alpha))\to\core_{\D}(M^\Tt_\alpha)$.
It is important that we do it in this manner,
as if we instead directly define the $1$-core of $M^\Tt_\alpha$ in the traditional fashion, we can get a different result.)

\subsection{Copying and weak Dodd-Jensen}

We now want to adapt the Shift Lemma and copying construction (for lifting an iteration tree on a premouse $M$ to one on a premouse $N$ via an embedding $\pi:M\to N$), and the usual weak Dodd-Jensen property for iteration strategies, to the premice under consideration. This is somewhat complicated
because we need to consider iteration maps and copy maps
which fail to even be $\Sigma_0$-elementary with respect to the active extender predicate.
These arise because of how we define $\Ult_0$
in order to ``avoid a protomouse'', which leads to embeddings which fail to preserve the critical point of an active short extender, but which have the following property:
\begin{dfn}
 Let $M,N$ be active short premice.
 Let $\pi:M\to N$. We say that $\pi$
 is \emph{$0$-deriving}
 iff $\pi:M^\passive\to N^\passive$ is elementary
 (so $\pi(\lambda^M)=\lambda^N$, as these are the largest cardinals of each), and letting $\kappa=\crit(F^M)$ and $\mu=\crit(F^N)$:
 \begin{enumerate}[label=--] \item $\mu\leq\pi(\kappa)<\lambda(F^N)$,
 and
 \item
 $\pi(i^M_{F^M}(X))=i^N_{F^N}(\pi(X)\cap\mu)$ for all $X\in\pow(\kappa)^M$.\qedhere
 \end{enumerate}
\end{dfn}

\begin{dfn}
Let $m\in(\om+1)\cup\{0^-\}$,
let $M$ be an $m$-sound premouse,
and $\Tt$ be an $m$-maximal tree on $M$. Let $\alpha+1<\lh(\Tt)$.
We say that $\alpha$ is \emph{$\Tt$-proto} iff $\deg^\Tt_{\alpha+1}\in\{0,0^-\}$, $E^\Tt_\alpha$ is long,
$M^{*\Tt}_{\alpha+1}$ is active short and $\crit(E^\Tt_\alpha)=\crit(F(M^{*\Tt}_{\alpha+1}))$.
\end{dfn}
The  following lemma mostly follows from the fine structure preservation facts established earlier.

\begin{lem}
Let $m\in(\om+1)\cup\{0^-\}$, let $M$ be an $m$-sound premouse, and
let $\Tt$ be a successor length $m$-maximal
tree on $M$. Let $\alpha,\beta<\lh(\Tt)$
with $\alpha\leq^\Tt\beta$ and $(\alpha,\beta]^\Tt\cap\dropset^\Tt=\emptyset$
and $\deg^\Tt_\alpha=\deg^\Tt_\beta=n$. Then:
\begin{enumerate}
\item If there is no $\gamma+1\in(\alpha,\beta]^\Tt$ such that $\gamma$ is $\Tt$-proto then $i^\Tt_{\alpha\beta}$ is a $\max(n,0)$-embedding.
\item If $\alpha$ is a successor
and there is no $\gamma+1\in[\alpha,\beta]^\Tt$ such that $\gamma$ is $\Tt$-proto then $i^{*\Tt}_{\alpha\beta}$ is a $\max(n,0)$-embedding.
\item Suppose there is $\gamma+1\in[\alpha,\beta]^\Tt$
such that $\gamma$ is $\Tt$-proto. Then there is a unique such $\gamma$, and letting $\delta=\pred^\Tt(\gamma+1)$, either:
\begin{enumerate}
\item we have:
\begin{enumerate}
\item $M$ is active short
and $(0,\beta]^\Tt$ does not drop in any way,
\item $n=m\in\{0,0^-\}$,
\item $i^\Tt_{0\delta}$ and $i^\Tt_{\gamma+1,\beta}$ are  $0$-embeddings,
\item $i^\Tt_{\delta,\gamma+1}$
and $i^\Tt_{0\beta}$ are $0$-deriving and are cofinal and elementary as maps between passive reducts \tu{(}that is, as maps $(M^\Tt_\delta)^{\passive}\to(M^\Tt_{\gamma+1})^{\passive}$ and $M^\passive\to(M^\Tt_\beta)^{\passive}$ respectively\tu{)}, but are not $0$-embeddings, and in fact
\[ i^\Tt_{0\delta}(\crit(F^M))=\crit(F^{M^\Tt_\delta})=\crit(E^\Tt_\gamma)=\crit(F(M^\Tt_{\gamma+1}))=\crit(F^{M^\Tt_\beta}),\]
but $i^\Tt_{0\beta}(\crit(F^M))>\crit(F(M^\Tt_\beta))$,
\item if $m=0$ \tu{(}so $n=0$\tu{)} then $F(M^\Tt_\beta)$ is Dodd-absent-sound and $\rho_{\D}^{M^\Tt_\beta}>0$ and $\rho_{\D}^M>0$,
\end{enumerate}
or
\item  we have:
\begin{enumerate}
\item $(0,\beta]^\Tt$ drops in some way, and in fact,
$\Tt$ drops in some way at $\gamma+1$, so $\alpha=\gamma+1$,
\item $n=0^-$,
\item $i^{*\Tt}_{\gamma+1,\beta}$ is $0$-deriving
and is cofinal and elementary as a map $(M^{*\Tt}_{\gamma+1})^{\passive}\to(M^\Tt_\beta)^{\passive}$, but is not a $0$-embedding, and in fact,
\[ i^{*\Tt}_{\gamma+1,\beta}(\crit(F^{M^{*\Tt}_{\gamma+1}}))>\crit(F^{M^{*\Tt}_{\gamma+1}})=\crit(E^\Tt_\gamma)=\crit(F^{M^\Tt_{\gamma+1}})=\crit(F^{M^\Tt_\beta}).\]
\item $i^{\Tt}_{\gamma+1,\beta}$ is a $0$-embedding,
\item $M^{*\Tt}_{\gamma+1}=\core_{\D}(M^\Tt_\beta)$
is Dodd-absent-sound,
and $i^{*\Tt}_{\gamma+1,\beta}$ is the Dodd-absent-core map,
\item $\rho_{\D}^{M^{*\Tt}_{\gamma+1}}=0=\rho_{\D}^{M^\Tt_\beta}$.
\end{enumerate}
\end{enumerate}
\end{enumerate}
\end{lem}

\begin{lem}\label{lem:0-deriving}
 Let $\pi:M\to N$ be $0$-deriving
 and $\kappa=\crit(F^M)$ and $\mu=\crit(F^N)$. Then:
 \begin{enumerate}
 \item\label{item:component_measures} For all $a\in[\lambda(F^M)]^{<\om}$
 and all $X\in\pow([\kappa]^{|a|})\cap M$,
 we have
 \[ a\in i^M_{F^M}(X)\iff \pi(a)\in i^N_{F^N}(\pi(X)\cap[\mu]^{|a|}).\]
 \item\label{item:pi``kappa}  $\pi``\kappa\sub\mu\leq\pi(\kappa)$, so
  $\mu\in\rg(\pi)$ iff $\mu=\pi(\kappa)$.
 \item\label{item:pi(X)=restriction_of_image_of_restriction} For all $n<\om$ and all $X\in\pow([\kappa]^{n})\cap M$, we have
 \[ \pi(X)=i^N_{F^N}(\pi(X)\cap[\mu]^n)\cap[\pi(\kappa)]^n.\]
 \end{enumerate}
 \end{lem}
 \begin{proof}
 Part \ref{item:component_measures}: This follows immediately from the definitions.
 
 Part \ref{item:pi``kappa}: Let $\gamma<\kappa$
 and  $X=\kappa\cut\gamma$.
 Then $X\neq\emptyset$, so $i^M_{F^M}(X)\neq\emptyset$, so
 \[ \emptyset\neq\pi(i^M_{F^M}(X))=i^N_{F^N}(\pi(X)\cap\mu),\]
 so $\pi(X)\cap\mu\neq\emptyset$,
 but $\pi(X)\cap\mu=(\pi(\kappa)\cut\pi(\gamma))\cap\mu=\mu\cut\pi(\gamma)$, and so
  $\pi(\gamma)<\mu$.
  
  Part \ref{item:pi(X)=restriction_of_image_of_restriction}:  We have $X=i^M_{F^M}(X)\cap[\kappa]^n$.
  So applying $\pi$ and the definitions,
  \[ \pi(X)=\pi(i^M_{F^M}(X))\cap[\pi(\kappa)]^n=i^N_{F^N}(\pi(X)\cap[\mu]^n)\cap[\pi(\kappa)]^n.\qedhere\]
\end{proof}

We now proceed to generalize the usual Shift Lemma to our context. This is somewhat complicated by having to deal with $0$-deriving embeddings.
We state 4 distinct variants (Lemmas \ref{lem:shift_lemma_(i)}--\ref{lem:shift_lemma_iv}), so that we don't need to deal with all possibilities at once.
\begin{lem}[Shift Lemma I]\label{lem:shift_lemma_(i)}
\label{item:case_m-lifting}
Let  $m<\om$ and $M,N$ be $m$-sound
 premice and $\pi:M\to N$ be $m$-lifting.
Let $P,Q$ be active  premice
 and $\sigma:P\to Q$ be
 such that either:
 \begin{enumerate}[label=\tu{(}\alph*\tu{)}]
 \item\label{item:case_short} $F^P,F^Q$ are short and $\sigma:P\to Q$ is $0$-deriving,
 or
 \item\label{item:case_long} $F^P,F^Q$ are long and $\sigma:P\to Q$ is $0$-lifting.
 \end{enumerate}
 Suppose that $F^P$ is an $M$-extender with
 $\dom(F^P)\leq\rho_m^M$
 and $F^Q$ is an $N$-extender with $\dom(F^Q)\leq\rho_m^N$.
 Suppose that $\pi\rest\dom(F^P)=\sigma\rest\dom(F^P)$.
 Let $\lambda=\lambda(F^P)$.
 Then there is a unique
 embedding
 \[\psi:\Ult_m(M,F^P)\to\Ult_m(N,F^Q) \]
 such that
 $\psi$ is $m$-lifting,
   $\psi\com i^{M,m}_{F^P}=i^{N,m}_{F^Q}\com\pi$,
  $\psi\rest\lambda=\sigma\rest\lambda$
  and
if \ref{item:case_long} holds, $\psi\rest\lambda^{+P}=\sigma\rest\lambda^{+P}$.
Moreover,
 \[ \psi\Big([a,f_{t,q}^M]^{M,m}_{F^P}\Big)=[\sigma(a),f_{t,\pi(q)}^{N}]^{N,n}_{F^Q}\]
for all $\rSigma_m$ Skolem terms $t$ and  $q\in M$, where  $f^M_{t,q}:_{\mathrm{p}}M\to M$ is the partial function where $f^M_{t,q}(x)=t^M(q,x)$ for $x\in M$, and likewise for $f^N_{t,\pi(q)}$,
and all $a\in[\lambda]^{<\om}$,
 and in fact all $a\in[\lambda^{+P}]^{<\om}$ in case \ref{item:case_long} holds.
\end{lem}
\begin{proof} Let $\kappa=\crit(F^P)$ and $\mu=\crit(F^Q)$.
\begin{case}\label{case:SLi_case_short_holds}Hypothesis \ref{item:case_short} holds;
so $F^P,F^Q$ are short.

Here one can just use the proof of the usual Shift Lemma (see \cite[Lemma 5.2]{fsit}),
 noting that for all $a\in[\lambda]^{<\om}$
 and all $X\in\pow([\kappa]^{|a|})^M$,
 we have
 \[ (X,a)\in F^P\iff (\sigma(X)\cap[\mu]^{|a|},\sigma(a))\in F^Q, \]
 by Lemma \ref{lem:0-deriving}, which is enough for the basic elementarity
 calculations.
 \end{case}
  \begin{case} Hypothesis \ref{item:case_long} holds;
  so $F^P,F^Q$ are long.

This is also via the usual proof,
  except when $m=0$,
  $M,N$ are active short
  and $\crit(F^M)=\crit(F^P)$,
  and hence (since $\pi,\sigma$ are both $0$-lifting) $\crit(F^N)=\crit(F^Q)$.
  This case is novel, since $F^M,F^N$ are short and $F^P,F^Q$ long, so
   we form $Y=\Ult_0(M,F^P)$
  and $Z=\Ult_0(N,F^Q)$ via  \ref{dfn:Ult_n}, avoiding the protomouse.
  But note that we still want to establish that $\psi$ (defined as usual) is $0$-lifting.
  Let $Y'$ and $Z'$ be the
  protomice associated to $Y$ and $Z$.
  Then certainly
  $\psi:Y'\to Z'$
  is $\Sigma_0$-elementary in the language with $\dot{\in},\dot{\es},\dot{F}$, and $\psi\com i^{M}_{F^P}=i^N_{F^Q}\com\pi$. But $F^Y=F^{Y'}\com (F^P\rest\lambda(F^P))$,  $F^Z=F^{Z'}\com (F^Q\rest\lambda(F^Q))$,  $F^P\rest\lambda(F^P)\in P$ and \[\psi(F^P\rest\lambda(F^P))=\sigma(F^P\rest\lambda(F^P))=F^Q\rest\lambda(F^Q), \]
  and it follows that $\psi:Y\to Z$ is $\Sigma_0$-elementary in the language with $\dot{\in},\dot{\es},\dot{F}$. It remains
  to see  $\psi(F_J^Y)=F_J^Z$.
  But if $M,N$ are type A then $F_J^Y=F^P\rest\lambda(F^P)$
  and $F_J^Z=F^Q\rest\lambda(F^Q)$,
  which suffices.
  And if $M,N$ are type B then $F_J^Y=i^{M}_{F^P}(F_J^M)\com (F^P\rest\lambda(F^P))$
  and $F_J^Z=i^N_{F^Q}(F_J^N)\com(F^Q\rest\lambda(F^Q))$,
  so $\psi(F_J^Y)=F_J^Z$.
  Finally if $M,N$ are type C then so are $Y,Z$, and $F_J^Y=\emptyset=F_J^Z$.\qedhere
\end{case}
 \end{proof}

\begin{lem}[Shift Lemma II]\label{lem:shift_lemma_(ii)}\label{item:case_0-deriving}
Let
 $M,N$ be active short premice and $\pi:M\to N$ be $0$-deriving.
Let $P,Q$ be active  premice
 and $\sigma:P\to Q$ where either:
 \begin{enumerate}[label=\tu{(}\alph*\tu{)}]
 \item\label{item:shift_lemma_(ii)_case_short} $F^P,F^Q$ are short and $\sigma:P\to Q$ is $0$-deriving,
 and if
 $\crit(F^P)=\crit(F^M)$
 then $\crit(F^Q)=\crit(F^N)$,
 or
 \item\label{item:shift_lemma_(ii)_case_long} $F^P,F^Q$ are long and $\sigma:P\to Q$ is $0$-lifting.
 \end{enumerate}
 Suppose that $F^P$ is an $M$-extender with
 $\dom(F^P)\leq\rho_0^M$
 and $F^Q$ is an $N$-extender with $\dom(F^Q)\leq\rho_0^N$.
 Suppose that:
 \begin{enumerate}[label=--]
 \item if $F^P,F^Q$ are short then
 $\pi(X)\cap\crit(F^Q)=\sigma(X)\cap\crit(F^Q)$ for all $X\in\dom(F^P)$ \tu{(}so $\pi(\crit(F^P))\geq\crit(F^Q)$\tu{)}, and
 \item if $F^P,F^Q$ are long then $\pi\rest\dom(F^P)=\sigma\rest\dom(F^P)$.
 \end{enumerate}
 Let $\lambda=\lambda(F^P)$.
 Then there is a unique
 embedding
 \[\psi:\Ult_0(M,F^P)\to\Ult_0(N,F^Q) \]
 such that
$\psi$ is $0$-deriving,
  $\psi\com i^{M,0}_{F^P}=i^{N,0}_{F^Q}\com\pi$, $\psi\rest\lambda=\sigma\rest\lambda$
  and
if \ref{item:shift_lemma_(ii)_case_long} holds, $\psi\rest\lambda^{+P}=\sigma\rest\lambda^{+P}$.
 Moreover,
 \[\psi\Big([a,f]^{M,0}_{F^P}\Big)=[\sigma(a),\pi(f)]^{N,0}_{F^Q}\]
 for all $f\in M$, all $a\in[\lambda]^{<\om}$,
 and in fact all $a\in[\lambda^{+P}]^{<\om}$ in case \ref{item:shift_lemma_(ii)_case_long} holds.\end{lem}
\begin{proof}

  \begin{casetwo} Hypothesis \ref{item:shift_lemma_(ii)_case_short} holds; so $F^P,F^Q$ are short.

 Suppose   $\crit(F^P)=\crit(F^M)$, so by  \ref{item:shift_lemma_(ii)_case_short},
 also  $\crit(F^Q)=\crit(F^N)$. Let $U_M=\Ult_0(M,F^P)$
 and $U_N=\Ult_0(N,F^Q)$.
We have  $\psi``\crit(F^{U_M})\sub\crit(F^{U_N})$, since
 $\crit(F^{U_M})=\lambda(F^P)$
 and $\crit(F^{U_N})=\lambda(F^Q)$
 and $\sigma``\lambda(F^P)\sub\lambda(F^Q)$
 and $\sigma\rest\lambda(F^P)\sub\psi$. Now the commuting diagram required for $\psi$ to be $0$-deriving results from the corresponding diagram for $\pi$ and the fact that every $A\in U_M|\crit(F^{U_M})^{+U_M}$
 is of form $i^{M,0}_{F^P}(f)(a)$
 for some $f\in M|\crit(F^M)^{+M}$
 and $a\in[\lambda(F^P)]^{<\om}$.
 That is, because $\pi$ is $0$-deriving,
 letting $\mu'=\crit(F^{U_N})$,
 for all $X\sub\crit(F^{U_M})$ with $X\in\rg(i^{M,0}_{F^P})$, we have
 \begin{equation}\label{eqn:0-deriving_comm}\psi(i^{U_M,0}_{F^{U_M}}(X))=i^{U_N,0}_{F^{U_N}}(\psi(X)\cap[\mu']^{<\om}).\end{equation}
 But then because every
 $X\in U_M\cap\pow(F^{U_M})$ is
 of the form $X=i^{M,0}_{F^P}(f)(a)$
 for some $f\in M$ and $a\in[\lambda(F^P)]^{<\om}$, and so $\psi(a)\in U_N|\mu'$, it follows that
 line (\ref{eqn:0-deriving_comm})
 also holds for all such $X$.

If instead $\crit(F^P)\neq\crit(F^M)$ then it is similar or easier; if $\crit(F^P)<\crit(F^M)$
 then  note that  $\crit(F^Q)\leq\sigma(\crit(F^P))<\crit(F^N)$, since by Lemma \ref{lem:0-deriving},
 we have $\sup\pi``\crit(F^M)\leq\crit(F^N)$.
 \end{casetwo}
\begin{casetwo} Hypothesis \ref{item:shift_lemma_(ii)_case_long} holds.

If $\crit(F^P)\neq\crit(F^M)$ then things are much as above.
So suppose $\crit(F^P)=\crit(F^M)$.
If $\crit(F^N)=\pi(\crit(F^M))$ then it is easy (we don't need to consider $F_J^M$, as no consideration of this is made in the definition of \emph{$0$-deriving}).
So, using Lemma \ref{lem:0-deriving},
we may assume that $\sup\pi``\kappa\leq\mu<\pi(\kappa)$ where $\kappa=\crit(F^M)$ and $\mu=\crit(F^N)$.
So $\Ult_0(M,F^P)$ is formed to avoid the protomouse, whereas $\Ult_0(N,F^Q)$ is formed directly. We get $\psi(\kappa)=\pi(\kappa)>\mu=\crit(F^N)=\crit(F^Z)$. Considering these things, it is now just a small diagram chasing computation to see
that $\psi$ is $0$-deriving; Lemma \ref{lem:0-deriving} part \ref{item:pi(X)=restriction_of_image_of_restriction} is useful. We leave the details to the reader.\qedhere
 \end{casetwo}
\end{proof}
\begin{lem}[Shift Lemma III]\label{lem:shift_lemma_iii}
Let  $M,N$ be active short premice
and $\pi:M\to N$ be $0$-deriving
with $\pi(\kappa)>\crit(F^N)$
where $\kappa=\crit(F^M)$.
Let $P,Q$ be active short premice
 and $\sigma:P\to Q$ be
  $0$-lifting.
 Suppose  $F^P$ is an $M$-extender with
 $\kappa=\crit(F^P)$
 and $F^Q$ is an $N$-extender.
 Suppose that $\pi\rest(M|\kappa^{+M})=\sigma\rest(M|\kappa^{+M})$.
 Then there is a unique $0$-deriving
 embedding
 \[\psi:P\to\Ult_0(N,F^Q) \]
 such that $\psi\rest\lambda=\sigma\rest\lambda$ where $\lambda=\lambda(F^P)$.
\end{lem}
\begin{proof}
 First let $W=\Ult_0(M,F^P)$ and $X=\Ult_0(N,F^Q)$  and define
 $\tau: W\to X$
 via the Shift Lemma, so
 \[ \tau\com i^{M,0}_{F^P}=i^{N,0}_{F^Q}\com\pi \]
 and $\tau\rest P^{\passive}=\sigma$.
 Although $\tau$ is fully elementary
 as a map $W^{\passive}\to X^{\passive}$,
 it is not $0$-deriving
 (by Lemma \ref{lem:0-deriving} part \ref{item:pi``kappa}, and since $\kappa<\crit(F^W)=\lambda$ but
 $\tau(\kappa)=\sigma(\kappa)=\pi(\kappa)>\crit(F^N)=\crit(F^X)$).
 We have $\lambda=\crit(F^W)$ and $P^{\passive}=W|\lambda^{+W}$.
 So define
 $\psi:P\to X$ by \[\psi=\tau\com i^{P,0}_{F^W}=\tau\com i^{W|\lambda^{+W},0}_{F^W}.\] Then since $\crit(i^{P,0}_{F^W})=\lambda$,
 we have $\psi\rest\lambda=\tau\rest\lambda=\sigma\rest\lambda$. (But $\psi(\lambda)>\sigma(\lambda)$.) Clearly $\psi$ is elementary as a map $P^{\passive}\to X^{\passive}$. So to see that $\psi$ is $0$-deriving, we have to verify that for all
$A\sub\kappa$ with \[A\in P|\kappa^{+P}=M|\kappa^{+M},\]
we have $\psi(i^{P|\kappa^{+P},0}_{F^P}(A))=i^{X,0}_{F^X}(\psi(A)\cap\mu)$, where $\mu=\crit(F^X)=\crit(F^N)$.
Well,
\[\begin{split}
\psi(i^{P|\kappa^{+P},0}_{F^P}(A))&=\tau(i^{W|\lambda^{+W},0}_{F^W}(i^{P|\kappa^{+P},0}_{F^P}(A)))\\
&=\tau(i^{M,0}_{F^P}(i^{M|\kappa^{+M},0}_{F^M}(A)))\\
&=i^{N,0}_{F^Q}(\pi(i^{M|\kappa^{+M},0}_{F^M}(A)))\\
&=i^{N,0}_{F^Q}(i^{N|\mu^{+N},0}_{F^N}(\pi(A)\cap \mu))\\
&=i^{X,0}_{F^X}(\pi(A)\cap\mu),\end{split}\]
the latter being because $\crit(F^Q)=\pi(\kappa)>\mu$. Since $\psi(A)=\sigma(A)=\pi(A)$,  we are done.
\end{proof}

The following further variant
is easy to see:

\begin{lem}[Shift Lemma IV]\label{lem:shift_lemma_iv}
 Let $m<\om$ and let $M,N$ be $m$-sound
 premice
 and $\pi:M\to N$ be $m$-lifting.
 Let $P,Q$ be active short premice
 and $\sigma:P\to Q$ be $0$-deriving.
 Let $\kappa=\crit(F^P)$ and $\mu=\crit(F^Q)$.
 Suppose that $F^P$ is an $M$-extender
 with $\kappa^{+M}\leq\rho_m^M$
 and $F^Q$ is an $N$-extender
 with $\mu^{+N}\leq\rho_m^N$.
 Suppose that $\pi(X)\cap\mu=\sigma(X)\cap\mu$
  for all $X\in\pow(\kappa)\cap M$.
  Then there is a unique embedding
  \[ \psi:\Ult_m(M,F^P)\to\Ult_m(N,F^Q)\]
  such that
 $\psi$ is $m$-lifting,
 $\psi\com i^{M,m}_{F^P}=i^{N,m}_{F^Q}\com\pi$ and
 $\psi\rest\lambda=\sigma\rest\lambda$
  where $\lambda=\lambda(F^P)$.
Moreover,
\[ \psi\Big([a,f^M_{t,q}]^{M,m}_{F^P}\Big)=[\sigma(a),f^N_{t,\pi(q)}]^{N,n}_{F^Q}\]
for all $a,t,q$ as in Shift Lemma I \tu{(}\ref{lem:shift_lemma_(i)}\tu{)}.
\end{lem}

We next discuss a version of the copying construction
 starting with $0$-deriving embeddings
$\pi:M\to N$ where $M,N$ are active short and $\pi(\crit(F^M))>\crit(F^N)$,
using the  Shift Lemmas above,
and in particular Shift Lemma III \ref{lem:shift_lemma_iii}. All extenders used in $\Tt$ will lift to copies used in $\Uu$, but in certain situations, $\Uu$ also uses extra extenders.
 These situations correspond to applications of Shift Lemma III.
 Here we will have a copy map $\pi_\beta:M^\Tt_\beta\to M^\Uu_\beta$ and some $\alpha\geq\beta$ with $\pred^\Tt(\alpha+1)=\beta$
 and $E^\Tt_\alpha$  short
 with $\crit(E^\Tt_\alpha)=\crit(F(M^\Tt_\beta))$ and $E^\Tt_\alpha$ being $M^\Tt_\beta$-total,
 and we will have $\pi_\beta(\crit(F(M^\Tt_\beta)))>\crit(F(M^\Uu_\beta))$,
 but $E^\Tt_\alpha$
 not the image of $F^M$,
and we will have a $0$-lifting embedding $\exit^\Tt_\alpha\to\exit^\Uu_\alpha$. Here we ``insert'' $E^\Uu_\alpha$,
and $\Uu$ has an extra node $(\alpha,0)$, indexed  with $\alpha<(\alpha,0)<(\alpha+1)$ (note this is just the index order, not tree order).
We apply Shift Lemma III
to produce a $0$-deriving embedding $\psi_\alpha:\exit^\Tt_\alpha\to M^\Uu_{\alpha 0}=\Ult_0(M^\Uu_\beta,E^\Uu_\alpha)$,
and then proceed with copying,
using $E^\Tt_\alpha$ and its copy $E^\Uu_{\alpha 0}=F(M^\Uu_{\alpha 0})$.
The set $S$ in \ref{dfn:0-deriving_copying_con} below is the set of all $\alpha$ at which this situation occurs.

\begin{dfn}[Copying construction for $0$-deriving embeddings]\label{dfn:0-deriving_copying_con}
 Let $M,N$ be  active short premice.
 Let $\pi:M\to N$ be $0$-deriving.
Let $\Tt$ be $0^-$-maximal on $M$
and $\Uu$ be $0^-$-maximal on $N$.
Then we say that $\Uu$ is the \emph{$\pi$-copy} of $\Tt$ (written $\Uu=\pi\Tt$), as witnessed by \emph{copy maps} \[\pi_\alpha:M^\Tt_\alpha\to M^\Uu_\alpha\]
for $\alpha<\lh(\Tt)$
and \emph{shift set} $S$ and \emph{shifted copy maps}
\[ \psi_\alpha:\exit^\Tt_\alpha\to M^\Uu_{\alpha0}
\]
for $\alpha\in S$,
in the following circumstances:
\begin{enumerate}
 \item $S\sub\lh(\Tt)^-=\{\alpha\bigm|\alpha+1<\lh(\Tt)\}$, and $\Uu$ is indexed by $\lh(\Tt)\cup (S\cross\{0\})$, with indices ordered as $\alpha<(\alpha,0)<\alpha+1$ when $\alpha\in S$.
 \item $({<^{\Uu}}\rest\lh(\Tt))={<^{\Tt}}$
 and for every $\alpha\in S$, $(\alpha,0)$ is an end-node of $<^{\Uu}$ (that is, there is no $x\in\lh(\Tt)\cup (S\cross\{0\})$ such that $(\alpha,0)<^{\Uu}x$).
 \item Letting $\alpha+1<\lh(\Tt)$
 and $\beta=\pred^\Tt(\alpha+1)$,
 we have
  $\alpha\in S$ iff
  \begin{enumerate}[label=--]
  \item $\pi(\crit(F^M))>\crit(F^N)$,\footnote{So if $\pi(\crit(F^M))=\crit(F^N)$ then $S=\emptyset$, and everything is much more routine.}
   \item 
$[0,\alpha+1]^\Tt\cap\dropset^\Tt=\emptyset$
(so in particular,
$[0,\beta]^\Tt\cap\dropset^\Tt=\emptyset$),
\item 
  $E^\Tt_\alpha$ is short,
  \item 
  $\crit(E^\Tt_\alpha)=\crit(F(M^\Tt_\beta))$,
  and
  \item either $[0,\alpha]^\Tt\cap\dropset^\Tt\neq\emptyset$
  or $E^\Tt_\alpha\neq F(M^\Tt_\alpha)$.
  \end{enumerate}
 \item For all $\alpha<\lh(\Tt)$, we have
 $[0,\alpha]^\Tt\cap\dropset^\Tt=\emptyset$
 iff $[0,\alpha]^\Uu\cap\dropset^\Uu=\emptyset$.
 \item For every $\alpha\in S$,
 we have $[0,(\alpha,0)]^\Uu\cap\dropset^\Uu=\emptyset$.
 \item If $\alpha<\lh(\Tt)$ and $[0,\alpha]^\Tt\cap\dropset^\Tt=\emptyset$ then:
 \begin{enumerate}
 \item $\pi_\alpha:M^\Tt_\alpha\to M^\Uu_\alpha$
 is $0$-deriving,
 \item 
$\pi_\alpha\com i^{\Tt}_{\beta\alpha}=i^\Uu_{\beta\alpha}\com\pi_\beta$
for all $\beta\leq^\Uu\alpha$.\footnote{So $\pi_\beta(\crit(F(M^\Tt_\beta)))>\crit(F(M^\Uu_\beta))$ iff $\pi(\crit(F^M))>\crit(F^N)$.}
 \end{enumerate}
 \item If $\alpha<\lh(\Tt)$ and $[0,\alpha]^\Tt\cap\dropset^\Tt\neq\emptyset$ and $d=\deg^\Tt_\alpha$ and $d'=\max(d,0)$ then:
 \begin{enumerate}
  \item  $d=\deg^\Uu_\alpha$,
 \item $\pi_\alpha:M^\Tt_\alpha\to M^\Uu_\alpha$
 is a near $d'$-embedding,\footnote{Note that this uses the argument of \cite{fs_tame} for propagating near embeddings; this uses the inductive hypothesis
 \ref{item:fstame_inductive} to  keep
 things going.}
 \item letting $\beta+1\leq^\Tt\alpha$
 be least such that $(\beta+1,\alpha]^\Tt$
 does not drop in model, degree or Dodd-degree, and $\gamma=\pred^\Tt(\beta+1)$, then:
 \begin{enumerate}
 \item  $\beta+1\leq^\Uu\alpha$
 is also least such that $(\beta+1,\alpha]^\Uu$ does not drop in model,
 degree or Dodd-degree, and $\gamma=\pred^\Uu(\beta+1)$, 
 \item $\pi_\gamma(M^{*\Tt}_{\beta+1})=M^{*\Uu}_{\beta+1}$,
 \item for all $\xi\in[\beta+1,\alpha]^\Tt$
 we have $\pi_\alpha\com i^\Tt_{\xi\alpha}=i^\Uu_{\xi\alpha}\com\pi_\xi$,
 \item $\pi_\alpha\com i^{*\Tt}_{\beta+1,\alpha}=i^{*\Uu}_{\beta+1,\alpha}\com(\pi_\gamma\rest M^{*\Tt}_{\beta+1})$,
 \end{enumerate}
 \item\label{item:fstame_inductive} for each $\varepsilon+1\leq^\Tt\alpha$
 with $(\varepsilon+1,\alpha]^\Tt\cap\dropset^\Tt=\emptyset$
 and $d'=\max(\deg^\Tt_{\varepsilon+1},0)$,
 for each $\rSigma_{d'+1}$ formula $\varphi$
 and each $q\in M^\Tt_\alpha$,
 there is $q'\in M^{*\Tt}_{\varepsilon+1}$
 and an $\rSigma_{d'+1}$ formula  $\varphi'$
 such that
 letting $\gamma=\pred^\Tt(\varepsilon+1)$ and $\mu=\spc(E^\Tt_{\varepsilon})$, for all $\xi<\mu$, we have
 \[ M^{\Tt}_{\alpha}\sats\varphi(q,\xi)\iff M^{*\Tt}_{\varepsilon+1}\sats\varphi'(q',\xi),\]
 and for all $\xi<\pi_\alpha(\mu)=\pi_\gamma(\mu)=\pi_{\varepsilon}(\mu)=\spc(E^\Uu_\varepsilon)$, we have
 \[ M^{\Uu}_{\alpha}\sats\varphi(\pi_\alpha(q),\xi)\iff M^{*\Uu}_{\varepsilon+1}\sats\varphi'(\pi_{\gamma}(q'),\xi).\]
 \end{enumerate}
 
  \item For each $\alpha+1<\lh(\Tt)$ (independent of whether $\alpha\in S$),
  we have  $E^\Uu_\alpha=\pi_\alpha(E^\Tt_\alpha)$ 
 (as usual, in case $E^\Tt_\alpha=F(M^\Tt_\alpha)$, this means that $E^\Uu_\alpha=F(M^\Uu_\alpha)$); note that therefore, we have
 \[ \pi_\alpha\rest\exit^\Tt_\alpha:\exit^\Tt_\alpha\to\exit^\Uu_\alpha \]
 and\[\lambda(E^\Uu_\alpha)=\pi_\alpha(\lambda(E^\Tt_\alpha)) \]
 (but   if   $\pi(\crit(F^M))>\crit(F^N)$ and $[0,\alpha]^\Tt\cap\dropset^\Tt=\emptyset$
 and $E^\Tt_\alpha=F(M^\Tt_\alpha)$ 
 then $\crit(E^\Uu_\alpha)>\pi_\alpha(\crit(E^\Tt_\alpha))$).
 \item If $\alpha\in S$ and $\beta=\pred^\Tt(\alpha+1)$ then $\psi_\alpha:\exit^\Tt_\alpha\to M^\Uu_{\alpha0}=\Ult_0(M^\Uu_\beta,E^\Uu_\alpha)$ is defined from $\pi_\alpha\rest\exit^\Tt_\alpha$ and $\pi_\beta$ via  Shift Lemma III \ref{lem:shift_lemma_iii},
$E^\Uu_{\alpha0}=F(M^\Uu_{\alpha 0})$,
 and $\pi_{\alpha+1}$ is defined
 via Shift Lemma II \ref{lem:shift_lemma_(ii)}\ref{item:shift_lemma_(ii)_case_short}.
 
 \item If $\alpha\notin S$ and either
  $[0,\alpha]^\Tt\cap\dropset^\Tt\neq\emptyset$ or
$E^\Tt_\alpha\neq F(M^\Tt_\alpha)$, then
  $M^{*\Uu}_{\alpha+1}=\pi_\gamma(M^{*\Tt}_{\alpha+1})$, $\deg^\Tt_{\alpha+1}=\deg^\Uu_{\alpha+1}$, and $\pi_{\alpha+1}$ is defined via Shift Lemma I \ref{lem:shift_lemma_(i)}
  or Shift Lemma II \ref{lem:shift_lemma_(ii)}
  as appropriate.
 \item If $\alpha\notin S$ and $[0,\alpha]^\Tt\cap\dropset^\Tt=\emptyset$ and $E^\Tt_\alpha=F(M^\Tt_\alpha)$, then
$\pi_{\alpha+1}$ is defined via Shift Lemma II
\ref{lem:shift_lemma_(ii)}\ref{item:shift_lemma_(ii)_case_short}.

  \item For all $\alpha<\beta<\lh(\Tt)$,
  \begin{enumerate}[label=--]
   \item if $\alpha\notin S$ then $\pi_\alpha\rest\exit^\Tt_\alpha\sub\pi_\beta$, and
   \item if $\alpha\in S$ then $\pi_\alpha\rest\lambda(E^\Tt_\alpha)
   \sub\psi_\alpha$
   and $\pi_\alpha(\lambda(E^\Tt_\alpha))<\psi_\alpha(\lambda(E^\Tt_\alpha))$ and $\psi_\alpha\rest\exit^\Tt_\alpha\sub\pi_\beta$.
  \end{enumerate}
\item At limit stages things are defined in the usual manner.\qedhere
\end{enumerate}
\end{dfn}

\begin{lem}[Copying construction for $0$-deriving embeddings]\label{lem:0-deriving_copying}
 Let $M,N$ be active short premice
 and $\pi:M\to N$ be $0$-deriving.
 Let $\Tt$ be a $0^-$-maximal tree on $M$.
 Then $\pi\Tt$ is well-defined and uniquely determined iff all models encountered are wellfounded \tu{(}just as for the standard copying construction\tu{)}. Thus, if $\Sigma$ is a $0^-$-maximal  iteration strategy for $N$, then the $\pi$-pullback $\Sigma_{\leftarrow\pi}$ of $\Sigma$ is a well-defined $0^-$-maximal strategy
 for $M$, and a putative $0^-$-maximal tree $\Tt$ on $M$ is via 
 $\Sigma_{\leftarrow\pi}$ iff $\pi\Tt$ is via $\Sigma$.
\end{lem}

We  now adapt weak Dodd-Jensen  \cite{wDJ} to $0$-deriving embeddings:
\begin{dfn}
 Let $M$ be a countable active short premouse.
 Let $\vec{x}=\left<x_n\right>_{n<\om}$
 be an enumeration of the universe of $M$.
 Say a $0^-$-maximal stack $\vec{\Tt}$ on $M$  is \emph{$(M,\vec{x})$-large} iff $M^{\vec{\Tt}}_\infty$ exists
 and there is a $0$-deriving embedding $\pi:M\to M^{\vec{\Tt}}_\infty$.
 
 Let $\Sigma$ be a $(0^-,\om_1+1)$-iteration strategy, or an optimal-$(0^-,\om_1,\om_1+1)^*$-strategy, for $M$. We say that $\Sigma$ has the \emph{weak Dodd-Jensen property with respect to $\vec{x}$} iff for all hereditarily  countable trees/stacks $\Tt$
 via $\Sigma$, if $\Tt$ is $(M,\vec{x})$-large
 then:
 \begin{enumerate}
  \item $[0,\infty]^\Tt$ does not drop in model (hence not in any way), and
  \item $i^\Tt_{0\infty}:M\to M^\Tt_\infty$
  is the lexicographically least $0$-deriving embedding $M\to M^\Tt_\infty$ with respect to $\vec{x}$.
 \end{enumerate}
 
 Often the particular choice of enumeration $\vec{x}$ is not  important, and we suppress it from the notation/terminology.
\end{dfn}

\begin{lem}
 Let  $M$ be a countable active short premouse and let $\vec{x}$ enumerate the universe of $M$ in ordertype $\omega$.
 Suppose that $M$ is $(0^-,\om_1,\om_1+1)^*$-iterable. Then there is a $(0^-,\om_1,\om_1+1)^*$-strategy for $M$ with the weak Dodd-Jensen property with respect to  $\vec{x}$.
\end{lem}
\begin{proof}
The standard proof (for which see \cite{wDJ})
adapts routinely,
using the key fact, by Lemma \ref{lem:0-deriving_copying}, that copying constructions
are well-behaved. We leave the details to the reader.
\end{proof}

\begin{dfn}
Let $M,\vec{x}$ be as above,
and suppose also that $M$ is Dodd-absent-sound
and $\rho_{\D}^M=0$.
For a degree $0$ iteration strategy $\Sigma$ for $M$,
we define \emph{weak Dodd-Jensen} (with respect to $\vec{x}$) as above, but consider only trees $\Tt$ such that $\deg^\Tt_\infty\neq 0^-$.
\end{dfn}
\begin{lem}
 Let  $M,\vec{x}$ be as above,
 and suppose also that $M$ is Dodd-absent-sound
 and $\rho_{\D}^M=0$.
 Suppose that $M$ is $(0,\om_1,\om_1+1)^*$-iterable. Then there is a $(0,\om_1,\om_1+1)^*$-strategy for $M$ with the weak Dodd-Jensen property with respect to  $\vec{x}$.
\end{lem}
\begin{proof}
This is essentially the usual proof, just noting that if $\Tt$ is a $0$-maximal tree on $M$ of successor length,
then $M^\Tt_\infty$ is Dodd-absent-sound iff $\deg^\Tt_\infty\neq 0^-$,
and in case $b^\Tt$ does not drop in any way,
$i^\Tt$ is a $0$-embedding.\footnote{The same does not work if $\rho_{\D}^M>0$,
since there can be $0$-maximal trees which do not drop in any way along $b^\Tt$,
but $i^\Tt$ is $0$-deriving but not $0$-maximal.}
\end{proof}
\begin{rem}
The two cases above handle the situations in which we will apply weak Dodd-Jensen for a $(d,\om_1,\om_1+1)^*$-strategy for a premouse $M$
which is active short
and $d\in\{0,0^-\}$.
(Thus, in some situations,
we will first reduce the case in which $M$ is active short with $\rho_{\D}^M>0$
to the case that $\rho_{\D}^M=0$.)
In other cases (in which either $d>0$ or $M$ is not active short)
everything is a much more direct generalization of the standard methods, so we won't discuss those further.
\end{rem}

\section{Fine structure from iterability}

We have now completed the development of the basic theory of premice and iteration trees at the level of $\kappa^+$-supercompactness. We now turn to the suite of comparison proofs, establishing the  basic fine structural properties --  Dodd-absent-solidity and Dodd-absent-universality  for short extenders, solidity and universality of the standard parameter, condensation, and the Mitchell-Steel initial segment condition for short extenders. The proofs have very similar structure to the corresponding ones from the short extender context, but as in Woodin's development, Neeman-Steel \cite{nsp1}, \cite{nsp1fs}, and Voellmer \cite{voellmer}, new issues and details must be dealt with. A large part of that here is like in those papers, but there are new details involved here, firstly in connection with the modified hierarchy, and secondly because we also prove various new facts whose analogues are not discussed in those papers (especially Theorems
\ref{tm:rho^M_D=0_implies_Dodd-absent-solid_and_universal}, \ref{tm:Dodd-absent-soundness},  \ref{tm:first_cond} under its hypothesis \ref{item:cond_moving_below_proj},  \ref{tm:second_cond}, and \ref{tm:MS-ISC}).
\subsection{Comparison}

We start with the basic comparison theorem, the argument for which underlies the later ones. The proof follows very much that of Woodin and Neeman-Steel \cite{nsp1}, \cite{nsp1fs}.
\begin{tm}[Basic comparison]\label{tm:comparison}
 Let $m,n\in(\om+1)\cup\{0^-\}$, and let $M,N$ be countable $m$-sound and $n$-sound premice respectively. Suppose that if $m\neq 0^-$ \tu{(}$n\neq 0^-$\tu{)} then $M$ \tu{(}$N$\tu{)} is Dodd-absent-sound.
 Let $\Sigma$ be an $(m,\om_1+1)$-strategy for $M$
 and $\Gamma$ be an $(n,\om_1+1)$-strategy for $N$.
 Then there are countable successor length trees $\Tt,\Uu$ on $M,N$ via $\Sigma,\Gamma$ respectively, such that either:
 \begin{enumerate}[label=--]
  \item $b^\Tt$ does not drop in model, degree or Dodd-degree, and $M^\Tt_\infty\ins M^\Uu_\infty$, or
  \item $b^\Uu$ does not drop in model,
  degree or Dodd-degree,
  and $M^\Uu_\infty\ins M^\Tt_\infty$.
 \end{enumerate}
\end{tm}
\begin{proof}
 We compare by least disagreement,
 noting that we are forming $m$-maximal and $n$-maximal trees $\Tt,\Uu$ respectively,
 as defined in Definition \ref{dfn:n-max},
 in particular using the iteration rules there
 and forming ultrapowers according to Definitions \ref{dfn:Ult_n} and \ref{dfn:0^-}.
 
 Suppose that the comparison lasts through $\om_1+1$ stages. Let $\pi:X\to V_\gamma$
 be elementary, with $X$ countable and transitive, and $(\Tt,\Uu)\in\rg(\pi)$. Let $\kappa=\crit(\pi)$,
 so $\kappa=\om_1^X<\om_1$.
 The usual calculations show that $\kappa\in[0,\om_1]^\Tt\cap[0,\om_1]^\Uu$,
 that $(\kappa,\om_1]^\Tt\cap\dropset^\Tt=\emptyset$
 and $(\kappa,\om_1]^\Uu\cap\dropset^\Uu=\emptyset$,
 $\deg^\Tt_{\kappa}=\deg^\Tt_{\omega_1}$,
 $\deg^\Uu_{\kappa}=\deg^\Uu_{\omega_1}$,
 and 
 that $i^\Tt_{\kappa\om_1}\sub\pi$ and $i^\Uu_{\kappa\om_1}\sub\pi$. Note that
 \[M^\Tt_\kappa|\kappa^{+M^\Tt_\kappa}=M^\Tt_{\omega_1}|\kappa^{+M^\Tt_{\omega_1}}=M^\Uu_{\omega_1}|\kappa^{+M^\Uu_{\omega_1}}=M^\Uu_\kappa|\kappa^{+M^\Uu_\kappa}.\]
 
 Let $\alpha+1=\min((\kappa,\om_1]^\Tt)$
 and $\beta+1=\min((\kappa,\om_1]^\Uu)$.
 The usual argument 
 with Jensen ISC shows that at least one of $E^\Tt_\alpha,E^\Uu_\beta$ is long. Say $E^\Tt_\alpha$ is long. Then $E^\Tt_\alpha\rest\lambda(E^\Tt_\alpha)\in M^\Tt_{\alpha+1}$, and $i^\Tt_{\alpha+1,\om_1}(E^\Tt_\alpha\rest\lambda(E^\Tt_\alpha))=E\rest\om_1$, where $E$ is the extender derived from $i^\Tt_{\kappa\om_1}$ (that is, $E\rest\om_1$ is the short extender derived from this embedding).
 If $E^\Uu_\beta$ were short, then $E^\Uu_\beta\notin M^\Uu_{\om_1}$,
 but by the compatibility, then $E^\Uu_\beta=E\rest\lambda(E^\Uu_\beta)$,
so this \emph{is} in $M^\Uu_{\om_1}$, contradiction.
So $E^\Uu_\beta$ is also long. Since both $E^\Tt_\alpha,E^\Uu_\beta$ are long, it now follows that
\[M^\Tt_\kappa||\kappa^{++M^\Tt_\kappa}=M^\Tt_{\omega_1}|\kappa^{++M^\Tt_{\omega_1}}=M^\Uu_{\omega_1}|\kappa^{++M^\Uu_{\omega_1}}=M^\Uu_\kappa||\kappa^{++M^\Uu_\kappa}, \]
and so writing $\theta=\omega_1$, also
\begin{equation}\label{eqn:agmt_thru_theta^++} M^\Tt_{\theta}||\theta^{++M^\Tt_\theta}=M^\Uu_{\theta}||\theta^{++M^\Uu_\theta}.\end{equation}

Suppose $E^\Tt_\alpha$ has a largest generator $\nu^-$. Then $E^\Tt_\alpha\rest\nu^-\in M^\Tt_{\alpha+1}$, by the long ISC. So $i^\Tt_{\alpha+1,\om_1}(E^\Tt_\alpha\rest\nu^-)=E\rest i^\Tt_{\alpha+1,\om_1}(\nu^-)\in M^\Tt_{\om_1}$. But $E\rest (i^\Tt_{\alpha+1,\om_1}(\nu^-)+1)\notin M^\Tt_{\om_1}$, because from this extender
we can derive $E\rest (\lambda(E^\Tt_\alpha)\cup\{i^\Tt_{\alpha+1,\om_1}(\nu^-)\})$,
which is equivalent to $E^\Tt_\alpha$,
which is not in $M^\Tt_{\alpha+1}$,
and hence not in $M^\Tt_{\om_1}$,
since $\pow(\lambda(E^\Tt_\alpha))\cap M^\Tt_{\alpha+1}=\pow(\lambda(E^\Tt_\alpha))\cap M^\Tt_{\om_1}$.

Suppose now that $E^\Uu_\beta$ also has a largest generator $\widetilde{\nu}^-$.
Then by the preceding paragraph applied also to $\Uu$, we have $i^\Uu_{\beta+1,\om_1}(\widetilde{\nu}^-)=i^\Tt_{\alpha+1,\om_1}(\nu^-)$; denote this ordinal $\xi$.
Let $\lambda=\min(\lambda(E^\Tt_\alpha),\lambda(E^\Uu_\beta))$. Then (much as in the previous paragraph) $E\rest(\lambda\cup\{\xi\})\notin M_{\om_1}^\Tt$, and also $\notin M^\Uu_{\om_1}$ (also after coding this extender naturally with a subset of $\lambda$). By the long Jensen ISC,
it follows that $\lambda(E^\Tt_\alpha)=\lambda(E^\Uu_\beta)$.
But now note that $E^\Tt_\alpha=E^\Uu_\beta$,
which is impossible.

So $\nu(E^\Uu_\beta)$ is a limit ordinal.
Let $\eta=\sup i^\Uu_{\beta+1,\om_1}``\nu(E^\Uu_\beta)$. Note that $\eta\leq\theta^{+M^\Uu_\theta}=\theta^{+M^\Tt_\theta}$, and for each $\xi<\eta$,
we have $E\rest\xi\in M^\Uu_{\om_1}$.
Therefore $\eta\leq i^\Tt_{\alpha+1,\om_1}(\nu^-)$, and therefore $E\rest\eta\in M^\Uu_{\om_1}$ (cf.~line (\ref{eqn:agmt_thru_theta^++})).

But  $E^\Uu_\beta\rest\nu(E^\Uu_\beta)\notin M^\Uu_{\beta+1}$, and now the argument in \cite[***]{extmax} shows that it follows
that $E\rest\eta\notin M^\Uu_{\om_1}$, a contradiction.

So by symmetry,  both $\nu(E^\Tt_\alpha),\nu(E^\Uu_\beta)$ are limit ordinals.
By the preceding analysis,
letting $\eta$ be as above, we also have that $\eta=\sup i^\Tt_{\alpha+1,\om_1}``\nu(E^\Tt_\alpha)$. Note also that $\nu(E^\Tt_\alpha)=\delta^{+M^\Tt_{\alpha+1}}$
where $\delta=\lambda(E^\Tt_\alpha)$,
and likewise $\nu(E^\Uu_\beta)=\mu^{+M^\Uu_{\beta+1}}$ where $\mu=\lambda(E^\Uu_\beta)$.
Let $\alpha'+1=\min((\alpha+1,\om_1]^\Tt)$
and $\beta'+1=\min((\beta+1,\om_1]^\Uu)$.
Then by our rules on (not) moving generators,
$E^\Tt_{\alpha'},E^\Uu_{\beta'}$ are both long.
Therefore the short part $E^\Tt_{\alpha'}\rest\lambda(E^\Tt_{\alpha'})$ of $E^\Tt_{\alpha'}$ is indexed
on $\es^{M^\Tt_{\alpha'+1}}$, at $\chi=\sup i_{E^\Tt_{\alpha'}}``\delta^{+M^\Tt_{\alpha+1}}$,
and likewise for $\Uu,\beta',\mu$.
But $\crit(i^\Tt_{\alpha'+1,\om_1})=\lambda(E^\Tt_{\alpha'})$,
and $\cof^{M^\Tt_{\alpha'+1}}(\chi)=\delta^{+M^\Tt_{\alpha+1}}=\delta^{+M^\Tt_{\alpha'+1}}$,
and so $i^\Tt_{\alpha'+1,\om_1}$ is continuous at $\chi$; likewise for $\Uu$ etc,
and let $\zeta$ be the index of the short part of $E^\Uu_{\beta'}$.
It follows that \[i^\Tt_{\alpha'+1,\om_1}(\chi)=\eta=i^\Uu_{\beta'+1,\om_1}(\zeta).\]
Since $i^\Tt_{\alpha'+1,\om_1}$ and $i^\Uu_{\beta'+1,\om_1}$ do not shift the critical points of the extenders indexed at $\chi,\zeta$ respectively,
therefore those extenders have the same critical points, which are
 $\lambda(E^\Tt_\alpha)$ and $\lambda(E^\Uu_\beta)$ respectively, so 
 $\lambda(E^\Tt_\alpha)=\lambda(E^\Uu_\beta)$.
But we now get $E^\Tt_\alpha=E^\Uu_\beta$,
because they are both reduced to $E$
via the (common) extender indexed at $i^\Tt_{\alpha'+1,\om_1}(\chi)=i^\Uu_{\beta'+1,\om_1}(\zeta)$. This is a contradiction, showing that the comparison terminates.

Lemma \ref{lem:dropped_iterated_core_embedding}
shows that if $b^\Tt$ drops in model, degree or Dodd-degree then $M^\Tt_\infty$ is not sound or not Dodd-absent-sound; likewise for $\Uu$.

So now suppose that $M^\Tt_\infty\pins M^\Uu_\infty$. Then $M^\Tt_\infty$ is sound and  Dodd-absent-sound,
so $b^\Tt$ does not drop (in model, degree or Dodd-degree),
so we are done. Likewise if $M^\Uu_\infty\pins M^\Tt_\infty$.

So  suppose  $M^\Tt_\infty=M^\Uu_\infty$.
We may assume
both $b^\Tt$ and $b^\Uu$ drop in model, degree or Dodd-degree. But then the usual core embedding argument,
but now using Lemma \ref{lem:dropped_iterated_core_embedding},
gives a contradiction.
(That is, we have $k=\deg^\Tt_\infty=\deg^\Uu_\infty$,
and $k\in\om\cup\{0^-\}$.
If $k\in\om$ then the $(k+1)$-core embedding $\core_{k+1}(M^\Tt_\infty)\to M^\Tt_\infty$,
or equivalently, $\core_{k+1}(M^\Uu_\infty)\to M^\Uu_\infty$,
is an iteration map for both $\Tt$ and $\Uu$,
leading to the usual contradiction.
If $k=0^-$,
then it is similar, but using the Dodd-absent core and Dodd-absent core embedding.)
\end{proof}

\subsection{Dodd solidity}

In this section we prove results establishing that  active short mice $M$ have nice Dodd-absent structure. We first deal with the case that the Dodd-absent projectum $\rho_{\D}^M=0$; recall that this means that some measure derived from $F^M$ is not in $M$. In Theorem
\ref{tm:rho^M_D=0_implies_Dodd-absent-solid_and_universal}, we prove that every such mouse $M$ is Dodd-absent-solid and Dodd-absent-universal.
In Theorems
\ref{tm:Dodd-absent-core_of_M_when_rho_D^M=0} and
\ref{tm:1-core_of_Dodd-abs-sound} together,
we show that  $\core_1(M)$ is also Dodd-absent-sound (and in \S\ref{sec:solidity}) we  will show that it is also $1$-sound).
Recall here that $\core_1(M)$
is defined as $\core_1(\core_{\D}(M))$;
Theorem \ref{tm:Dodd-absent-core_of_M_when_rho_D^M=0} alone
first establishes the relevant properties for $\core_{\D}(M)$.
After this, in Theorem
\ref{tm:Dodd-absent-soundness},
we deal with the case that $\rho_{\D}^M>0$, so $\rho_{\D}^M>\kappa^{+M}$ where $\kappa=\crit(F^M)$.
Here we also assume that $M$ is $1$-sound, and we will deduce that $M$ is Dodd-absent-sound.

\begin{rem}\label{rem:standard_Dodd_proj_and_param}
 Let $M$ be an active short premouse
 and $\kappa=\crit(F^M)$.
 Recall that $\tau^M$ is the least $\tau\geq\kappa^{+M}$
 such that $F^M$ is generated by $\tau\cup t$ for some $t\in[\OR^M]^{<\om}$,
 and $t^M$ is the least $t\in[\OR^M]^{<\om}$ such that $F^M$ is generated by $\tau^M\cup t$.\end{rem}

 \begin{lem}\label{lem:tau^M=max(rho_1,kappa^+)}
  Let $M$ be a $1$-sound active short premouse and $\kappa=\crit(F^M)$. Then $\tau^M=\max(\kappa^{+M},\rho_1^M)$.
  In particular, if $\nu(F^M)$
  is a limit ordinal then $\nu(F^M)=\tau^M=\rho_1^M>\kappa^{+M}$.
 \end{lem}
\begin{proof}
We have $\tau^M\geq\kappa^{+M}$ by definition.
 Using the $1$-soundness of $M$,
 we get that $F^M$ is generated by
 $\rho_1^M\cup\{x\}$ for some $x$,
 and so $\tau^M\leq\max(\rho_1^M,\kappa^{+M})$. But if $\tau^M<\rho_1^M$
 then  $F^M\rest(\tau^M\cup t^M)\in M$,
 which is impossible.
\end{proof}

\begin{lem}\label{lem:condensation_for_self-solid}
 Let  $M,H$ be premice of the same kind
 and $\pi:H\to M$ be $0$-lifting and c-preserving.
 Suppose that $M$ is $(0^-,\om_1,\om_1+1)^*$-iterable.
 Suppose that 
  $\rho=\rho_1^H<\rho_1^M$,
  $\rho$ is an $M$-cardinal,
   $\crit(\pi)=\rho^{+H}<\rho^{+M}$,
   and there is $x\in H$ such that $H=\Hull_1^H(\rho\cup\{x\})$.
Let $J\pins M$ be such that $\rho^{+H}=\rho^{+J}$ and $\rho_\om^J=\rho$.
Then $\Th_{\rSigma_1}^H(\rho\cup\{x\})$
is $\bfrSigma_1^J$, after coding
this theory naturally as a subset of $\rho$.
\end{lem}
\begin{proof}
We may assume that $M$ is countable,
and fix a $(0^-,\om_1+1)$-strategy $\Sigma$
for $M$ with weak Dodd-Jensen.
 Compare the phalanx $\ph=((M,{<\rho}),H)$
 versus $M$, lifting ($0^-$-maximal) trees  $\Uu$ on  $\ph$ to ($0^-$-maximal) trees 
 $\Uu^*$ on $M$ via $\Sigma$.
 Here the ``${<\rho}$'' refers to the \emph{space} of extenders $E$: if $\spc(E)<\rho$ then $E$ applies to $M$,
 and otherwise to $H$ or a later model of the tree. In particular, if $\rho=\kappa^{+M}=\kappa^{+H}$ and $E$ is long with $\crit(E)=\kappa$, then $E$ will \emph{not} apply to $M$ (in fact it will apply to $H$).\footnote{Actually, we could have been a little more careful here,
  applying such a long extender to $J$,
  at the appropriate degree.  But for the present argument don't need to.}
 Let $(\Uu,\Tt)$ be the resulting comparison
 of $(\ph,M)$ (which succeeds in countably many stages, by a slight variant of the
 proof of comparison termination within the proof of Theorem \ref{tm:comparison}).
 
 The usual arguments with weak-Dodd-Jensen
 yield that $b^\Uu$ is above $H$,
 is non-model-dropping (hence does not drop in degree, nor Dodd-degree), and $M^\Uu_\infty\ins M^\Tt_\infty$. For every $\alpha+1<\lh(\Uu)$, we have that $E^\Uu_\alpha$ is close to $M^{*\Uu}_{\alpha+1}$; the proof of this is almost as in Lemma \ref{lem:closeness},
 with the usual slight tweak in the case that $\alpha$ is above $H$
 and $\spc(E^\Uu_\alpha)<\rho$,
 in which case we first get, in essentially the usual manner, that $E^\Uu_\alpha$ is close to $H$, and then deduce that it is close to $M$, since $\pi:H\to M$ is $0$-lifting.
 
 Therefore, since $\rho_1^H=\rho$,
 the preservation lemmas in \S\ref{sec:long_mice}
 give that $\rho_1(M^\Uu_\infty)=\rho_1^H=\rho$. If $M^\Uu_\infty\pins M^\Tt_\infty$
 then it follows that $H=M^\Uu_\infty\pins M$, so $H=J$. So we may assume that $M^\Uu_\infty=M^\Tt_\infty$.
 Since $\rho_1^M>\rho$, it follows that $b^\Tt$ drops in model. Recall $\crit(\pi)=\rho^{+H}$. So letting $\eta=\rho^{+H}$,
 we have $H|\eta=M||\eta$,
 so $\eta<\lh(E^\Uu_0)$ and $\eta\leq\lh(E^\Tt_0)$.
 It easily follows that $J=\core_\om(M^\Uu_\infty)$,
 and the subsets of $\rho$ which are $\bfrSigma_1^H$ are exactly those which are $\bfrSigma_1^{J}$, which suffices.
\end{proof}
\begin{lem}\label{lem:condensation_for_self-solid_2}
 Let  $M,H$ be $n$-sound premice of the same kind
 and $\pi:H\to M$ be $n$-lifting c-preserving $\pvec_n$-preserving.
 Suppose that either:
 \begin{enumerate}[label=--]
 \item $H,M$ are not active short, and $M$ is $(n,\om_1,\om_1+1)^*$-iterable, or
 \item $H,M$ are active short, $n=0$ and $M$ is $(0^-,\om_1,\om_1+1)^*$-iterable, or
 \item $H,M$ are active short, $n>0$, $H,M$ are Dodd-absent-sound and $M$ is $(n,\om_1,\om_1+1)^*$-iterable.
 \end{enumerate}
 Suppose that 
  $\rho=\rho_{n+1}^H<\rho_{n+1}^M$,
  $\rho$ is an $M$-cardinal,
   $\crit(\pi)=\rho^{+H}<\rho^{+M}$,
   and there is $x\in H$ such that $H=\Hull_{n+1}^H(\rho\cup\{x\})$.
Let $J\pins M$ be such that $\rho^{+H}=\rho^{+J}$ and $\rho_\om^J=\rho$.
Then $\Th_{\rSigma_{n+1}}^H(\rho\cup\{x\})$
is $\bfrSigma_{n+1}^J$, after coding
this theory naturally as a subset of $\rho$.
\end{lem}
\begin{proof}
The case that $H,M$ are active short and $n=0$ is covered by the previous lemma.
The other cases are very similar, and we leave them to the reader.
\end{proof}

We now come to the main result in this section, and our first application of comparison:
\begin{tm}\label{tm:rho^M_D=0_implies_Dodd-absent-solid_and_universal}
  Let $M$ be a $(0^-,\om_1,\om_1+1)^*$-iterable active short premouse
  with $\rho_{\mathrm{D}}^M=0$.
 Then $M$ is Dodd-absent-solid and Dodd-absent-universal.
\end{tm}

\begin{proof}
By Lemma \ref{lem:tau^M=max(rho_1,kappa^+)},
we may assume that $F^M$ has a largest generator.
 We may also assume $M$ is countable, and fix a $(0^-,\om_1+1)$-strategy $\Sigma$ for $M$ with weak Dodd-Jensen with respect to some enumeration $\vec{x}=\left<x_n\right>_{n<\om}$ of the universe of $M$
 with $p_{\D}^M\ins\vec{x}$.

 Let $p_{\mathrm{D}}^M=\{\alpha_0,\alpha_1,\ldots,\alpha_{n-1}\}$. Let $\alpha_n=0$.
 Given $i\leq n$ let \[F_i=F^M\rest(\{\alpha_0,\ldots,\alpha_{i-1}\}\cup\alpha_i) \]
 and $U_i=\Ult_0(M,F_i)$
 and $j_i:M\to U_i$ be the ultrapower map.
 Let \[\pi_i:U_i\to U'=\Ult_0(M,F^M) \]
 be the factor map, so if $i<n$ then $\crit(\pi_i)=\alpha_i$,
 and also $\crit(\pi_n)\geq\kappa$ where $
 \kappa=\crit(F^M)$. We know $F_n\notin M$.
 Let $i_0\leq n$ be least such that $F_{i_0}\notin M$. Let $U=U_{i_0}$ and $j=j_{i_0}$ and $\pi=\pi_{i_0}$.
 If $i_0<n$ then let 
 $\eta=\alpha_{i_0}$.
 If $i_0=n$ 
 then let $\eta=\kappa^{+M}$.
 
 \begin{clm}
 $i_0<n$ iff $\kappa^{+M}<\eta$ (and if $i_0=n$ then $\kappa^{+M}=\eta$).\end{clm}
 \begin{proof}
 If $i_0=n$ then $\eta=\kappa^{+M}$ by definition.
  So suppose $i_0<n$ and $\eta\leq\kappa^{+M}$. Then easily $\eta=\kappa$, and in fact $\kappa=\eta=\alpha_{i_0}\in p_{\D}^M$, so $\kappa=\min(p_{\D}^M)$.
  Since $i_0<n$,
  by minimality of $p_{\D}^M$,
  we have $F^M\rest(p_{\D}^M\cut\{\kappa\})\in M$. But this is equivalent to $F^M\rest(\kappa\cup(p_{\D}^M\cut\{\kappa\}))$, which verifies Dodd-absent-solidity at $i_0=n-1$, contradicting the choice of $i_0$.
 \end{proof}

Now we need to see that $i_0=n$ 
  and $U|\kappa^{++U}=M|\kappa^{++M}$.
So we may assume that either $\crit(\pi)=\eta$ (which certainly happens if $i_0<n$)
or $\eta^{+U}<\eta^{+M}$. Note that if $\eta=\gamma^{+U}=\gamma^{+M}$ for some $M$-cardinal $\gamma$, then $\eta\neq\crit(\pi)$,
and therefore $i_0=n$, so $\gamma=\kappa$
and $\eta=\kappa^{+U}=\kappa^{+M}$.

Now as in the usual proof of Dodd-solidity at the short extender level,
we will compare  $M$ with a certain phalanx $\ph$ derived from $M$.
Using  weak Dodd-Jensen, we will show that the comparison ends with the desired outcome, and then we will analyse that outcome to prove what we want (that $i_0=n$ and $U|\kappa^{++U}=M|\kappa^{++M}$).

 We define $\ph$ according to cases as follows; these cases are exhaustive by the preceding remarks:
 \begin{enumerate}[label=Ph\arabic*.,ref=Ph\arabic*]
 \item\label{case:Dodd-solidity_phalanx_eta_limit_card} If $\eta$ is a  cardinal of $M$ then
 $\ph=((M,{<\eta}),U)$. Note that this includes the possibility
 that $\eta=\kappa^{+M}=\kappa^{+U}$,
 in which case long extenders $E$ with $\crit(E)=\kappa$
 apply to $U$.

  \item\label{item:case_eta_non-card_active}  If $\gamma=\card^M(\eta)<\eta=\gamma^{+U}<\gamma^{+M}$ and $M|\eta$ is active, then
  \[ \ph=((M,{<\gamma}),(\Ult_0(M,F^{M|\eta}),\gamma),U).\]
  \item\label{item:case_eta_non-card_passive} If $\gamma=\card^M(\eta)<\eta=\gamma^{+U}<\gamma^{+M}$ and $M|\eta$ is passive,
  then
  \[ \ph=((M,{<\gamma}),((R,r),\gamma),U),\]
  where $(R,r)\pins(M,0)$ is least such that
  $\eta\leq\OR^R$ and $\rho_{r+1}^R=\gamma$.
  \end{enumerate}
  
  Our first task is to see that $\ph$ is iterable. This will take some work, eventually being achieved in Claim \ref{clm:Gamma_is_strat_for_ph} below.
  Let $\xi\in[\kappa^{+M},\lambda(F^M))$
  be an $M$-cardinal.
 Define the phalanxes $\mathfrak{Q}=((M,{<\xi}),M)$ and $\mathfrak{Q}'=((M,{<\xi}),U')$ where $U'=\Ult_0(M,F^M)$. For trees $\Tt'$ on $\mathfrak{Q}'$, we have two roots in $<^{\Tt'}$, which are $-1$ and $0$.
 We set $M^{\Tt'}_{-1}=M$ and $M^{\Tt'}_0=U'$. Given $\alpha<\lh(\Tt')$,
  $\mathrm{root}^{\Tt'}(\alpha)$
  denotes the root of $\alpha$  in $\Tt'$
  (that is, the unique $\beta\in\{-1,0\}$ such that $\beta<^{\Tt'}\alpha$).
  Likewise for trees $\Tt$ on $\mathfrak{Q}$, but in this case $M^{\Tt}_{-1}=M^{\Tt}_0=M$.
 We say that a pair of padded $0^-$-maximal trees $(\Tt,\Tt')$ on $\mathfrak{Q},\mathfrak{Q}'$ is a \emph{translation pair} iff:
 \begin{enumerate}
  \item $\lh(\Tt)=\lh(\Tt')$,
  \item $E^\Tt_\alpha\neq\emptyset$ for all $\alpha+1<\lh(\Tt)$ (but $\Tt'$ can use some padding),
  \item for all $\alpha+1<\lh(\Tt)$,
  if $E^{\Tt'}_\alpha\neq\emptyset$ then $E^\Tt_\alpha=E^{\Tt'}_\alpha$,
    \item if $\lh(\Tt)>1$ then $\xi<\lh(E^\Tt_0)$,
  \item\label{item:when_padding_in_Tt'} for all $\alpha+1<\lh(\Tt)$,
  we have $E^{\Tt'}_\alpha=\emptyset$
  iff \[\mathrm{root}^{\Tt}(\alpha)=0\text{ and }[0,\alpha]^\Tt\cap\dropset^\Tt=\emptyset
  \text{ and }E^\Tt_\alpha=F(M^\Tt_\alpha),\]
  \item\label{item:when_root^T(alpha)=0} Let $\alpha<\lh(\Tt)$
  be such that
   $\mathrm{root}^\Tt(\alpha)=0$.
  Then:
  \begin{enumerate}[label=--]
   \item
$\mathrm{root}^{\Tt'}(\alpha)=0$,
  \item $[0,\alpha]^{\Tt}=[0,\alpha]^{\Tt'}$,
  \item  $[0,\alpha]^{\Tt}\cap\dropset^\Tt=[0,\alpha]^{\Tt'}\cap\dropset^{\Tt'}$,
  \item $E^\Tt_\beta=E^{\Tt'}_\beta$ for all $\beta+1\leq^\Tt\alpha$,
  \item either:
  \begin{enumerate}
  \item $[0,\alpha]^{\Tt}\cap\dropset^\Tt=\emptyset=
  [0,\alpha]^{\Tt'}\cap\dropset^{\Tt'}$
  and $M^{\Tt'}_\alpha=\Ult_0(M,F(M^\Tt_\alpha))$, or
  \item $[0,\alpha]^{\Tt}\cap\dropset^\Tt\neq\emptyset\neq [0,\alpha]^{\Tt'}\cap\dropset^{\Tt'}$
  and $M^\Tt_\alpha=M^{\Tt'}_\alpha$
  and $\deg^\Tt_\alpha=\deg^\Tt_{\alpha'}$.
  \end{enumerate}
  \end{enumerate}
  \item Let $\alpha<\lh(\Tt)$ be such that
   $\mathrm{root}^\Tt(\alpha)=-1$
   and $\mathrm{root}^{\Tt'}(\alpha)=0$.
  Then
  there is a unique $\beta+1\leq^{\Tt'}\alpha$
  such that $E^{\Tt'}_\beta=\emptyset$;
  fixing this $\beta$, we have:
  \begin{enumerate}[label=--]
  \item $\pred^{\Tt'}(\beta+1)=\beta$,
  \item $\pred^\Tt(\beta+1)=-1$
  (and by clause \ref{item:when_padding_in_Tt'}, we have
   $\mathrm{root}^\Tt(\beta)=0$ and $[0,\beta]^\Tt\cap\dropset^\Tt=\emptyset$ and $E^\Tt_\beta=F(M^\Tt_\beta)$,
   so also clause \ref{item:when_root^T(alpha)=0}
   attains at $\beta$),
  \item  $[\beta+1,\alpha]^{\Tt}=[\beta+1,\alpha]^{\Tt'}$,
  \item $(\beta+1,\alpha]^{\Tt}\cap\dropset^\Tt=(\beta+1,\alpha]^{\Tt'}\cap\dropset^{\Tt'}$,
  and
  \item  $E^\Tt_\gamma=E^{\Tt'}_\gamma$
  for all $\gamma+1\in(\beta+1,\alpha]^\Tt$,
  \item $M^\Tt_\alpha=M^{\Tt'}_\alpha$
  and $\deg^\Tt_\alpha=\deg^{\Tt'}_\alpha$.
  \end{enumerate}
  \item Let $\alpha<\lh(\Tt)$ be such that
  $\mathrm{root}^{\Tt}(\alpha)=-1=\mathrm{root}^{\Tt'}(\alpha)$. Then:
  \begin{enumerate}[label=--]
    \item $(-1,\alpha]^{\Tt'}=(-1,\alpha]^{\Tt}$,
    \item $E^\Tt_\beta=E^{\Tt'}_\beta$ for all $\beta+1\leq^{\Tt}\alpha$,
   \item $(-1,\alpha]^{\Tt}\cap\dropset^\Tt=(-1,\alpha]^{\Tt'}\cap\dropset^{\Tt'}$
  \item $M^{\Tt}_\alpha=M^{\Tt'}_\alpha$
  and $\deg^\Tt_\alpha=\deg^{\Tt'}_\alpha$.
  \end{enumerate}
  \end{enumerate}

  Clearly $\mathfrak{Q}$ is $(0^-,\om_1+1)$-iterable, and we consider
  $\Sigma$ as acting on trees on $\mathfrak{Q}$ in the obvious manner.
  
  \begin{clm}
   $\mathfrak{Q}'$ is $(0^-,\om_1+1)$-iterable, via the strategy $\Sigma'$,
   where a putative $0^-$-maximal tree on $\mathfrak{Q}'$ is via $\Sigma'$
   iff there is $\Tt$ on $\mathfrak{Q}$ via $\Sigma$ such that $(\Tt,\Tt')$
   is a translation pair.
  \end{clm}
\begin{proof}
 This is mostly straightforward,
 but there is one detail
 which is new in the long extender context. Let $\alpha<\lh(\Tt)$ be such that $\mathrm{root}^\Tt(\alpha)=0$ and
 letting $\beta+1=\min((0,\alpha]^\Tt)$,
 then $E^\Tt_\beta$ is long with $\crit(E^\Tt_\beta)=\kappa$.
 Then we form $M^\Tt_{\beta+1}=\Ult_0(M,E^\Tt_\beta)$ so as to avoid the protomouse. We form $M^{\Tt'}_{\beta+1}=\Ult_0(U,E^\Tt_\beta)$
 in the usual manner.
 We need to see here that
 \[ M^{\Tt'}_{\beta+1}=\Ult_0(M,F(M^\Tt_{\beta+1})).\]
 But  $F(M^\Tt_{\beta+1})=F((M^\Tt_{\beta+1})^*)\com E$,
 where $E$ is the short part of $E^\Tt_\beta$, and $(M^\Tt_{\beta+1})^*$ is the protomouse associated to $\Ult_0(M,E^\Tt_\beta)$.
 We also have
 \[ M^{\Tt'}_{\beta+1}=\Ult_0(\Ult_0(M,F^M),E^\Tt_\beta).\]
 So we need to see that the 2-step ``iteration'' of $M$ given by first applying $F^M$, and then applying $E^\Tt_\beta$, is equivalent to
 that given by first applying $E$,
 and then applying $F((M^{\Tt}_{\beta+1})^*)$. But this is indeed true, basically by the usual proof
 from the short extender context
 (show that the resulting overall extenders are equivalent; they are both short). (See for example \cite{extmax}
 or \cite{fullnorm} for versions of this in the short extender context.)\end{proof}

We now want to define a $0^-$-maximal strategy $\Gamma$ for $\ph$, by lifting trees $\Tt$ on $\ph$  via $\Gamma$
to 
 trees $\Tt'$ on $\mathfrak{Q}'$ via $\Sigma$, where $\xi$
is taken appropriately
(recall $\mathfrak{Q}'=((M,{<\xi}),U')$,
and $\xi\in[\kappa^{+M},\lambda(F^M))$
was a given $M$-cardinal).
We will produce copy maps $\pi_\alpha$,
possibly with $\pi_\alpha:M^\Tt_\alpha\to M^{\Tt'}_\alpha$, or a variant thereof depending on the case.

Suppose first that we are in case \ref{case:Dodd-solidity_phalanx_eta_limit_card}
of the definition of $\ph$,
so $\ph=((M,{<\eta}),U)$
and $\eta$ is an $M$-cardinal.
We have the factor map $\pi:U\to U'$,
so $\pi(\eta)$ is a $U'$-cardinal
and $\pi(\eta)<\lambda(F^M)$,
so $\pi(\eta)$ is
an $M$-cardinal.  In this case take $\xi=\pi(\eta)$. So if $\eta\leq\crit(\pi)$ then note that we can lift trees $\Tt$ to trees $\Tt'$  in the usual manner, starting with base copy maps $(\pi_{-1},\pi_0)=(\id,\pi)$. Now suppose that $\crit(\pi)<\eta$, and therefore $\eta=\kappa^{+M}$ and $\crit(\pi)=\kappa$. For all copy maps
$\pi_\alpha$ with $\alpha\geq 0$,
we will have $\pi\rest\kappa^{+M}\sub\pi_\alpha$. And we will have $\pi_\alpha:M^\Tt_\alpha\to M^{\Tt'}_\alpha$. Thus, as long as $\crit(E^\Tt_\alpha)<\kappa$,
we can copy to $E^{\Tt'}_\alpha=\pi_\alpha(E^\Tt_\alpha)$
and define $\pi_{\alpha+1}$ as usual.
Suppose $\crit(E^\Tt_\alpha)=\kappa$.
If $E^\Tt_\alpha$ is long then $\pred^\Tt(\alpha+1)=0$ and $\pred^{\Tt'}(\alpha+1)=0$,
and we again proceed as usual. Suppose $E^\Tt_\alpha$ is short.
Then $\pred^\Tt(\alpha+1)=-1$,
but note that $\pi_\alpha(\crit(E^\Tt_\alpha))=\pi(\kappa)>\kappa^{+M}$, and if we set $E^{\Tt'}_\alpha=\pi_\alpha(E^\Tt_\alpha)$
we would have $\pred^{\Tt'}(\alpha+1)=0$.
Instead we use
the fact that in this case,
the short extender $E$
derived from $\pi$ is in $\es^{U'}$,
and $\pi\rest\kappa^{+M}=\pi_\alpha\rest\kappa^{+M}$,
and we set $E^{\Tt'}_\alpha=\pi_\alpha(E^\Tt_\alpha)(E)$.
So $\crit(E^{\Tt'}_\alpha)=\kappa$
and $E^{\Tt'}_\alpha$ is short.
So $\pred^{\Tt'}(\alpha+1)=-1$,
and we define
\[ \pi_{\alpha+1}:\Ult_0(M,E^\Tt_\alpha)\to\Ult_0(M,\pi_\alpha(E^\Tt_\alpha)(E)) \]
via Shift Lemma IV \ref{lem:shift_lemma_iv}
applied to $\pi_{-1}=\id$ and the $0$-deriving embedding
\[\pi_\alpha\rest\exit^\Tt_\alpha:\exit^\Tt_\alpha\to \exit^{\Tt'}_\alpha.\]
Otherwise things are as usual.

Now suppose we are in case
\ref{item:case_eta_non-card_passive}
of the definition of $\ph$.  So \[\ph=((M,{<\gamma}),((R,r),\gamma),U),\]
$\gamma$ is an $M$-cardinal,
$\eta=\gamma^{+U}<\gamma^{+M}$ and $M|\eta$ is passive). So $\kappa^{+M}\leq\gamma$ in this case,
and $\eta=\crit(\pi)$. Use $\xi=\pi(\eta)=\gamma^{+M}$.
Trees $\Tt$ on $\ph$ have 3 roots: $-2$, $-1$, $0$.
We write $M^\Tt_{-2}=M$,
$M^\Tt_{-1}=R$, $\deg^\Tt_{-1}=r+1$,
$M^\Tt_0=U$. We lift to padded trees $\Tt''$ on $\mathfrak{Q}''=((M,{<\gamma}),(M,\gamma),U')$,
with $\lh(E^{\Tt''}_0)>\gamma^{+M}$.
We start with copy maps $\pi_{-2}=\id:M\to M$, $\pi_{-1}=\id:R\to R\pins M$, and $\pi_0=\pi:U\to U'$.
We will have $\mathrm{root}^\Tt(\alpha)=\mathrm{root}^{\Tt'}(\alpha)$ for all $\alpha$.
 We will have $\pi_\alpha:M^\Tt_\alpha\to M^{\Tt''}_\alpha$, except in case
$\alpha$ is a successor
  with $\pred^\Tt(\alpha)=-1$,
  in which case 
   $\pi_\alpha:M^\Tt_\alpha\to R^{\Tt''}_\alpha\pins M^{\Tt''}_\alpha$,
  where $R^{\Tt''}_{-1}=R$
  and $R^{\Tt''}_\alpha$ will be specified in what follows.
  We will pad in $\Tt''$ exactly
  at those $\beta$ such that
  $\spc(E^\Tt_\beta)=\gamma$
  (that is, we set $E^{\Tt''}_\beta=\emptyset$ at exactly those $\beta$),
  and here we will define $M^{\Tt''}_{\beta+1}=M^{\Tt''}_\beta$
  and $\pred^{\Tt''}(\beta+1)=-1$.
  In this case
  we will still define an exchange ordinal $\lambda^{\Tt''}_\beta$,
  to be used in place of the usual one.
  
  The copying process proceeds as usual,
  except in case  $\spc(E^\Tt_\beta)=\gamma$, so that $\pred^\Tt(\beta+1)=-1$ and $M^{*\Tt}_{\beta+1}=R$ (and $\deg^\Tt_{\beta+1}=r$). As just mentioned,
  we pad in $\Tt''$ at this stage, setting $E^{\Tt''}_\beta=\emptyset$, so
 $M^{\Tt''}_{\beta+1}=M^{\Tt''}_\beta$,
 and we set $\pred^{\Tt''}(\beta+1)=-1$ (the latter is really just for consistency with $\Tt$). Let $F=\pi_\beta(E^\Tt_\beta)$.
 Define the exchange ordinal associated to $\beta$ in $\Tt''$ as $\lambda^{\Tt''}_{\beta}=$ the largest cardinal of $\pi_\beta(\exit^\Tt_\beta)$ (note that this is just the exchange ordinal that we would have had if we had set $E^{\Tt''}_\beta=F$).
 We will define $R^{\Tt''}_{\beta+1}\pins M^{\Tt''}_{\beta+1}=M^{\Tt''}_\beta$, in fact with $R^{\Tt''}_{\beta+1}\pins\pi_\beta(\exit^\Tt_\beta)$,
 with $\rho_{r+1}(R^{\Tt''}_{\beta+1})=\rho_\om(R^{\Tt''}_{\beta+1})=\lambda^{\Tt''}_\beta$.
 Note that $\gamma<\crit(\pi)=\crit(\pi_\beta)=\gamma^{+U}$ and $\pi(\gamma^{+U})=\pi_\beta(\gamma^{+U})=\gamma^{+M}$,
 so $\spc(F)=\gamma$ and $F$ is $M$-total. Recall that $R\pins M$ with $\rho_{r+1}^R=\rho_\om^R=\gamma$,
 so letting $S=i^{M,0}(R)$,
 we have $S\pins\Ult_0(M,F)$
 and $\rho_{r+1}^S=\rho_\om^S$
 is a cardinal of $\Ult_0(M,F)$.
 There are three cases to consider.
  
  \begin{casethree}\label{case:short_ext_in_anomalous_lift}
 $\gamma$ is a limit cardinal of $M$.

 We define
 \[ \widetilde{R}=R^{\Tt''}_{\beta+1}=\Ult_r(R,F).\]
 Note that we only use generators ${<\lambda(F)}$
 to define this ultrapower.
 Note that $F$ is an $M$-extender,
 and hence \emph{not} an $R$-extender, since $\pow(\gamma)\cap R\psub\pow(\gamma)\cap M$. 
 
 Since $R\pins M$ with $\rho_{r+1}^R=\rho_\om^R=\gamma$, we have $S\pins\pi_\beta(\exit^\Tt_\beta)$
 and $\rho_{r+1}^S=\rho_\om^S=\lambda(F)=\lambda_\beta^{\Tt''}$.
 We claim that $\widetilde{R}\ins S$,
 so $\widetilde{R}\pins\pi_\beta(\exit^\Tt_\beta)$.
 For let
 \[ \sigma: \widetilde{R}\to S \]
 be the factor map; that is,
 \[ \sigma([a,f_{t,q}^R]^{R,r}_F)=i^{M,0}_F(f_{t,q}^R)(a) \]
 for all $a\in[\lambda(F)]^{<\om}$
 and $q\in R$ and $\rSigma_r$ Skolem terms $t$. Then $\widetilde{R}$
 is $(r+1)$-sound with $\rho_{r+1}(\widetilde{R})=\lambda(F)=\rho_{r+1}(S)$ and $\sigma$ is $r$-lifting with $\crit(\sigma)>\lambda(F)$, 
 and note that $\sigma,\widetilde{R}\in\pi_\beta(\exit^\Tt_\beta)$,
 so by internal condensation for $M$,
 we get $\widetilde{R}\ins S$.
 Now just let $\pi_{\beta+1}:M^\Tt_{\beta+1}\to R^{\Tt''}_{\beta+1}$ be defined via the natural variant of the Shift Lemma,
 applied to the maps $\id:R\to R$ and $\pi_\beta\rest\exit^\Tt_\beta$. Note that $\pi_{\beta+1}$
 is an $r$-embedding (it is $r$-lifting as usual, but by commutativity and
 since $R^{\Tt''}_{\beta+1}=\Ult_r(R,F)$,
 we get that $\rho_{r}(R^{\Tt''}_{\beta+1})=\sup \pi_{\beta+1}``\rho_{r+1}(M^\Tt_{\beta+1})$).
  \end{casethree}
 
 \begin{casethree}
  $E^\Tt_\beta$
 is long with $\crit(E^\Tt_\beta)=\mu$
 where $\mu^{+M}=\gamma$,
 and $E^\Tt_\beta$ has a largest generator.
 
In this case $i^{M,0}_F(\gamma)=i^{M,0}_F(\mu^{+M})=\lambda^{+\Ult_0(M,F)}=\lh(F)$ where $\lambda=\lambda(F)$.
So
  $\rho_\om^S=\rho_{r+1}^S=\lh(F)$.
 Again define
 \[ \widetilde{R}=\Ult_r(R,F),\]
 but note here that we literally use
 $F$ to form the ultrapower, not just $F\rest\nu(F)$;
 this makes a difference since
 $\dom(F)\not\sub R$.
 As before,
 $\widetilde{R}$ is $(r+1)$-sound with $\rho_{r+1}^{\widetilde{R}}=\rho_\om^{\widetilde{R}}=\lh(F)$.
  We get $\widetilde{R}\ins S$ by condensation, and have the $r$-embedding $i^{R,r}_{F}:R\to\widetilde{R}$. However,
  if we had set $E^\Tt_\beta=F$ and $R^{\Tt''}_{\beta+1}=\widetilde{R}$, then
  the fact that
  $\rho_{r+1}^{\widetilde{R}}=\lh(F)>\lambda(F)$ would not be so convenient for later copying,
  because if $\Tt$
  uses a short extender $E^\Tt_\gamma$
  with $\beta<\gamma$
  and $\crit(E^\Tt_\gamma)=\lambda(E^\Tt_\beta)$, then we would have $\pred^\Tt(\gamma+1)=\beta+1$
  and $\deg^\Tt_{\gamma+1}=r$,
  and by copying $E^\Tt_\gamma$
  in the usual way,
  we would get $\pred^{\Tt''}(\gamma+1)=\beta+1$ but 
  $\widetilde{R}\pins M^{*\Tt''}_{\gamma+1}$. 
  Instead we set $E^{\Tt''}_\beta=\emptyset$ and replace $\widetilde{R}$
  with some $R'\pins\widetilde{R}$
  with $\rho_{r+1}^{R'}=\rho_\om^{R'}=\lambda(F)$,
  so that in the latter situation,
  we will get $M^{*\Tt''}_{\gamma+1}=R^{\Tt''}_{\beta+1}$ and $\deg^{\Tt''}_{\gamma+1}=r$.
  For this we use some more condensation. Given $\xi\in(\lambda(F),\lh(F))$, let
  \[ R'_\xi=\cHull_{r+1}^{\widetilde{R}}((\xi+1)\cup\{\pvec_{r+1}^{\widetilde{R}}\})\]
  and $\sigma'_\xi:R'_\xi\to\widetilde{R}$ be the uncollapse.
  
  We claim that
  for all sufficiently large $\xi<\lh(F)$, letting $\chi=\xi^{+R'_\xi}$,
  we have that $R'_\chi$ is $(r+1)$-sound with $\rho_{r+1}^{R'_\chi}=\lambda(F)$ and $\chi=\min(p_{r+1}^{R'_\chi})$
  and $\sigma'_\chi(q)=p_{r+1}^{\widetilde{R}}$
  where $q=p_{r+1}^{R'_\chi}\cut\{\chi\}$, and if $\widetilde{R}$ is active short then
  $R'_\chi$ is Dodd-absent-sound. For
  first let $\xi$ be large enough that
  all $(r+1)$-solidity witnesses
  for $p_{r+1}^{\widetilde{R}}$ are in $\rg(\sigma'_\xi)$,
  and if $\widetilde{R}$ is active short,
  then likewise for all Dodd-absent-solidity witnesses for $p_{\D}^{\widetilde{R}}$.
  Let $\chi=\lambda(F)^{+R'_\xi}$.
  Let $J\pins\widetilde{R}$
  be such that $\chi=\lambda(F)^{+J}$ and $\rho_\om^J=\lambda(F)$.
  Let $\sigma'_\xi(q)=p_{r+1}^{\widetilde{R}}$.
  By Lemma \ref{lem:condensation_for_self-solid_2}, $t=\Th_{\rSigma_{r+1}}^{R'_\xi}(\lambda(F)\cup\{\xi,q\})$
  is $\bfrSigma_{r+1}^J$.
  But now  the desired properties hold for $\chi$ (note: not $\xi$).
  In particular, $R'_\chi$
   is $(r+1)$-sound with $\sigma'_\chi(p_{r+1}^{R'_\chi})=p_{r+1}^{\widetilde{R}}\cup\{\chi\}$,
  because $\rg(\sigma'_\xi)\sub\rg(\sigma'_\chi)$ and $t\in R'_\chi$
  and $t$ yields the bottom $(r+1)$-solidity
  witness  for $p_{r+1}^{\widetilde{R}}\cup\{\chi\}$. And if $\widetilde{R}$ is active short, then the Dodd-absent-soundness of $R'_\chi$ is only
  not immediate if $\crit(F^{\widetilde{R}})\leq\lambda(F^M)$ and $r=0$
  (if $r>0$ then elementarity considerations suffice).
  So suppose this is the case.
  If $\crit(F^{\widetilde{R}})=\lambda(F^M)$ then $F^{R'_\chi}$
  is generated by $(\sigma'_\chi)^{-1}(p_{\D}^{\widetilde{R}})$,
  by $\rSigma_1$-elementarity,
  and since $\xi$ was taken large enough,
  $(\sigma'_\chi)^{-1}(p_{\D}^{\widetilde{R}})$ is also Dodd-absent-solid for $R'_\chi$.
  So suppose $\crit(F^{\widetilde{R}})<\lambda(F^M)$.
  Then $R'_\chi$ is Dodd-absent-solid with $\pi(p_{\D}^{R'_\chi})=p_{\D}^{\widetilde{R}}\cup\{\chi\}$, for reasons similar to those just mentioned, and that the Dodd-absent-solidity witness at $\chi$ is also coded by $t$.
  
  Now let $\chi$ be least such.
  Since $\crit(\sigma'_\chi)=\chi^{+R'_\chi}$, internal condensation then gives
  that $R'_\chi\pins\widetilde{R}$, so $R'_\chi\pins\pi_\beta(\exit^\Tt_\beta)$. Also, since $i^{R,r}_{F}:R\to\widetilde{R}$ is an $r$-embedding and $\rg(i^{R,r}_F)\sub\rg(\sigma'_\chi)$,
  $\sigma'_\chi$ is also an $r$-embedding,
  as is the factor map $\tau_\chi:R\to R'_\chi$
  given by $\tau_\chi=(\sigma'_\chi)^{-1}\com i^{R,r}_F$. Now let $\chi$ be least
  as above such that $\nu(F)\sub(\chi+1)$ and $\sup\pi_\beta``\lh(E^\Tt_\beta)\sub(\chi+1)$. Then we define $R^{\Tt''}_{\beta+1}=R'_\chi$.
  Note that $R'_\chi=\Ult_r(R,G)$ where $G=F\rest\chi^{+R'_\chi}$ and $\tau_\chi=i^{R,r}_G$.
  Using this characterization, define $\pi_{\beta+1}:M^{\Tt}_{\beta+1}\to R'_\chi$
  via the natural version of the Shift Lemma. By commutativity,
  $\pi_{\beta+1}$ is an $r$-embedding.
 \end{casethree}
 
 \begin{casethree}
 $E^\Tt_\beta$ is long with $\crit(E^\Tt_\beta)=\mu$ where $\mu^{+M}=\gamma$, and $E^\Tt_\beta$ has no largest generator.

 This is more like Case \ref{case:short_ext_in_anomalous_lift}.
 We have  $\lambda^{\Tt''}_{\beta}=\nu(F)=\lambda^{+\pi_\beta(\exit^\Tt_\beta)}$
 where $\lambda=\lambda(F)$.
 Define
\[\widetilde{R}=R^{\Tt''}_{\beta+1}=\Ult_r(R,F\rest\nu(F)).
\]
As in Case \ref{case:short_ext_in_anomalous_lift}, we get $\widetilde{R}\ins S\pins\pi_\beta(\exit^\Tt_\beta)\ins M^{\Tt''}_{\beta+1}$,
and $\rho_{r+1}(\widetilde{R})=\rho_\om(\widetilde{R})=\nu(F)$.
We define $\pi_{\beta+1}:M^{\Tt}_{\beta+1}\to R^{\Tt''}_{\beta+1}$
via the natural version of the Shift Lemma, and this is an $r$-embedding.
 \end{casethree}
 
 This completes all cases.
 Note that
 $\pi_{\beta+1}:M^\Tt_{\beta+1}\to R^{\Tt''}_{\beta+1}\pins M^{\Tt''}_{\beta+1}$ is an $r$-embedding
 with $\pi_\beta\rest(\lambda^{\Tt}_\beta+1)\sub\pi_{\beta+1}$,
 where $\lambda^{\Tt}_\beta$ is the exchange ordinal associated to $E^\Tt_\beta$ in $\Tt$,
 and
  $\rho_{r+1}(R^{\Tt''}_{\beta+1})=\pi_{\beta+1}(\lambda^{\Tt}_\beta)$.
  This helps ensure that if $\gamma+1<\lh(\Tt)$ with $\pred^\Tt(\gamma+1)=\beta+1$,
  then $\pred^{\Tt''}(\gamma+1)=\beta+1$
  and $M^{*\Tt''}_{\gamma+1}=\pi_{\beta+1}(M^{*\Tt}_{\gamma+1})$
  and $\deg^{\Tt''}_{\gamma+1}=\deg^{\Tt}_{\gamma+1}$.

  The rest of the copying in Case \ref{item:case_eta_non-card_passive}
  of the definition of $\ph$
  is routine, and we omit further details.
  
  It just remains to discuss
  Case \ref{item:case_eta_non-card_active} of the definition of $\ph$. This case is  parallel to case \ref{item:case_eta_non-card_passive}, but with a difference. We will not use any padding in $\Tt''$. 
 We have $\gamma=\card^M(\eta)<\eta=\crit(\pi)=\gamma^{+U}<\gamma^{+M}=\pi(\eta)$, $M|\eta$ is active, and
  \[ \ph=((M,{<\gamma}),(\Ult_0(M,F^{M|\eta}),\gamma),U).\]
 We have $\kappa^{+M}\leq\gamma$. Use $\xi=\gamma^{+M}$.
Write $M^\Tt_{-2}=M$,
$M^\Tt_{-1}=\Ult_0(M,F^{M|\eta})$,
$M^\Tt_0=U$. We lift to padded trees $\Tt''$ on $\mathfrak{Q}''=((M,{<\gamma}),U')$,
with $\lh(E^{\Tt''}_0)>\gamma^{+M}$,
and $\Tt''$ will not use any extenders with space $\gamma$.
We start with copy maps $\pi_{-2}=\id:M\to M$ and $\pi_0=\pi:U\to U'$.
We will have $\mathrm{root}^{\Tt'}(\alpha)=0$ if $\mathrm{root}^{\Tt}(\alpha)=0$, and $\mathrm{root}^{\Tt'}(\alpha)=-2$ otherwise.
 We will have $\pi_\alpha:M^\Tt_\alpha\to M^{\Tt''}_\alpha$.
 
 As in case \ref{item:case_eta_non-card_passive},
 the copying process
 is as usual except when
 $\spc(E^\Tt_\beta)=\gamma$,
 so consider this case.
 So $\pred^\Tt(\beta+1)=-1$
and
\[ M^\Tt_{\beta+1}=\Ult_0(M,E^\Tt_\beta\com F^{M|\eta}).\]
 Let $F=\pi_\beta(E^\Tt_\beta)$.
 We will set
   $E^{\Tt''}_\beta$
  to be either $i^{M,0}_F(F^{M|\eta})$ or an initial segment thereof,
 so $\spc(E^{\Tt''}_\beta)<\gamma$,
 so $\pred^{\Tt''}(\beta+1)=-2$
 and $M^{*\Tt''}_{\beta+1}=M$.
 We will ensure that the exchange ordinal associated to $E^{\Tt''}_\beta$
 is just $\pi_\beta(\nu)$,
 where $\nu$ is the exchange ordinal
 associated to $E^\Tt_\beta$.
 
 Let $R=M|\eta$,
 so $R$ is active with largest cardinal $\rho_\om^R=\rho_1^R=\gamma$.
 Let $S=i^{M,0}_F(R)$,
 so $S\pins\Ult_0(M,F)$
 is active with largest cardinal $\rho_\om^S=\rho_1^S=i^{M,0}_F(\gamma)$.
 There are three cases to consider.
  
  \begin{casefour}\label{case:short_ext_in_anomalous_lift_active}
 $\gamma$ is a limit cardinal of $M$.
 
 Set $E^{\Tt''}_\beta=F^S$.
 Because $\gamma=\rho_1^R$ is a limit cardinal of $M$,  $R$ is either active short or active long with a largest generator, and likewise $S$,
 so the exchange ordinal associated to $F^S$ is just $i^{M,0}_F(\gamma)=\lambda(F)$,
 and $F$ is short as $\spc(E^\Tt_\beta)=\gamma$,
 so $\lambda(F)=\pi_\beta(\lambda(E^\Tt_\beta))$ and $\lambda(E^\Tt_\beta)$ is the exchange ordinal for $E^\Tt_\beta$.
 Let \[j:M\to\Ult_0(M,F^{M|\eta}),\]
  \[k:\Ult_0(M,F^{M|\eta})\to\Ult_0(\Ult_0(M,F^{M|\eta}),E^\Tt_\beta),\]
 \[ j':M\to\Ult_0(M,F^S) \]
 be the ultrapower maps.
 Let $\nu^-$ be the largest generator
 of $F^{M|\eta}$, if this extender is long (notice that if it is long then it has a largest generator),
 and $\nu^-=0$ if $F^{M|\eta}$ is short.
 We define $\pi_{\beta+1}:M^\Tt_{\beta+1}\to M^{\Tt''}_{\beta+1}$ in the natural way; that is,
 \[ \pi_{\beta+1}:\Ult_0(\Ult_0(M,F^{M|\eta}),E^\Tt_\beta)\to\Ult_0(M,F^S),\]
 is given by setting
 \[ \pi_{\beta+1}\Big(k(j(f))\big(a,k(\nu^-),b\big)\Big)=j'(f)\big(a,i^{M,0}_F(\nu^-),\pi_\beta(b)\big)\]
 for all $f\in M$,
 $a\in[\gamma]^{<\om}$ and
 $b\in[\lambda(E^\Tt_\beta)]^{<\om}$.
  \end{casefour}
 
 \begin{casefour}\label{case:long_ext_largest_gen_in_anomalous_lift_active}
  $E^\Tt_\beta$
 is long with $\crit(E^\Tt_\beta)=\mu$
 where $\mu^{+M}=\gamma$,
 and $E^\Tt_\beta$ has a largest generator.
 
 Since $R$ has largest cardinal $\gamma=\mu^{+M}$, $F^R$ is long with $\nu(F^R)=\gamma$, so $F^R$ has no largest generator, and so $S\pins\Ult_0(M,F)$ and $F^S$ is long with $\nu(F^S)=i^{M,0}_F(\gamma)=\lh(F)$. We set $E^{\Tt''}_\beta$ to be a sufficiently large initial segment of $F^S$ which has a largest generator:
 just let $E^{\Tt''}_\beta$
 be (the trivial completion of)
 $F^S\rest(\nu^-(F)+1)$
 (note $\nu^-(F)=\pi_\beta(\nu^-(E^\Tt_\beta))$.
 Note that the exchange ordinal
 associated to $E^{\Tt''}_\beta$
 is then $\pi_\beta(\lambda(E^\Tt_\beta))$.
 Now define $\pi_{\beta+1}$ by
 \[ \pi_{\beta+1}(k(j(f))(k(a),b))=j'(i^{M,0}_F(a),\pi_\beta(b)) \]
 for all $a\in[\gamma]^{<\om}$
 and $b\in[\nu^-(E^\Tt_\beta)+1]^{<\om}$, where $j,k,j'$ are as before.
 \end{casefour}
 
 \begin{casefour}
 $E^\Tt_\beta$ is long with $\crit(E^\Tt_\beta)=\mu$ where $\mu^{+M}=\gamma$, and $E^\Tt_\beta$ has no largest generator.

 Set $E^{\Tt''}_\beta=F^S$.
 Define $\pi_{\beta+1}$ like in Case \ref{case:long_ext_largest_gen_in_anomalous_lift_active}, but now
 $b$ comes from $[\nu(E^\Tt_\beta)]^{<\om}$.
 \end{casefour}
 
 This completes all cases.
 Note that
 $\pi_{\beta+1}:M^\Tt_{\beta+1}\to  M^{\Tt''}_{\beta+1}$ is $0$-embedding
 with $\pi_\beta\rest(\lambda^{\Tt}_\beta+1)\sub\pi_{\beta+1}$,
 where $\lambda^{\Tt}_\beta$ is the exchange ordinal associated to $E^\Tt_\beta$ in $\Tt$.

  The rest of the copying in Case \ref{item:case_eta_non-card_passive}
  of the definition of $\ph$
  is routine.
  
  Let $\Gamma$ be the (putative)
  $0^-$-maximal strategy for $\ph$
  given by lifting to $0^-$-maximal
  trees on $M$ via $\Sigma$, in the manner described above. Then:
  
  \begin{clm}\label{clm:Gamma_is_strat_for_ph}
   $\Gamma$ is a $(0^-,\omega_1+1)$-strategy for $\ph$.
  \end{clm}

  \begin{clm}\label{clm:pow(eta)^U_sub_M}
   We have:
   \begin{enumerate}
    \item Suppose $\ph$ is defined via case \ref{case:Dodd-solidity_phalanx_eta_limit_card}. Then
    $U|\eta^{+U}=M||\eta^{+U}$,
    so $\pow(\eta)^U\sub M$.
    \item In case \ref{item:case_eta_non-card_active}
   of the definition of $\ph$, we have
   $U|\eta^{+U}=U|\gamma^{++U}=\Ult_0(M,F^{M|\eta})||\eta^{+U}$,
   so
   $\pow(\gamma)^U\sub M$ and in fact $\pow(\eta)^U\sub\Ult_0(M,F^{M|\eta})$,
    \item In case \ref{item:case_eta_non-card_passive}
    of the definition of $\ph$, we have
    $U|\eta^{+U}=U|\gamma^{++U}=R||\eta^{+U}$, so $\pow(\gamma)^U\sub M$ and in fact $\pow(\eta)^U\sub R$.
   \end{enumerate}
  \end{clm}
\begin{proof}
 By internal condensation (as specified in the definition of \emph{premouse},
 \ref{dfn:pm}) applied to restrictions of $\pi$.
\end{proof}

As usual, we want to know that extenders used in $\Tt$ on $\ph$ are close to the models to which they apply:
\begin{clm}\label{clm:Dodd-solidity_closeness}
Let $\Tt$ be a $0^-$-maximal
tree on $\ph$ with $\eta^{+U}<\lh(E^\Tt_0)$. Let $\alpha+1<\lh(\Tt)$.
Then $E^\Tt_\alpha$ is close to $M^{*\Tt}_{\alpha+1}$.
\end{clm}
\begin{proof}
This is a slight adaptation of the proof of Closeness \ref{lem:closeness}, showing
by induction on $\alpha+1<\lh(\Tt)$
that $E=E^\Tt_\alpha$ is close to $M^{*\Tt}_{\alpha+1}$.
We just focus on the details which are different. These arise when $\spc(E)<\eta$ (so $\pred^\Tt(\alpha+1)<0$)
and $\pred^\Tt(\alpha+1)\neq\mathrm{root}^\Tt(\alpha)$,
so consider this situation.

 \begin{casefive}
  $\ph$ is defined via Case \ref{case:Dodd-solidity_phalanx_eta_limit_card};
  that is, $\eta$ is a cardinal of $M$
  and $\ph=((M,{<\eta}),U)$.
  
  Suppose $\mathrm{root}^\Tt(\alpha)=0$
  (so $M^\Tt_\alpha$ is above $U$)
  and $\spc(E)<\eta$,
  so $\pred^\Tt(\alpha+1)=-1$,
  so $M^{*\Tt}_{\alpha+1}=M$.
  By induction and the usual proof of Closeness, every component
  measure $E_a$ of $E$
  is in $U|\eta^{+U}$, since $E_a\sub U|\eta$
  and $\crit(F^U)>\eta$.
  But then $E_a\in M$
  by Claim \ref{clm:pow(eta)^U_sub_M},
  so $E$ is close to $M$.
 \end{casefive}
 
 \begin{casefive}
  $\ph$ is defined via case \ref{item:case_eta_non-card_active},
  so $\ph=((M,{<\gamma}),(\Ult_0(M,F^{M|\eta}),\gamma),U)$.

  Suppose $\mathrm{root}^\Tt(\alpha)=0$
  and $\pred^\Tt(\alpha+1)<0$.
  So $M^\Tt_\alpha$ is above $U$.
  Because $\gamma<\crit(F^U)$,
  we get each $E_a\in U$.
  But then by Claim \ref{clm:pow(eta)^U_sub_M},
  if $\pred^\Tt(\alpha+1)=-2$ then $E_a\in M$, and otherwise $E_a\in\Ult_0(M,F^{M|\eta})$.
  
  Now suppose $\mathrm{root}^\Tt(\alpha)=-1$ and $\pred^\Tt(\alpha+1)=-2$. Then
  as usual, $E$ is close to $\Ult_0(M,F^{M|\eta})$. If each $E_a\in\Ult_0(M,F^{M|\eta})$,
  then each $E_a\in M$. So suppose some $E_a\notin\Ult_0(M,F^{M|\eta})$. Then
  $E$ is a non-dropping image of $F$ where $F=F^{\Ult_0(M,F^{M|\eta})}$.
  So $\spc(E)=\spc(F)<\gamma$. Therefore $\spc(F^M)=\spc(F)=\spc(E)<\spc(F^{M|\eta})$.
  But then every $\bfSigma_1^{\Ult_0(M,F^{M|\eta})}$ subset of $\spc(E)$ is $\bfSigma_1^M$,
  and so $E$ is close to $M$.
 \end{casefive}

  \begin{casefive}
  $\ph$ is defined via case \ref{item:case_eta_non-card_passive},
  so $\ph=((M,{<\gamma}),((R,r),\gamma),U)$.
  
 The argument in the previous
 case also works here when $\mathrm{root}^\Tt(\alpha)=0$ and $\pred^\Tt(\alpha+1)<0$.
 If $\mathrm{root}^\Tt(\alpha)=-1$
 and $\pred^\Tt(\alpha+1)=-2$
 then, also as there, $E$ is close to $R$. But $R\in M$, so each $E_a$ is in $M$.\qedhere
 \end{casefive}
\end{proof}

We now compare $\ph$ versus $M$,
producing trees $\Uu$ and $\Tt$ respectively, via $\Gamma$ and $\Sigma$.

\begin{clm}
 The comparison terminates with $b^\Uu$ above $U$, $b^\Uu,b^\Tt$ are non-dropping, and $i^\Uu\com j=i^\Tt$.\end{clm}
 \begin{proof}
 Comparison termination is as in the proof of Theorem \ref{tm:comparison}.
 Also as there, it is not the case that $b^\Uu$ drops in model, degree or Dodd-degree, and $b^\Tt$ drops in model.
 (Here all base nodes are at degree $0^-$, except that in case \ref{item:case_eta_non-card_passive},
 $R$ is at degree $r>0^-$.)
 Weak Dodd-Jensen arguments give that $b^\Uu$ is not above $M$,
 and that if $b^\Uu$ is above $U$
 then neither $b^\Tt$ nor $b^\Uu$ drops, and that $i^\Uu\com j=i^\Tt$.
 If $b^\Uu$ is above $(R,r)$
 then by weak Dodd-Jensen
 (and considering how images of $R$ are lifted), $b^\Tt$ drops, but then because $R$ is $(r+1)$-sound and Dodd-sound, we have the same contradiction
 as when a drop occurs on both sides.
 So suppose $\ph$ is defined via case \ref{item:case_eta_non-card_active}
 and $b^\Uu$ is above $\Ult_0(M,F^{M|\eta})$. Again weak Dodd-Jensen
 gives that $b^\Uu,b^\Tt$ are non-dropping and $i^\Uu\com i^{M,0}_{F^{M|\eta}}=i^\Tt$. Let $\beta+1=\min((-1,\infty]^\Uu)$. So $\spc(E^\Uu_\beta)=\gamma$.
 Clearly $E^\Tt_0=F^{M|\eta}$.
 Now use calculations as in the proof of Theorem \ref{tm:comparison}
 to see that $1<^\Tt\infty$,
 so $F^{M|\eta}$ is the first extender used along $b^\Tt$, and letting
 $\alpha+1=\min((1,\infty]^\Tt)$, 
 then $E^\Uu_\beta=E^\Tt_\alpha$
 (so in fact $\beta=\alpha$),
 contradicting comparison.
 (Note here that if $\gamma$ is a limit cardinal of $M$ then $F^{M|\eta}$ is either short or long with a largest generator,
 whereas if $\gamma=\mu^{+M}$
  where $\mu$ is an $M$-cardinal
  then
  $F^{M|\eta}$ is long with no largest generator, and $\nu(F^{M|\eta})=\gamma$. These cases  handled similarly to how they were handled in the proof of Theorem \ref{tm:comparison}.)
 \end{proof}

 Recall from the start of the proof that $U=U_{i_0}$ and $j=j_{i_0}:M\to U$
 is the ultrapower map, and so $F_{i_0}$ is derived from $j$.
 We have $F_{i_0}\notin M$.
 If $i_0<n$,
 we want to reach a contradiction
 by showing that $F_{i_0}\in N$.
 Well, we now know that $F_{i_0}$
 can also be derived from $i^\Tt$,
 possibly shifting generators up with $i^\Uu$. The hope is to show that if $i_0<n$, then enough of $i^\Tt$ can be internalized into $M$ that we can deduce that $F_{i_0}\in M$.
 For this we need to analyse more carefully the extenders used in $\Tt$,
 and especially along $b^\Tt$.

 Let $\lambda_\infty=i^{\Uu}_{0\infty}(j(\kappa))$.
 
  \begin{clm}\label{clm:chain_of_exts_in_branches_DS}\ 
\begin{enumerate}
\item For every $\gamma+1\in b^\Uu$,
we have $\lambda(E^\Uu_\gamma)<\lambda_\infty$;
equivalently,
  $\crit(E^\Uu_\gamma)<i^\Uu_{0\delta}(j(\kappa))$ where $\delta=\pred^\Uu(\gamma+1)$.
  
 \item  For every $\gamma+1\in b^\Tt$,
 $E^\Tt_\gamma$ is short and
 $\crit(E^\Tt_\gamma)=i^\Tt_{0\delta}(\kappa)$
 where $\delta=\pred^\Tt(\gamma+1)$.
 Moreover,
  $\lambda(E^\Tt_\gamma)\leq\lambda_\infty$.
 \end{enumerate}
\end{clm}
\begin{proof}
 Let $\xi$ be least in $b^\Uu$ such that
 either $M^\Uu_\infty=M^\Uu_\xi$
 or $i^\Uu_{0\xi}(j(\kappa))\leq\crit(i^\Uu_{\xi\infty})$.
 
 We claim that
 \begin{equation}\label{eqn:either_end_of_Uu_or_low_crit}\text{ either }M^\Uu_\infty=M^\Uu_\xi\text{
 or }i^\Uu_{0\xi}(j(\kappa))=\crit(i^\Uu_{\xi\infty})<\lambda_\infty.\end{equation}
 For suppose 
  $M^\Uu_\infty\neq M^\Uu_\xi$
 and $\lambda_\infty=i^\Uu_{0\xi}(j(\kappa))<\crit(i^\Uu_{\xi\infty})$.
 Let $\zeta\in b^\Tt$ be least such that $i^\Tt(\kappa)=\lambda_\infty$,
 so $\lambda_\infty=i^\Tt_{0\zeta}(\kappa)<\crit(i^\Tt_{\zeta\infty})$, whereas $\crit(i^\Tt_{\zeta'\infty})=i^\Tt_{0\zeta'}(\kappa)$ for all $\zeta'<^\Tt\zeta$.
 Now note that $M^\Uu_\xi=M^\Tt_\zeta$,
 which is impossible. This establishes line (\ref{eqn:either_end_of_Uu_or_low_crit}).

 Let $\theta=i^\Uu_{0\xi}(j(\kappa))$. We now claim that
 \begin{equation}\label{eqn:Tt_using_only_short_exts_when_crit<theta} E^\Tt_\delta\text{ is short for every }\delta+1\in b^\Tt\text{ with }\crit(E^\Tt_\delta)<\theta.\end{equation}
 For suppose otherwise and let $\delta$ be least such.
 Let $\mu=\crit(E^\Tt_\delta)$
 and $\beta=\pred^\Tt(\delta+1)$.
 Note that $i^\Tt_{0\beta}(\kappa)=\mu<\theta$ and $i^\Tt_{\beta\infty}(\mu)=\lambda_\infty$.
 Now $i^\Tt_{0\beta}$ is continuous at $\kappa^{+M}$, so $\mu^{+M^\Tt_\beta}=\sup i^\Tt_{0\beta}``\kappa^{+M}$.
 But $i^\Tt_{\beta\infty}$ is discontinuous at $\mu^{+M^\Tt_\beta}$
 and letting
 \[ \eta=\sup i^\Tt_{\beta\infty}``\mu^{+M^\Tt_\beta}=\sup i^\Tt``\kappa^{+M}, \]then $F^{M^\Tt_\infty|\eta}$ is  the short extender derived from $i^\Tt_{\beta\infty}$,
 and in particular has critical point $\mu$. But $i^\Uu_{0\xi}$ is continuous at $\kappa^{+M}$. So
 \[ \eta=\sup i^{\Uu}_{\xi\infty}``i^{\Uu}_{0\xi}(j(\kappa^{+M}))=\sup i^{\Uu}_{0\xi}``\theta^{+M^\Uu_\xi}.\]
 Since $M^\Uu_\infty|\eta$
 is active, there is a long extender used along $(\xi,\infty]^\Uu$.
 But letting $\gamma+1\in(\xi,\infty]^\Uu$
 be least such, then $F^{M^\Uu_\infty|\eta}$ has critical point $\crit(E^\Uu_\gamma)\geq\theta$,
 a contradiction. This gives line (\ref{eqn:Tt_using_only_short_exts_when_crit<theta}).
 
 It now suffices to see that $M^\Uu_\infty=M^\Uu_\xi$,
 so suppose otherwise,
 so $i^\Uu_{\xi\infty}\neq\id$.
 
  Let $E$ be the short extender derived from $i^\Uu\com j$.
 Note that $E\rest\theta$ is whole.
 The same short extender is derived from $i^\Tt$, so
 using line (\ref{eqn:Tt_using_only_short_exts_when_crit<theta}),
 it easily follows
 that $\lambda(E^\Tt_\delta)\leq\theta$ for every $\delta+1\in b^\Tt$
 with $\crit(E^\Tt_\delta)<\theta$.
Since $M^\Uu_\infty\neq M^\Uu_\xi$,
it follows that there is $\zeta+1\in b^\Tt$ with $\theta\leq\crit(E^\Tt_\zeta)$.
Let $\zeta$ be least such
and $\chi=\pred^\Tt(\zeta+1)$.
Now note that $M^\Uu_\xi=M^\Tt_\chi$,
again a contradiction.
\end{proof}

  Recall that
  $p^M_{\mathrm{D}}=\{p_0,p_1,\ldots,p_{n-1}\}$, $\rho_{\D}^M=0$,
  $\eta=p_{i_0}$ if $i_0<n$,
  and $\eta=\kappa^{+M}$ otherwise.
  Let $q=\{p_0,p_1,\ldots,p_{i_0-1}\}$ and $\pi(\bar{q})=q$. Recall that $i_0$ is the least $i\leq n$ such that $F_{i}\notin M$.
  
  \begin{clm}\label{clm:if_i_0<n_no_use_of_F^M}
   If $i_0<n$ (equivalently, $q\neq p_{\mathrm{D}}^M$) then there is no $\alpha+1\in b^\Tt$
   such that $[0,\alpha]^\Tt\cap\dropset^\Tt=\emptyset$ and $E^\Tt_\alpha=F(M^\Tt_\alpha)$.
  \end{clm}
\begin{proof}
 Suppose otherwise and let $\alpha_0$ be least such.
 
 \begin{sclm}
  Let $q'=i^\Tt_{0\alpha_0}(q)$.
  Then  $q'=i^\Uu_{0\infty}(\bar{q})$.
 \end{sclm}
\begin{proof}
This is trivial if $q=\emptyset$,
so suppose $q\neq\emptyset$.

Suppose $i^\Uu_{0\infty}(\bar{q})<q'$.
Let $\eta'=\sup i^\Uu_{0\infty}``\eta$;
so $\eta\leq\min(q')$.
Note that by the minimality of $\alpha_0$ and normality of $\Tt$, there is no $\beta+1\in(0,\alpha_0]^\Tt$ such that $[0,\beta]^\Tt\cap\dropset^\Tt=\emptyset$ and $E^\Tt_\beta=F(M^\Tt_\beta)$.

Therefore by a finite support argument, there is a linear ``iteration''
$\vec{\mu}$ internal to $M$ which captures $(M^\Tt_{\alpha_0},i^\Uu_{0\infty}(\bar{q})\cup\{\eta'\},\eta+1)$.
That is, there are finite sequences $\left<N_i\right>_{i\leq k+1}$
and $\vec{\mu}=\left<\mu_i\right>_{i\leq k}$
such that $N_0=M$, $\mu_i\in N_i\sats$``$\mu_i$ is an ultrafilter'', $N_{i+1}=\Ult_0(N_i,\mu_i)$, and letting $i^{M,0}_{\vec{\mu}}:M\to N_{k+1}$ be the ultrapower map, then there is also a $0$-embedding $\sigma:N_{k+1}\to M^\Tt_{\alpha_0}$
such that $\sigma\com i^{M,0}_{\vec{\mu}}=i^\Tt_{\alpha_0}$,
and $i^\Uu_{0\infty}(\bar{q})\cup\{\eta'\}\in\rg(\sigma)$
and $\eta+1\sub\rg(\sigma)$.

Now $(M,q)$ is Dodd-solid,
so let $w\in M$ be the  Dodd-solidity witness
corresponding to where $i^\Uu_{0\infty}(\bar{q})$
drops strictly below $q'$.
That is, let $j$ be largest such that $i^{\Uu}_{0\infty}(\bar{q})\rest j=q'\rest j$
(so $i^{\Uu}_{0\infty}(\bar{q})_j<(q')_j$),
and let 
\[ w=F^M\rest ((q\rest j)\cup q_j).\]
With $\sigma$ as above,
let $\sigma(\bar{\eta}')=\eta'$ and $\sigma(r)=i^\Uu_{0\infty}(\bar{q})$.
If $\eta'=\eta\leq\crit(i^\Uu)$,
then let $\bar{E}$ be the extender derived from $i^{M,0}_{\vec{\mu}}(w)$ 
with generators $\eta\cup\{r\}$, and note that this collapses to $F_{i_0}$,
so $F_{i_0}\in M$, a contradiction.
So $\crit(i^\Uu)<\eta$,
so letting $\beta+1=\min((0,\infty]^\Uu)$, then $E^\Uu_\beta$ is long with $\spc(E^\Uu_\beta)=\eta$.
Therefore $\eta<\eta'$ and
$M^\Uu_\infty|\eta'$ is active
with the short extender derived from $i^\Uu_{0\infty}$. This reflects into $N_{k+1}=\Ult_0(M,\vec{\mu})$ at $\bar{\eta}'$. Let $G=F^{\Ult_0(M,\vec{\mu})|\bar{\eta}'}$ and let 
$\bar{E}$ be the extender 
derived from $i^M_{\vec{\mu}}(w)$
with generators $(G``\eta)\cup\{r\}$,
and note that this is equivalent to $F_{i_0}$,
so $F_{i_0}\in M$ again, a contradiction.

So
$q'\leq i^\Uu_{0\infty}(\bar{q})$.
Suppose $q'<i^\Uu_{0\infty}(\bar{q})$.
By Claim \ref{clm:chain_of_exts_in_branches_DS},
we can fix a linear ``iteration'' $\vec{\mu}$ internal to $U$,
and in fact, with $\vec{\mu}\in U|j(\kappa)$, which captures $(M^\Uu_\infty,p',\eta+1)$,
where $p'=i^\Tt_{0\alpha_0}(p_{\mathrm{D}}^M)$.
Let $\sigma:\Ult_0(U,\vec{\mu})\to M^\Uu_\infty$ be the final copy map.
So $\crit(\sigma)>\eta$
and $p'\in\rg(\sigma)$.
Let $\sigma(p'')=p'$.
Let $N$ be the natural candidate
for a premouse with $N^{\passive}=U|j(\kappa)^{+U}$
and whose active extender is equivalent to $F_{i_0}$.
Thus, letting $\bar{\pi}=\pi\rest N$,
then $\bar{\pi}:N\to M$
is $0$-deriving.
We have $\vec{\mu}\in N|j(\kappa)=U|j(\kappa)$. Let $N_\infty=\Ult_0(N,\vec{\mu})$. Note that $F^M\rest p_{\D}^M$ is equivalent to $F^{N_\infty}\rest p''$.
Let $M_\infty=\Ult_0(M,\bar{\pi}(\vec{\mu}))$ and $p'''=\tau(p'')$
where $\tau:M_\infty\to N_\infty$ is the final copy map. Then $F^{N_\infty}\rest p''$ is equivalent to $F^{M_\infty}\rest p'''$.
It follows that there is a $0$-deriving embedding $k:M\to M_\infty$
with $k(p_{\D}^M)=p'''$.
(Use $\rg(k)=\{i^{M_\infty,0}_{F^{M_\infty}}(f)(p''')\bigm|f\in M_\infty|\kappa^{+M}\}$.)
But $p'''<i^{M,0}_{\pi(\vec{\mu})}(p_{\D}^M)$, which contradicts the weak Dodd-Jensen property for
$\Sigma$ with respect to $\vec{x}$,
since
 $p_{\D}^M\ins\vec{x}$.
\end{proof}

Now since $i_0<n$, we have $\kappa^{+M}<\eta$.
Let $p'=i^\Tt_{0\alpha_0}(p_{\D}^M)$.
Let $\vec{\mu}\in U|j(\kappa)$
be a linear ``iteration''
of $U$ capturing $(M^\Uu_{\infty},
p',\eta+1)$.
Let $\sigma:\Ult_0(U,\vec{\mu})\to M^\Uu_\infty$ be the final copy map,
so $\crit(\sigma)>\eta$
and $p'\in\rg(\sigma)$
and $\sigma\com i^{U,0}_{\vec{\mu}}=i^\Uu_{0\infty}$.
Let $\sigma(r)=p'$.
Let $a\in[\eta]^{<\om}$
be such that:
\begin{enumerate}[label=--]
 \item 
$\vec{\mu}$
is generated by $a\cup\bar{q}$,
in that $\vec{\mu}=j(f)(a\cup\bar{q})$
for some $f\in M|\kappa^{+M}$,
and\item letting $s$ be the canonical seed of $\vec{\mu}$, then  $r=i^{U,0}_{\vec{\mu}}(j(g)(a\cup\bar{q}))(s)$
for some $g\in M|\kappa^{+M}$.
\end{enumerate}
Then by the minimality of $p_{\D}^M$,
we have $F^M\rest (q\cup a)\in M$.
But note that using $F^M\rest (q\cup a)$, we can recover an extender equivalent to $F^M\rest p_{\D}^M$
(using the natural preimage of $r$).
So $F^M\rest p_{\D}^M\in M$, 
a contradiction.
\end{proof}

 \begin{clm}\label{clm:i_0=n}
  $i_0=n$.
 \end{clm}
\begin{proof}
 Suppose $i_0<n$.
 
 Suppose $\crit(i^\Uu)\geq\eta$.
 By Claim \ref{clm:if_i_0<n_no_use_of_F^M},
 we can find a linear ``iteration''
 $\vec{\mu}\in M|\lambda(F^M)$
 capturing $(M^\Tt_\infty,i^\Uu(\bar{q}),\eta+1)$,
 so letting $\sigma:\Ult_0(M,\vec{\mu})\to M^\Tt_\infty$ be the final copy map, we have $i^\Uu(\bar{q})\in\rg(\sigma)$,
 $\crit(\sigma)>\eta$, and $\sigma\com i^{M,0}_{\vec{\mu}}=i^\Tt$.
 But then note that we can derive $F^M\rest (q\cup\eta)$ as a subextender of that of $i^{M,0}_{\vec{\mu}}$, so $F^M\rest(q\cup\eta)\in M$, contradicting the choice of $i_0$.
 
 So $\crit(i^\Uu)<\eta$.
 So letting $\beta+1=\min((0,\infty]^\Uu)$, then $E^\Uu_\beta$ is long and $\crit(E^\Uu_\beta)=\mu$,
 and $\eta=\mu^{+U}<\mu^{+M}$. Let $\eta'=\sup i^\Uu``\eta$, so $F^{M^\Uu_\infty|\eta'}$ is the short part of $i^\Uu$. Let $\vec{\mu}$ etc be like before, but now also with $\eta'\in\rg(\sigma)$.
 Let $\sigma(\bar{\eta})=\eta'$.
 Then using $i^{M,0}_{\vec{\mu}}$
 and $F^{\Ult_0(M,\vec{\mu})|\bar{\eta}}$, we can recover $F^M\rest(q\cup\eta)$, so $F^M\rest(q\cup\eta)\in M$, contradicting the choice of $i_0$ again.
\end{proof}

By the previous claim, $M$ is Dodd-absent-solid,  $U=\Ult_0(M,F^M\rest p_{\D}^M)$, 
 $\eta=\kappa^{+M}=\kappa^{+U}$ and $\ph=((M,{<\kappa^{+M}}),U)$,
and it just remains to see that
$M$ is Dodd-absent-universal; that is, that $U|\kappa^{++U}=M|\kappa^{++M}$.
So we may assume that $\kappa^{++U}<\kappa^{++M}$, and so either $\crit(\pi)=\kappa$
or $\crit(\pi)=\kappa^{++U}$.
Let $\pi(\bar{p})=p_{\D}^M$.

 \begin{clm}There is $\alpha+1\in b^\Tt$
  such that $E^\Tt_\alpha=F(M^\Tt_\alpha)$
  and $[0,\alpha]^\Tt\cap\dropset^\Tt=\emptyset$.
 \end{clm}
 \begin{proof}
 Otherwise argue like in the proof of Claim \ref{clm:i_0=n} 
 to show  $F^M\rest p_{\mathrm{D}}^M\in M$,
 using that $F^M\rest p_{\mathrm{D}}^M$ is the same as the measure derived from $i^\Tt$ at $i^\Uu(\bar{p})$.
 \end{proof}

 Let $\alpha_0$ be least witnessing the claim.
 Let $p'=i^\Tt_{0\alpha_0}(p_{\mathrm{D}}^M)$.

  \begin{clm}\label{clm:short_ext_from_pi_is_in_M}
 Suppose $\crit(\pi)=\kappa$.
 Then the short extender $E$ derived from  $\pi$
 is a whole proper segment of $F^M$, so $E\in M$.\end{clm}
 \begin{proof}
 Let $F=F^M\rest p_{\mathrm{D}}^M$.
 Given $X\in\pow(\kappa)\cap M$,
  we have $F^M(X)=\pi(F(X))$.
  So $F^M(X)\cap\pi(\kappa)=\pi(F(X))\cap\pi(\kappa)=\pi(F(X)\cap\kappa)=\pi(X)$.
  So $E$ is a proper segment of $F^M$.
  Now let us see that $E$ is whole.
  Let $f:\kappa\to\kappa$ with $f\in M$.
  Suppose there is $\alpha<\pi(\kappa)$ such that $F^M(f)(\alpha)\geq\pi(\kappa)$.
  Letting $\alpha'$ be the least such,
  then $\alpha'\in[\kappa,\pi(\kappa))$,
  since $F^M(f)\rest\kappa=f:\kappa\to\kappa$.
  We have $\pi(F(f))=F^M(f)$.
  But let $g:\kappa\to\kappa$
  be the function where given $\eta<\kappa$,
  $g(\eta)$ is the least $\alpha<\eta$
  such that $f(\alpha)\geq\eta$, if there is such an $\alpha$, and $g(\alpha)=0$ otherwise.
  Then $\pi(F(g))(\pi(\kappa))=F^M(g)(\pi(\kappa))=
  \alpha'$, so $\alpha'\in\rg(\pi)$, contradiction.
  (Or more simply, since $\pi(X)=F(X)\cap\pi(\kappa)$ for all $X\sub\kappa$, if $f:\kappa\to\kappa$
  then $F^M(X)\cap\pi(\kappa)=\pi(f):\pi(\kappa)\to\pi(\kappa)$, so $\pi(\kappa)$ is closed under $F^M(f)$, as desired.)
  
  Since $E$ is a whole proper segment of $F^M$,
 $E\in M$ by the Jensen ISC.
 \end{proof}

 \begin{clm}\label{clm:p'=i^Uu(p-bar)}
  $p'=i^\Uu(\bar{p})$.
 \end{clm}
 \begin{proof}
 Suppose $p'>i^\Uu(\bar{p})$.
 Note that $F^M\rest p_{\mathrm{D}}^M$ is the measure derived from $i^\Tt$ at $i^\Uu(\bar{p})$. But we can again use a finite support calculation to find $\vec{\mu}\in M|\lambda(F^M)$
 capturing $(M^\Tt_{\alpha_0}, i^\Uu(\bar{p}),\kappa^{+M})$, producing $\bar{i}:M\to \Ult_0(M,\vec{\mu})$,
 and then use the image $\bar{i}(w)$ of the appropriate
 Dodd-absent-solidity witness $w$ for $p_{\mathrm{D}}^M$ in $M$, to compute $F^M\rest p_{\mathrm{D}}^M$ in $M$; contradiction. 
 
 Now suppose instead that $p'<i^\Uu(\bar{p})$. The foregoing argument
 doesn't immediately work here, since we don't know that the relevant Dodd-absent-solidity witnesses
 are in $U$. But instead, we can lift the relevant internal ``iteration'' of $U$ to one on $U'$, where we do have those witnesses. That is, since $j(\kappa)=\crit(F^U)$,
 $F^U$ and its images never get used in an essential manner in $\Uu$.
 That is,  we can do a finite support calculation to produce an internal ``iteration'' of $U$ capturing $p'$. Let $\muvec\in U$ be the resulting finite sequence of measures,
 let $\sigma:\Ult_0(U,\muvec)\to M^\Uu_\infty$
 be the final copy map, and let $\sigma(\bar{p})=p'$.
 So $\bar{p}<i^U_{\muvec}(\bar{p})$. Now we copy $\muvec$ up with $\pi$,
 producing $\pi\muvec$ on $U'$,
 and the maps commute. So $\pi\muvec\in U'$
 is internal to $U'$, and everything commutes, so since the short extender
 derived from $\pi$ is in $M$ (by Claim \ref{clm:short_ext_from_pi_is_in_M}) then we get $F^M\rest p_{\mathrm{D}}^M\in M$, contradiction.
 \end{proof}

 Fix some $f,f^+,f^{++}\in M|\kappa^{+M}$ such that $j(f)(\bar{p})=\kappa$
 and $j(f^+)(\bar{p})=\kappa^{+U}=\kappa^{+M}$ and $j(f^{++})(\bar{p})=\kappa^{++U}$,
 so also $i^M_{F^M}(f)(p_{\mathrm{D}}^M)=\pi(\kappa)$
and $i^M_{F^M}(f^+)(p_{\mathrm{D}}^M)=\pi(\kappa)^{+M}$
and $i^M_{F^M}(f^{++})(p_{\mathrm{D}}^M)=\pi(\kappa)^{++M}$.
 \begin{clm}\label{clm:p'=i^U(pi^-1(p^M_D))}
Suppose $\kappa<\crit(\pi)$. Then $U|\kappa^{++U}=M|\kappa^{++M}$.\end{clm}
\begin{proof}
Suppose otherwise; then $\xi=\kappa^{++U}=\crit(\pi)$
and $\pi(\kappa^{++U})=\kappa^{++U'}=\kappa^{++M}$. 
Let $\xi'=i^\Tt_{0\alpha_0}(\xi)$,
and note that $\kappa'=i^\Tt(f)(p')=i^\Tt_{0\alpha_0}(\kappa)$ and
\begin{equation}\label{eqn:where_eta'_is} (\kappa')^{+M^\Tt_\infty}=i^\Tt(f^+)(p')<\xi'<i^\Tt(f^{++})(p') \end{equation}
and
\begin{equation}\label{eqn:set_bd_by_eta'} \{i^\Tt(g)(p')\bigm|g\in M\}\cap[i^\Tt(f^+)(p'),i^\Tt(f^{++})(p'))\sub\xi'.\end{equation}
On the other hand, $\kappa=j(f)(\pi^{-1}(p_{\mathrm{D}}^M))$ and likewise for $\kappa^{+U}$ and $\kappa^{++U}$, so
\[ \{j(g)(\pi^{-1}(p_{\mathrm{D}}^M))\bigm|g\in M\}\cap[j(f^+)(\pi^{-1}(p_{\mathrm{D}}^M)),j(f^{++})(\pi^{-1}(p_{\mathrm{D}}^M))) \]
is unbounded in $\kappa^{++U}$, and $i^\Uu$ is continuous at $\kappa^{++U}$, and by Claim \ref{clm:p'=i^U(pi^-1(p^M_D))}, $i^\Uu(\pi^{-1}(p_{\mathrm{D}}^M))=p'$, so applying $i^\Uu$ pointwise to the previous line, we get that
\[ \{i^\Uu(j(g))(p')\bigm|g\in M\}\cap[i^\Uu(j(f^+))(p'),i^\Uu(j(f^{++}))(p')) \]
is unbounded in $i^\Uu(j(f^{++}))(p')$.
Since $i^\Uu\com j=i^\Tt$, this contradicts the conjunction of lines (\ref{eqn:where_eta'_is})
and (\ref{eqn:set_bd_by_eta'}).
\end{proof}

So from now on we can assume that $\kappa=\crit(\pi)$.
\begin{clm}
 Suppose $\kappa=\crit(\pi)$. 
 Then $\crit(i^\Uu)=\kappa$,
 and so the first extender used along $b^\Uu$ is long with critical point $\kappa$.
\end{clm}
\begin{proof}
With the help of Claim \ref{clm:p'=i^Uu(p-bar)}, we have
\[i^\Uu(\kappa)=i^\Uu(j(f))(i^\Uu(\bar{p}))=i^\Tt(f)(p')=i^\Tt_{0\alpha_0}(\pi(\kappa))\geq\pi(\kappa)>\kappa, \]
so $i^\Uu(\kappa)>\kappa$, as desired.
\end{proof}

Now let $\gamma_0+1=\min((0,\infty]^\Uu)$, so $\crit(E^\Uu_{\gamma_0})=\kappa$ and $E^\Uu_{\gamma_0}$ is long.
Then the short extender $E'$ derived from $i^\Uu$
is in $\es^{M^\Uu_\infty}$, and is indexed
at $\sup i^\Uu``\kappa^{+M}$.
In the following proof,
recall  from Claim  \ref{clm:short_ext_from_pi_is_in_M} that if $\crit(\pi)=\kappa$
then $E=F^M\rest\pi(\kappa)$ coincides with the short extender derived from $\pi$.

\begin{clm}\label{clm:pred(alpha_0+1)=0}
$\pred^\Tt(\alpha_0+1)=0$,
so $\crit(E^\Tt_{\alpha_0})=\kappa$.
\end{clm}
\begin{proof}
Suppose not. Let $\beta_0=\pred^\Tt(\alpha_0+1)$.
Since $\beta_0\in b^\Tt$, all extenders used along $(0,\beta_0]^\Tt$ are short (see Claim \ref{clm:chain_of_exts_in_branches_DS}).
Let $N=M^\Tt_{\beta_0}$
and $k=i^\Tt_{0\beta_0}:M\to N$.
We have $\kappa=\crit(k)$, and note that $\beta_0\leq^\Tt\alpha_0$.
So $\crit(k(E))=k(\kappa)=\crit(E^\Tt_{\alpha_0})>\kappa$.
Now note that letting $\bar{M}=M|\kappa^{+M}$, we have
\[\lh(E)=\sup\pi``\kappa^{+M}=\sup\Big(\Big\{i^{M,0}_{F^M}(g)(p_{\D}^M)\Bigm|g\in \bar{M}\Big\}\cap\pi(\kappa)^{+M}\Big).\]
Since $k$  is continuous at  $\kappa^{+M}$ and hence also at $\lh(E)$, we have
\begin{equation}\label{eqn:lh(k(E))} \lh(k(E))=\sup k\com\pi``\kappa^{+M}=\sup\Big(\Big\{i^{N,0}_{F^N}\com k(g)(k(p_{\D}^M))\Bigm|g\in \bar{M}\Big\}\cap k(\pi(\kappa))^{+N}\Big).\end{equation}

Now suppose that $\beta_0=\alpha_0$ or $k(\kappa)<\crit(i^\Tt_{\beta_0\alpha_0})$. Then
letting $N'=M^\Tt_{\alpha_0}$ and $k'=i^\Tt_{0\alpha_0}$, line (\ref{eqn:lh(k(E))}) still holds after replacing each ``$N$'' with ``$N'$''
and each ``$k$'' with ``$k'$''.
Since $p'=k'(p_{\D}^M)$
and $j(f^+)(\bar{p})=\kappa^{+M}$
and $i^{M,0}_{F^M}(f^+)(p_{\D}^M)=\pi(\kappa)^{+M}$,
this gives that
\[ \lh(k'(E))=\sup\Big(\Big\{i^\Tt(g)(p')\Bigm|g\in\bar{M}\Big\}\cap i^\Tt(f^+)(p')\Big).\]
Therefore
\[ \lh(k'(E))=\sup\Big(\Big\{i^\Uu\com j(g)(i^\Uu(\bar{p}))\Bigm|g\in\bar{M}\Big\}\cap i^\Uu\com j(f^+)(i^\Uu(\bar{p}))\Big).\]
But \[ j(f^+)(\bar{p})=\kappa^{+M}=\sup\Big(\Big\{j(g)(\bar{p})\Bigm|g\in\bar{M}\Big\}\cap j(f^+)(\bar{p})\Big)\]
and $\lh(E')=\sup i^\Uu``\kappa^{+M}$, 
and it follows that $\lh(E')=\lh(k'(E))$,
so $E'=k'(E)$. But $\crit(E')=\kappa$ whereas $\crit(k'(E))=\crit(E^\Tt_{\alpha_0})>\kappa$, a contradiction.

So in fact $\beta_0<^\Tt\alpha_0$
and $k(\kappa)=\crit(i^\Tt_{\beta_0\alpha_0})$, so letting
 $\gamma_0+1=\min((\beta_0,\alpha_0]^{\Tt})$, then $E^\Tt_{\gamma_0}$ is long with $k(\kappa)=\crit(E^\Tt_{\gamma_0})$,
so $M^\Tt_{\gamma_0+1}=\Ult_0(M^\Tt_{\beta_0},E^\Tt_{\gamma_0})$ is formed avoiding the protomouse. Let $\ell:M\to M^\Tt_{\gamma_0+1}$ be the ultrapower map. Then $\crit(\ell(E))=\ell(\kappa)$. Let $G$ be the short part of $E^\Tt_{\gamma_0}$, so $G\in\es(M^\Tt_{\gamma_0+1})$,
and $\lambda(G)=\ell(\kappa)$,
and note that $\lh(\ell(E)(G))=\sup\ell``\lh(E)$.
Moreover, $i^\Tt_{\gamma_0+1,\alpha_0}$ is continuous at $\lh(\ell(E)(G))$,
since $\cof^{M^\Tt_{\gamma_0+1}}(\lh(\ell(E)(G)))=\crit(E^\Tt_{\gamma_0})^{+M^\Tt_{\gamma_0+1}}$.
Also, $\crit(\ell(E)(G))=\crit(E^\Tt_{\gamma_0})$,
so $\crit(i^\Tt_{\gamma_0+1,\alpha_0}(\ell(E)(G)))=\crit(E^\Tt_{\gamma_0})$.
But now much as before,
we get that $i^\Tt_{\gamma_0+1,\alpha_0}(\ell(E)(G))=E'$, whereas $\crit(E')=\kappa$, a contradiction.
\end{proof}

Now if either $\alpha_0=0$ or $\crit(i^\Tt_{0\alpha_0})>\kappa$,
we immediately have that $\kappa^{++M}=\kappa^{++M^\Tt_\infty}=\kappa^{++M^\Uu_\infty}$,
so that $\kappa^{++U}=\kappa^{++M}$, as desired.
So suppose that $\alpha_0>0$ and $\crit(i^\Tt_{0\alpha_0})=\kappa$.
Let $\gamma_0+1=\min((0,\alpha_0]^\Tt)$.
By the Claim \ref{clm:pred(alpha_0+1)=0},
$\crit(E^\Tt_{\alpha_0})=\kappa$,
and so $E^\Tt_{\gamma_0}$ is long
with $\crit(E^\Tt_{\gamma_0})=\kappa$.
But $E^\Tt_{\gamma_0}$ is $M$-total, 
so measures all of $M|\kappa^{++M}$, and so $\kappa^{++M}=\kappa^{++M^\Tt_\infty}$, so  $\kappa^{++U}=M|\kappa^{++M}$ like before. This completes the proof of the lemma.
\end{proof}

We now want to show that  the Dodd-absent-core is well-behaved:

\begin{tm}\label{tm:Dodd-absent-core_of_M_when_rho_D^M=0}
 Suppose $M$ is $(0^-,\om_1,\om_1+1)^*$-iterable
 active short, and $\rho_{\mathrm{D}}^M=0$.
 Let $C=\core_{\mathrm{D}}(M)$ and $\pi:C\to M$ be the Dodd-absent core map. Let $\kappa=\crit(F^M)$.
 Then: 
 \begin{enumerate}
 \item\label{item:C_sats_Jensen_ISC}$C$ satisfies the Jensen ISC,
 and hence is a premouse,
 \item\label{item:C_is_Dodd-absent-sound} $C$ is Dodd-absent sound,
 \item $\rho_{\mathrm{D}}^C=0$ and $\pi(p_{\mathrm{D}}^C)=p_{\mathrm{D}}^M$,
 \item Either $F^M$ has no whole proper segment
 or has a largest one,
 \item Either $F^C$ has no whole proper segment or has a largest one,
 \item If $\kappa<\crit(\pi)$
 then $\pi(F_J^C)=F_J^M$.
 \item If $\kappa=\crit(\pi)$ and $F_J^C=\emptyset$
 then $F_J^M=F^M\rest\pi(\kappa)$.
 \item\label{item:kappa=crit(pi)_F^C_J_non-empty} If $\kappa=\crit(\pi)$
 and $F_J^C\neq\emptyset$
 then $F_J^M=F^M\rest\pi(\lambda)$
 where $F_J^C=F^C\rest\lambda$.
 \item\label{item:transfer_definability} Every $\bfSigma_1^M$ subset of $\kappa$ is $\bfSigma_1^C$.
 \end{enumerate}
\end{tm}
\begin{proof}
 The main thing is to verify parts
 \ref{item:C_sats_Jensen_ISC}
 and \ref{item:transfer_definability}.

 Assuming part \ref{item:C_sats_Jensen_ISC},
 let us deduce that parts \ref{item:C_is_Dodd-absent-sound}--\ref{item:kappa=crit(pi)_F^C_J_non-empty} hold.
 By Theorem \ref{tm:rho^M_D=0_implies_Dodd-absent-solid_and_universal}, $M$ is Dodd-absent-solid
and Dodd-absent-universal.
 By the Jensen ISC, we know that $C$ is a premouse,
 and clearly $\rho_{\mathrm{D}}^C=0$,
 and by Dodd-absent-universality of $M$,
 we have $p_{\mathrm{D}}^C=\pi^{-1}(p_{\mathrm{D}}^M)$.
 Also, $C$  is $(0^-,\om_1,\om_1+1)^*$-iterable
 (by Theorem \ref{lem:0-deriving_copying}),
 so $C$ is also Dodd-absent-solid,
 and hence Dodd-absent-sound.
 
 Now $F^M$ cannot have unbounded whole proper segments, since $\rho_1^M\leq\kappa^{+M}$,
 and likewise for $C$.
 If $F^C\rest\lambda$
 is whole, where $\kappa<\lambda<\lambda(F^C)$,
 as witnessed by $F\in C$, and $E$ is the short part of $E_\pi$ (which, recall, is just $F^M\rest\pi(\kappa)$), then note that $\pi(F)(E)=\pi(F)\com E$ is a segment of $F^M$,
 and since $E\in\es^M$,
 also $\pi(F)(E)\in\es^M$.
 Now suppose that $G\in\es^M$ is a whole proper segment of $F^M$ with $\lh(E)<\lambda(G)<\lambda(F^M)$.
 Let $\lambda=\pi^{-1}``\lambda(G)$.
 Then $\kappa<\lambda<\lambda(F^C)$,
 since $\pi(\kappa)=\lambda(E)$
 and $\pi(\lambda(F^C))=\lambda(F^M)$.
 It is easy to see that $F^C\rest\lambda$ is whole.
 So since $C$ satisfies the Jensen ISC,
 $F^C\rest\lambda\in C$, so $F^C\rest\lambda\in\es^C$ (as $C$ is iterable),
 and now note that $\pi(F^C\rest\lambda)(E)$
 is a whole proper segment of $F^M$.
 
 Now let us verify parts
 \ref{item:C_sats_Jensen_ISC}
 and \ref{item:transfer_definability}.
 Let $U=\Ult(M,F^C)$ and $j:M\to U$ be the ultrapower map. Compare the phalanx $((M,\kappa),U)$
 with $M$, using/lifting to a strategy for $M$ with weak Dodd-Jensen as in the proof of Theorem \ref{tm:rho^M_D=0_implies_Dodd-absent-solid_and_universal}
 (note that in particular,
 short extenders with critical point $\kappa$ are applied to $M$,
 even in case $\crit(\pi)=\kappa$).
 Also as there,
 we get $b^\Uu$
 above $U$ and $i^\Uu\com j=i^\Tt$, and there is $\alpha+1\in b^\Tt$ such that $[0,\alpha]^\Tt\cap\dropset^\Tt=\emptyset$
 and $E^\Tt_\alpha=F(M^\Tt_\alpha)$.
 Let $\alpha_0$ be least such.
 As before, letting $p'=i^\Tt_{0\alpha_0}(p_{\mathrm{D}}^M)$ and $\pi(\bar{p})=p_{\D}^M$, we have $p'=i^\Uu_{0\infty}(\bar{p})$;
 also for all $\beta+1\in b^\Uu$
 and $\delta=\pred^\Uu(\beta+1)$,
 we have $\crit(E^\Uu_\beta)<i^\Uu_{0\delta}(j(\kappa))$.

 Now suppose $C$ fails the Jensen ISC.
 Let $\lambda<j(\kappa)$ be least such that $F^C\rest\lambda$ is whole but $F^C\rest\lambda\notin C$. We have $U|j(\kappa)^{+U}=C^{\passive}$.
 So $F^C\rest\lambda\notin U$.
 
  As before, every extender used along $b^\Tt$
 is short. Let $\beta_0+1=\min((0,\infty]^\Tt)$.
 Now we get $\sup i^\Uu_{0\infty}``\lambda=\lambda(E^\Tt_{\beta_0})$,
 because fairly standard calculations show that $\sup i^\Uu_{0\infty}``\lambda$ is the location of the least failure of Jensen ISC for the branch embedding $i^\Uu_{0\infty}\com j=i^\Tt_{0\infty}$.
 (That is, suppose $\lambda$ is a limit of whole segments of $F^C$, so those are all in $C$.
 One can argue in the style of \cite[\S2]{extmax} 
 to see that this can't be increased in the ultrapower to $\geq\sup i^\Uu``\lambda$;
 however, letting $\delta_0+1=\min((0,\infty]^\Uu)$, if $E^\Uu_{\delta_0}$ is long with $\crit(E^\Uu_{\delta_0})=\kappa$,
 then the ultrapower by $E^\Uu_{\delta_0}$ involves a slight variant of the methods of \cite[\S2]{extmax}, which we now describe. Suppose that the extender $H_1$ derived from $i^\Uu_{0,\delta_0+1}\com j$
 of length $\sup i^\Uu_{0,\delta_0+1}``\lambda$ is in $M^\Uu_{\delta_0+1}$. 
 Since the short part $E$ of $E^\Uu_{\delta_0}$ is in $M^\Uu_{\delta_0+1}$,
 it follows that $\ell\in M^\Uu_{\delta_0+1}$, where letting $\bar{M}=M|\kappa^{+M}$, 
 \[\ell: \Ult_0(\bar{M},E)\to\Ult_0(\bar{M},H_1) \]
 is the canonical factor map
 \[\ell(i^{\bar{M}}_E(f)(b))=i^{\bar{M}}_{H_1}(f)(b).\]
 Thus, the extender $H'_1$ derived from $\ell$ is also in $M^\Uu_{\delta_0+1}$. 
  Let $H'_1=[a,f]^{U,0}_{E^\Uu_{\delta_0}}$
  (where $f\in U$ and $a\in[\nu(E^\Uu_{\delta_0})]^{<\om}$).
  Now letting $H_0$ be the short extender derived from $j$ of length $\lambda$ (so $H_0\notin U$ but $H_0\rest\gamma\in U$ for each $\gamma<\lambda$),
 note that for each $\xi<\kappa^{+M}$ and each $\gamma<\lambda$, there is $A\in (E^\Uu_{\delta_0})_a$
 such that \[f(x)=H_0\rest((U|\gamma)\cross[\gamma]^{<\om}) \]
 for all $x\in A$. 
 But $f,(E^\Uu_{\delta_0})_a\in U$
 and it follows that $H_0\in U$, a contradiction.
 Now suppose instead that  $\mu<\lambda$ and either
  $\mu=\kappa$ and there is no whole segment of $F^C$ in $U$,
  or $\kappa<\mu$ and $F^C\rest\mu$ is the largest whole  segment of $F^C$ in $U$. Then we have $\lambda=\sup(\{j(f)(\mu)\bigm|f\in M\})$.
 Either $F^C\rest\alpha\in C$
 for all $\alpha<\lambda$, in which case we can argue as in the limit case above,
 or there is some $\alpha<\lambda$
 with $F^C\rest\alpha\notin C$,
 and then again using instances of the \cite{extmax}
 argument show that we can't get more
 beyond the image of $\alpha$.)

 Suppose that $\kappa<\crit(i^\Uu_{0\infty})$.
 Then since $\kappa=j(f)(\bar{p})$
 for some $f\in M$, and $i^\Tt=i^\Uu\com j$,
 we get $i^\Tt(f)(p')=\kappa$,
 and note that this yields that
 $\kappa<\crit(\pi)$
 and $\kappa=\crit(E^\Tt_{\alpha_0})$
 and $\pred^\Tt(\alpha_0+1)=0$,
 and that there is no long extender used along $(0,\alpha_0]^\Tt$
 with critical point $\kappa$,
 and therefore if $0<^\Tt\alpha_0$
 then in fact $\kappa<\crit(i^\Tt_{0\alpha_0})$.
 By the preceding paragraph,
 it follows that $\lambda(E^\Tt_{\alpha_0})=\sup i^\Uu``\lambda$.
 But $i^\Uu(\bar{p})=p'\sub\lambda(E^\Tt_{\alpha_0})$, so
 $\bar{p}\sub\lambda$. So $p_{\mathrm{D}}^M\sub\sup\pi``\lambda$, and note that $F^M\rest\pi(\lambda)$ is also a Jensen segment of $F^M$. But then if $\pi(\lambda)<\lambda(F^M)$
 then by the Jensen ISC for $M$,
 we get $F^M\rest p_{\mathrm{D}}^M\in M$, contradiction. So $\pi(\lambda)=\lambda(F^M)$.
 But then $\lambda=\lambda(F^C)$, so $F^C\rest\lambda$
 is not a proper segment of $F^C$, contradiction.
 
 So $\kappa=\crit(i^\Uu_{0\infty})$.
 So letting $\delta_0+1=\min((0,\infty]^\Uu)$,
 then $E^\Uu_{\delta_0}$
 is long with $\crit(E^\Uu_{\delta_0})=\kappa$.
 So letting $\xi=\sup i^\Uu``\kappa^{+M}$, then $E'=F^{M^\Uu_\infty|\xi}$ is  the short extender derived from $i^\Uu$, or equivalently, from $i^\Uu\com j=i^\Tt$. So like in the proof of Claim \ref{clm:pred(alpha_0+1)=0}
 of the proof of Theorem \ref{tm:rho^M_D=0_implies_Dodd-absent-solid_and_universal},
 we get $\crit(E^\Tt_{\alpha_0})=\kappa$ and
 $\pred^\Tt(\alpha_0+1)=0$.
 But then as before, $\lambda(E^\Tt_{\alpha_0})=\sup i^\Uu``\lambda$, which leads to a contradiction,
establishing the Jensen ISC for $C$.

Part \ref{item:transfer_definability}:
because $C$ has the Jensen ISC,
considerations as above give that $\crit(E^\Tt_{\alpha_0})=\kappa$,
and that $E^\Tt_{\alpha_0}$ is the only extender used along $b^\Tt$.
Let $x\in M$ and $A\sub\kappa$
be $\rSigma_1^M(\{x\})$.
If $\crit(i^\Tt_{0\alpha_0})=\kappa$
then let $G\in M^\Tt_{\alpha_0}$
be the short part of $i^\Tt_{0\alpha_0}$, and otherwise let $G=\emptyset$.
By a finite support calculation 
and since all extenders $E^\Uu_\delta$
used along $b^\Tt$
are such that $\crit(E^\Uu_\delta)<i^\Uu_{0\beta}(j(\kappa))$ where $\beta=\pred^\Uu(\delta+1)$,
we can find $\vec{\mu}\in U|j(\kappa)$ capturing $(M^\Uu_\infty,(i^\Tt(x),G),\kappa+1)$.
Since $C^{\passive}=U|j(\kappa)^{+U}$, we can also apply $\vec{\mu}$
to $C$, with the obvious agreement. 
But the $\Sigma_1^M(\{x\})$
definition of $A$
shifts to a $\Sigma_1^{M^\Tt_{\alpha_0}}(\{i^\Tt_{0\alpha_0}(x),G\})$ definition of $A$.
This yields a $\bfSigma_1^{\Ult_0(C,\vec{\mu})}$ definition of $A$,
which in turn yields a $\bfSigma_1^C$ definition of $A$.
\end{proof}

\begin{lem}\label{lem:rho_D,rho_1_corresp}
 Let $M$ be an active short premouse. Then
 $\max(\kappa^{+M},\rho_1^M)=\max(\kappa^{+M},\rho_{\D}^M)$.
 
 In particular, we have $\kappa^{+M}<\rho_1^M$ iff $0<\rho_{\D}^M$ iff $\kappa^{+M}<\rho_{\D}^M$,
 and if $\kappa^{+M}<\rho_1^M$ then $\rho_1^M=\rho_{\D}^M$.
\end{lem}
\begin{proof}
 Clearly $\rho_1^M\leq\max(\kappa^{+M},\rho_{\D}^M)$.
 Suppose $\rho_1^M<\max(\kappa^{+M},\rho_{\D}^M)$.
 Let $p\in[\OR]^{<\om}$
 be such that $p$ generates  $\{p_1^M,F_J^M\}$ (in the sense that $p_1^M=i^{M}_{F^M}(f)(p)$
 and $F_J^M=i^M_{F^M}(g)(p)$ for some $f,g\in M$). Then $F=F^M\rest\rho_1^M\cup\{p\}\in M$,
 but from $F$ and $M|\kappa^{+M}$
 one can produce the transitive
 collapse $C$ of some $X\preccurlyeq_{1}M$ with $\rho_1^M\cup\{p_1^M\}\sub X$, giving $C\in M$, which is impossible.
\end{proof}

Recall that $t^M$ was recalled in Remark \ref{rem:standard_Dodd_proj_and_param}.
The following lemma is similar
to one in \cite{zeman_dodd},
and a variant for Mitchell-Steel indexing in
 \cite{fsfni_v5}:
\begin{lem}\label{lem:char_t^M}
 Let $M$ be an active short premouse and $\kappa=\crit(F^M)$.
 Suppose that either:
 \begin{enumerate}[label=--]
  \item 
$M$ is $1$-sound,
 or
 \item $M$ is $\kappa^{+M}$-sound, or
 \item $M$ is
 Dodd-absent-sound.
 \end{enumerate}
 Then:
 \begin{enumerate}\item\label{item:char_t} $p_1^M\cut\kappa^{+M}\sub t^M$
 and letting $p_1^M\cut\kappa^{+M}=\{q_0,q_1,\ldots,q_{k-1}\}$
  and $t^M=\{\vec{t}_0,q_0,\vec{t}_1,q_1,\ldots,q_{k-1},\vec{t}_k\}$
  with each $\vec{t}_i=\{t_{i0},\ldots,t_{i,j_i-1}\}$, with everything in strictly decreasing order (so in particular, $q_0<\min(\vec{t}_0)$, etc), then:
 \begin{enumerate}[label=--]
  \item if $0<k$ (that is, $p_1^M\cut\kappa^{+M}\neq\emptyset$) then $\vec{t}_0$ is the least $s\in[\OR]^{<\om}$
  with $\{F_J^M\}$  generated by $s\cup q_0$
  \tu{(}so $q_0\notin s$\tu{)},
  \item if $i+1<k$ then $\vec{t}_{i+1}$ is the least $s$
  with  $\{F_J^M\}$  generated by \[\{\vec{t}_0,q_0,\ldots,\vec{t}_i,q_i,s\}\cup q_{i+1}\]
  (so $q_{i+1}\notin s$)
  \item $\vec{t}_k$ is the least $s$
  with $\{F_J^M\}$ generated by \[\{\vec{t}_0,q_0,\ldots,\vec{t}_{k-1},q_{k-1},s\}\cup\tau^M.\]
 \end{enumerate}
  \item\label{item:M_Das_rho_D=0}
If $M$ is Dodd-absent-sound with $\rho_{\D}^M=0$, then
 $M$ is $\kappa^{+M}$-sound and $\rho_1^M\leq\kappa^{+M}$.
 \item\label{item:M_Das_rho_D>0} 
 If $M$ is Dodd-absent-sound with $\rho_{\D}^M>0$,
 then $M$ is $1$-sound
 with $\rho_1^M=\rho_{\D}^M$.
 \end{enumerate}
\end{lem}
\begin{proof}
Part \ref{item:char_t}:
In the case that $M$ is $1$-sound or $\kappa^{+M}$-sound, this is basically via the proof from the short extender context (see \cite{zeman_dodd},
 \cite{fsfni_v5}).
 The version assuming that $M$ is Dodd-absent-sound and $\rho_{\D}^M=0$ may be new,
 but its proof is almost the same.
 
 Suppose first that $M$ is $1$-sound or $\kappa^{+M}$-sound. We prove recursively from top down that $t^M$ has the claimed structure. Let $t=\{\vec{t}_0,q_0,\ldots,q_{k-1},\vec{t}_k\}$ be defined recursively as in the statement of the lemma.
 Note that by Lemma \ref{lem:tau^M=max(rho_1,kappa^+)}, $t\cup\tau^M$ generates (all of) $F^M$,
 so certainly $t^M\leq t$.
 Suppose $t^M<t$. Let $i<\lh(t)$ be such that $t^M\rest i=s=t\rest i$ and either $i=\lh(t^M)$
 or $t^M(i)<t(i)$. Then $t(i)\in p_1^M$,
 since otherwise $t^M\cup\tau^M$ does not generate $\{F_J^M\}$,
 contradicting that it generates  $F^M$.
 But then from the $1$-solidity witness at $(p_1^M\cut(t(i)+1),t(i))$, we can recover $F^M\rest t^M\cup\tau^M$ (the points of $t^M\cut t(i)$
 which are not in $p_1^M$ can be recursively recovered using the fact that the theory has the parameter $F_J^M$ available for free). This is impossible,
 since $F^M\rest t^M\cup\tau^M$ is equivalent to $F^M$.
 
 Now suppose instead that $M$ is Dodd-absent-sound. Then $t^M=p_{\D}^M\cut\kappa^{+M}$ and $\tau^M=\max(\rho_1^M,\kappa^{+M})$.
 Let $r\in[\lambda(F^M)]^{<\om}$
 be least such that $\{F^M_J\}$ is generated by $r\cup\tau^M$.
 So $r\leq t^M$.
 Let $\vec{t}_0$ be the largest $s\ins t^M$ such that $s\ins r$.
 Note that if $\vec{t}_0=t^M$ then $p_1^M\cut\tau^M=\emptyset$,
 and otherwise $p_1^M\cut\tau^M\sub(t^M(\lh(\vec{t}_0))+1)$.
 In particular, if $\vec{t}_0=t^M$ then  $p_1^M\cut\tau^M=\emptyset$,
 so we are done.
 Suppose $\vec{t}_0\pins t^M$.
 Let $q_0=t^M(\lh(\vec{t}_0))$,
 so either $r=\vec{t}_0$
 or $r(\lh(\vec{t}_0))<q_0$.
 Therefore $\{F^M_J\}$ is generated by $\vec{t}_0\cup q_0$, and so there is $H\preccurlyeq_1M$ with the universe of $H$
 generated by this set. Since the corresponding Dodd-absent-solidity witness is in $M$, so is the transitive collapse of $H$.
 Since we have seen that $\max(p_1^M)\leq q_0$,
 it follows that $\max(p_1^M)= q_0$ and the corresponding $1$-solidity witness is in $M$. We now repeat this process relative to $(\vec{t}_0,q_0)$, producing $\vec{t}_1,q_1$, etc.  
 
 Parts \ref{item:M_Das_rho_D=0}
 and \ref{item:M_Das_rho_D>0}
 follow easily from the considerations above.
 \end{proof}
 
 We can now deduce that the $1$-core is also well-behaved with respect to Dodd-absent structure (in the circumstances from Theorem \ref{tm:Dodd-absent-core_of_M_when_rho_D^M=0}):
 \begin{tm}\label{tm:1-core_of_Dodd-abs-sound}
  Let $M$ be an active short $(0,\om_1,\om_1+1)^*$-iterable
  Dodd-absent-sound premouse with $\rho_{\D}^M=0$. 
  Let $C=\core_1(M)$ then $\pi:C\to M$ the core map.  Then $C$ is Dodd-absent-sound with $\rho_{\D}^C=0$
  and $\pi(p_{\D}^C)=p_{\D}^M$.
 \end{tm}
 \begin{proof}
   We have that $C$ is $(0^-,\om_1,\om_1+1)^*$-iterable.
 So by Theorem \ref{tm:rho^M_D=0_implies_Dodd-absent-solid_and_universal},
 $C$ is Dodd-absent-solid and Dodd-absent-universal.
 By Lemma \ref{lem:char_t^M},
  $p_{\D}^M\in\rg(\pi)$ (we might have $\kappa\in p_{\D}^M$, but $\kappa\in\rg(\pi)$ anyway).
  Since $\rho_{\D}^M=0$ and $M$ is Dodd-absent-sound,
  it easily follows that $F^C$ is generated by $\pi^{-1}(p_{\D}^M)$,
  so $\rho_{\D}^C=0$
  and $p_{\D}^C\leq\pi^{-1}(p_{\D^M})$. It just remains to see that $p_{\D}^C=\pi^{-1}(p_{\D}^M)$,
  so suppose $p_{\D}^C<\pi^{-1}(p_{\D}^M)$.
  
Let
 $D=\core_{\D}(C)$
 and $\sigma:D\to C$ the Dodd-absent-core embedding.
 Let
 \[ A=\Th_1^M(\rho_1^M\cup\{p_1^M\}),\]
 so $A$ is $\bfSigma_1^C$.
 By Theorem \ref{tm:Dodd-absent-core_of_M_when_rho_D^M=0},
 $A$ is $\bfSigma_1^D$.
 Let $\bar{M}=M|\kappa^{+M}$ and \[H=\Big\{i^{M,0}_{F^M}(f)(\pi(p_{\D}^C))\Bigm|f\in \bar{M}\Big\}.\]
 Let $\widetilde{D}$ be the structure
 with passivization the transitive collapse of $H$, and with the transitive collapse of $F^M\rest\pi(p_{\D}^C)$ active.
 Let $\widetilde{\sigma}:\widetilde{D}\to M$ be the uncollapse map. Then we get a unique $0$-embedding $\tau:D\to\widetilde{D}$ such that $\widetilde{\sigma}\com\tau=\pi\com\sigma$. Therefore $A$ is $\bfSigma_1^{\widetilde{D}}$. But since $\pi(p_{\D}^C)<p_{\D}^M$,
 we have $\widetilde{D}\in M$,
 so $A\in M$, a contradiction.  
 \end{proof}

This completes our discussion of the case that $\rho_{\D}^M=0$.
We now want to consider the Dodd-absent structure of active short mice $M$ with $\rho_{\D}>0$. Here we will also assume that $M$ is $1$-sound, and will deduce that $M$ is Dodd-absent-sound. The plan here is to simply reduce to the case that $\rho_{\D}^M$, by taking an appropriate hull.

\begin{dfn}
 Let $M$ be a $1$-sound premouse
 and let $\theta$ be an $M$-cardinal with $\theta\leq\rho_1^M$. Let $q\in[\OR^M]^{<\om}$.
 We say that $(M,\theta,q)$ is \emph{$1$-self-solid} iff letting $H=\cHull_1^M(\theta\cup\{q\})$
 and $\pi:H\to M$ be the uncollapse,
 then $H$ is $1$-sound and $\pi(p_1^H)=q$.
\end{dfn}

\begin{lem}\label{lem:existence_of_1-self-solid-parameter}
 Let $M$ be a $1$-sound premouse
 and let $\theta$ be an $M$-cardinal with $\theta\leq\rho_1^M$. Let $x\in M$.
 Then there is $q\in[\OR^M]^{<\om}$
 such that $(M,\theta,q)$ is $1$-self-solid
 and  $x\in\Hull_1^M(\theta\cup\{q\})$
 and $p_1^M=q\cut\rho_1^M$.
\end{lem}
\begin{proof}
 This is by the proof of \cite[Lemma***]{extmax}, using Lemma \ref{lem:condensation_for_self-solid}.
 That is, we will define $k<\om$ and  a sequence $\left<(\theta_i,q_i)\right>_{i\leq k}$
 such that:
 \begin{enumerate}[label=--]
  \item $(\theta_0,q_0)=(\rho_1^M,p_1^M)$,
  \item $(M,\theta_i,q_i)$ is $1$-self-solid, and moreover, the $1$-solidity witnesses for $q_i$ are all in $\Hull_1^M(\theta_i\cup q_i)$,
  \item $q_i\cut\rho_1^M=p_1^M$,
  \item $x\in\Hull_1^M(\theta_i\cup\{q_i\})$,
  \item $\theta_{i+1}<\theta_i$,
  \item $q_{i+1}=q_i\cup\{\gamma\}$ for some $\gamma$ such that $\theta_i<\gamma<\theta_i^{+M}$,
  \item $\theta_{k}=\theta$.
 \end{enumerate}
Clearly this suffices.

Starting with $(\theta_0,q_0)=(\rho_1^M,p_1^M)$, this trivially satisfies the inductive requirements. So suppose we have defined $\left<\theta_i,q_i\right>_{i<j}$
as above, but $\theta<\theta_{j-1}$.
Let $\theta_j$ be the least $\theta'\in[\theta,\theta_{j-1})$
such that $\theta'$ is an $M$-cardinal
and letting $w_{j-1,0},\ldots,w_{j-1,n-1}$ enumerate the $1$-solidity witnesses for $q_{j-1}$,
we have
\[ x,w_{j-1,0},\ldots,w_{j-1,n-1}\in\Hull_1^M((\theta')^{+M}\cup q_{j-1}). \]
Now let $\gamma\in[\theta',(\theta')^{+M})$
be least such that
\[ x,w_{j-1,0},\ldots,w_{j-1,n-1}\in\Hull_1^M((\gamma+1)\cup q_{j-1}).\]
Let
\[ H=\cHull_1^M((\gamma+1)\cup q_{j-1}) \]
and $\pi:H\to M$ the uncollapse.
Then $\rho_1^H=\theta_j<\rho_1^M$
and $H=\Hull_1^H(\theta_j\cup\{y\})$
for some $y$. Let $\eta=\theta_j^{+H}$.
Define $q_{j}=q_{j-1}\cup\{\eta\}$.

Let us verify that $(M,\theta_j,q_j)$
is $1$-self-solid.
We easily have the $1$-solidity witnesses
corresponding  to each $\alpha\in q_{j-1}$
in the relevant hull, so we just need to see that
\[ \Th_1^M(\eta\cup q_{j-1})\in\Hull_1^M(\theta_j\cup q_j).\]
Let
$J\pins M$ be such that $\eta=\theta_j^{+J}$ and $\rho_\om^J=\theta_j$.
Using Lemma \ref{lem:condensation_for_self-solid}, note then that
\[ H,J\in\Hull_1^M(\theta_j\cup q_j)\]
(since $J$ can be recovered in a $\Sigma_1$ fashion from the parameter $\eta$,
and then since $J$ is sound with $\rho_\om^J=\theta_j$ (actually $\rho_1^J=\theta_j$), using parameters ${<\theta_j}$, this can be used to recover $H$), but
$H=\cHull_1^M(\eta\cup q_{j-1})$,
which suffices.
\end{proof}

 \begin{rem}
Note that if $\alpha\in p_1^M\cut\kappa^{+M}$
 then $\{F_J^M\}$ is generated by $(t^M\cut(\alpha+1))\cup\alpha$.
 Otherwise, note that $\alpha$ is the least ordinal such that $F_J^M$ is generated by $(t^M\cut(\alpha+1))\cup(\alpha+1)$.
 But from $p_1^M\cut(\alpha+1)$, we can recover $t^M\cut(\alpha+1)$ in an $\rSigma_1^M$-manner, because we have access to the parameter $F_J^M$ for free, and hence can produce $\vec{t}_0$, and then we get $q_0$ from $p_1^M\cut(\alpha+1)$, and then we recover $\vec{t}_1$, etc, throughout all of $t^M\cut(\alpha+1)$.
\end{rem}

\begin{tm}\label{tm:Dodd-absent-soundness}
 Let $M$ be a $(0^-,\om_1,\om_1+1)^*$-iterable active short $1$-sound
 premouse with $\rho_{\D}^M>0$
 \tu{(}so $\rho_1^M=\rho_{\D}^M>\kappa^{+M}$\tu{)}.
 Then $M$ is Dodd-absent-sound.
\end{tm}

\begin{proof}
Using Lemma \ref{lem:existence_of_1-self-solid-parameter},
 let $\bar{M}\in M$ be a $1$-sound active short premouse 
 and $\pi:\bar{M}\to M$
 be an $\rSigma_1$-elementary map
 with
 $\crit(\pi)>\kappa^{+M}=\rho_1^{\bar{M}}$,
  $p_1^M\in\rg(\pi)$,
  $\pi(p_1^{\bar{M}})\cut\rho_1^M=p_1^M$, $F^M\rest t^M\in\rg(\pi)$
 and $\rho_1^M\in\rg(\pi)$.
 Note that $\rho_{\D}^{\bar{M}}=0$.
So by Theorem \ref{tm:rho^M_D=0_implies_Dodd-absent-solid_and_universal}, $\bar{M}$ is Dodd-absent-universal and Dodd-absent-solid.
 
Lemma \ref{lem:char_t^M}
characterizes $t^M$,
and note that this characterization is preserved by $\pi$;
 that is, letting $\pi(\bar{\rho})=\rho_1^M$
 and $\pi(\bar{p})=t^M$,
 then $\pi(p_1^{\bar{M}}\cut\bar{\rho})=p_1^M$
 and $(p_1^{\bar{M}}\cut\bar{\rho},\bar{p},\bar{\rho})$ is described over $\bar{M}$ 
 in the manner that $(p_1^M,t^M,\rho_1^M=\tau^M)$ is described over $M$ by Lemma \ref{lem:char_t^M}.
 Since $\bar{M}$ is Dodd-absent-solid,
 the following claim easily completes the proof:
 \begin{clmtwo} $\bar{p}\ins p_{\mathrm{D}}^{\bar{M}}$.
 \end{clmtwo}
\begin{proof}
 Suppose not. Let $i$ be largest such that $\bar{p}\rest i=p_{\mathrm{D}}^{\bar{M}}\rest i$.
 So $\bar{p}\rest i\neq \bar{p}$.
 Clearly $p_{\mathrm{D}}^{\bar{M}}\neq\bar{p}\rest i$, since we put a measure into $\bar{M}$
 which is equivalent to $F^{\bar{M}}\rest\bar{p}$.
 Noting that $F^C$ is generated by $\bar{\rho}\cup\bar{p}$, therefore
 $p_{\mathrm{D}}^{\bar{M}}(i)<\bar{p}(i)$.
 Let $\sigma:C\to \bar{M}$ be the Dodd-absent core and embedding.

 Now $\bar{p}(i)\notin p_1^{\bar{M}}$,
 because otherwise, using $p_1^{\bar{M}}\cut(\bar{p}(i)+1)$, we can recover $\bar{p}\rest i$
 in an $\rSigma_1$ manner
 (using the description given by \ref{lem:char_t^M}, as before), and 
 we have the $1$-solidity witness $w$ at this point in $\bar{M}$, but since $p_{\mathrm{D}}^{\bar{M}}(i)<\bar{p}(i)$,  and since $w$ encodes $\bar{p}\rest i$, we can therefore compute $F^{\bar{M}}\rest p_{\mathrm{D}}^{\bar{M}}$ from $w$, a contradiction.
 
Now suppose $\kappa<\crit(\sigma)$. Then by Theorem \ref{tm:Dodd-absent-core_of_M_when_rho_D^M=0},
 we have $\sigma(F_J^C)=F_J^{\bar{M}}$, and $\sigma(p_{\D}^C)=p_{\D}^{\bar{M}}$. But then since $p_{\D}^{\bar{M}}(i)<\bar{p}(i)$ and $F_J^{\bar{M}}\in\rg(\sigma)$, by the description of $(p_1^{\bar{M}}\cut\bar{\rho},\bar{p},\bar{\rho})$ mentioned above, we must have $\bar{p}(i)\in p_1^{\bar{M}}$,
 a contradiction.
 
 So $\kappa=\crit(\sigma)$.
 We have $\sigma(p_{\mathrm{D}}^C)=p_{\mathrm{D}}^{\bar{M}}$,
  so $p_{\mathrm{D}}^{\bar{M}}$ generates $\{\sigma(F_J^C)\}$.  But since  $\bar{p}(i)\notin p_1^{\bar{M}}$, the description of Lemma \ref{lem:char_t^M} gives that $\bar{p}(i)$ is the least $\gamma$
 such that $F_J^{\bar{M}}$ is generated from $(\bar{p}``i)\cup(\gamma+1)$. So $p_{\D}^{\bar{M}}$ does not generate $\{F^{\bar{M}}_J\}$.
 
  Suppose $F_J^C=\emptyset$.
 Then by Theorem \ref{tm:Dodd-absent-core_of_M_when_rho_D^M=0}, $F_J^{\bar{M}}=E_\sigma=$ the short extender derived from $\sigma$,
 which is indexed at $\lh(F_J^{\bar{M}})<\sigma(\kappa)^{+\bar{M}}$, and  $\{F_J^C\}$ is generated by the ordinal $\lh(F_J^C)$.
 Since $\{F_J^{\bar{M}}\}$ is not generated by ordinals in $(\bar{p}``i)\cup \bar{p}(i)$,
 we must have $\bar{p}(i)\leq\lh(F_J^{\bar{M}})$.
 But $\sigma(\kappa)\leq p_{\mathrm{D}}^{\bar{M}}(i)<\bar{p}(i)$, as $p_{\mathrm{D}}^{\bar{M}}\in\rg(\sigma)$.
 But then in fact $p_{\mathrm{D}}^{\bar{M}}(i)=\sigma(
\kappa)$, since each $\xi$
with $\sigma(\kappa)<\xi<\lh(F^{\bar{M}}_J)=\sup\sigma``\kappa^{+M}$
is generated by $(\sigma(\kappa)+1)$.
Since $p_{\D}^{\bar{M}}\in\rg(\sigma)$, it also follows that $p_{\D}^{\bar{M}}=(\bar{p}\rest i)\cup\{\sigma(\kappa)\}$.
But we have $F^{\bar{M}}\rest \bar{p}\in \bar{M}$
by construction, and therefore $F^{\bar{M}}\rest(\bar{p}\rest(i+1))\in\bar{M}$,
but $\{\bar{p}(i)\}$ generates $\sigma(\kappa)=\card^{\bar{M}}(\bar{p}(i))$,
and therefore $F^{\bar{M}}\rest(\bar{p}\rest(i+1))$
determines $F^{\bar{M}}\rest((\bar{p}\rest i)\cup\{\sigma(\kappa)\})$,
so this measure is in $\bar{M}$.
This contradicts the definition of $p_{\D}^{\bar{M}}$.

 So $F_J^C\neq\emptyset$.
 So by Theorem \ref{tm:Dodd-absent-core_of_M_when_rho_D^M=0}, $F_J^{\bar{M}}=\sigma(F_J^C)(E_\sigma)$, where $E_\sigma$ is as before. We have $E_\sigma\in\es^{\bar{M}}$. Therefore $F_J^{\bar{M}}$
 is generated by $\sigma(F_J^C)$ and $\lh(E_\sigma)$.
 But $\sigma(F_J^C)$ is  generated by $\sigma(p_{\mathrm{D}}^C)$. So like before,
 we get $\sigma(\kappa)=p_{\mathrm{D}}^{\bar{M}}(i)<\bar{p}(i)\leq\lh(E_\sigma)<\sigma(\kappa)^{+\bar{M}}$,
 which leads to contradiction like before.
 \end{proof}
 As mentioned earlier, the claim completes the proof (since the Dodd-absent-solidity witnesses for $\bar{p}$ lift under $\pi$ to such witnesses for $t^M$,
 yielding Dodd-absent-solidity for $M$).
\end{proof}
\subsection{Solidity}\label{sec:solidity}

In this section we prove Theorem \ref{tm:solidity},
which establishes, roughly,
that mice are solid and universal, in the fine structure sense.

\begin{dfn}
Let $M$ be an $n$-sound premouse with $\om<\rho_n^M$.
Given $s\in[\OR^M]^{<\om}$
and $\sigma\in\OR^M$, define
\[ S^M_{n+1}(s,\sigma)=\cHull^M_{n+1}(\{s,\pvec_n^M\}\cup\sigma),\]
the \emph{standard $(n+1)$-hull for $M$
at $(s,\sigma)$}
and $\pi^M_{n+1}(s,\sigma):S^M_{n+1,s}\to M$ as the uncollapse map.

Let $p=p_{n+1}^M\conc\left<\rho_{n+1}^M\right>$.
Given $i<\lh(p)$, define 
\[
 W_{n+1,i}^M=S^M_{n+1}(p\rest i, p(i))\]
 and $\pi_{n+1,i}^M=\pi^M_{n+1}(p\rest i,p(i))$.
\end{dfn}

\begin{tm}\label{tm:solidity}
Let $n\in\{0^-\}\cup\om$ and $M$ be such that either:
\begin{enumerate}[label=(\roman*)]
 \item\label{item:case_n-sound_Dodd-sound} $0\leq n$ and $M$ is an $n$-sound
 $(n,\om_1,\om_1+1)^*$-iterable premouse with $\om<\rho_n^M$,
 and if $M$ is active short and $n=0$ then $M$ is Dodd-absent-sound and $\rho_1^M\leq\kappa^{+M}$ where $\kappa=\crit(F^M)$, or
 \item\label{item:case_non-Dodd-sound} $n=0^-$ and $M$ is a $(0^-,\om_1,\om_1+1)^*$-iterable active short premouse
 with $\rho_1^M=\rho_{\D}^M>\kappa^{+M}$ where $\kappa=\crit(F^M)$.
\end{enumerate}
 Then writing $0^-+1=1$, either:
 \begin{enumerate}[label=--]
 \item $N$ is $(n+1)$-solid and $(n+1)$-universal, or
  \item $N$ is stretched-$(n+1)$-solid and stretched-$(n+1)$-universal.
 \end{enumerate}
\end{tm}

\begin{proof}
 Suppose not. We may assume $M$ is countable
and fix an $(n,\om_1+1)$-strategy $\Sigma$
for $M$ with weak Dodd-Jensen with respect to $\vec{x}$, where $\vec{x}\rest\lh(p_{n+1}^M)=p_{n+1}^M$.

Let $i_0$ be the least $i\leq\lh(p_{n+1}^M)$
 such that $W=W_{n+1,i}^M\notin M$.
Let $\eta=p_{n+1}^M(i_0)$, if $i<\lh(p_{n+1}^M)$,
and $\eta=\rho_{n+1}^M$ otherwise.
Let \[\pi=\pi_{n+1,i_0}^M:W\to M.\] So $\eta\leq\crit(\pi)$ and if $i_0<\lh(p_{n+1}^M)$
then $\eta=\crit(\pi)$.

We begin with Claims \ref{clm:W_is_Dodd-absent-sound} and \ref{clm:if_eta=gamma^+M} below, which establish some simplifying facts.

\begin{clmthree}\label{clm:W_is_Dodd-absent-sound}
Suppose $M$ is active short
and $n=0$, so $\rho_1^M\leq\kappa^{+M}$. Then $\rho_1^M,\eta<\kappa$
and $W$ is Dodd-absent-sound. Moreover, if $M|\eta$ is active
then  $\Ult_0(M,F^{M|\eta})$
is Dodd-absent-sound.
\end{clmthree}
Recall here that we are arguing by contradiction overall;
this is needed to conclude that $\rho_1^M,\eta<\kappa$ in the claim.
\begin{proof}
By the hypotheses of the theorem and the claim, $M$ is Dodd-absent-sound
and by Lemma \ref{lem:char_t^M},
 $M$ is $\kappa^{+M}$-sound.
 Since we are proceeding by contradiction, therefore
 $\rho_1^M\leq\kappa$
 and $\eta<\kappa^{+M}$. So we just need to see that $W$ is Dodd-absent-sound.
Also by the characterization of $p_{\mathrm{D}}^M\cut\kappa^{+M}$
 given in  \ref{lem:char_t^M},
  $p_{\mathrm{D}}^M\in\Hull_1^M(p_1^M\cut\kappa^{+M})$.
 
 Suppose that $\rho_1^M=\kappa$.
 Let us first observe that
  $\kappa^{+M}\sub\Hull_1^M(\kappa\cup\{p_1^M\})$.
 Clearly $\kappa\in\Hull_1^M(\emptyset)$,
 and since $\rho_1^M=\kappa$,
 $\Hull_1^M(\kappa\cup\{p_1^M\})\notin M$,
 and therefore this hull
 is cofinal in $\OR^M$, and therefore its intersection with $\kappa^{+M}$ is cofinal in $\kappa^{+M}$, which suffices.
 If $p_1^M\cap\kappa^{+M}=\emptyset$
 then, with the previous paragraph, this completes everything. Suppose instead that $p_1^M\cap\kappa^{+M}=\{\eta\}$.
 Then it suffices to see that $M$ is $1$-solid
 at $(p_1^M\cut\kappa^{+M},\eta)$.
 We have that $\Hull_1^M(\eta\cup\{p_1^M\cut\kappa^{+M}\})$ is bounded in $\OR^M$, since otherwise
 the hull would be unbounded in $\kappa^{+M}$,
 and hence would contain the point $\eta$, which is impossible.
 But then this hull is in $M$,
 a contradiction.
 
 So $\rho_1^M<\kappa$. If $\eta\in p_1^M\cap(\kappa,\kappa^{+M})$
 then $M$ is $1$-solid at $(p_1^M\cut\kappa^{+M},\eta)$, just as above.
 So ${\eta<\kappa}$. Likewise,
 we can assume that the hulls in question are unbounded in $\OR^M$. Also for $1$-universality, the corresponding hull will be unbounded in $\OR^M$, since it is not in $M$ by definition.
 
So either $\eta\in p_1^M\cap\kappa$,
and let $s=p_1^M\cut(\eta+1)$, or $\eta=\rho_1^M$,  and let $s=p_1^M$,
 and let $W=\cHull_1^M(\eta\cup\{s\})$ and $\pi:W\to M$ the uncollapse.
 Then $F^W$ is generated by $\pi^{-1}(p_{\mathrm{D}}^M)$
 (recall $p_{\D}^M\in\Hull_1^M(p_1^M\cut\kappa^{+M})$), so $p_{\mathrm{D}}^W\leq\pi^{-1}(p_{\mathrm{D}}^M)$.
 If $p_{\mathrm{D}}^W=\pi^{-1}(p_{\mathrm{D}}^M)$ then $W$ is Dodd-absent-sound
 (since by Theorem \ref{tm:rho^M_D=0_implies_Dodd-absent-solid_and_universal},
  $W$ is Dodd-absent-solid).
  So assume that $p_{\mathrm{D}}^W<\pi^{-1}(p_{\mathrm{D}}^M)$. We claim that the Dodd-absent-core $C$ of $W$ is in $M$.
  For let $\sigma:C\to W$ be the uncollapse map. Let $q=p_{\mathrm{D}}^W=\sigma(p_{\mathrm{D}}^C)$
  and  $q'=\pi(q)$, so $q'<p_{\mathrm{D}}^M$,
  and therefore $F^M\rest q'\in M$.
  Let $C'$ be the structure with $C'|\kappa=M|\kappa$ and $F^{C'}$ equivalent to $F^M\rest q'$. So $C'\in M$, so it suffices to compute $C$ from $C'$.
  Let $\sigma':C'\to M$ be the natural map,
  so $x\in\rg(\sigma')$ iff $x=i^M_{F^M}(f)(q')$ for some $f\in M|\kappa^{+M}$. Then either
  \begin{enumerate}
  \item we have
  \begin{enumerate}
   \item $\crit(F^W)\in\rg(\sigma)$ and $\sigma:C\to W$ is cofinal $\Sigma_1$-elementary in the language without $\dot{F}_J$, and
   \item $\kappa\in\rg(\sigma')$ and
   $\sigma':C'\to M$ is cofinal $\Sigma_1$-elementary
   in the language without $\dot{F}_J$,
   \end{enumerate}
   or\item we have
   \begin{enumerate}
   \item $\crit(F^W)\notin\rg(\sigma)$
   and the short extender $E$ derived from $\sigma$ is in $\es^W$, and $E$ is a whole segment of $F^W$,
   and letting $W^*$ be the corresponding protomouse,
   i.e.~$(W^*)^{\passive}=W^\passive$
   and $F^{W^*}\com E=F^W=\sigma\com F^C$,
   then $\sigma:C\to W^*$
   is cofinal $\Sigma_1$-elementary in the language without $\dot{F}_J$,
    and
   \item $\kappa\notin\rg(\sigma')$
   and the short extender $E'$ derived from $\sigma'$ is in $\es^M$, and $E'$ is a whole segment of $F^M$,
   and letting $M^*$ be the corresponding 
   protomouse, i.e,~$(M^*)^{\passive}=M^\passive$ and $F^M=\com E'=F^M=\sigma'\com F^{C'}$,
   then $\sigma':C'\to M^*$
   is cofinal $\Sigma_1$-elementary in the language without $\dot{F}_J$,
   \end{enumerate}
  \end{enumerate}
and moreover, $\rg(\pi\com\sigma)\sub\range(\sigma')$; in fact \[\rg(\sigma)=\Big\{i^W_{F^W}(f)(q)\Bigm|f\in W|\crit(F^W)^{+W}\Big\},\]
so \[\rg(\pi\com\sigma)=\Big\{i^M_{F^M}(f)(q')\Bigm| f\in (M|\kappa^{+M})\cap\rg(\pi)\Big\},\]
and recall \[\rg(\sigma')=\Big\{i^M_{F^M}(f)(q')\Bigm|f\in M|\kappa^{+M}\Big\}.\]
And $\pi,\sigma,\sigma'$ are cofinal with elementarity as indicated above.
It follows that we get a unique map $\bar{\pi}:C\to C'$ such that $\sigma'\com\bar{\pi}=\pi\com\sigma$,
and moreover, $\bar{\pi}$ is $\Sigma_1$-elementary in the language without $\dot{F}_J$.
But $\crit(\bar{\pi})=\eta=\crit(\pi)$
(it seems that maybe $\crit(F^W)=\eta$,
but since $\eta<\kappa\in\rg(\pi)$,
certainly $\crit(\pi)=\eta\leq\crit(F^W)$).
Now by Theorem \ref{tm:Dodd-absent-core_of_M_when_rho_D^M=0},
every $\bfrSigma_1^W$ subset
of $\crit(F^W)$
is $\bfrSigma_1^C$.
Therefore \[ t=\Th_1^W(\eta\cup\pi^{-1}(p_1^M\cut(\eta+1)))\text{ is }\bfSigma_1^C.\]
But then $\bar{\pi}``t$ is $\bfSigma_1^{C'}$,
and therefore $\bar{\pi}``t\in M$,
but then $W\in M$, as desired.

Finally suppose $M|\eta$ is active.
Then since $\eta<\kappa=\crit(F^M)$,
the Dodd-absent-soundness of
$\Ult_0(M,F^{M|\eta})$
follows from that of $M$.
 \end{proof}

\begin{clmthree}\label{clm:if_eta=gamma^+M}
If $\eta=\gamma^{+M}$
 then $\eta=\rho_{n+1}^M$
 and $i_0=\lh(p_{n+1}^M)$,
 and if $\eta\in W$ then 
 $\pi\rest\eta^{+W}=\id$.
 \end{clmthree}
 \begin{proof}
 Suppose $i_0<\lh(p_{n+1}^M)$.
 So $\eta\in p_{n+1}^M$.
 If  $\eta\in W$
  then $\eta=\crit(\pi)$, but then $\pi(\eta)=\gamma^{+M}=\eta$, contradiction. So $\eta=\OR^W$,
  so $i_0=0$;
 that is, $\eta=\max(p_{n+1}^M)$,
 and note that $n=0$
 and $M$ is passive and $\eta$ is the largest cardinal of $M$.
 But then $\Hull_1^M(\eta)=M|\eta\in M$, giving $1$-solidity at $\eta$, a contradiction.
 
 So $i_0=\lh(p_{n+1}^M)$,
 so $\eta=\rho_{n+1}^M$,
 and the rest is clear.
 \end{proof}
 
We now define a phalanx $\ph$, according to the following cases,
analogous to those considered in the main comparison proof for Dodd-absent-solidity, \ref{tm:rho^M_D=0_implies_Dodd-absent-solid_and_universal}.
As usual, the exchange ordinal notation ``${\xi}$''
means ``$\leq\xi$'', and both ``${<\xi}$''
and ``${\leq\xi}$'' refer to the space of extenders.
If $\eta$ is not a limit cardinal of $M$ then let $\gamma$ be the largest $M$-cardinal with $\gamma<\eta$. Then:
\begin{enumerate}[label=Ph\arabic*.,ref=Ph\arabic*]
 \item\label{item:solidity_phalanx_eta_card} If $\eta$ is an $M$-cardinal then
 $\ph=((M,{<\eta}),W)$.
 \item\label{item:solidity_phalanx_eta_active} If $\eta=\gamma^{+W}<\gamma^{+M}$ and $M|\eta$ is active then 
 \[\ph=((M,{<\gamma}),(\Ult_n(M,F^{M|\eta}),\gamma),W).\]
  \item\label{item:solidity_phalanx_eta_passive} If $\eta=\gamma^{+W}<\gamma^{+M}$ and
  $M|\eta$ is passive
 then letting $(R,r)\pins (M,0)$
 be least such that $\eta\leq\OR^R$ and $\rho_{r+1}^R=\gamma<\rho_r^R$, 
 \[ \ph=((M,{<\gamma}),((R,r),\gamma),W).\]
 \end{enumerate}
 
 In the phalanxes, all models are at
\begin{enumerate}[label=--]
 \item degree $n$, in case \ref{item:case_n-sound_Dodd-sound} of the theorem, and
 \item degree $0^-$, in case
 \ref{item:case_non-Dodd-sound},
\end{enumerate}
 except when indicated as a pair $(R,r)$, which denotes the model $R$ at degree $r$ (and here $r\geq 0$).
We only consider $n$- or $0^-$-maximal trees $\Uu$ on $\ph$
with $E^\Uu_0\in\es_+^W$
and $\eta^{+W}<\lh(E^\Uu_0)$. \emph{Iterability} for $\ph$ is meant under these conditions.
As before, we consider trees $\Uu$ with roots in $\{0,-1,-2\}$,
with $M^\Uu_0=W$, etc.
Note that all models which are active short and at degrees $m\geq 0$ in $\ph$ are Dodd-absent-sound;
this follows from Claim \ref{clm:W_is_Dodd-absent-sound}
(that is, the model $R$, if it exists, and $M$, $W$ and $\Ult_n(M,F^{M|\eta})$ if $M$ is active short and $n\geq 0$).

Much like in the proof of Dodd-absent-solidity, before proceeding to the comparison, we establish some basic facts and deal  with some tasks.
Like in  Claim \ref{clm:pow(eta)^U_sub_M} of that proof, we have:
  \begin{clmthree}
   We have:
   \begin{enumerate}
    \item Suppose $\ph$ is defined via case \ref{item:solidity_phalanx_eta_card}. Then
    $W|\eta^{+W}=M||\eta^{+W}$,
    so $\pow(\eta)^W\sub M$.
    \item In case \ref{item:solidity_phalanx_eta_active}
   of the definition of $\ph$, we have
   $W|\eta^{+W}=W|\gamma^{++W}=\Ult_n(M,F^{M|\eta})||\eta^{+W}$,
   so
   $\pow(\gamma)^W\sub M$ and in fact $\pow(\eta)^W\sub\Ult_n(M,F^{M|\eta})$,
    \item In case \ref{item:solidity_phalanx_eta_passive}
    of the definition of $\ph$, we have
    $W|\eta^{+W}=W|\gamma^{++W}=R||\eta^{+W}$, so $\pow(\gamma)^W\sub M$ and in fact $\pow(\eta)^W\sub R$.
   \end{enumerate}
  \end{clmthree}

  We now establish the iterability of $\ph$ and closeness of extenders:
\begin{clmthree}
$\ph$ is $(\om_1+1)$-iterable
(under the conditions mentioned above).
In fact, there is an $(\om_1+1)$-iteration strategy $\Gamma$ for $\ph$ such that trees $\Uu$ on $\ph$ via $\Gamma$ lift to trees $\Uu^*$ on $M$ via $\Sigma$
in a manner similar to that done in the proof of Theorem \ref{tm:rho^M_D=0_implies_Dodd-absent-solid_and_universal}.
\end{clmthree}
\begin{proof}[Proof sketch]
Follow a similar procedure
to that from Theorem \ref{tm:rho^M_D=0_implies_Dodd-absent-solid_and_universal},
although slightly simplified,
as we have $\pi:W\to M$
instead of the map $\pi:U\to U'$
from there, so we can more directly reduce to trees on $M$. Also,
$\pi\rest\eta=\id$
here, which further simplifies things. (In the former proof it was possible that $\eta=\kappa^{+M}$
and $\crit(\pi)=\kappa$.
In the current proof, the analogue
of thar arises when taking
a hull corresponding to ``stretched-solidity''. But we are not presently considering that hull;
instead, we would be 
 presently dealing with $\eta=\min(p_{n+1}^M)$,
and so $\pi\rest\eta=\id$.) We leave the remaining details to the reader.
\end{proof}

\begin{clmthree}[Closeness analysis] Let $\Uu$ on $\ph$ be $n$- or $0^-$-maximal as is relevant, with $\eta^{+W}<\lh(E^\Uu_0)$.
Let $\alpha+1<\lh(\Uu)$. Then
$E^\Uu_\alpha$ is close to $M^{*\Uu}_{\alpha+1}$.
\end{clmthree}
\begin{proof}
 This is similar to the proof of Claim \ref{clm:Dodd-solidity_closeness} of the proof of Theorem \ref{tm:rho^M_D=0_implies_Dodd-absent-solid_and_universal}.
 Since most of that argument works,
 the key new case to consider
 is when $E^\Uu_\alpha$
 is close to $W$
 but $\spc(E^\Uu_\alpha)<\eta$,
 so that $\pred^\Uu(\alpha+1)<0$,
 and there is some $a$ such that $(E^\Uu_\alpha)_a\notin W$.
 Note then that $n=0$ or $n=0^-$
 and $M$ is active.
If $M^{*\Uu}_{\alpha+1}=M$
 (so $E^\Uu_\alpha$ is $M$-total)
 then because $\pi:W\to M$ is $\rSigma_1$-elementary,
 $E^\Uu_\alpha$ is also close to $M$. So suppose we are in case \ref{item:solidity_phalanx_eta_active} or \ref{item:solidity_phalanx_eta_passive}
 and $\pred^\Uu(\alpha+1)=-1$,
 so $\spc(E^\Uu_\alpha)=\gamma$,
 and $M^{*\Uu}_{\alpha+1}=\Ult_n(M,F^{M|\eta})$
 or $M^{*\Uu}_{\alpha+1}=R$ respectively. If $\gamma$
 is a limit cardinal then $\spc(E^\Uu_\alpha)=\crit(E^\Uu_\alpha)=\gamma$ and $\eta=\gamma^{+W}$,
 but then $F^W$ and $F^M$ are short
 and we have already ruled out this case (we have $\pi(\eta)=\gamma^{+M}$
 and $\eta\in p_1^M$
 and $\rg(\pi)$ is bounded in $\OR^M$, since $\rg(\pi)$ is bounded in $\gamma^{+M}$; so $W\in M$, a contradiction).
 But similarly,
 if $\gamma=\mu^{+M}$
 where $\mu$ is an $M$-cardinal
 then $E^\Uu_\alpha$
 is long with $\crit(E^\Uu_\alpha)=\mu$,
 and $\eta=\mu^{++W}<\mu^{++M}$,
 and $\pi(\eta)=\mu^{++M}$,
 and again $\rg(\pi)$ is bounded in $\OR^M$, so $W\in M$, contradiction. 
\end{proof}

Like in the proof of Theorem \ref{tm:rho^M_D=0_implies_Dodd-absent-solid_and_universal}, we have:
\begin{clmthree}
 There is a successful comparison $(\Uu,\Tt)$
 of $(\ph,M)$, via $(\Gamma,\Sigma)$,
 $b^\Uu$ is above $W$, $b^\Uu$ does
  not drop in model or degree,
 and $M^\Uu_\infty\ins M^\Tt_\infty$.
\end{clmthree}

\begin{clmthree}
  $M^\Uu_\infty=M^\Tt_\infty$.
\end{clmthree}
\begin{proof}
Otherwise $M^\Uu_\infty\pins M^\Tt_\infty$.
But then by the preservation of definability from $W$ to $M^\Uu_\infty$,
it follows that $W\in M^\Tt_\infty$,
and therefore, since $W$ is equivalent to a subset of $\eta$,
it follows that $W\in M$, a contradiction.
\end{proof}

\begin{clmthree}
 $b^\Tt$ does not drop in model,
 degree or Dodd-degree.
\end{clmthree}
\begin{proof}
Suppose  otherwise.
Then $b^\Tt$ drops in model, since $n=\deg^\Uu_\infty$ and $M^\Uu_\infty=M^\Tt_\infty$,
so $\deg^\Tt_\infty\geq n$.
 By \S\ref{sec:long_mice},
 the theory \[T=\Th_{\rSigma_{n+1}}^M(\eta\cup\{\pvec_n^M,p_{n+1}^M\rest i_0\}) \]
determining $W$
 is $\bfrSigma_{n+1}^{M^\Uu_\infty}$,
 and so $\rho_{n+1}^{M^\Uu_\infty}\leq\eta$
 (as $T\notin M^\Uu_\infty$).

 Suppose that $M|\eta$ is passive. Then $\eta<\lh(E^\Tt_0)$, and since $\rho_{n+1}^{M^\Uu_\infty}\leq\eta$,
 we have $C=\core_\om(\core_{\mathrm{D}}(M^\Tt_\infty))\pins M$, and $T$ is definable over $C$, so $W\in M$, as desired.
 
 Now suppose that $M|\eta$ is active.
 Then a similar calculation gives that either $C\pins M$ or $C\pins M^\Tt_1|\eta^{+M^\Tt_1}$,
 and so $C\in M$, in either case, and so $W\in M$.
\end{proof}

We now need to analyse the comparison more carefully, and in particular how the key objects are shifted by the iteration maps.
Let $q=p_{n+1}^M\rest i_0$ and $\pi(\bar{q})=q$.

\begin{clmthree}\label{clm:i^Uu(q-bar)=i^Tt(q)}
 $i^\Uu(\bar{q})= i^\Tt(q)$.
\end{clmthree}
\begin{proof}
Suppose $i^\Uu(\bar{q})<i^\Tt(q)$.
 Because $(M,q)$ is $(n+1)$-solid,
 $(M^\Tt_\infty,i^\Tt(q))$ is $(n+1)$-solid,
 so 
 then $t=\Th^{M^\Tt_\infty}_{n+1}(\{i^\Uu(\bar{q}),\pvec_n^{M^\Tt_\infty}\}\cup\sup i^\Uu``\eta)\in M^\Tt_\infty$.
 Clearly then $\crit(i^\Uu_{0\infty})<\eta$,
 so $\eta=\gamma^{+W}<\gamma^{+M}$
 and the first extender $E^\Uu_\alpha$
 used along $b^\Uu$ is long. But then the short part of $i^\Uu_{0\infty}$ is in $\es^{M^\Uu_\infty}$, and hence, working in $M^\Uu_\infty$, $t$ can be pulled back to compute $t'=\Th_{\rSigma_{n+1}}^W(\{\bar{q},\pvec_n^W\}\cup\eta)$, but $\pow(\eta)\cap M^\Uu_\infty\sub\pow(\eta)\cap W$ (in fact we have equality since $E^\Uu_\alpha$ is long),
 so $t'\in W$, a contradiction.
 
 Conversely suppose $i^\Uu(\bar{q})>i^\Tt(q)$.
 Let $\pi_{\infty}:M^\Uu_\infty\to M^{\Uu^*}_\infty$
 be the final copying map.
 Then $\pi_{\infty}\com i^\Uu=i^{\Uu^*}\com\pi$,
 and $\pi(\bar{q})=q$,
 so $i^{\Uu^*}(q)=\pi_\infty(i^\Uu(\bar{q}))>\pi_\infty(i^\Tt(q))$. Considering the choice of our enumeration $\vec{x}$ of $M$, this contradicts weak Dodd-Jensen with respect to $\vec{x}$.
\end{proof}

\begin{clmthree}\label{clm:Tt_shift_rho}
Let $\rho=\rho_{n+1}^M$. Then:
\begin{enumerate}
 \item\label{item:rho_is_lim_card} If $\rho$ is a limit cardinal of $M$
 then $\rho_{n+1}^{M^\Tt_\infty}=\sup i^\Tt``\rho$.
 \item\label{item:rho_is_successor_card} If $\rho=\delta^{+M}$ where $\delta$ is an $M$-cardinal then letting $\gamma\in b^\Tt$
 be least such that either $[0,\gamma]^\Tt=b^\Tt$
 or $\spc(E^\Tt_\theta)=i^\Tt_{0\gamma}(\rho)$
 then $\rho_{n+1}(M^\Tt_\infty)=i^\Tt_{0\gamma}(\rho)=\sup i^\Tt_{0\gamma}``\rho$.
 (But in this case, note that $\rho_{n+1}(M^\Tt_\infty)<\sup i^\Tt_{0\infty}``\rho$.)
\end{enumerate}
\end{clmthree}
\begin{proof}
 For part \ref{item:rho_is_lim_card}
 use the argument in \cite[\S2]{extmax}.
 For part \ref{item:rho_is_successor_card},
 that argument applies to $i^\Tt_{0\gamma}$,
 and then we get preservation of $\rho_{n+1}$ between $M^\Tt_{\gamma}$ and $M^\Tt_\beta$ for all $\beta\in(\gamma,\infty]^\Tt$,
 by \S\ref{sec:long_mice}.
\end{proof}

\begin{clmthree}\label{clm:compare_rhos_of_M,W}
We have:
\begin{enumerate}
 \item\label{item:rho_n+1^W_pres} $\rho_{n+1}^W=\rho_{n+1}(M^\Uu_\infty)\leq\eta$,
 \item\label{item:rho_n+1^M<rho_n+1^W_when_i_0_small} if $i_0<\lh(p_{n+1}^M)$ then $\rho_{n+1}^M<\eta$ and:
 \begin{enumerate}
 \item\label{item:rho_n+1^M<rho_n+1^W_case} if $(\rho_{n+1}^M)^{+M}\leq\eta$ then 
 $\rho_{n+1}^M<\rho_{n+1}^W$,
 \item\label{item:rho_n+1^W>=rho_n+1^M_case}  if $\rho_{n+1}^M<\eta<(\rho_{n+1}^M)^{+M}$
 (so $\eta=(\rho_{n+1}^M)^{+W}$) then $\rho_{n+1}^W\in\{\rho_{n+1}^M,\eta\}$.
 \end{enumerate}
 \item\label{item:rho_n+1^M=rho_n+1^W_when_i_0_max} if $i_0=\lh(p_{n+1}^M)$ then $\rho_{n+1}^M=\rho_{n+1}^W$.
 \end{enumerate}
\end{clmthree}
\begin{proof}
Part \ref{item:rho_n+1^W_pres}: By \S\ref{sec:long_mice} and the closeness of extenders used in $\Uu$.

Part \ref{item:rho_n+1^M<rho_n+1^W_when_i_0_small}:
Suppose $i_0<\lh(p_{n+1}^M)$,
so $\eta\in p_{n+1}^M$.
If $\eta=\rho_{n+1}^M$ then $W\in M$ by the minimality of $p_{n+1}^M$ (since $p_{n+1}^M\cut\{\eta\}<p_{n+1}^M$). Now part \ref{item:rho_n+1^M<rho_n+1^W_case} is by the minimality of $p_{n+1}^M$, and part \ref{item:rho_n+1^W>=rho_n+1^M_case} is clear.

Part \ref{item:rho_n+1^M=rho_n+1^W_when_i_0_max}:
Straightforward.
\end{proof}

\begin{clmthree}
 Suppose $i_0=\lh(p_{n+1}^M)$,
 so $\eta=\rho$ where $\rho=\rho_{n+1}^W=\rho_{n+1}^M$.
 Then $W||\rho^{+W}=M||\rho^{+M}$, so $M$ is $(n+1)$-solid and $(n+1)$-universal.
 \end{clmthree}
 \begin{proof}
We have $\rho\leq\crit(i^\Uu)$
by the definition of $\ph$ and as $\rho$ is an $M$-cardinal. If $\crit(i^\Tt)\geq\rho$
then the claim immediately follows, so suppose $\crit(i^\Tt)<\rho$.
Note that by the earlier claims,
\[ \rho_{n+1}^M=\rho_{n+1}^W=\rho_{n+1}(M^\Uu_\infty)=\rho_{n+1}(M^\Tt_\infty),\]
and therefore $\rho=\kappa^{+M}$ where $\kappa=\crit(E^\Tt_\alpha)$ where $\alpha+1=\min((0,\infty]^\Tt)$, and $E^\Tt_\alpha$ is long. But then because $b^\Tt$ is non-dropping
and $E^\Tt_\alpha$ is long,
we get $\kappa^{++M}=\kappa^{++\exit^\Tt_\alpha}$,
so $M||\kappa^{++M}=M^\Tt_\infty||\kappa^{++M^\Tt_\infty}=W||\kappa^{++M}$,
and so in fact $W||\kappa^{++W}=M||\kappa^{++M}$,
so $W||\rho^{+W}=M||\rho^{+M}$, as desired.\end{proof}

\begin{clmthree}\label{clm:M_is_simulated_solid}
 Suppose $i_0<\max(p_{n+1}^M)$,
 so $\eta\in p_{n+1}^M=\crit(\pi)$.
 Then $M$ is stretched-$(n+1)$-solid
 with $\min(p_{n+1}^M)=\eta$.
\end{clmthree}
\begin{proof}
 Suppose first that $\rho_{n+1}^M<\eta<(\rho_{n+1}^M)^{+M}$. So $\rho_{n+1}^W\in\{\eta,\rho_{n+1}^M\}$ and $\eta=(\rho_{n+1}^M)^{+W}$. Note that by Claim \ref{clm:Tt_shift_rho},
 $\rho_{n+1}^M\leq\spc(E^\Tt_\alpha)$
 where $\alpha+1=\min((0,\infty]^\Tt)$
 (otherwise $\rho_{n+1}^{M^\Tt_\infty}>\eta$).
 So $\rho_{n+1}^M=\rho_{n+1}^{M^\Tt_\infty}=\rho_{n+1}^W$. And since $b^\Tt$ does not drop,
 therefore $M||(\rho_{n+1}^M)^{+M}=M^\Tt_\infty||(\rho_{n+1}^M)^{+M}$, so $\pow(\rho)^M\sub W$,
 contradicting the fact that $\eta=(\rho_{n+1}^M)^{+W}<(\rho_{n+1}^M)^{+M}$.
 
 So $(\rho_{n+1}^M)^{+M}\leq\eta$,
 so by Claim \ref{clm:compare_rhos_of_M,W} etc,
  $\rho_{n+1}^M<\rho_{n+1}^W=\rho_{n+1}(M^\Uu_\infty)=\rho_{n+1}(M^\Tt_\infty)$,
 so  $\spc(E^\Tt_\alpha)<\rho_{n+1}^M$ where
 $\alpha+1=\min((0,\infty]^\Tt)$.
 This gives $\rho_{n+1}(M^\Tt_{\alpha+1})=\sup i^\Tt_{0,\alpha+1}``\rho_{n+1}^M\geq\sigma^{+M^\Tt_{\alpha+1}}$
 where $\sigma=i^\Tt_{0,\alpha+1}(\spc(E^\Tt_\alpha))$.
 
 But $\rho_{n+1}^W\leq\eta\leq\lh(E^\Tt_0)\leq\lh(E^\Tt_\alpha)$, and if $0<\alpha$ then $\lh(E^\Tt_0)<\lh(E^\Tt_\alpha)$. Putting things together, we get
 \begin{enumerate}[label=--]
  \item 
$\alpha=0\text{ 
 and }\rho_{n+1}^W=\eta=\lh(E^\Tt_0)\text{ and }\rho_{n+1}^M=\sigma^{+\exit^\Tt_0}=\sigma^{+M}$
 (where $\sigma=\spc(E^\Tt_0)$),
 and
 \item $\eta\leq\spc(E^\Tt_\beta)
 \text{ for all }\beta+1\in(1,\infty]^\Tt$.
 \end{enumerate}
 Thus, we now retire the variable $\alpha$, replacing it with $0$.
 
 \begin{sclmthree}\label{sclm:E^T_0_is_short}
  $E^\Tt_0$ is short.
 \end{sclmthree}
\begin{proof}
Suppose not. So $\rho_{n+1}^M=\kappa^{++M}$
where $\kappa=\crit(E^\Tt_0)$
and letting $\lambda=\lambda(E^\Tt_0)$,
we have $\gamma=\lambda^{+M}$ is an $M$-cardinal
and $\eta=\gamma^{+W}=\lambda^{++W}$.
Since $\eta\leq\spc(E^\Tt_\beta)$ for all $\beta+1\in(1,\infty]^\Tt$,
if $M^\Tt_\infty\neq M^\Tt_{1}$
then in fact $\eta<\crit(i^\Tt_{1\infty})$,
and by the rules for $\ph$,
also $\eta<\crit(i^\Uu)$. Note that $E^\Tt_0$
is long with no largest generator.

Because $\rho_{n+1}^W=\eta$,
for each $\alpha<\eta$, we have
\[t_\alpha= \Th_{\rSigma_{n+1}}^W(\alpha\cup \bar{q})\in W;\]
recall $\pi(\bar{q})=q=p_{n+1}^M\rest i_0$
and $\pi:W\to M$ is the uncollapse.
But $W|\eta=M||\eta$,
so also $t_\alpha\in M$.
But 
\[ t_\alpha=\Th_{\rSigma_{n+1}}^M(\alpha\cup q).\]
So $i^\Tt_{0\infty}(t_\alpha)$ (need not equal but) can be used to compute
\[ t'_\alpha=\Th_{\rSigma_{n+1}}^{M^\Tt_\infty}(i^\Tt(\alpha)\cup i^\Tt(q)),\]
so $t'_\alpha\in M^\Tt_\infty$.
By Claim \ref{clm:i^Uu(q-bar)=i^Tt(q)},
$i^\Tt(q)=i^\Uu(\bar{q})$,
and with $\alpha<\eta$ large enough, we have $i^\Tt(\alpha)>\eta$ (since already $i^\Tt_{0,1}(\gamma)>\eta$ where $\gamma$ is as above),
and therefore
\[ t^*_\eta=\Th_{\rSigma_{n+1}}^{M^\Uu_\infty}(\eta\cup i^\Uu(\bar{q}))\in M^\Uu_\infty,\]
but since this is coded as a subset of $\eta$, 
therefore
\[ \Th_{\rSigma_{n+1}}^{W}(\eta\cup\bar{q})\in W,\]
impossible!
\end{proof}

So $E^\Tt_0$ is short, so letting $\kappa=\crit(E^\Tt_0)$, we have $\rho_{n+1}^M=\kappa^{+M}$ and letting $\lambda=\lambda(E^\Tt_0)$, we have $\lambda=\gamma$ is an $M$-cardinal and $\eta=\lh(E^\Tt_0)=\lambda^{+W}=\lambda^{+M^\Tt_1}$.  We have:
\begin{enumerate}[label=--]
\item $\eta\leq\spc(E^\Uu_\beta)$ for all $\beta+1\in[0,\infty]^\Uu$ (by the rules of $\ph$), and
 \item 
$\eta\leq\spc(E^\Tt_\beta)$ for all $\beta+1\in(\alpha+1,\infty]^\Tt$ (as $\rho_{n+1}$ has to agree;
it might now be that $E^\Tt_\beta$ is long
with $\spc(E^\Tt_\beta)=\eta$, where $\beta+1=\min((\alpha+1,\infty]^\Tt)$).
\end{enumerate}

\begin{sclmthree}\label{sclm:sub-solid}
$M$ is ``sub-solid''
at $(q,\eta)$; that is, for each $\alpha<\eta$ we have 
 \[ \Th_{\rSigma_{n+1}^M}(\alpha\cup\{q,\pvec_n^M\})\in M.\]\end{sclmthree}
\begin{proof}
As in the proof of Subclaim \ref{sclm:E^T_0_is_short}.
\end{proof}

By the previous subclaim, 
it just remains to see that:
\begin{enumerate}[label=(\roman*)]
 \item\label{item:virtual_sol_i} $\eta=\min(p_{n+1}^M)$; i.e., recalling $\rho_{n+1}^M=\kappa^{+M}$, we have\[\Th_{\rSigma_{n+1}}^M(\kappa^{+M}\cup\{q,\eta,\pvec_n^M\})\notin M,\]
 \item\label{item:virtual_sol_ii} $X=\eta\cap\Hull_{n+1}^M(X\cup\{q,\pvec_n^M\})$ where $X=F^{M|\eta}``\kappa^{+M}$,\item\label{item:virtual_sol_iii}$C||\kappa^{++C}=M||\kappa^{++M}$
 where $C$ is the transitive collapse
 of the hull just mentioned. 
\end{enumerate}
Note that part \ref{item:virtual_sol_i} follows
easily from the conjunction of parts \ref{item:virtual_sol_ii}
and \ref{item:virtual_sol_iii},
since $\eta$ determines $F^{M|\eta}$ and
hence the set $X$, and
so by the latter two parts, 
the theory mentioned in \ref{item:virtual_sol_i}
collapses $\kappa^{++M}$.

\begin{sclmthree}\label{sclm:crit(i^U)<eta}
  $\crit(i^\Uu)<\eta$,
  so $\spc(E^\Uu_{\gamma_0})=\eta$,
  where $\gamma_0+1=\min((0,\infty]^\Uu)$.
\end{sclmthree}
\begin{proof}
Also like in the proof of Subclaim \ref{sclm:E^T_0_is_short}.
\end{proof}

By Subclaim \ref{sclm:crit(i^U)<eta},
the short part $G$ of (the extender derived from) $i^\Uu$ is indexed at $\sup i^\Uu``\eta$.

\begin{sclmthree}\label{sclm:sup_i^U``eta=sup_i^T``eta} $\sup i^\Uu``\eta=\sup i^\Tt``\eta$.
 (Note this is $i^\Tt=i^\Tt_{0\infty}$,
 not $i^\Tt_{1\infty}$.)
\end{sclmthree}
\begin{proof}
Since $(z_{n+1}^W(\bar{q}),\zeta_{n+1}^W(\bar{q}))=(\emptyset,\eta)$ (where the ``$(\bar{q})$''refers to the version of $(z,\zeta)$ relativized to the parameter $\bar{q}$),
by (the proof of) \cite[\S2]{extmax}, we have
\[ (z_{n+1}^{M^\Uu_\infty}(i^\Uu(\bar{q})),\zeta_{n+1}^{M^\Uu_\infty}(i^\Uu(\bar{q})))=(\emptyset,\sup i^\Uu``\eta).\]
And
\[ (z_{n+1}^M(q),\zeta_{n+1}^M(q))=(\emptyset,\eta), \]
as in the proof of Subclaim \ref{sclm:E^T_0_is_short} and since $W\notin M$
(and hence $\Th_{\rSigma_{n+1}}^M(\eta\cup q)\notin M$). Therefore
\[ (z_{n+1}^{M^\Tt_\infty}(i^\Tt(q)),\zeta_{n+1}^{M^\Tt_\infty}(i^\Tt(q)))=(\emptyset,\sup i^\Tt``\eta).\]
So  $\sup i^\Uu``\eta=\sup i^\Tt``\eta$, as desired. 
\end{proof}

\begin{sclmthree}\label{sclm:a_subclaim_defining_some_C} Let $F=F^{M|\eta}$ and $X=F``\kappa^{+M}$.
Let
\[ H^W=\Hull_{n+1}^W(X\cup\{\bar{q},\pvec_n^W\})\]
and
\[ H^M=\Hull_{n+1}^M(X\cup\{q,\pvec_n^M\}) \]
(note these are the uncollapsed hulls). Then:
\begin{enumerate}
\item\label{item:pointwise_images_of_X}
$i^\Uu``X=G``X=i^\Tt``X$,
\item\label{item:pointwise_images_of_hulls} $i^\Uu``H^W=i^\Tt``H^M=\Hull_{n+1}^{M^\Tt_\infty}((G``X)\cup\{i^\Tt(q),\pvec_n^{M^\Tt_\infty}\})$, and
\item\label{item:also_agmt_with_i^Tt_1,infty}
 $X=\eta\cap\Hull_{n+1}^M(X\cup\{q,\pvec_n^M\})$.
 \item\label{item:agmt_thru_kappa^++}Let $C$ be the (common) transitive collapse
 of $H^W,H^M$. Then $C||\kappa^{++C}=M||\kappa^{++M}$.\end{enumerate}
\end{sclmthree}
\begin{proof}
We have
\[ i^\Uu``H^W=\Hull_{n+1}^{M^\Uu_\infty}((G``X)\cup\{i^\Uu(\bar{q}),\pvec_n^{M^\Uu_\infty}
\})\]
and
\[ i^\Tt``H^M=\Hull_{n+1}^{M^\Tt_\infty}((i^\Tt``X)\cup\{i^\Tt(q),\pvec_n^{M^\Tt_\infty}\}) \]
and of course, $M^\Uu_\infty=M^\Tt_\infty$
and $i^\Uu(\bar{q})=i^\Tt(q)$ and $G``X=i^\Uu``X$, so  part \ref{item:pointwise_images_of_hulls} follows from part \ref{item:pointwise_images_of_X}; let us establish  the latter.

Note that $i^\Tt_{01}(\eta)=\sup i^\Tt_{01}``\eta$, and $i^\Tt_{01}(F)$ is indexed in $\es(M^\Tt_1)$ at this ordinal.
Since $E^\Tt_0=F$ and when $j$ is elementary, ``$j\rest\rg(j)=j(j)\rest\rg(j)$'', we have \[i^\Tt_{01}(F)\rest X=i^\Tt_{01}\rest X.\]
(More precisely, if $F(\alpha)=\beta$ where $\alpha<\kappa^{+M}$ then $i^\Tt_{01}(F)(i^\Tt_{01}(\alpha))=i^\Tt_{01}(\beta)$,
but $i^\Tt_{01}(\alpha)=F(\alpha)=\beta$,
so $i^\Tt_{01}(F)(\beta)=i^\Tt_{01}(\beta)$.)

So letting $F'=i^\Tt_{0,\beta_0+1}(F)=i^\Tt_{1,\beta_0+1}(i^\Tt_{01}(F))$, we have \[F'\com i^\Tt_{1,\beta_0+1}\rest X=i^\Tt_{0,\beta_0+1}\rest X.\]
In more detail, if $\beta\in X$
then since $i^\Tt_{01}(F)(\beta)=i^\Tt_{01}(\beta)$, we have \[i^\Tt_{1,\beta_0+1}(i^\Tt_{01}(F))(i^\Tt_{1,\beta_0+1}(\beta))=i^\Tt_{1,\beta_0+1}(i^\Tt_{01}(\beta))=i^\Tt_{0,\beta_0+1}(\beta).\]
But $F'$ is indexed on $\es(M^\Tt_{\beta_0+1})$ at $\eta'=i^\Tt_{1,\beta_0+1}(i^\Tt_{01}(\eta))$, and
letting $\xi=\sup i^\Tt_{1,\beta_0+1}``\eta$,
 the short part $E$ of $E^\Tt_{\beta_0}$ is indexed on $\es(M^\Tt_{\beta_0+1})$ at $\xi$.
We have $\crit(F')<\xi<\crit(F')^{+M^\Tt_{\beta_0+1}}$,
so $\lambda(F')<F'(\xi)<\lh(F')$.
Let $\bar{G}=F'(E)=F^{M^\Tt_{\beta_0+1}|F'(\xi)}$.
Note that $\bar{G}$ is a short extender
with domain $M||\eta$,
and \[i^{M||\eta}_{\bar{G}}=i_{F'}\com i_E\rest(M||\eta).\] So by the previous calculations,
\[ \bar{G}\rest X=i^\Tt_{0,\beta_0+1}\rest X.\]
Also, note that $\lh(E)=\xi=\sup i^\Tt_{1,\beta_0+1}``\eta$
and \[\lh(\bar{G})=F'(\xi)=\sup F'``\xi=\sup i^\Tt_{1,\beta_0+1}\com i^\Tt_{01}``\eta,\]
and $\cof^{M^\Tt_{\beta_0+1}}(\xi)=\cof^{M^\Tt_{\beta_0+1}}(F'(\xi))=\eta$,
since $\crit(\bar{G})=\lambda<\eta=\lambda^{+M^\Tt_{\beta_0+1}}$. So
$i^\Tt_{\beta_0+1,\infty}$ is continuous at $\lh(\bar{G})$. Therefore $\sup i^\Tt``\eta=i^\Tt_{\beta_0+1,\infty}(\lh(\bar{G}))$, so by Subclaim \ref{sclm:sup_i^U``eta=sup_i^T``eta}, $i^\Tt_{\beta_0+1,\infty}(\bar{G})=G$.
But since $\bar{G}\rest X=i^\Tt_{0,\beta_0+1}\rest X$, therefore $G\rest X=i^\Tt\rest X$, as desired.

Part \ref{item:also_agmt_with_i^Tt_1,infty}:
Suppose there is $\beta<\eta$ and a successor ordinal $\alpha<\kappa^{+M}$ such that:
\begin{enumerate}[label=--]
\item $\{\beta\}$ is $\rSigma_{n+1}^M$-definable from the parameters $F(\alpha),q,\pvec_n^M$, and
\item $\beta<\sup F``\alpha$ but $\beta\notin F``\alpha$,
\end{enumerate}
Then writing $(\alpha',\beta',F',q')=i^\Tt_{0,\beta_0+1}(\alpha,\beta,F,q)$, we have
\begin{enumerate}[label=--]
\item $\{\beta'\}$ is $\rSigma_{n+1}^{M^\Tt_{\beta_0+1}}$-definable from the parameters  $F'(\alpha'),q',\pvec_n^{M^\Tt_{\beta_0+1}}$, and
\item $\beta'<\sup F'``\alpha'$ but $\beta'\notin F'``\alpha'$.
\end{enumerate}
Since $\bar{G}=F'\com E$ and $\alpha'=E(i^\Tt_{01}(\alpha))$ and $\alpha$ is a successor ordinal, we have
\begin{enumerate}[label=--]
\item $\{\beta'\}$ is $\rSigma_{n+1}^{M^\Tt_{\beta_0+1}}$-definable from the parameters  $\bar{G}(i^\Tt_{01}(\alpha)),q',\pvec_n^{M^\Tt_{\beta_0+1}}$, and
\item $\beta'<\sup \bar{G}``i^\Tt_{01}(\alpha)$ but $\beta'\notin \bar{G}``i^\Tt_{01}(\alpha)$.
\end{enumerate}
Lifting further with $i^\Tt_{\beta_0+1,\infty}$
and writing $(\beta'',q'')=i^\Tt(\beta,q)=i^\Tt_{\beta_0+1,\infty}(\beta',q')$, therefore
\begin{enumerate}[label=--]
\item $\{\beta''\}$ is $\rSigma_{n+1}^{M^\Tt_\infty}$-definable from the parameters $G(i^\Tt_{01}(\alpha)),q'',\pvec_n^{M^\Tt_\infty}$, and
\item $\beta''<\sup G``i^\Tt_{01}(\alpha)$ but $\beta''\notin G``i^\Tt_{01}(\alpha)$.
\end{enumerate}
But $M^\Tt_\infty=M^\Uu_\infty$ and the parameters used in the latter definition are in $\rg(i^\Uu)$
(since $G\rest\eta=i^\Uu\rest\eta$), 
so $\beta''\in\rg(i^\Uu)$. Therefore, since $\beta''<\sup i^\Uu``i^\Tt_{01}(\alpha)$,
we must have $\beta''\in i^\Uu``i^\Tt_{01}(\alpha)=G``i^\Tt_{01}(\alpha)$,
a contradiction.

Part \ref{item:agmt_thru_kappa^++}:
We will consider $C$ as the transitive collapse
of $H^W$, and show that $C|\kappa^{++C}=M|\kappa^{++M}$.
Since $M^\Tt_1=\Ult_n(M,E^\Tt_0)$
and $E^\Tt_0$ is short with $\crit(E^\Tt_0)=\kappa$
and $\rho_{n+1}^M=\kappa^{+M}$,
note that $M^\Tt_1|\lambda^{++M^\Tt_1}=M^\Tt_1|\eta^{+M^\Tt_1}=\Ult_0(M|\kappa^{++M},E^\Tt_0)$,
and letting $j:M|\kappa^{++M}\to M^\Tt_1|\lambda^{++M^\Tt_1}$ be the resulting ultrapower map, then
$j=i^\Tt_{01}\rest(M|\kappa^{++M})$.
(This is because the $\bfrSigma_n^M$ functions $f:[\kappa]^{<\om}\to\kappa^{+M}$
are all bounded in $\kappa^{+M}$,
and hence $f\in M$,
and $\rho_{n+1}^M<\rho_n^M$,
since $\eta\in p_{n+1}^M\neq\emptyset$,
so $\kappa^{+M}<\rho_n^M$,
which implies that the $\bfrSigma_n^M$ functions $f:[\kappa]^{<\om}\to\kappa^{++M}$ are bounded in $\kappa^{++M}$, and hence also $f\in M$.)
But
\[ M^\Tt_1|\lambda^{++M^\Tt_1}=M^\Tt_\infty|\lambda^{++M^\Tt_\infty}=M^\Uu_\infty|\lambda^{++M^\Uu_\infty}=W|\lambda^{++W}.\]
So $W|\lambda^{++W}=\Ult_0(M|\kappa^{++M},E^\Tt_0)$ and
 $j:M|\kappa^{++M}\to W|\lambda^{++W}$
is the  associated ultrapower map,
and $X=j``\kappa^{+M}$.
We have that $C$ is the transitive collapse of $H^W$; let $\sigma:C\to H^W$ be the uncollapse map.
By parts
\ref{item:pointwise_images_of_hulls}
 and 
\ref{item:also_agmt_with_i^Tt_1,infty},
$H^W\cap\eta=H^M\cap\eta=X$,
so $C|\kappa^{+C}=M|\kappa^{+M}$ and $\sigma\rest(\kappa^{+C}+1)=j\rest(\kappa^{+M}+1)$,
and $\sigma,j$ are continuous at $\kappa^{+M}$.
It easily follows that $\sigma,j$ agree on $\pow(\kappa^{+M})\cap M\cap C$.

Let $\xi=\min(\kappa^{++C},\kappa^{++M})$. Then we claim $C||\xi=M||\xi$. For otherwise,
let $(R,s)\pins (C|\kappa^{++C},0)$ and $(S,s)\pins (M|\kappa^{++M},0)$ be such that $R||\kappa^{++R}=S||\kappa^{++S}$
and \[ \rho_{r+1}^R=\kappa^{+R}=\kappa^{+C}=\kappa^{+M}=\kappa^{+S}=\rho_{s+1}^S\]
but $R\neq S$. Let $U_R=\Ult_r(R,E^\Tt_0)$
and $U_S=\Ult_s(S,E^\Tt_0)$.
Then by the argument of \cite[***]{premouse_inheriting} and internal condensation
for $M$, we get $U_R\pins W|\lambda^{++W}$
and $U_S\pins M^\Tt_1|\lambda^{++M^\Tt_1}$
and $U_R\neq U_S$, but 
 $i^{R,r}_{E^\Tt_0},i^{S,s}_{E^\Tt_0}$
both continuous  at $\kappa^{++R}=\kappa^{++S}$.
This contradicts the fact that $W||\kappa^{++W}=M||\kappa^{++M}$.

So $C||\xi=M||\xi$. If $\kappa^{++M}=\xi<\kappa^{++C}$, then (by the foregoing discussion) $j\sub\sigma$, but because $j$ is cofinal in $\lambda^{++W}$, this is a contradiction.
So suppose $\kappa^{++C}=\xi<\kappa^{++M}$. 
Then $H^W\cap\lambda^{++W}$ is bounded in $\lambda^{++W}$ (since $\sigma\sub j$ and $j``\kappa^{++M}\sub\kappa^{++W}$).
It follows that $H^M\cap\lambda^{++M}$ is bounded in $\lambda^{++M}$ (since $i^\Uu(\lambda)=i^\Tt(\lambda)$, $i^\Tt$ is continuous at $\lambda^{++M}$ and $i^\Tt``H^M=i^\Uu``H^W$). So we can fix $\beta<\lambda^{++M}$ such that
\[ \Hull_{n+1}^M(X\cup\{q,\pvec_n^M\})\cap\lambda^{++M}\sub\beta.\]
Let us assume that $\crit(i^\Tt_{1\infty})=\lambda$, since the other case is a simplification of this one.
It follows that writing $(X',q',\lambda',\beta')=i^\Tt_{0,\beta_0+1}(X,q,\lambda,\beta)$, we have
\[ \Hull_{n+1}^{M^\Tt_{\beta_0+1}}(X'\cup\{q',\pvec_n^{M^\Tt_{\beta_0+1}}\})\cap(\lambda')^{++M^\Tt_{\beta_0+1}}\sub\beta'\]
(and note that $X'=F'``(\kappa')^{+M^\Tt_{\beta_0+1}}$
where $(F',\kappa')=i^\Tt_{0,\beta_0+1}(F,\kappa)$).
Therefore
\[ \Hull_{n+1}^{M^\Tt_{\beta_0+1}}((\bar{G}``\eta)\cup\{q',\pvec_n^{M^\Tt_{\beta_0+1}}\})\cap(\lambda')^{++M^\Tt_{\beta_0+1}}\sub\beta',\]
since $\bar{G}=F'\com E$, with notation as before.
Lifting further with $i^\Tt_{\beta_0+1,\infty}$, therefore
\begin{equation}\label{eqn:Hull_at_M_infty_bounded} \Hull_{n+1}^{M^\Tt_\infty}((G``\eta)\cup\{i^\Tt(q),\pvec_n^{M^\Tt_\infty}\})\cap i^\Tt(\lambda)^{++M^\Tt_\infty}\sub i^\Tt(\beta).\end{equation}
But the hull indicated in line (\ref{eqn:Hull_at_M_infty_bounded}) is exactly $\rg(i^\Uu)$, since $M^\Tt_\infty=M^\Uu_\infty$
and $i^\Tt(q)=i^\Uu(\bar{q})$ and $G``\eta=i^\Uu``\eta$. Also $i^\Uu(\lambda)=i^\Tt(\lambda)$, so
 \[ \rg(i^\Uu)\cap i^\Uu(\lambda)^{++M^\Uu_\infty}\sub i^\Tt(\beta),\]
 and we have $i^\Uu(\lambda)^{+M^\Uu_\infty}<i^\Tt_{0\infty}(\beta)<i^\Uu_{0\infty}(\lambda)^{++M^\Uu_\infty}$.
 This contradicts the fact that $i^\Uu$ is continuous at $\lambda^{++W}$.
\end{proof}

Let $\sigma:C\to M$ be the uncollapse
embedding, where $C$ was defined in Subclaim \ref{sclm:a_subclaim_defining_some_C}.

\begin{sclmthree}
$C$ is $(n+1)$-solid and $(n+1)$-sound,
with $\rho_{n+1}^C=\rho_{n+1}^M=\kappa^{+M}$
$\sigma(\pvec_n^C)=\pvec_n^M$ and $\sigma(p_{n+1}^C)=q=p_{n+1}^M\cut\{\eta\}$.
\end{sclmthree}
\begin{proof}
We have that $C$ is $n$-sound with $\sigma(\pvec_n^C)=\pvec_n^M$ as usual.
We also have $C=\Hull_{n+1}^C(\kappa^{+M}\cup\{\pvec_n^C,\sigma^{-1}(q)\})$ where $q$ is as in the subclaim.

So $\rho_{n+1}^C\leq\kappa^{+M}=\rho_{n+1}^M$. If $\rho_{n+1}^C<\kappa^{+M}$ then we get a contradiction because the same missing set would be $\bfrSigma_{n+1}^M$-definable. So $\rho_{n+1}^C=\kappa^{+M}$.

Likewise, $p_{n+1}^C\leq \sigma^{-1}(q)$. But if $p_{n+1}^C<\sigma^{-1}(q)$
then note that $t=\Th_{\rSigma_{n+1}}^C(\kappa^{+M}\cup\{\pvec_n^C,p_{n+1}^C\})$
is $\rSigma_{n+1}^M(\{\pvec_n^M,\sigma(p_{n+1}^C),\eta\})$,
but $(\sigma(p_{n+1}^C),\eta)<q$, so $t\in M$
by the minimality of $p_{n+1}^M$,
so $t\in C$, contradiction.
So $p_{n+1}^C=\sigma^{-1}(q)$.

It follows that $C$ is either $(n+1)$-solid
or stretched-$(n+1)$-solid
(by the global induction of our solidity proof,
and since $C$ is also $(n,\om_1,\om_1+1)^*$-iterable and $|p_{n+1}^C|<|p_{n+1}^M|$).
So we just need to show that $C$ is not stretched-$(n+1)$-solid, so suppose it is.

Let $\theta=\min(p_{n+1}^C)$.
So $H=F^{C|\theta}$
is a short $C$-total extender with $\crit(H)=\kappa$, etc, and letting
\[ D=\cHull_{n+1}^C((H``\kappa^{+M})\cup\{r,\pvec_n^C\}) \]
and $\tau:D\to C$ be the uncollapse, where
$r=p_{n+1}^C\cut\{\theta\}$,
then $D||\kappa^{++D}=C||\kappa^{++C}=M||\kappa^{++M}$, so $D\notin M$ and
\[ t=\Th_{\rSigma_{n+1}}^D(\kappa^{+M}\cup\{\tau^{-1}(r),\pvec_n^D\})\notin M.\]

We will show that $t\in M$, for a contradiction.
Note that $\sigma(H)=F^{M|\sigma(\theta)}$
is a short $M$-total extender with $\crit(\sigma(H))=\sigma(\kappa)=\lambda$.
Let $F^*=H(F)$, an extender equivalent to ``$H\com F$'', and which is indexed on $\es^M$
with $\sigma(\lambda(H))<\lh(F^*)<\sigma(\theta)=\sigma(\lh(H))$, and since $\theta=\min(p_{n+1}^C)$,
we have $\sigma(\theta)=\min(p_{n+1}^M\cut\{\eta\})$, and therefore $(p_{n+1}^M\cut\{\sigma(\theta),\eta\})\cup\{\lh(F^*)\}<p_{n+1}^M$.

It suffices to see that $t$ is $\rSigma_{n+1}^M(\{(p_{n+1}^M\cut\{\sigma(\theta),\eta\})\cup\{\lh(F^*)\},\pvec_n^M\})$, since then by the minimality of $p_{n+1}^M$, we have $t\in M$.
But given any $\rSigma_{n+1}$ formula $\varphi$
and given $\xi<\kappa^{+M}=\kappa^{+C}=\kappa^{+D}$, we have
\[ D\sats\varphi(\xi,\tau^{-1}(p_{n+1}^C\cut\{\theta\}),\pvec_n^D) \]
iff
\[ C\sats\varphi(H(\xi),p_{n+1}^C\cut\{\theta\},\pvec_n^C\}) \]
iff
\[ M\sats\varphi(\sigma(H(\xi)),p_{n+1}^M\cut\{\sigma(\theta),\eta\},\pvec_n^M),\]
but $\sigma(H(\xi))=\sigma(H)(\sigma(\xi))=\sigma(H)(F(\xi))=F^*(\xi)$, so the above statements are equivalent to
\[ M\sats\varphi(F^*(\xi),p_{n+1}^M\cut\{\sigma(\theta),\eta\},\pvec_n^M), \]
which shows that $t$ is definable in the desired manner.

\end{proof}

This completes the proof of Claim \ref{clm:M_is_simulated_solid}.
\end{proof}

This completes the proof of solidity/stretched-solidity.
\end{proof}

\subsection{Condensation}

In \cite{voellmer},
Voellmer formulates and proves condensation
for plus-one premice (which are of course
closely related to the premice considered here).
We will formulate and prove analogous results here, but will also generalize and strengthen the results somewhat. First, the ``Anomalous Case 4'' defined in \cite[\S6, p.~35]{voellmer}
was not dealt with in \cite{voellmer},
but we do handle the analogue here.
Second, the definition of \emph{premouse}
in \cite{voellmer} demands projectum free spaces, so in particular, for the embeddings $\pi:H\to M$ considered for the condensation results in \cite{voellmer} (written $\sigma:H\to M$ there), $H$ is assumed to have projectum free spaces. Of course, we do not discuss projectum free spaces here at all,
and so this does not feature in our condensation results either. Thirdly,
and probably of most interest,
$\crit(\pi)=\kappa<\kappa^{+H}=\rho_{n+1}^H=\kappa^{+M}$, \emph{as long as the short extender $E$ derived from $\pi$ is in $\es^M$}.
This is a new situation which arises naturally
for premice at the present level.

\begin{tm}\label{tm:first_cond}  Let $n<\om$ and let $M$ be an $(n,\om_1,\om_1+1)^*$-iterable  $n$-sound premouse with $\om<\rho_n^M$.
 Let $W$ be an $(n+1)$-sound premouse. Suppose that if $M$ is active short then $M,W$ are Dodd-absent-sound.
 Let $\pi:W\to M$ be an $n$-lifting $c$-preserving $\pvec_n$-preserving embedding
 such that either:
 \begin{enumerate}[label=(\alph*)]
  \item 
$\eta=\rho_{n+1}^W\leq\crit(\pi)$, or
\item\label{item:cond_moving_below_proj} $\crit(\pi)=\mu<\eta=\rho_{n+1}^W=\mu^{+W}=\mu^{+M}$ and $E\in\es^M$,
where $E$ is the short extender derived from $\pi$.
\end{enumerate}
 Then either:
 \begin{enumerate}
  \item $W\pins M$, or
  \item $M|\eta$ is active
  and $W\pins\Ult(M|\eta,F^{M|\eta})$, or
  \item\label{item:clause_3_cond}
 $W=\core_{n+1}(M)$ and $\pi$ is the core map, or
  \item\label{item:clause_4_cond}
 $M$ is stretched-$(n+1)$-solid,
  $\eta=\min(p_{n+1}^M)$, and letting $p=p_{n+1}^M\cut\{\eta\}$,
  \[W=\cHull_{n+1}^M(\eta\cup\{\pvec_n^M,p\}),\]
  or equivalently, $W=\Ult_n(\core_{n+1}(M),F^{M|\eta})$; moreover, $\pi$ is the uncollapse map.
 \end{enumerate}
\end{tm}

We will also prove the following analogue:
\begin{tm}\label{tm:second_cond}Let $M$ be a $(0^-,\om_1,\om_1+1)^*$-iterable active short  premouse.
Let $W$ be a Dodd-absent-sound active short premouse with $\kappa=\crit(F^W)=\crit(F^M)$.
Let $\pi:W\to M$ be $0$-deriving \tu{(}hence $c$-preserving\tu{)} and $\mu=\crit(\pi)$ and suppose either:
\begin{enumerate}[label=--]
 \item 
$\kappa^{+W}=\kappa^{+M}<\mu$ and $\rho^W_{\mathrm{D}}\leq\mu$, or
\item $\kappa^{+W}=\kappa^{+M}<\mu<\mu^{+W}=\mu^{+M}=\rho^W_{\mathrm{D}}$ and the short extender $E$ derived from $\pi$ is in $\es^M$, or
\item $\mu=\kappa<\kappa^{+W}=\kappa^{+M}$ and $\rho_{\mathrm{D}}^W=0$ and the short extender $E$ derived from $\pi$ is in $\es^M$.
\end{enumerate}
Let $\eta=\max(\kappa^{+M},\rho_{\mathrm{D}}^W)$. Then either:
\begin{enumerate}[label=\arabic*'.,ref=\arabic*']
 \item $W\pins M$, or
 \item $M|\eta$ is active and $W\pins\Ult(M|\eta,F^{M|\eta})$, 
 \item\label{item:clause_3'}  $\rho_{\D}^W=0$, $W=\core_{\mathrm{D}}(M)$ and $\pi$ is the Dodd-absent-core map, 
 \item\label{item:clause_4'} $\rho_{\D}^W>0$, 
 $W=\core_1(M)$ is $1$-sound and $\pi$ is the core embedding, or
 \item\label{item:clause_5'} $\rho_{\D}^W>0$, 
 $M$ is stretched-$1$-solid, $\eta=\min(p_1^M)$,
 and letting $p=p_1^M\cut\{\eta\}$,
 \[ W=\cHull_1^M(\eta\cup\{p\}),\]
 or equivalently, $W=\Ult_0(\core_1(M),F^{M|\eta})$; moreover, $\pi$ is the uncollapse map.
\end{enumerate}
\end{tm}
\begin{rem}
 Note that  in Theorem \ref{tm:second_cond}, there is no analogue of clause \ref{item:clause_4_cond} of Theorem \ref{tm:first_cond}; this is because the analogue would be a proper protomouse, but $W$ is assumed to be a premouse, so this option is ruled out.
 But maybe one could formulate a more general version in which $W$ is allowed to be a proper protomouse, and this would become a valid option.
\end{rem}

\begin{proof}[Proof of Theorems \ref{tm:first_cond} and \ref{tm:second_cond}]
Let us first make some remarks about Theorem  \ref{tm:first_cond} in case $n=0$ and $W,M$ are active short.
If $\kappa^{+W}<\rho_1^W$ then
Theorem \ref{tm:second_cond}
suffices to give the desired conclusions. (For since $W$ is Dodd-absent-sound and $1$-sound,
$\rho_1^W=\rho_{\D}^W$.
But if $\rho_1^W\leq\crit(\pi)$
then $\pi\rest\kappa^{++W}=\id$,
so $\kappa^{+W}=\kappa^{+M}$
and $\crit(F^M)=\kappa$,
so Theorem \ref{tm:second_cond} applies.
If instead $\crit(\pi)=\mu<\rho_1^W=\eta=\mu^{+W}=\mu^{+M}$,
then again $\kappa^{+M}=\kappa^{+W}\leq\mu$,
but therefore $\kappa^{+W}<\mu$,
and again Theorem \ref{tm:second_cond} applies. In either case, its conclusions suffice.) So suppose $\rho_1^W\leq\kappa^{+W}$,
so $\rho_{\D}^W=0$.
We have $\pi(\kappa)=\crit(F^W)$,
since $\pi$ is $0$-lifting. If also $\rho_1^M\leq\pi(\kappa)^{+M}$,
then we can work with $0$-maximal trees in the comparison arguments to follow (and iteration maps which do not drop in any way are $0$-embeddings, and we have weak Dodd-Jensen). Suppose instead $\rho_1^M>\pi(\kappa)^{+M}$. Then
we can find $x\in M$
and replace $M$ with $\bar{M}=\cHull_1^{M}(\pi(\kappa)^{+M}\cup\{x\})$,
in such a manner that we still get $\bar{\pi}:W\to \bar{M}$,
with $\sigma\com\bar{\pi}=\pi$
where $\sigma:\bar{M}\to M$ is the uncollapse map,
and we still have the hypotheses of the theorem with $\bar{M}$ replacing $M$.
Thus, working with $\bar{M}$ instead of $M$, we can  work safely with $0$-maximal trees in the comparison arguments. So we will assume we have made these arrangements in what follows.\footnote{Recall that if we form $0$-maximal trees on $M$ where $\pi(\kappa)^{+M}<\rho_1^M$, we might form ultrapowers ``avoiding the protomouse'' along the main branch, even at degree $0$.
We did not establish weak Dodd-Jensen in this situation.}

  We may assume $M$ is countable,
 and fix an $(n,\om_1+1)$-strategy $\Sigma$ for $M$ with the weak Dodd-Jensen property with respect to an enumeration of $M$ (see the previous paragraph).

 The proof follows a similar structure to the main phalanx comparisons used earlier,
 \ref{tm:rho^M_D=0_implies_Dodd-absent-solid_and_universal}
 and \ref{tm:solidity}.
We define a phalanx $\ph$, according to cases mostly similar
to  the earlier proofs.  As before, if $\eta$ is not a limit cardinal of $M$, then $\gamma$
denotes the largest $M$-cardinal ${<\eta}$.

\begin{enumerate}[label=Ph\arabic*.,ref=Ph\arabic*]
 \item\label{item:condensation_phalanx_eta_card} If $\eta$ is an $M$-cardinal then
 $\ph=((M,{<\eta}),W)$.
 \item\label{item:condensation_phalanx_eta_active} If $\eta=\gamma^{+W}<\gamma^{+M}$ and $M|\eta$ is active then:
\[ \ph=((M,{<\gamma}),((\Ult_n(M,F^{M|\eta}),n'),\gamma),W),\]
where:
\begin{enumerate}[label=--]
\item if $n\geq 0$ and $\Ult_n(M^{M|\eta})$
 is Dodd-absent-sound then $n'=n$, and
 \item otherwise, $n'=0^-$.
 \end{enumerate}
  \item\label{item:condensation_phalanx_eta_passive} If $\eta=\gamma^{+W}<\gamma^{+M}$ and
  $M|\eta$ is passive
 then letting $(R,r)\pins (M,0)$
 be least such that $\eta\leq\OR^R$ and $\rho_{r+1}^R=\gamma<\rho_r^R$, 
 \[ \ph=((M,{<\gamma}),((R,r),\gamma),W).\]
 \end{enumerate}

As usual, we only consider $n$-maximal $\Uu$ on $\ph$
with $E^\Uu_0\in\es_+^W$
and $\eta^{+W}<\lh(E^\Uu_0)$,
 \emph{iterability} for $\ph$ is in this sense, $R$ is at degree $r$,
 $\Ult_n(M,F^{M|\eta})$ at degree $n'$, 
 and all other models of $\ph$ at degree $n$. Note that all active short models $P$ of $\ph$ are Dodd-absent-sound, except
 possibly in the case that $M$ is active short and either:
 \begin{enumerate}[label=--]
  \item $n=0^-$ and either $P=M$ or $P=\Ult_0(M,F^{M|\eta})$, or
  \item $n=0$, $P=\Ult_0(M,F^{M|\eta})$ and $\spc(F^{M|\eta})\geq\max(\kappa^{+M},\rho_{\D}^M)$ where $\kappa=\crit(F^M)$.
 \end{enumerate}
Thus, all models of $\ph$ which are at degree $n\geq 0$ are Dodd-absent-sound. Note that if $n=0^-$ then $W,M$ are at degree $n=0^-$,
irrespective of Dodd-absent-soundness. (Likewise, if $n\geq 0$,
then $W$ is at degree $n$, even though it is $(n+1)$-sound.)

 Using  lifting methods as before
 (in particular using the methods
 of Theorem \ref{tm:rho^M_D=0_implies_Dodd-absent-solid_and_universal}
 in case $\pi\rest\eta\neq\id$)
 we have:
\begin{clmfour}
$\ph$ is $(\om_1+1)$-iterable.
In fact, there is an $(\om_1+1)$-iteration strategy $\Gamma$ for $\ph$ such that trees $\Uu$ on $\ph$ via $\Gamma$ lift to trees $\Uu^*$ on $M$ via $\Sigma$ in a manner
analogous to that in the proofs of Theorem \ref{tm:rho^M_D=0_implies_Dodd-absent-solid_and_universal}
and Theorem \ref{tm:solidity}.
\end{clmfour}

  \begin{clmfour}\label{clm:cond_pow(eta)^U_sub_M}
   We have:
   \begin{enumerate}
    \item In case \ref{item:condensation_phalanx_eta_card} of the definition of $\ph$,
    $W|\eta^{+W}=M||\eta^{+W}$,
    so $\pow(\eta)^W\sub M$.
    \item In case \ref{item:condensation_phalanx_eta_active}
   of the definition of $\ph$, we have
   $W|\eta^{+W}=W|\gamma^{++W}=\Ult_n(M,F^{M|\eta})||\eta^{+W}$,
   so
   $\pow(\gamma)^W\sub M$ and in fact $\pow(\eta)^W\sub\Ult_n(M,F^{M|\eta})$,
    \item In case \ref{item:condensation_phalanx_eta_passive}
    of the definition of $\ph$, we have
    $W|\eta^{+W}=W|\gamma^{++W}=R||\eta^{+W}$, so $\pow(\gamma)^W\sub M$ and in fact $\pow(\eta)^W\sub R$.
   \end{enumerate}
  \end{clmfour}

\begin{proof}
 In case $\gamma=\crit(\pi)<\eta=\gamma^{+M}$,
 this is by the theorem being proved, applied to segments $W'\pins W$ and $M'=\pi(W')\pins M$
 and maps $\pi':W'\to M'$ with $\pi'=\pi\rest W'$, first (i) where $\rho_\om^{W'}=\gamma$, (this gives $W|\eta=M|\eta$), and then (ii)
 where $\rho_\om^{W'}=\eta$ (this gives $W|\eta^{+W}=M||\eta^{+W}$).
 Note that the hypotheses apply for both of these cases, in case (ii) because
 the short extender $E'$ derived from $\pi'$ is just $E'=E\in\es^{M}$,
 and $\pi'(\eta)=\pi(\eta)=\pi(\gamma)^{+M}$,
 and $\pi(\gamma)<\lh(E)<\pi(\gamma)^{+M}$,
 so $E'=E\in\es^{M'}$ also. And considering the projecta involved, the only valid conclusion
 of the theorem is that $W'\pins M'$,
 so $W'\pins M$.
\end{proof}

\begin{clmfour}Let $\Uu$ be an $n$-maximal tree on $\ph$ with $\lh(E^\Uu_0)>\eta^{+W}$.
Let $\alpha+1<\lh(\Uu)$.
Then either:
\begin{enumerate}
 \item $E^\Uu_\alpha$ is close to $M^{*\Uu}_{\alpha+1}$, or
 \item $n=0$, $W\in M$, $W$ is active,
  $\mathrm{root}^\Uu(\alpha)=0$,
 $[0,\alpha]^\Uu\cap\dropset^\Uu=\emptyset$,  $E^\Uu_\alpha=F(M^\Uu_\alpha)$, $\spc(F^W)=\gamma$, $\pred^\Uu(\alpha+1)=-1$,
 $(E^\Uu_\alpha)_a\in M$ for each $a\in[\nu(E^\Uu_\alpha)]^{<\om}$, and either:
 \begin{enumerate}[label=--]
  \item $\ph$ is defined as in case \ref{item:condensation_phalanx_eta_active}, $M^{*\Uu}_{\alpha+1}=\Ult_n(M,F^{M|\eta})$, or
 \item $\ph$ is defined as in case \ref{item:condensation_phalanx_eta_passive},
  $M^{*\Uu}_{\alpha+1}=R$,
 $\rho_{r+1}^{M^\Uu_{\alpha+1}}=\gamma=\rho_{r+1}^R$ and the usual correspondence holds between $p_{r+1}^R$ and $p_{r+1}^{M^\Uu_{\alpha+1}}$; that is, 
 either 
 \begin{enumerate}[label=--]
 \item $p_{r+1}(M^\Uu_{\alpha+1})=i^{*\Uu}_{\alpha+1}(p_{r+1}^R)$
 and $M^\Uu_{\alpha+1}$ is $(r+1)$-solid, 
 or
 \item $E^\Uu_\alpha$ is long
 and $p_{r+1}(M^\Uu_{\alpha+1})=i^{*\Uu}_{\alpha+1}(p_{r+1}^R)\cup\{\xi\}$
 where $\xi$ is the index of the short part of $E^\Uu_\alpha$,
 and $M^\Uu_{\alpha+1}$
 is stretched-$(r+1)$-solid.
\end{enumerate}
\end{enumerate}
\end{enumerate}
\end{clmfour}
\begin{proof}
 \begin{casesix}
  $\ph$ is defined as in case \ref{item:condensation_phalanx_eta_card} ($\eta$ is an $M$-cardinal).
  
  The main situation
  which differs from the usual proof of closeness is in the case that $\eta=\mu^{+M}$, $\pi\rest\eta\neq\id$,
  $\mathrm{root}^\Uu(\alpha)=0$
  and $E^\Uu_\alpha$ is short with $\crit(E^\Uu_\alpha)=\mu$,
  so $\pred^\Uu(\alpha+1)=-1$
  and $M^{*\Uu}_{\alpha+1}=M$.
  As usual, $E^\Uu_\alpha$ is close to $W$.
  We then use the $\rSigma_1$-elementarity of $\pi:W\to M$,
  along with the fact that the short extender $E_\pi$ derived from $\pi$ is in $M$ (in fact in $\es^M$), to deduce that $E^\Uu_\alpha$ is close to $M$.
 \end{casesix}

  \begin{casesix}
  $\ph$ is defined as in case \ref{item:condensation_phalanx_eta_passive} ($\eta$ is a non-$M$-cardinal but $M|\eta$ is passive).
  
 We claim that for all $\alpha+1<\lh(\Uu)$, if $E^\Uu_\alpha$ is not close to $M^{*\Uu}_{\alpha+1}$
 then:
 \begin{enumerate}[label=(\roman*)]
 \item\label{item:non-closeness_basic_props} $n=0$, $W,M$ are active, $\mathrm{root}^\Uu(\alpha)=0$,
 $[0,\alpha]^\Uu\cap\dropset^{\Uu}=\emptyset$, $E^\Uu_\alpha=F(M^\Uu_\alpha)$, $\spc(E^\Uu_\alpha)=\spc(F^W)=\gamma$, $\pred^\Uu(\alpha+1)=-1$, and $M^{*\Uu}_{\alpha+1}=R$,
 \item\label{item:non-closeness_extra_props} $W\in M$,
 $(E^\Uu_\alpha)_a\in M$ for each $a\in[\nu(E^\Uu_\alpha)]^{<\om}$,
 $\rho_{r+1}^{M^\Uu_{\alpha+1}}=\gamma=\rho_{r+1}^R$, and the usual correspondence between $p_{r+1}^R$ and $p_{r+1}^{M^\Uu_{\alpha+1}}$ holds.
 \end{enumerate}
 
 Suppose first that $\alpha$ is least such that $E^\Uu_\alpha$ is not close to $M^{*\Uu}_{\alpha+1}$. Then part \ref{item:non-closeness_basic_props} for $\alpha$ follows by the usual arguments (the minimality is useful here because the usual proof of closeness is by induction).  For part \ref{item:non-closeness_extra_props}, we get $W\in M$ because $\eta=\gamma^{+W}=\dom(F^W)=\dom(E^\Uu_\alpha)<\gamma^{+M}=\dom(F^M)$,
 so $\rg(\pi)$ is bounded in $\dom(F^M)$, so $\rg(\pi)$ is bounded in $\OR^M$ (and $n=0$).
 Since $E^\Uu_\alpha$ is close to $W$ and $W\in M$, each component measure $(E^\Uu_\alpha)_a$ is in $M$. We have $\rho_{r+1}^{M^\Uu_{\alpha+1}}\leq\gamma$
 as usual. Since also $R\in M$
 and $\gamma$ is an $M$-cardinal,
 it follows that $\gamma=\rho_{r+1}(M^\Uu_{\alpha+1})$.
 From here, the usual calculations
 yield the usual correspondence
 between $p_{r+1}^R$ and $p_{r+1}(M^\Uu_{\alpha+1})$.
 
 We now proceed by induction.
 The extra wrinkle to consider is as follows: suppose
 that $\beta+1<\lh(\Uu)$  is such that $E^\Uu_\beta$ is not close to $M^{*\Uu}_{\beta+1}$,
 $\mathrm{root}^\Uu(\beta)=-1$,
 $(-1,\beta]^\Uu\cap\dropset^\Uu=\emptyset$, $E^\Uu_\beta=F(M^\Uu_\beta)$, and $\pred^\Uu(\beta+1)=-2$,
 so $M^{*\Uu}_{\beta+1}=M$.
 The problem here is that we can't just use the usual inductive argument, because letting $\alpha+1=\min(-1,\beta]^\Uu$,
 it could be that $E^\Uu_\alpha$ is not close to $R$. Note that either:
 \begin{enumerate}[label=(\alph*)]
  \item\label{item:gamma_limit_card} $\gamma$ is a limit cardinal of $M$, $F^M,F^W,E^\Uu_\alpha$ are short, $\crit(F^W)=\crit(E^\Uu_\alpha)=\gamma$ and $R$ is active with $\crit(F^R)=\crit(E^\Uu_\beta)<\gamma$, or
  \item\label{item:gamma_succ_card} $\gamma=\mu^{+M}$ where $\mu$ is an $M$-cardinal,
  $F^M,F^W,E^\Uu_\alpha$
  are long, $\spc(F^W)=\spc(E^\Uu_\alpha)=\gamma$, and $R$ is active with $\spc(F^R)=\spc(E^\Uu_\beta)\leq\mu$.
 \end{enumerate}

 Suppose \ref{item:gamma_limit_card} holds. We have $\rho_1(M^\Uu_\beta)<\gamma$, but $\rho_{r+1}^R=\gamma=\rho_{r+1}(M^\Uu_{\alpha+1})=\rho_{r+1}(M^\Uu_\beta)$, a contradiction.
 
 So \ref{item:gamma_succ_card}
 holds, and much as above,
 $F^R$ is short with $\crit(F^R)=\mu$ $\rho_1(M^\Uu_\beta)=\rho_1^R=\gamma=\mu^{+M}$, and $r=0$,
 so $\Ult_0(M^{*\Uu}_{\alpha+1},E^\Uu_\alpha)$ is formed avoiding the protomouse,
 so that indeed $\crit(F(M^\Uu_{\alpha+1}))=\mu$.
 As usual, $E^\Uu_\beta$ is close to $M^\Uu_{\alpha+1}$.
 But since $R\in M$ and $(E^\Uu_\alpha)_a\in M$ for each $a$, it follows that every $\bfrSigma_1(M^\Uu_{\alpha+1})$
 subset of $\gamma$ is in $M$,
 and therefore $E^\Uu_\beta$ is close to $M$, a contradiction.
\end{casesix}

 \begin{casesix}
  $\ph$ is defined as in case \ref{item:condensation_phalanx_eta_active} ($M|\eta$ is active).
  
  This is like the previous case,
  but simpler: as before,
  if  $\alpha$ is least such that $E^\Uu_\alpha$ is not close to $M^{*\Uu}_{\alpha+1}$,
  then we have a situation much as before; in particular,
  $M,W,\exit^\Uu_\alpha$
  are active and $\spc(F^M)=\spc(F^W)=\spc(\exit^\Uu_\alpha)=\gamma$. (Recall here
  that if $E$ is close to $\Ult_n(M,F^{M|\eta})$ then $E$ is close to $M$,
  since $F^{M|\eta}\in M$.) 
  But $M|\eta$
  is active and $\eta=\gamma^{+W}$,
  so $\crit(F)<\gamma$ where $F=F^{M|\eta}$
  so $\spc(F^{\Ult_n(M,F)})=j(\gamma)>\gamma$. So if $(-1,\beta]^\Uu\cap\dropset^\Uu=\emptyset$ and $E^\Uu_\beta=F(M^\Uu_\beta)$
  then $\spc(E^\Uu_\beta)>\gamma$, so $\pred^\Uu(\beta+1)\neq -2$.\qedhere
 \end{casesix}
\end{proof}

\begin{clmfour}
 There is a successful comparison $(\Uu,\Tt)$
 of $(\ph,M)$, via $(\Gamma,\Sigma)$,
 $b^\Uu$ is above $W$, $b^\Uu$ does
  not drop in model, degree
  or Dodd-degree,
 and $M^\Uu_\infty\ins M^\Tt_\infty$.
\end{clmfour}
\begin{proof}
 This is basically as usual;
 note, for example,
 if $n=0^-$ then $W,M$ are both 
 at degree $0^-$ in $\ph$,
 so that if $b^\Uu$ drops in degree or Dodd-degree, then in fact it drops in model. If $n=0$ then $W,M$ are at degree $0$, and if $M^\Uu_\infty=M^\Tt_\infty$ and
 $b^\Uu,b^\Tt$ do not drop in model but one of them drops in Dodd-degree, then so does the other (as $M^\Uu_\infty=M^\Tt_\infty$ is non-Dodd-absent-sound),
 and the Dodd-core embedding
 of $M^\Uu_\infty=M^\Tt_\infty$
 is a common iteration map, a contradiction.
\end{proof}

\begin{clmfour}
 If $M^\Uu_\infty\pins M^\Tt_\infty$
 or $b^\Tt$ drops  in model,
 degree  or Dodd-degree,  then
 the conclusion of condensation holds.
\end{clmfour}
\begin{proof}
Assume $M^\Uu_\infty\pins M^\Tt_\infty$.
Then
$M^\Uu_\infty$ is sound,
and so by Lemma \ref{lem:dropped_iterated_core_embedding}
and because $W$ is $(n+1)$-sound
with $\rho_{n+1}^W=\eta$,
or $W$ is Dodd-absent-sound with $\max(\kappa^{+W},\rho_{\D}^W)=\eta$,
note that $\Uu$ is trivial,
so $W\pins M^\Tt_\infty$.
But since $\rho_{n+1}^W=\eta\leq\lh(E^\Tt_0)$, it follows that either $\Tt$ is trivial and $W\pins M$, or $\Tt$ uses only one extender, whose index is $\eta$, and $W\pins M^\Tt_1$, which suffices.

If instead $M^\Uu_\infty=M^\Tt_\infty$
but $b^\Tt$ drops in model, degree or Dodd-degree,
then since $\deg^\Uu_\infty=n$,
in fact $b^\Tt$ drops in model.
By the material in \S\ref{sec:long_mice}
and since $\rho_{n+1}^{W}\leq\eta\leq\lh(E^\Tt_0)$, we easily get that $W=\core_{n+1}(M^\Tt_\infty)$
and the conclusion of condensation holds.
\end{proof}

 So we may assume that $b^\Tt$ does not drop in model, degree or Dodd-degree, and $M^\Uu_\infty=M^\Tt_\infty$.
 Therefore:
 \begin{enumerate}[label=--]
  \item If $n\geq 0$ then
$\eta=\rho_{n+1}^W=\rho_{n+1}(M^\Uu_\infty)=\rho_{n+1}(M^\Tt_\infty)$. So $\rho_{n+1}^M\leq\eta$.
\item If $n=0^-$ then $\rho_{\D}^W=\rho_{\D}(M^\Uu_\infty)=\rho_{\D}(M^\Tt_\infty)$
and $\kappa=\crit(F^W)=\crit(F(M^\Uu_\infty))=\crit(F(M^\Tt_\infty))$.
So $\rho_{\D}^M\leq\rho_{\D}^W$
(and recall that in this case we assume that $\crit(F^W)=\kappa=\crit(F^M)$).
\end{enumerate}

 \begin{clmfour}
 Suppose $n\geq 0$ and $\rho_{n+1}^M=\eta$. Then $W=\core_{n+1}(M)$
 and $\pi$ is the core embedding.
 That is, clause \ref{item:clause_3_cond}
 of the conclusion of condensation holds.
 \end{clmfour}
 \begin{proof}
 Since $\rho_{n+1}^M=\eta$,
 $\eta$ is an $M$-cardinal.
 
 \begin{sclmfour}\label{sclm:W_notin_M_and_pi_n-embedding} $W\notin M$
 and $\pi$ is an $n$-embedding.
 \end{sclmfour}
\begin{proof}
It suffices to see that $W\notin M$,
because then $\pi$ is unbounded in $\rho_n^M$,
and hence an $n$-embedding. (Here if $\crit(\pi)<\eta$,
we use the fact that $E\in M$ where $E$ is the short part of $\pi$.) So suppose $W\in M$.
Since $\rho_{n+1}(M^\Tt_\infty)=\eta=\rho_{n+1}^M$,
every $\bfrSigma_{n+1}^M$ subset of $\eta$
is $\bfrSigma_{n+1}(M^\Tt_\infty)$,
and since $M^\Tt_\infty=M^\Uu_\infty$,
therefore it is $\bfrSigma_{n+1}^W$.
But since $\rho_{n+1}^M=\eta$,
it follows that $W\notin M$.
\end{proof}
 
 Suppose $\eta\leq\crit(i^\Uu)$.
 Then $M^\Uu_\infty$ is $(n+1)$-solid. Therefore $\crit(i^\Tt)\geq\eta$,
 $M$ is $(n+1)$-solid and $i^\Uu(p_{n+1}^W)=p_{n+1}(M^\Uu_\infty)=p_{n+1}(M^\Tt_\infty)=i^\Tt(p_{n+1}^M)$.  This gives the claim in this case.
 
 Now suppose $\crit(i^\Uu)<\eta$.
 Then by the rules of $\ph$,  $\eta=\gamma^{+M}=\gamma^{+W}$ and $\crit(\pi)=\gamma$ and the short part $E$ of $\pi$ is in $\es^M$. Since $W$ is $(n+1)$-sound and  $\pi$ is an $n$-embedding, letting $z=\pi(p_{n+1}^W)$ and $\zeta=\sup\pi``\rho_{n+1}^W$ (and $\rho_{n+1}^W=\eta$), then $(z_{n+1}^M,\zeta_{n+1}^M)\geq(z,\zeta)$. But because $E\in M$ but $W\notin M$, note then that $(z_{n+1}^M,\zeta_{n+1}^M)=(z,\zeta)$
 (see \S\ref{subsec:notation}).
 Since $\rho_{n+1}^M=\eta<\zeta$,
 therefore $M$ is not $(n+1)$-solid,
 so $M$ is stretched-$(n+1)$-solid
 with $p_{n+1}^M=z\cup\{\zeta\}$.
 Therefore $W=\core_{n+1}(M)$ and $\pi$ is the core map, as desired.
\end{proof}

 \begin{clmfour}\label{clm:rho_n+1^M<eta}
Suppose $n\geq 0$ and $\rho_{n+1}^M<\eta$.
Then clause \ref{item:clause_4_cond}
 of the conclusion of condensation holds.
 \end{clmfour}
\begin{proof}
Like in the proof of (stretched-)solidity,
$M|\eta$ is active and $E^\Tt_0=F^{M|\eta}$
and $1\leq^\Tt\infty$,
and letting $\theta=\crit(E^\Tt_0)$, either $E^\Tt_0$ is short
and $\rho_{n+1}^M=\theta^{+M}$,
or $E^\Tt_0$ is long with no largest generator and $\rho_{n+1}^M=\theta^{++M}$. Moreover,
if $1<^\Tt\infty$ then $\eta\leq\spc(E^\Tt_\beta)$ where $\beta+1=\min((1,\infty]^\Tt)$.
And since $\eta$ is not an $M$-cardinal,
we have $\eta=\crit(\pi)$.

\begin{sclmfour}
 $W\notin M$ and $\pi$ is an $n$-embedding.
\end{sclmfour}
\begin{proof}
Like before, it is enough to see $W\notin M$.
Let $t=\Th_{\rSigma_{n+1}}^M(\rho_{n+1}^M\cup\{\pvec_{n+1}^M\})$.
So $t\notin M$.
But $t$ reduces to $t'=\Th_{\rSigma_{n+1}}^{M^\Tt_1}(\eta\cup\{\pvec_{n+1}^{M^\Tt_1}\})$, via $E^\Tt_0$.
So it suffices to see that $t'$ is $\bfrSigma_{n+1}^W$. But this is like before.
\end{proof}

Now arguing as before, it follows that
$z_{n+1}^M=\pi(p_{n+1}^W)$
and
$\zeta_{n+1}^M=\eta>\rho_{n+1}^M$,
so $M$ is stretched-$(n+1)$-solid
with $\eta=\min(p_{n+1}^M)$.
So the desired conclusion follows.
\end{proof}

 \begin{clmfour}
 Suppose $n=0^-$ and $\rho_{\D}^M=\eta>\kappa^{+M}$, where $\kappa=\crit(F^M)=\crit(F^W)$. Then
 $W$ is $1$-sound,
 $W=\core_1(M)$
 and $\pi$ is the core embedding.
 That is, clause \ref{item:clause_4'}
of the conclusion of condensation holds.
 \end{clmfour}
 \begin{proof}
 Recall that $\eta=\max(\kappa^{+W},\rho_{\D}^W)$.
 So since $\eta=\rho_{\D}^M>\kappa^{+M}=\kappa^{+W}$,
 in fact $\eta=\rho_{\D}^W$,
 and also 
 $\eta$ is an $M$-cardinal.
 By Lemmas \ref{lem:rho_D,rho_1_corresp} and \ref{lem:char_t^M},
 $\rho_1^W=\rho_1^M=\eta$
 and $W$ is $1$-sound.
 
 \begin{sclmfour} $W\notin M$
 and $\pi$ is a $0$-embedding.
 \end{sclmfour}
 \begin{proof}
 We get $W\notin M$ by
 an argument analogous
to that for Subclaim \ref{sclm:W_notin_M_and_pi_n-embedding}.
Also,
 Since $\kappa^{+W}=\kappa^{+M}<\eta$,
 we have $\pi(\kappa)=\kappa$,
 so $\pi$ is $0$-lifting (not just $0$-deriving), and in particular,
 $\pi(F_J^W)=F_J^M$. It follows
 that $\pi$ is a $0$-embedding.
 \end{proof}
 From here on we can argue as before to see that $W=\core_1(M)$ and $\pi$ is the core embedding.
\end{proof}

 \begin{clmfour}
Suppose $n=0^-$ and $\eta=\rho_{\D}^W>\kappa^{+M}$
but $\rho_{\D}^M<\eta$.
Then clause \ref{item:clause_5'}
 of the conclusion of condensation holds.
 \end{clmfour}
\begin{proof}
 Again we have $\rho_1^W=\eta$ and $W$ is $1$-sound. But then the proof is like that of Claim \ref{clm:rho_n+1^M<eta}.
\end{proof}

 \begin{clmfour}
Suppose $n=0^-$ and $\rho_{\D}^W=0$, so $\eta=\kappa^{+W}=\kappa^{+M}$
and
 $\rho_{\D}^M=0$.
Then clause \ref{item:clause_3'}
 of the conclusion of condensation holds.
 \end{clmfour}

 \begin{proof}
 By \S\ref{sec:long_mice},
 $W=\core_{\D}(M^\Uu_\infty)$ and $i^\Uu$ is the Dodd-absent-core embedding.
 Also by Theorem \ref{tm:rho^M_D=0_implies_Dodd-absent-solid_and_universal},
 $M$ is Dodd-absent-solid
 and Dodd-absent-universal,
 and it easily follows
 that $p_{\D}^{M^\Tt_\infty}=i^\Tt(p_{\D}^M)$.
 But then $\core_{\D}(M)=\core_{\D}(M^\Tt_\infty)=W$,
 and $\pi:W\to M$ is the Dodd-absent-core embedding.
\end{proof}

This completes all cases,
finishing the proof of condensation.
\end{proof}

\subsection{The initial segment condition}

The key result of this section is  Theorem \ref{tm:MS-ISC}, which is essentially the translation of the usual
Mitchell-Steel ISC  for short extenders to our context, for Dodd-absent-sound $1$-sound mice
(there are easy counterexamples if one drops these assumptions).
Before this, however, we prove a simple and coarse kind of ISC, also for short extenders, in Theorem \ref{tm:F^M_rest_nu_in_M_or_type_Z}.
\begin{rem}\label{rem:M_X}
 Let $M$ be an active short premouse and $X\sub M$. Let $\kappa=\crit(F^M)$ and
     \[H'_X=\{j(f)(x)\bigm|f\in M|\kappa^{+M}\text{ and }x\in X^{<\om}\}\]
    and let $H_X=(H'_X,\es^M\cap H'_X,F^M\rest H'_X)$. 
    Let $M_X$ be the transitive collapse
    of $H_X$. Let $\pi_X:M_X\to M$ be the uncollapse map. Then $\pi$ is $\Sigma_1$-elementary in the language $\{{\in},\dot{\es},\dot{F}\}$,
    but note that this language excludes the constant symbol $\dot{F}_J$ (recall that this is interpreted as the largest whole initial segment in a premouse, when it exists). So if $F_J^M\in\rg(\pi_X)$
    then $\pi_X$ is $\rSigma_1$-elementary (which includes the constant symbol $\dot{F}_J$, interpreted in $M_X$ as $\pi_X^{-1}(F_J^M)$). \end{rem}
    
    The first fact is a kind of ``coarse'' ISC:
    \begin{tm}\label{tm:F^M_rest_nu_in_M_or_type_Z} Let $M$ be an active short premouse,
    $\kappa=\crit(F^M)$
    and $\kappa<\nu<\nu(F^M)$.
  Then either:
  \begin{enumerate}[label=--]
   \item $F^M\rest\nu\in M$, or
   \item $F_J^M\neq\emptyset$
   and $\nu=\lambda(F_J^M)+1$ 
   and $F^M\rest\nu$ is type Z
   (note $F^M\rest\lambda(F^M_J)\in M$).
  \end{enumerate}
  Therefore letting $U=\Ult_0(M,F^M\rest\nu)$ and $\mu=\card^M(\nu)$,
  we have $\mu^{+U}<\mu^{+M}$.
    \end{tm}

    \begin{proof}
  Applying Remark \ref{rem:M_X} with $X=\nu$,
    we have $(M_\nu)^{\passive}=U|\lambda^{+U}$ where $\lambda=\lambda(F^M\rest\nu)$
    and $U=\Ult_0(M,F^M\rest\nu)$. Suppose for the moment that $F_J^M\in H_\nu$,
    so $\pi_\nu:M_\nu\to M$ is $\rSigma_1$-elementary.
    Since $\nu<\nu(F^M)$,
    we have $H_\nu=\rg(\pi_\nu)\neq M$.
 Therefore
 $\rho_1^M\cup p_1^M\not\sub\nu$.
    If $\nu<\rho_1^M$ then $F^M\rest\nu\in M$, so suppose $\rho_1^M\leq\nu$.
    So $p_1^M\not\sub\nu$, so letting $\gamma=\max(p_1^M)$, we have $\gamma\geq\nu$. So by $1$-solidity, $\Th_{\rSigma_1}^M(\gamma)\in M$, but then $F^M\rest\nu\in M$, as desired.
    
    Now suppose $F_J^M\notin H_\nu$.
    Then $F_J^M\neq\emptyset$.
    If $\lambda(F_J^M)<\mu=\card^M(\nu)$ then $F_J^M\in M|\mu$, so $F_J^M\in H_\nu$, a contradiction. And if $\nu\leq\lambda(F_J^M)$ then $F^M\rest\nu=F^J_M\rest\nu\in M$.
    So suppose that $\mu=\lambda(F_J^M)<\nu$. Then by induction, we have $F_J^M\in\es^M$,
    and note that $\mu$ is  a generator of $F^M$.
    We have $\lh(F_J^M)=\mu^{+\Ult(M,F_J^M)}$. Note that either $\lh(F_J^M)<\mu^{+\Ult(M,F^M\rest(\mu+1))}$,
    in which case $F_J^M\in \Ult(M,F^M\rest\nu)$, which suffices,
    or $\lh(F_J^M)=\mu^{+\Ult(M,F^M\rest(\mu+1))}$,
    in which case $\lh(F_J^M)$
    is the next generator of $F^M$.
    So if $\nu>\mu+1$ then we have $F^M_J\in H_\nu$, which suffices.
    
    So we may assume $\nu=\mu+1$,
    $F_J^M=F^M\rest\mu$ 
    and $\lh(F_J^M)=\mu^{+\Ult(M,F^M\rest(\mu+1))}$,
    which means that $F^M\rest\nu=F^M\rest(\mu+1)$ is type Z,
    as desired.
    
    Finally, if $F^M\rest\nu\in M$,
    then 
    clearly $\mu^{+U}<\mu^{+M}$,
    and in the other case, $\mu^{+U}=\lh(F^M_J)<\mu^{+M}$.
\end{proof}
We want to establish a variant of Lemma \ref{lem:existence_of_1-self-solid-parameter}:
 \begin{lem}\label{lem:existence_of_1-self-solid-Dodd-absent-solid-parameter}
 Let $M$ be a $1$-sound Dodd-absent-sound active short premouse
 with $\kappa=\crit(F^M)<\rho_1^M$.
 Let $x\in M$.
 Then there is $q\in[\OR^M]^{<\om}$
 such that $(M,\kappa^{+M},q)$ is $1$-self-solid, $x\in\Hull_1^M(\kappa^{+M}\cup\{q\})$, $p_1^M=q\cut\rho_1^M$
 and $\cHull_1^M(\kappa^{+M}\cup\{q\})$
 is Dodd-absent-solid.
\end{lem}
\begin{proof}
We may assume $\rho_1^M>\kappa^{+M}$.
 Follow the construction in the proof of Lemma 
 \ref{lem:existence_of_1-self-solid-parameter} with $\theta=\kappa^{+M}$,
 and at each step, and when adding a point $\alpha$ to $q$ with $\alpha<\rho_1^M$,
 choose $\alpha$ large enough that $\{F_J^M\}$ is generated by $(q\cut(\alpha+1))\cup\alpha$.
 Letting $C=\cHull_1^M(\kappa^{+M}\cup\{q\})$ and $\pi:C\to M$ the uncollapse,
 and letting $\pi(\bar{\rho})=\rho_1^M$,
 this ensures that $\bar{\rho}\cap p_{\D}^C=\bar{\rho}\cap q=\bar{\rho}\cap p_1^C$, and so by $1$-solidity,
 gives Dodd-absent-solidity below $\bar{\rho}$ (cf.~Lemma \ref{lem:char_t^M}).
\end{proof}

The \emph{Mitchell-Steel ISC}, introduced below,
is essentially the usual ISC from \cite{fsit}, combined with how type Z extenders are shown to behave in \cite{deconstructing}.
\begin{dfn}\label{dfn:MS-ISC}
 Let $M$ be an active short premouse.
 Let $\kappa=\crit(F^M)$.
 We say that $M$ satisfies the
 \emph{Mitchell-Steel ISC} (\emph{MS-ISC})
 iff for all $\nu\in[\kappa^{+M},\nu(F^M))$
 with $\nu=\nu(F^M\rest\nu)$ (that is, $\nu$ is the strict sup of generators of $F^M\rest\nu$), then $F^M\rest\nu\in M$,
 and in fact, letting
  $G$ be the trivial completion of $F^M\rest\nu$, we have $G\in M$ and
  \begin{enumerate}[label=--]
  \item if $F^M\rest\nu$ is non-type-Z
and  $M|\nu$ is passive then $G\in\es^M$,
  \item if $F^M\rest\nu$ is non-type-Z and $M|\nu$ is active then $G\in\es^{\Ult(M|\nu,F^{M|\nu})}$,
  \item if $F^M\rest\nu$ is type Z
  then letting $\nu=\xi+1$
  and $U=\Ult(M,F^M\rest\xi)$,
  there is a $U$-total type 1 short extender $H\in\es^{U}$
  such that $\crit(H)=\xi$ and $F^M\rest(\xi+1)=(H\com(F^M\rest\xi))\rest(\xi+1)$.\qedhere
 \end{enumerate}
\end{dfn}

We cannot expect the MS-ISC to hold for arbitrary short mice $M$, but we will show that if 
$M$ is also Dodd-absent-sound and $1$-sound (and sufficiently iterable), then the MS-ISC holds.

\begin{tm}\label{tm:MS-ISC}
 Let $M$ be a $(0,\om_1,\om_1+1)^*$-iterable
 Dodd-absent-sound $1$-sound
 active short premouse.
 Then $M$ satisfies the MS-ISC.
\end{tm}
\begin{proof}
We may assume $M$ is countable.

Let $\kappa=\crit(F^M)$ and
  suppose $\nu(F^M)>\kappa^{+M}$.
  Let $\nu\in[\kappa^{+M},\nu(F^M))$
  be such that $\nu(F^M\rest\nu)=\nu$.
  We will show by induction on $\nu$ that the MS-ISC holds at $\nu$.

We will first reduce to the case that $\rho_1^M\leq\kappa^{+M}$, where $\kappa=\crit(F^M)$.
For suppose $\kappa^{+M}<\rho_1^M$.
We have $F=F^M\rest\nu\in M$, by Dodd-absent-soundness.
Let $\gamma$ be the least generator of $F^M$ such that $\gamma\geq\nu$. 
Let $x=\{F,\nu,\gamma,F_J^M\}$.
Using Lemma \ref{lem:existence_of_1-self-solid-Dodd-absent-solid-parameter},
let $q$ be such that $(M,\kappa^{+M},q)$ is $1$-self-solid with $x\in\rg(\pi)$
and $C$ is Dodd-absent-sound,
 where $C=\cHull_1^M(\kappa^{+M}\cup\{q\})$
and $\pi:C\to M$ is the uncollapse.
So $C$ is Dodd-sound and $1$-sound
and $p^C_{\mathrm{D}}=\pi^{-1}(q)$,
and $x\in\rg(\pi)$. We have replicated the hypotheses of the theorem with $C$ replacing $M$, and $\rho_1^C\leq\kappa^{+C}$ where $\kappa=\crit(F^C)$.
And it is enough to see that $C$ satisfies the MS-ISC,
because then note that if $X\in C$ witnesses this at $\pi^{-1}(\nu)$,
then $\pi$ lifts the witness of MS-ISC back up to 
    a witness at $\nu$. (Note here that because $F^M\rest\nu\in\rg(\pi)$,
    $\pi^{-1}(\nu)$ is a strict sup of generators of $F^C$, and also because $\gamma\in\rg(\pi)$
    and $\gamma\geq\nu$,
    $\pi^{-1}(\gamma)$ is a generator of $F^C$ and $\pi^{-1}(\gamma)\geq\pi^{-1}(\nu)$.)
 
  So from now on, we may assume that $\rho_{\mathrm{D}}^M=0$.
  So we can fix a $(0,\om_1+1)$-strategy $\Sigma_M$ for $M$ with weak Dodd-Jensen.

  Let $U=\Ult_0(M,F^M\rest\nu)$
  and let $j:M\to U$ be the ultrapower map.
 Let $\mu=\card^M(\nu)$.
  Let
  \[ \pi:U\to\Ult_0(M,F^M) \]
  be the factor map. So $\pi$ is $\rSigma_1$-elementary and $\nu\leq\crit(\pi)$.
 
 \begin{clmfive}\label{clm:U_M_agmt} We have $\mu^{++U}<\OR^U$ and $\mu^{++U}<\mu^{+M}$ and either:
  \begin{enumerate}
   \item $\mu=\nu=\crit(\pi)$ and $U|\mu^{+U}=M||\mu^{+U}$, or
   \item $\mu\leq\nu\leq\crit(\pi)=\mu^{+U}$
   and $\pi(\mu^{+U})=\mu^{+M}$
   and letting $\xi=\mu^{+U}$, either:
   \begin{enumerate}
   \item $M|\xi$ is passive and $U|\mu^{++U}=M||\mu^{++U}$, or
   \item $M|\xi$ is active and $U|\mu^{++U}=P||\mu^{++U}$ where $P=\Ult(M|\xi,F^{M|\xi})$.
   \end{enumerate}
  \end{enumerate}
  \end{clmfive}
  \begin{proof}
   Because $F=F^M\rest\nu\in M$
   and by (internal) condensation. (The fact that $\mu^{++U}<\OR^U$
  is because $i^M_{F^M\rest\nu}(\kappa)\geq\mu$,
  and $M$ has many cardinals $>\kappa$. The fact that $\mu^{++`U}<\mu^{+M}$ follows from the rest.)
  \end{proof}

  If $\mu=\nu$ is an $M$-cardinal, let $\eta=\nu$,
  and otherwise let $\eta=\nu^{+U}$.
  We define a phalanx $\ph$ according to the following cases:
  \begin{enumerate}[label=Ph\arabic*.,ref=Ph\arabic*]
  \item\label{item:ISC_case_1}
  If $\eta=\mu=\nu$ is an $M$-cardinal then define
  $\ph=((M,<\eta),U)$.
  \item\label{item:ISC_case_2} If $\mu<\nu\leq\eta=\mu^{+U}$ and $M|\eta$ is active, then define
\[ \ph=((M,{<\mu}),((\Ult_0(M,F^{M|\eta}),n'),\mu),U) \]
where $n'=0$ if $\Ult_0(M,F^{M|\eta})$ is Dodd-absent-sound, and $n'=0^-$ otherwise. (Recall that $M$ is Dodd-absent-sound,
so $n'=0$ iff $\spc(F^{M|\eta})\leq\spc(F^M)=\crit(F^M)$.)
    \item\label{item:ISC_case_3} If $\mu<\nu\leq\eta=\mu^{+U}$
and $M|\eta$ is passive, define
  \[ \ph=((M,{<\mu}),((R,r),\mu),U)\]
  where $(R,r)\pins(M,0)$ is least such that $\mu^{+U}\leq\OR^R$ and $\rho_{r+1}^R=\mu$.

 \end{enumerate}

 All models of $\ph$ are at degree $0$, except that $R$ is at degree $r$, and $\Ult_0(M,F^{M|\eta})$
 at $n'$. Note that those at degree $d\geq 0$ are Dodd-absent-sound.
 Note that case \ref{item:ISC_case_1}
 holds if $\nu=\kappa^{+M}$ where $\kappa=\crit(F^M)$; that is, when $F^M\rest\nu$ is the normal measure derived from $F^M$. Generators of $F^M\rest\nu$ can only be moved (along a non-dropping branch above $U$)
 in cases \ref{item:ISC_case_2}
 and \ref{item:ISC_case_3}, and only when the first extender used along the branch is long with critical point $\mu$. We will only consider $0$-maximal trees $\Uu$ on $\ph$ with $\eta^{+U}<\lh(E^\Uu_0)$.
 
 Like in
earlier proofs, we start with some basic facts on $\ph$, its iterability, and closeness.
Like
Claim \ref{clm:pow(eta)^U_sub_M}
in the proof of Dodd-absent-solidity, we have:
  \begin{clmfive}
   We have:
   \begin{enumerate}
    \item In  case \ref{item:ISC_case_1} of the definition of $\ph$, we have
    $U|\eta^{+U}=M||\eta^{+U}$,
    so $\pow(\eta)^U\sub M$.
    \item In case \ref{item:ISC_case_2},
   $U|\eta^{+U}=U|\mu^{++U}=\Ult_0(M,F^{M|\eta})||\eta^{+U}$,
   so
   $\pow(\mu)^U\sub M$ and in fact $\pow(\eta)^U\sub\Ult_0(M,F^{M|\eta})$,
    \item In case \ref{item:ISC_case_3},
    $U|\eta^{+U}=U|\mu^{++U}=R||\eta^{+U}$, so $\pow(\mu)^U\sub M$ and in fact $\pow(\eta)^U\sub R$.
   \end{enumerate}
  \end{clmfive}

\begin{clmfive}
$\ph$ is $(\om_1+1)$-iterable.
In fact, there is an $(\om_1+1)$-iteration strategy $\Gamma$ for $\ph$ such that trees $\Uu$ on $\ph$ via $\Gamma$ lift to trees $\Uu^*$ on $M$ via $\Sigma$
in a manner similar to that in the proof of Theorem \ref{tm:rho^M_D=0_implies_Dodd-absent-solid_and_universal}.
\end{clmfive}

The proof of Claim \ref{clm:Dodd-solidity_closeness} of the proof of Theorem \ref{tm:rho^M_D=0_implies_Dodd-absent-solid_and_universal} gives:
\begin{clmfive} Let $\Uu$ be a $0$-maximal tree on $\ph$, and let $\alpha+1<\lh(\Uu)$. Then $E^\Uu_\alpha$ is close to $M^{*\Uu}_{\alpha+1}$.
  \end{clmfive}
  
  We compare $(\ph,M)$ using $(\Gamma,\Sigma)$. As usual:
  
 \begin{clmfive}\label{clm:compatibility}The comparison terminates with $b^\Uu$ above $U$, $b^\Uu,b^\Tt$  non-dropping in model, degree or Dodd-degree, and $i^\Uu\com j=i^\Tt$.\end{clmfive}

 So $\deg^\Uu_\infty=0=\deg^\Tt_\infty$, and $M^\Uu_\infty=M^\Tt_\infty$ is Dodd-absent-sound.

  \begin{clmfive} If $F^M\rest\nu$ is whole
   then $F^M\rest\nu\in\es^M$.
  \end{clmfive}
\begin{proof}Suppose $F^M\rest\nu$ is whole.
   It follows that $\nu$ is an $M$-cardinal, so $\ph=((M,{<\nu}),U)$. Since  $U|\nu=M|\nu$, we have $\lh(E^\Tt_0)>\nu$, and since $\nu$ is an $M$-cardinal, in fact $\nu\leq\lambda(E^\Tt_0)$. 
   So letting $\alpha+1=\min(b^\Tt)$,
   by Claim \ref{clm:compatibility}, $F^M\rest\nu=E^\Tt_\alpha\rest\nu$.
   If $\nu=\lambda(E^\Tt_\alpha)$ we are done, so suppose $\nu<\lambda(E^\Tt_\alpha)$. Then since $F^M\rest\nu=E^\Tt_\alpha\rest\nu$ is a whole
   segment of $E^\Tt_\alpha$, by the Jensen ISC, $F^M\rest\nu\in M^\Tt_{\alpha+1}$, and therefore, as $b^\Tt$ does not drop, $F^M\rest\nu\in M^\Tt_\infty$. But then $F^M\rest\nu\in U$, a contradiction, completing the proof in this case.
  \end{proof}

    So from now on, we may assume that \[ F^M\rest\nu\text{ is not whole}, \]
    so $j(\kappa)>\nu$
 where $\kappa=\crit(F^M)$.
    We may in fat assume that
    \[ F^M\text{ has no whole proper segment of length }\geq\nu.\]
    (Otherwise, let $F^M\rest\bar{\lambda}$ be such, and replace $M$ with the corresponding structure, noting that this is also iterable.) So if $F^M$ has a whole proper segment then it has a largest (which is $F^M_J$), and $\lambda(F^M_J)\leq\mu$ and $\lambda(F^M_J)<\nu$.

   \begin{clmfive}Either:
    \begin{enumerate}[label=(\alph*)]
     \item $F^M$ has no whole proper segment (so $F_J^M=\emptyset$), or
     \item $F^M_J=F^M\rest\bar{\lambda}$ 
     and $\lh(F^M_J)<\nu$, so $F^M_J\in\es^{U||\nu}$, or
     \item\label{item:type_Z} $\nu=\mu+1$
     where $F^M_J=F^M\rest\mu$
     and $\lh(F^M_J)=\mu^{+\Ult(M,F^M\rest\mu)}=\mu^{+U}$, and $F^M\rest\mu\notin U$.
     \end{enumerate}
     \end{clmfive}
     
 Note that in case \ref{item:type_Z} above,
 $F^M\rest\nu$ is type Z.

 \begin{proof}
   Suppose  $F^M_J\neq\emptyset$.
    If
     $\lh(F^M_J)<\nu$ then $F^M_J\in\es^U$. If instead $\nu\leq\lh(F^M_J)$
     then $\lambda(F^M_J)<\nu\leq\lh(F^M_J)$, and note that then  $\nu=\lambda(F^M_J)+1$,
     since $F^M$ has no generators in the interval $(\lambda(F^M_J),\lh(F^M_J))$. So  $F^M_J\notin U$ iff $\lambda(F^M_J)+1=\nu$ and $F^M\rest\nu$ is type Z.
 \end{proof}

 As usual, we now need to analyse the results of the comparison more carefully.
 Let $\lambda_\infty=i^\Uu(j(\kappa))=i^\Tt(\kappa)$ (and recall we are now in the case that $j(\kappa)>\nu$).
 Let $\alpha_0+1=\min((0,\infty]^\Tt)$
 and
  $\beta_0+1=\min((0,\infty]^\Uu)$.
 The first claim below follows directly from the setup of the exchange ordinals for $\ph$:
 \begin{clmfive}\label{clm:if_gens_moved}
  If $\crit(i^\Uu_{0\infty})<\nu$
  then $E^\Uu_{\beta_0}$ is long,
  and (for some $\mu'$)
  $\eta=(\mu')^{+U}$ and $\crit(E^\Uu_{\beta_0})=\mu'$.
 \end{clmfive}

 The proof of Claim \ref{clm:chain_of_exts_in_branches_DS} of the proof of Theorem \ref{tm:rho^M_D=0_implies_Dodd-absent-solid_and_universal} gives:
\begin{clmfive}\label{clm:chain_of_exts_in_branches}\
\begin{enumerate}
\item For every $\gamma+1\in b^\Uu$,
we have $\lambda(E^\Uu_\gamma)<\lambda_\infty$;
equivalently,
  $\crit(E^\Uu_\gamma)<i^\Uu_{0\delta}(j(\kappa))$ where $\delta=\pred^\Uu(\gamma+1)$.
  
 \item  For every $\gamma+1\in b^\Tt$,
 $E^\Tt_\gamma$ is short and
 $\crit(E^\Tt_\gamma)=i^\Tt_{0\delta}(\kappa)$
 where $\delta=\pred^\Tt(\gamma+1)$.
 Moreover,
  $\lambda(E^\Tt_\gamma)\leq\lambda_\infty$.
 \end{enumerate}
\end{clmfive}

\begin{clmfive}\label{clm:if_no_gens_moved_then_two_options}
 Suppose $\nu\leq\crit(i^\Uu)$.
 Then either:
 \begin{enumerate}
  \item we have:
  \begin{enumerate}
  \item $\alpha_0+1=^\Tt\infty$, so $E^\Tt_{\alpha_0}$ is the only extender used along $b^\Tt$, so $M^\Tt_{\infty}=M^\Tt_{\alpha_0+1}$,
  \item $\nu\leq\nu(E^\Tt_{\alpha_0})\leq\lambda(E^\Tt_{\alpha_0})$,
  \item $\nu<\lambda(E^\Tt_{\alpha_0})$,
   \item $E^\Tt_{\alpha_0}$ is short,
   and \item $E^\Tt_{\alpha_0}\rest\nu=F^M\rest\nu$,
  \end{enumerate}
  or
  \item there is $\alpha_1$ such that:
  \begin{enumerate}
   \item 
 $\alpha_0+1<^\Tt\alpha_1+1=^\Tt\infty$
  and $\pred^\Tt(\alpha_1+1)=\alpha_0+1$,
  so 
  $E^\Tt_{\alpha_0},E^\Tt_{\alpha_1}$ are the
  only two extenders used along $b^\Tt$,
  so $M^\Tt_\infty=M^\Tt_{\alpha_1+1}$,
  \item $\lambda(E^\Tt_{\alpha_0})=\mu<\mu+1=\nu$,
 
  \item $\crit(E^\Tt_{\alpha_1})=\mu$,
  \item $E^\Tt_{\alpha_0},E^\Tt_{\alpha_1}$ are both short, and \item $F^M\rest\nu=(E^\Tt_{\alpha_1}\com E^\Tt_{\alpha_0})\rest\nu$ is type Z.
  \end{enumerate}
 \end{enumerate}
\end{clmfive}
\begin{proof}
 Let $\alpha=\alpha_0$. We have $\mu\leq\lambda(E^\Tt_{\alpha})$
 since $\mu$ is an $M$-cardinal
 and $\mu<\lh(E^\Tt_0)$. So  $F^M\rest\mu=E^\Tt_{\alpha}\rest\mu$.
 
 Suppose $\lambda(E^\Tt_\alpha)>\mu$.
 Then clearly $\nu<\lambda(E^\Tt_\alpha)$
 and $F^M\rest\nu=E^\Tt_\alpha\rest\nu$, 
 and
$E^\Tt_\alpha$ is short
 by Claim \ref{clm:chain_of_exts_in_branches}.
 Let us verify that $E^\Tt_\alpha$ is the only extender used along $b^\Tt$.
 Otherwise, letting $E$ be the short extender derived from $i^\Tt$,
 $E^\Tt_\alpha$ is a whole proper segment of $E$.
 But (the trivial completion of) $F^M\rest\nu$ has no whole proper segment 
 of form $F^M\rest\bar{\lambda}$ with $\bar{\lambda}\geq\nu$, and since $\crit(i^\Uu)\geq\nu$,
 it follows that $E$, which is also the short extender derived from $i^\Uu\com j$, has no whole proper segment of length $\geq\nu$. Since $\nu<\lambda(E^\Tt_\alpha)$, this is a contradiction.

 So suppose $\lambda(E^\Tt_\alpha)=\mu$. 
 Since $F^M\rest\nu$ is not whole, we must then have $\mu<\nu$. And $j(\kappa)>\nu$, but $i^\Tt_{0,\alpha+1}(\kappa)=\mu$, so $M^\Tt_{\alpha+1}\neq M^\Tt_\infty$. So let $\alpha'+1=\min((\alpha+1,\infty]^\Tt)$;
 then
  $\mu=\crit(E^\Tt_{\alpha'})$.
 Clearly $(E^\Tt_{\alpha'}\com E^\Tt_\alpha)\rest\nu=F^M\rest\nu$.
Again $E^\Tt_{\alpha},E^\Tt_{\alpha'}$
are short by Claim \ref{clm:chain_of_exts_in_branches}.
Arguing much as before, these are the only two extenders used along $b^\Tt$. But $E^\Tt_{\alpha'}$
has no generators in the interval $(\mu,\mu^{+M^\Tt_{\alpha+1}})=(\mu,\eta)$, so neither does $F^M\rest\nu$, so $\nu=\mu+1$.
It also follows that $(E^\Tt_{\alpha_1}\com E^\Tt_{\alpha_0})\rest\nu$ is type Z.
\end{proof}

\begin{clmfive}\label{clm:full_analysis_ISC_nu_leq_crit_and_two_exts_in_Tt}
 Suppose $\nu\leq\crit(i^\Uu)$
 and that $b^\Tt$ uses two distinct extenders $E^\Tt_{\alpha_0},E^\Tt_{\alpha_1}$.
 Then $\alpha_0=0$ and $\alpha_1=1$,
 $E^\Tt_0\in\es^M$ and $E^\Tt_1\in\es^{M^\Tt_1}$,
 $E^\Tt_0$ is a superstrong extender in $M$
 with
 $\lambda(E^\Tt_0)=\nu(E^\Tt_0)=\mu$,
 and $E^\Tt_1$ is a normal measure in $M^\Tt_1$ with $\crit(E^\Tt_1)=\mu$.
 Moreover, $\Uu$ is trivial, so $U=\Ult(\Ult(M,E^\Tt_{0}),E^\Tt_{1})$
 and $F^M\rest\nu$ is equivalent to $E^\Tt_1\com E^\Tt_0$.
\end{clmfive}
\begin{proof}
By Claim \ref{clm:if_no_gens_moved_then_two_options},
$\mu^{+U}=\lh(E^\Tt_{\alpha_0})$,
and therefore $\alpha_0=0$.
 Let $D$ be the normal measure derived from $E^\Tt_{\alpha_1}$. 
 It suffices to see that $D\in\es_+(M^\Tt_{\alpha_1})$, since then $U=\Ult(M^\Tt_1,D)$, so $\alpha_1=1$ and $D=E^\Tt_1$. If $E^\Tt_{\alpha_1}\neq F(M^\Tt_{\alpha_1})$ or $(0,\alpha_1]^\Tt$ drops in model, then this follows by induction (on segments of $M$).
 And if $E^\Tt_{\alpha_1}=F(M^\Tt_{\alpha_1})$
 and $(0,\alpha_1]^\Tt$ is non-dropping
 then it follows by our induction on $\nu$,
 since $\mu>\kappa$ is a limit cardinal of $M$. 
\end{proof}

\begin{clmfive}\label{clm:main_claim_when_nu_leq_crit(i^U)}
 Suppose $\nu\leq\crit(i^\Uu)$
 and $b^\Tt$ uses only one extender $E^\Tt_{\alpha_0}$. Then $\Uu$ is trivial and either $[0,\alpha_0]^\Tt$ drops in model or $E^\Tt_{\alpha_0}\neq F(M^\Tt_{\alpha_0})$. Moreover:
 \begin{enumerate}
  \item\label{item:E^Tt_alpha_0_only_ext_isnt_active}
 If $E^\Tt_{\alpha_0}\neq F(M^\Tt_{\alpha_0})$ then
  $\nu(E^\Tt_{\alpha_0})=\nu$ and either
 \begin{enumerate}
  \item\label{item:alpha_0=0_etc} $\alpha_0=0$, so $F^M\rest\nu\in\es^M$, or
  \item\label{item:alpha_0=1_etc} $M|\nu$ is active, $E^\Tt_0=F^{M|\nu}$, $E^\Tt_1=F^M\rest\nu$ and $U=M^\Tt_2$.
 \end{enumerate}
 \item\label{item:active_ext} If $E^\Tt_{\alpha_0}=F(M^\Tt_{\alpha_0})$ and $(0,\alpha_0]^\Tt\cap\dropset^\Tt\neq\emptyset$ then $\nu=\mu+1$, $F^M\rest\nu$ is type Z,
 but $F^M\rest\mu$ is not whole,
 and $\alpha_0=1$,  $E^\Tt_0$ is a normal measure with $\crit(E^\Tt_0)=\mu$,
 $M^{*\Tt}_1$ is active short with $F(M^{*\Tt}_1)$ is the trivial completion of $F^M\rest\mu$, and $E^\Tt_1=F(M^\Tt_1)$ is equivalent to $E^\Tt_0\com(F^M\rest\mu)$.
\end{enumerate}
\end{clmfive}
\begin{proof}
Part \ref{item:E^Tt_alpha_0_only_ext_isnt_active}:
Suppose that $E^\Tt_{\alpha_0}\neq F(M^\Tt_{\alpha_0})$. 
 Suppose that $\nu<\nu(E^\Tt_{\alpha_0})$.
 Then by the MS-ISC for $\exit^\Tt_{\alpha_0}$
 (which holds by induction and since $\exit^\Tt_{\alpha_0}\pins M^\Tt_{\alpha_0}$), we would have $E^\Tt_{\alpha_0}\rest\nu\in\es^{\exit^\Tt_{\alpha_0}}$
 if $F^M\rest\nu$ is non-type Z, but this is impossible, as $F^M\rest\nu\notin U$. So $F^M\rest\nu$ is type Z, so $\nu=\mu+1$
 and $\mu^{+U}=\mu^{+\Ult(M,F^M\rest\mu)}$. But then also by the MS-ISC applied to $\exit^\Tt_{\alpha_0}$, we have $F^M\rest\mu\in\es^{M^\Tt_{\alpha_0}}$, contradicting the fact that $\mu^{+U}=\mu^{+\Ult(M,F^M\rest\mu)}$.
 
 So $\nu=\nu(E^\Tt_{\alpha_0})$.
 Therefore $M^\Tt_{\alpha_0+1}=U$ and $\Uu$ is trivial.
 If $\alpha_0=0$ then clause
 \ref{item:alpha_0=0_etc} holds.
 Suppose $\alpha_0> 0$.
 We have $\mu<\lh(E^\Tt_0)$
 and therefore $\nu\leq\mu^{+U}\leq\lh(E^\Tt_0)$.
 But then $\lh(E^\Tt_0)=\mu^{+U}=\nu$; for
 otherwise $\nu<\lh(E^\Tt_0)$
 and $\lh(E^\Tt_0)$ is a cardinal
 in  $M^\Tt_{\alpha_0}$,
 and $\lh(E^\Tt_0)<\lh(E^\Tt_{\alpha_0})$ and $E^\Tt_{\alpha_0}\in M^\Tt_{\alpha_0}$, 
 and $\rho_1(\exit^\Tt_{\alpha_0})\leq\nu(E^\Tt_{\alpha_0})$,
 so $\lh(E^\Tt_0)\leq\nu(E^\Tt_{\alpha_0})$,
 so $\nu<\nu(E^\Tt_{\alpha_0})$,
 contradiction.
 Similarly, 
 $\alpha_0=1$, establishing
 clause \ref{item:alpha_0=1_etc}.
 
 Part \ref{item:active_ext}:
 Suppose that $(0,\alpha_0]^\Tt\cap\dropset^\Tt\neq\emptyset$
 and $E^\Tt_{\alpha_0}=F(M^\Tt_{\alpha_0})$.
 Let $\gamma_0+1=\min((0,\alpha_0]^\Tt)$.
 We have $\mu<\lh(E^\Tt_{\gamma_0})$.
 
 We claim that $\gamma_0+1\in\dropset^\Tt$
 and $(\gamma_0+1,\alpha_0]^\Tt\cap\dropset^\Tt=\emptyset$.
 For suppose otherwise.
 Let $\gamma_1+1\leq^\Tt\alpha_0$ be least such that $(\gamma_1+1,\alpha_0]^\Tt\cap\dropset^\Tt=\emptyset$.
 Then $\lambda(E^\Tt_{\gamma_0})\leq\crit(E^\Tt_{\gamma_1})$
 and  $\lh(E^\Tt_{\gamma_0})\leq\rho_1(M^\Tt_{\alpha_0})\leq\nu$,
 since $F^M\rest\nu\notin M^\Tt_{\alpha_0}$.
 So $\lh(E^\Tt_{\gamma_0})\leq\nu$.
 Therefore $\gamma_0=0$
 and $\lambda(E^\Tt_{\gamma_0})=\mu$
 and 
 $\lh(E^\Tt_{0})=\mu^{+U}=\nu=\rho_1(M^\Tt_{\alpha_0})$
and  $\pred^\Tt(\gamma_1+1)=1$.
 So if $\nu<\crit(E^\Tt_{\gamma_1})$
 then $F^M\rest\nu=F(M^{*\Tt}_{\beta_1+1})\rest\nu$ and, since $M^{*\Tt}_{\gamma_1+1}\pins M^\Tt_1$, this leads to a contradiction
 like in the proof of part \ref{item:E^Tt_alpha_0_only_ext_isnt_active}.
 So $\crit(E^\Tt_{\gamma_1})=\mu$
 and since $\beta_1+1\in\dropset^\Tt$,
 therefore $E^\Tt_{\gamma_1}$ is long. 
 But now note that $F^M\rest\nu=F(M^\Tt_{\gamma_1+1})\rest\nu$.
 But by
 the ISC for $M^{*\Tt}_{\gamma_1+1}$,
 which holds in particular at $\mu$,
 we get that the ISC holds for $M^\Tt_{\gamma_1+1}$ at $i^{*\Tt}_{\gamma_1+1}(\mu)$, and hence at $\nu$, so $F^M\rest\nu\in\es(M^\Tt_{\gamma_1+1})$, which is impossible.

 Now $\crit(E^\Tt_{\gamma_0})<\nu$,
 because otherwise $F^M\rest\nu=F(M^{*\Tt}_{\gamma_0+1})\rest\nu$, and since $M^{*\Tt}_{\gamma_0+1}\pins M$, this leads to a contradiction like in the proof of part \ref{item:E^Tt_alpha_0_only_ext_isnt_active}. Since $\gamma_0+1\in\dropset^\Tt$,
 it follows that either:
 \begin{enumerate}[label=--]
  \item for some $\gamma$,
  we have $\eta=\mu=\nu=\gamma^{+M}=\gamma^{+U}<\gamma^{++U}<\gamma^{++M}$ and $\crit(E^\Tt_{\gamma_0})=\gamma$ and $E^\Tt_{\gamma_0}$ is long, or
  \item $\mu<\nu\leq\mu^{+U}<\mu^{+M}$ and $\crit(E^\Tt_{\gamma_0})=\mu$, or
  \item for some $\gamma$,
  we have $\gamma^{+U}=\gamma^{+M}=\mu<\nu\leq\mu^{+U}<\mu^{+M}$
  and $E^\Uu_{\gamma_0}$ is long with $\crit(E^\Uu_{\gamma_0})=\gamma$.
 \end{enumerate}
Note that it follows that $\nu<\lambda(E^\Tt_{\gamma_0})$.

 Now if $\kappa<\crit(E^\Tt_{\gamma_0})<\nu(F(M^{*\Tt}_{\gamma_0+1}))$
 then by the MS-ISC for $F(M^{*\Tt}_{\gamma_0+1})$ at $\crit(E^\Tt_{\gamma_0})$,
 we get that $M^\Tt_{\gamma_0+1}$ satisfies
 the MS-ISC at $\lambda(E^\Tt_{\gamma_0})$,
 and hence also at $\nu<\lambda(E^\Tt_{\gamma_0})$, which again leads to contradiction.
 
 Suppose instead that $\nu(F(M^{*\Tt}_{\gamma_0+1}))\leq\crit(E^\Tt_{\gamma_0})$.
 We have
 $M^{*\Tt}_{\gamma_0+1}\pins M$,
 and since $\rho_1(M^{*\Tt}_{\gamma_0+1})\leq\nu(F(M^{*\Tt}_{\gamma_0+1}))$,
 therefore $\nu(F(M^{*\Tt}_{\gamma_0+1}))=\crit(E^\Tt_{\gamma_0})=\mu<\nu$.
And $F^M\rest\nu(F(M^{*\Tt}_{\gamma_0+1}))$ is equivalent to $F(M^{*\Tt}_{\gamma_0+1})$. 
But $M^{*\Tt}_{\gamma_0+1}\npins M||\mu^{+U}$, 
since $M^{*\Tt}_{\gamma_0+1}\npins M^\Tt_\infty$. 
So $\mu^{+U}\leq\lh(F(M^{*\Tt}_{
\gamma_0+1}))$. It follows that $\nu=\mu+1$,
since otherwise the next generator of $F^M$ is $\mu^{+\Ult(M,F^M\rest\mu)}$,
and if $\mu^{+\Ult(M,F^M\rest\mu)}<\nu$ then note that $M^{*\Tt}_{\gamma_0+1}\pins U||\nu$  (it is the least segment of $M$ beyond $\mu^{+\Ult(M,F^M\rest\mu)}$ which projects to $\mu$), contradiction. Likewise, also $F^M\rest\nu$ is type Z.
Note that also $E^\Tt_{\gamma_0}\neq F(M^\Tt_{\gamma_0})$ and $E^\Tt_{\gamma_0}$ is a normal measure, $\gamma_0=0$,
$\exit^\Tt_{0}\pins M^{*\Tt}_{\gamma_0+1}=M^{*\Tt}_1$, and $E^\Tt_0\com F(M^{*\Tt}_{1})$ is equivalent to $F^M\rest\nu$, so $\Uu$ is trivial and $\Tt$
uses exactly two extenders, which are $E^\Tt_0$ and $E^\Tt_1=F(\Ult_0(M^{*\Tt}_1,E^\Tt_{0}))$.

 Finally suppose $\kappa=\crit(E^\Tt_{\gamma_0})=\gamma$.
 So
 $E^\Tt_{\gamma_0}$ is long. Since $\crit(F(M^\Tt_{\gamma_0+1}))=\kappa$,
 we have  $\rho_1(M^{*\Tt}_{\gamma_0+1})=\kappa^{+M}=\mu$  and $M^{*\Tt}_{\gamma_0+1}$ is active short (which we already knew, as $E^\Tt_{\alpha_0}=F(M^\Tt_{\alpha_0})$ is short)
 and $\deg^\Tt_{\gamma_0+1}=0^-$,
 and $M^\Tt_{\gamma_0+1}$ is formed to avoid the protomouse. So note $\mu=\kappa^{+M}\leq\nu\leq\kappa^{++U}$. We have $F^M\rest\nu=F(M^\Tt_{\gamma_0+1})\rest\nu$.
 But since $M^\Tt_{\gamma_0+1}$ was formed to avoid the protomouse $U^*$,
 $F(M^\Tt_{\gamma_0+1})=F^{U^*}\com E$
 where $E$ is the short part of $E^\Tt_{\gamma_0}$. We have $E\in\es^{M^\Tt_{\gamma_0+1}}$,
 and now $E\rest\nu=F^M\rest\nu$,
 which gives a contradiction again.
 
 This completes the proof of part \ref{item:active_ext}.
 To complete the proof of the overall claim, 
 suppose that $(0,\alpha_0]^\Tt\cap\dropset^\Tt=\emptyset$
 and $E^\Tt_{\alpha_0}=F(M^\Tt_{\alpha_0})$,
 and we will reach a contradiction.
 
\begin{sclmfive}$\alpha_0=0$.\end{sclmfive}
\begin{proof}
Suppose otherwise.
 We first show that $\nu\leq\crit(i^\Tt_{0\alpha_0})$.
 Suppose
 $\crit(i^\Tt_{0\alpha_0})<\nu$.
 We have
  $\kappa\leq\crit(i^\Tt_{0\alpha_0})$.
 If $\kappa<\crit(i^\Tt_{0\alpha_0})$
 then by induction, the MS-ISC holds for $M$
 at $\crit(i^\Tt_{0\alpha_0})$,
 and hence it holds for $M^\Tt_{\alpha_0}$ at $\lambda(E^\Tt_0)$. As before, it follows that $1\leq^\Tt\alpha_0$
 and $\lambda(E^\Tt_0)<\nu$, so $\lambda(E^\Tt_0)=\mu$. But then 
 since the MS-ISC holds for $M$
 at $\mu$, in fact it holds for $M^\Tt_{\alpha_0}$ at $i^\Tt_{0\alpha_0}(\mu)>\mu^{+U}$, which is a contradiction.
 So $\kappa=\crit(i^\Tt_{0\alpha_0})$,
 so letting $\gamma_0+1=\min((0,\alpha_0]^\Tt)$, then $\crit(E^\Tt_{\gamma_0})=\kappa$
 and $E^\Tt_{\gamma_0}$ is long
 and the ultrapower is formed to avoid the protomouse. Like before, this gives that $\gamma_0=0$ and $\lambda(E^\Tt_0)=\mu$.
But now we can argue analogously to before:
since $F^M\rest\mu\in\es^M$,
we have $E'=i^\Tt_{01}(F^M\rest\mu)\in\es^{U}$,
and although this extender has critical point $\mu$ (not $\kappa$), and hence is not a segment of $F^U$,
we anyway have $i_{E'}(F^M\rest\mu)$
is a segment of $F^U$ and by coherence,
$i_{E'}(F^M\rest\mu)\in\es^U$.
Since $U|\lh(E^\Tt_0)=M||\lh(E^\Tt_0)$,
the rest follows from the MS-ISC for this extender.

So $\nu\leq\crit(i^\Tt_{0\alpha_0})$.
So $\nu(E^\Tt_{\alpha_0})>\nu$,
so  $U\neq M^\Tt_\infty$,
so $i^\Uu_{0\infty}$ is non-identity,
with $\nu\leq\crit(i^\Uu_{0\infty})$ by hypothesis.
But now let $\theta=\min(\crit(i^\Uu),\crit(i^\Tt))$.
Then
\[U|\theta^{+U}=M^\Uu_{\infty}|\theta^{+M^\Uu_\infty}=M^\Tt_\infty|\theta^{+M^\Tt_\infty}=M^\Tt_{\alpha_0}|\theta^{+M^\Tt_{\alpha_0}}=M|\theta^{+M}.\]
But $\mu\leq\theta$, so this contradicts
the fact that $\mu^{+U}<\mu^{+M}$,
establishing the subclaim.
 \end{proof}
 
 So $M^\Uu_\infty=\Ult(M,F^M)\neq U$,
 so $\Uu$ is non-trivial. But letting $\theta=\crit(i^\Uu)\geq\nu$, we get that $U|\theta^{+U}=M|\theta^{+M}$, a contradiction like before.
\end{proof}

Claims \ref{clm:full_analysis_ISC_nu_leq_crit_and_two_exts_in_Tt} and \ref{clm:main_claim_when_nu_leq_crit(i^U)} together complete the proof of the theorem in case
$\nu\leq\crit(i^\Uu)$.
So from now on, we assume $\Uu$ is non-trivial
with $\crit(i^\Uu)<\nu$. We will derive a contradiction,
in Claims \ref{clm:if_F^M_rest_mu_in_U}--\ref{clm:if_F^M_rest_mu_not_in_U_and_is_non-whole} below.

Since $\crit(i^\Uu)<\nu$, the first extender used along $b^\Uu$ is long. Recall that
Claim \ref{clm:chain_of_exts_in_branches}
still holds in the current context,
but Claims \ref{clm:if_no_gens_moved_then_two_options}--\ref{clm:main_claim_when_nu_leq_crit(i^U)}
are not relevant.

Let $\alpha_0+1=\min(b^\Tt)$.

Recall the \emph{natural length} $\nu(F)$
of the active extender $F=F^N$ of a premouse $N$ is conventionally defined as \[\nu(F)=\max(\theta^{+N},\sup_{\gamma\in G}(\gamma+1)) \]
where $\theta=\crit(F^N)$ and $G$ is the set of generators of $F$.
The following variant of this definition is more convenient in what follows.

\begin{dfn}\label{dfn:nu-tilde}
 Let $N$ be active short and $\theta=\crit(F^N)$.
 Let $G$ be the set of generators of $F^N$.
 Then we define
 \[\widetilde{\nu}(F^N)=\sup_{\gamma\in G}(\gamma+1).\qedhere\]
\end{dfn}

The following claim describes how the MS-ISC is propagated (or not) along dropping iterates:
\begin{clmfive}\label{clm:drop_change_ISC}
 Suppose $[0,\alpha_0]^\Tt\cap\dropset^\Tt\neq\emptyset$ and $M^\Tt_{\alpha_0}$ is active short.
 Then:
 \begin{enumerate}
  \item If $\deg^\Tt_{\alpha_0}>0^-$ then $M^\Tt_{\alpha_0}$ satisfies the MS-ISC.
  \item Suppose $\deg^\Tt_{\alpha_0}=0^-$.
  Let $\beta_0+1\leq^\Tt\alpha_0$
  be least such that $(\beta_0+1,\alpha_0]^\Tt\cap\dropset^\Tt=\emptyset$
  and $\deg^\Tt_{\beta_0+1}=0^-$.
  Then:
  \begin{enumerate}
  \item The following are equivalent:
  \begin{enumerate}[label=--]
   \item 
$M^\Tt_{\alpha_0}$ satisfies the MS-ISC,
  \item $\spc(E^\Tt_\beta)<\nu(F(M^{*\Tt}_{\beta+1}))$
  for all $\beta+1\in[\beta_0+1,\alpha_0]^\Tt$,
  \item $\spc(E^\Tt_\beta)<\widetilde{\nu}(F(M^{*\Tt}_{\beta+1}))$ for all $\beta+1\in[\beta_0+1,\alpha_0]^{\Tt}$  (see Definition \ref{dfn:nu-tilde}).
  \end{enumerate}
  \item If $M^\Tt_{\alpha_0}$ satisfies the MS-ISC
  then $\nu(F(M^\Tt_{\alpha_0}))=\sup i^{*\Tt}_{\beta_0+1,\alpha_0}``\nu(F(M^{*\Tt}_{\beta_0+1}))$.
  \item\label{item:c_lose_ISC} Suppose $F(M^\Tt_{\alpha_0})$ fails the MS-ISC
  and let $F=F(M^\Tt_{\alpha_0})$ and $\xi_0+1\leq^\Tt\alpha_0+1$
  be least such that $\widetilde{\nu}=\widetilde{\nu}(\bar{F})\leq\spc(E^\Tt_{\xi_0})
  $ where $\bar{F}=F(M^{*\Tt}_{\xi_0+1})$.  Let $\theta=\crit(E^\Tt_{\xi_0})$.
  Let $\bar{N}=M^{*\Tt}_{\xi_0+1}$
  and $N=M^\Tt_{\alpha_0}$. Let
  \[ k=i^{*\Tt}_{\xi_0+1,\alpha_0}:\bar{N}\to N \]
  be the iteration map.
  Let $\nu'=\sup k``\widetilde{\nu}$.
  Then:
  \begin{enumerate}
   \item\label{item:char_nu'_in_c_lose_ISC}
$\nu'$ is the least $\nu''$
  such that $F$ fails the MS-ISC
  at $\nu''$; moreover,
  $F\rest\nu'\notin N$,

  \item\label{item:when_nu-tilde_leq_theta} If $\widetilde{\nu}\leq\theta$, so $\widetilde{\nu}=\nu'$, then letting $\nu^*=\theta$, we have:
  \begin{enumerate}
  \item $\theta=\nu^*$ is the least generator  $\theta''$ of $F$ such that $\theta''\geq\widetilde{\nu}$,
and
\item $\theta=\nu^*$ is a cardinal of $N$,
in fact inaccessible in $N$,
 \item
either:
  \begin{itemize}
  \item$\widetilde{\nu}$ is a limit of generators of $\bar{F}$ and of $F$ and is a cardinal in $\bar{N}$ and $N$, or
  \item $\widetilde{\nu}=\gamma+1$ for some $\gamma$
  and $\bar{F}$ and $F$ have no generators in the interval $(\gamma,\gamma^{+\bar{N}}]=(\gamma,\gamma^{*N}]$
  (and of course in fact $\gamma^{+\bar{N}}<\theta$).
  \end{itemize}
  \end{enumerate}
  \item\label{item:char_nu*_in_c_lose_ISC} If $\theta<
  \widetilde{\nu}\leq\theta^{+\bar{N}}=\spc(E^\Tt_{\xi_0})$, so $E^\Tt_{\xi_0}$ is long, then
  letting \[\nu^*=\sup k``\theta^{+\bar{N}},\] we have:
  \begin{enumerate}
  \item $\nu^*$ is the least  generator of $F$ which is $\geq\nu'$,
  and letting $\widehat{\theta}=k(\theta)$, we have
  \[ \widehat{\theta}<\nu'\leq\nu^*<\widehat{\theta}^{+N}, \]
  and $F^{N|\nu^*}$ is the short extender derived from $k$,
  \item either:
  \begin{itemize}
  \item $\widetilde{\nu}=\theta^{+\bar{N}}$ is a limit
  of generators of $\bar{F}$
  and $\nu'=\nu^*$  a limit of generators of $F$,
  or
  \item $\widetilde{\nu}=\gamma+1$
  and $\bar{F}$ has no generators
  $>\gamma$ (hence no generators in the interval $(\gamma,\gamma^{+\bar{N}}=\theta^{+\bar{N}}]$),
  and $F$
  has no generators in the interval $(k(\gamma),\nu^*)$,
  but \[k(\gamma)+1=\nu'<\nu^*<k(\gamma)^{+N}\] 
  \end{itemize}
  \end{enumerate}
  \item\label{item:when_F_rest_nu'_type_Z} $F\rest\nu'$ is type Z iff $\theta=\crit(\bar{F})$ (hence the ultrapower $M^\Tt_{\xi_0+1}=\Ult_0(M^{*\Tt}_{\xi_0+1},E^\Tt_{\xi_0})$ is formed to avoid the protomouse and $\crit(F)=\theta$).
  \end{enumerate}
  \end{enumerate}
 \end{enumerate}
\end{clmfive}
\begin{proof}
 This is mostly by standard calculations.
 Note that the key new case that can arise
 in part \ref{item:c_lose_ISC}
 is where $M^{*\Tt}_{\xi_0+1}$ is active type 1;
 i.e. $\widetilde{\nu}=\chi+1$ where $\chi=\crit(\bar{F})$, and also $\chi=\theta=\crit(E^\Tt_{\xi_0})$,
 and $E^\Tt_{\xi_0}$ is long with $\chi^{+M^{*\Tt}_{\xi_0+1}}=\spc(E^\Tt_{\xi_0})$.
 In this case, $M^\Tt_{\xi_0+1}$ is formed
 in the manner to avoid the protomouse.
 As stated in the claim, we get here that
 the least failure of the MS-ISC
 for $F$ is with respect to $F\rest(i^{*\Tt}_{\xi_0+1,\alpha_0}(\chi)+1)$. This is, moreover,
 a type Z proper segment of $F$, and $E=F\rest i^{*\Tt}_{\xi_0+1,\alpha_0}(\chi)\in\es^{M^\Tt_{\alpha_0}}$, and $\lh(E)$ is the least generator of $F$ strictly above $i^{*\Tt}_{\xi_0+1,\alpha_0}(\theta)$, and $E$ is also the short extender derived from $i^{*\Tt}_{\xi_0+1,\alpha_0}$.
\end{proof}

Recall that $U=\Ult_0(M,F^M\rest\nu)$
and $\mu=\card^M(\nu)$. 
We have $\beta_0+1=\min(b^\Uu)$ and $E^\Uu_{\beta_0}$ is long.
Let $\mu'=\crit(E^\Uu_{\beta_0})$,
so either $\mu'=\mu<\nu\leq\mu^{+U}<\mu^{+M}$,
or $(\mu')^{+U}=(\mu')^{+M}=\mu=\nu$. Let $E$ be the (short) extender derived from $j$. Let $E^*$ be that derived from $i^\Uu\com j=i^\Tt$.
Also let $\chi=\sup i^\Uu``((\mu')^{+U})$
and $\nu'=\sup i^{\Uu}``\nu$.
In the following claim, recall that
$\kappa\leq\mu'$.

\begin{clmfive}\label{clm:Uu_gens_movement}  We have:
\begin{enumerate}
 \item 
$\chi<i^\Uu(\mu')^{+M^\Uu_\infty}$
 and in fact, $M^\Uu_\infty|\chi$ is active
 with the short extender derived from $i^\Uu$.
 \item\label{item:char_chi_least_gen>=nu'} $\chi$ is the least generator of $E^*$ which is $\geq\nu'$,
  \item\label{item:when_E_rest_mu'_in_U} Suppose $E\rest\mu'\in U$. Then either:
  \begin{enumerate}
  \item we have:
  \begin{enumerate}[label=--]
  \item $\kappa^{+M}<\nu$,
  \item $\nu=(\mu')^{+U}$ is a limit of generators of $E$,
  \item $\nu'$ is a limit of generators
  of $E^*$,
  \item $\chi=\nu'$,
  \end{enumerate}or
  \item we have:
  \begin{enumerate}[label=--]
  \item $\kappa^{+M}<\nu$,
  \item $\nu=\gamma+1$ for some $\gamma\in[\mu',(\mu')^{+U})$ (and $\gamma$ is the largest generator of $E$),
  \item $\nu'=i^\Uu(\gamma)+1$ and  $i^\Uu(\gamma)$
  is a generator of $E^*$ and
  \item  $E$ is non-type Z and $E^*\rest(i^\Uu(\gamma)+1)$ is non-type Z, 
  \item $\chi$ is the least generator of $E^*$ which is $>i^\Uu(\gamma)$,
  \end{enumerate}
  or
  \item we have:
  \begin{enumerate}[label=--]
  \item $\kappa+1=\nu$
  and $\kappa=\mu'$ is the unique generator of $E$,
  \item $\nu'=i^\Uu(\kappa)+1$
  and $i^\Uu(\kappa)$ is a generator of $E^*$,
  \item $E$ is a normal measure
  and $E^*\rest (i^\Uu(\kappa)+1)$
  is type Z,
  \item $\chi$ is the least generator of $E^*$ which is $>i^\Uu(\kappa)$,
  \end{enumerate}
  \end{enumerate}
  and moreover, in each case,
  \begin{enumerate}[label=--] \item $\nu$ is the least $\bar{\nu}$
  such that $E\rest\bar{\nu}\notin U$, and 
  \item $\nu'$
  is the least $\bar{\nu}$ such that $E^*\rest\bar{\nu}\notin M^\Uu_\infty$,
  \end{enumerate}
  and

  \item\label{item:when_E_rest_mu_notin_U} Suppose $E\rest\mu'\notin U$. Then:
  \begin{enumerate}[label=--]
  \item $\kappa^{+M}<\nu$,
  \item $\nu=\mu'+1$ and $\nu'=i^\Uu(\mu')+1$,
  \item $i^\Uu(\mu')$ is a generator of $E^*$,
  \item $E$ is type Z and $E^*\rest(i^\Uu(\mu')+1)$ is type Z,
  \item 
   $\mu'$ is the least $\bar{\nu}$
  such that $E\rest\bar{\nu}\notin U$, and
  \item $\mu'=\sup i^\Uu``\mu'$ is also the least $\bar{\nu}$ such that $E^*\rest\bar{\nu}\notin M^\Uu_\infty$.
  \item $\chi$ is the least generator of $E^*$
  which is $>i^\Uu(\mu')$,
  \item letting $\bar{U}=\Ult_0(M,F^M\rest\mu')$
  and $k:\bar{U}\to U$ be the factor map (so $\crit(k)=\mu'$), and $F^*$ be the short extender derived from $i^\Uu\com k$, then $\lambda(F^*)>i^\Uu(\mu')$
  and $F^*\rest i^\Uu(\mu')=F^{M^\Uu_\infty|\chi}\in\es^{M^\Uu_\infty}$ but $F^*\rest (i^\Uu(\mu')+1)\notin M^\Uu_\infty$.
  \end{enumerate}
  \end{enumerate}
\end{clmfive}
\begin{proof}
 This follows by induction and from standard calculations,
 and because $E^\Uu_{\beta_0}$ is long with $\crit(E^\Uu_{\beta_0})=\mu'$.
 In part \ref{item:when_E_rest_mu_notin_U},
 to see the last clause,
 note that the normal measure $D$ derived from $k$ is not in $U$, and hence not in $M^\Uu_\infty$,
 since $\pow((\mu')^{+U})^{M^\Uu_\infty}\sub\pow((\mu')^{+U})^U$, and therefore $F^*\rest((i^\Uu(\mu')+1))\notin M^\Uu_\infty$,
 since in fact $F^*\rest\{i^\Uu(\mu')\}\notin M^\Uu_\infty$, since this is equivalent to $D$.
\end{proof}

Claims \ref{clm:if_F^M_rest_mu_in_U}--\ref{clm:if_F^M_rest_mu_not_in_U_and_is_non-whole} below each reach a contradiction, and they cover all cases, thereby completing the proof.

\begin{clmfive}\label{clm:if_F^M_rest_mu_in_U}
 Suppose $F^M\rest\mu'\in U$.
 Then:
 \begin{enumerate}
  \item\label{item:E^T_alpha_0_lambda>i(mu)}
$\lambda(E^\Tt_{\alpha_0})>i^\Uu(\mu')$, and
therefore $\lambda(E^\Tt_{\alpha_0})>i^\Uu(\mu')^{+M^\Uu_\infty}$,
\item\label{item:E^T_alpha_0_is_active_ext}
  $E^\Tt_{\alpha_0}=F(M^\Tt_{\alpha_0})$.

 \item\label{item:alpha_0_is_only_ext} $\alpha_0+1=^{\Tt}\infty$, so $E^\Tt_{\alpha_0}$ is the only extender used along $b^\Tt$.
 \item\label{item:no_drop_to_alpha_0} $[0,\alpha_0]^\Tt\cap\dropset^\Tt=\emptyset$,
 \item\label{item:0=1}  $0=1$.
 \end{enumerate}
\end{clmfive}
\begin{proof}
 Part \ref{item:E^T_alpha_0_lambda>i(mu)}: Because if $\kappa<\mu'$ then $i^\Uu(F^M\rest\mu')=E^*\rest i^\Uu(\mu')\in \es^{M^\Uu_\infty}$, and if $\kappa=\mu'$
 then $E^*\rest i^\Uu(\kappa)=F^{M^\Uu_\infty|\chi}\in\es^{M^\Uu_\infty}$
 (since if $\kappa=\mu'$
 then $\Ult_0(U,E^\Uu_{\beta_0})$
 is formed avoiding the protomouse).
 
 Part \ref{item:E^T_alpha_0_is_active_ext}:
 Since $\delta=i^\Uu(\mu')^{+M^\Uu_\infty}<\lambda(E^\Tt_{\alpha_0})$,
 note that $E^*\rest\delta=E^\Tt_{\alpha_0}\rest\delta$. Suppose $E^\Tt_{\alpha_0}\neq F(M^\Tt_{\alpha_0})$.
 Then  $E^\Tt_{\alpha_0}$ satisfies
 the MS-ISC. But by Claim \ref{clm:Uu_gens_movement}, $\chi$ is a generator of $E^*$. So by the MS-ISC,
 $E^*\rest\chi\in M^\Tt_{\alpha_0+1}$,
 so $E^*\rest\chi\in M^\Uu_\infty$,
 contradicting Claim \ref{clm:Uu_gens_movement}.
 
 Part \ref{item:alpha_0_is_only_ext}: This is like in the proof of Claim \ref{clm:if_no_gens_moved_then_two_options}.
 
 Part \ref{item:no_drop_to_alpha_0}:
 Suppose $[0,\alpha_0]^\Tt\cap\dropset^\Tt\neq\emptyset$. Apply Claim \ref{clm:drop_change_ISC}
 (let $\nu'_0$ be the ``$\nu'$'' defined in its part \ref{item:c_lose_ISC},
 and let
 $\nu^*$ be as in its part
 \ref{item:when_nu-tilde_leq_theta} or \ref{item:char_nu*_in_c_lose_ISC}, as is relevant),
 in conjunction with Claim \ref{clm:Uu_gens_movement}
 (we  defined $\nu',\chi$ there, and its part \ref{item:when_E_rest_mu'_in_U} applies).
It follows that $\nu'_0=\nu'$ and $\nu^*=\chi$, and $\nu^*$ is
defined as in \ref{item:char_nu*_in_c_lose_ISC}
of Claim \ref{clm:drop_change_ISC}.
(For by Claim \ref{clm:drop_change_ISC}
part \ref{item:char_nu'_in_c_lose_ISC},
$\nu'_0$ is the least $\nu''$
such that $E^*\rest\nu''\notin M^\Tt_\infty$, and by its parts
\ref{item:when_nu-tilde_leq_theta}
and \ref{item:char_nu*_in_c_lose_ISC},
$\nu^*$
is the least generator of $E^*$ which is $\geq\nu'_0$.
On the other hand, by Claim \ref{clm:Uu_gens_movement}
part \ref{item:when_E_rest_mu'_in_U},
$\nu'$ is the least $\bar{\nu}$
such that $E^*\rest\bar{\nu}\notin M^\Uu_\infty$
(hence $\nu'_0=\nu'$)
and by its part \ref{item:char_chi_least_gen>=nu'},
$\chi$ is the least generator of $E^*$ which is $\geq\nu'$ (hence $\chi=\nu^*$). But $M^\Uu_\infty|\chi$ is active,
so $\nu^*$ is not as in \ref{item:when_nu-tilde_leq_theta}
of Claim \ref{clm:drop_change_ISC}.)

But now using the extender $H=F^{M^\Uu_\infty|\chi}$,
we can pull back the extender derived from $i^\Uu\com j=i^\Tt$ under $H$, to see that $E=F^M\rest\nu$
has trivial completion exactly $F(M^{*\Tt}_{\xi_0+1})$.
So it follows that $\Uu$ is trivial and the last extender of $\Tt$ is just $F(M^{*\Tt}_{\xi_0+1})$,
contradicting our present case assumption,
which requires in particular that $\Uu$ is non-trivial.

 Part \ref{item:0=1}: Recall that
 at the start of the proof,
  we reduced to the case that $\rho_{\mathrm{D}}^M=0$. So $F^M\rest p_{\mathrm{D}}^M\notin M$.
 But assuming that $[0,\alpha_0]^\Tt\cap\dropset^\Tt=\emptyset$,
 then $F^M\rest p_{\mathrm{D}}^M=E^\Tt_{\alpha_0}\rest
 p'$ where $p'=i^\Tt_{0\alpha_0}(p_{\mathrm{D}}^M)$.
 (Recall that this is still true in the case that $0<^\Tt\alpha_0$ and $\crit(i^\Tt_{0\alpha_0})=\crit(F^M)$.)
 So then $F^M\rest p_{\mathrm{D}}^M$
 is derived from $i^\Uu\com j$ at $p'$.
 But $F^M\rest\nu\in M$,
 so $M_\nu\in M$ (recall $M_\nu$ was defined at the start of the proof).
 As in the proof of Dodd-absent-solidity (Theorem \ref{tm:rho^M_D=0_implies_Dodd-absent-solid_and_universal}),
 all component measures of extenders applied along $b^\Uu$ are in the model to which they apply,
 and in fact if $\alpha+1\leq^\Uu\infty$
 then the component measures of $E^\Uu_\alpha$
 are in $M^\Uu_\alpha|i^\Uu_{0\alpha}(j(\kappa))$.
  Therefore there is a  finite linear ``iteration'' $\vec{\mu}\in U|j(\kappa)$ capturing $(M^\Uu_\infty,p',\kappa^{+M})$, so letting $\sigma:\Ult_0(U,\vec{\mu})\to M^\Uu_\infty$
  be the final copy map, we have $p'\in\rg(\sigma)$ and $\sigma\com i^{U,0}_{\vec{\mu}}=i^\Uu$.
  This can then be converted into an equivalent ``iteration'' of $M_\nu$
  (here $M_\nu$ is defined as in Remark \ref{rem:M_X}). Let $M'=\Ult_0(M_\nu,\vec{\mu})$.
  Now note that $F^M\rest p_{\mathrm{D}}^M=F^{M'}\rest \sigma^{-1}(p')$,
  and since $M_\nu\in M$, therefore $F^M\rest p_{\mathrm{D}}^M\in M$, a contradiction.
\end{proof}

\begin{clmfive}\label{clm:if_F^M_rest_mu_not_in_U_and_is_whole}
 Suppose  $F^M\rest\mu'\notin U$ and $F^M\rest\mu'$ is whole (so $\mu'=\mu<\nu$). Then:
 \begin{enumerate}
  \item\label{item:exactly_two_exts_in_b^Tt} There are at least 2 extenders
  used along $b^\Tt$; let $E^\Tt_{\alpha_0},E^\Tt_{\alpha_1}$
  be the first 2,
  \item\label{item:alpha_0=0_and_E^T_0=F^M_rest_mu}$\alpha_0=0$ and $E^\Tt_{0}=F^M\rest\mu$,
  \item\label{item:E^T_alpha_1_short_and_crit=mu}
  $E^\Tt_{\alpha_1}$ is short with $\crit(E^\Tt_{\alpha_1})=\mu$,
  \item\label{item:analyse_E^T_alpha_1} letting $\beta_1+1=\min((0,\alpha_1]^\Tt)$,
  then $\beta_1+1\in\dropset^\Tt$ and $\deg^\Tt_{\beta_1+1}=0^-$, but
  $(\beta_1+1,\alpha_1]^\Tt\cap\dropset^\Tt=\emptyset$, and
   $M^{*\Tt}_{\beta_1+1}\pins M^\Tt_1$ is active short
  type 1 with $\crit(F(M^{*\Tt}_{\beta_1+1}))=\mu=\crit(E^\Tt_{\beta_1})$
  and $E^\Tt_{\beta_1}$ is long,
  so $M^\Tt_{\beta_1+1}$ is formed avoiding the protomouse,
  \item\label{item:0=1_in_proof} $0=1$.
  \end{enumerate}
\end{clmfive}

\begin{proof}
Parts 
\ref{item:exactly_two_exts_in_b^Tt},\ref{item:alpha_0=0_and_E^T_0=F^M_rest_mu}, \ref{item:E^T_alpha_1_short_and_crit=mu}:
Since $F^M\rest\mu\notin U$
but $\mu<\nu$, by induction,
$\nu=\mu+1$ and $F^M\rest\nu$ is type Z.
So $\bar{E}=F^M\rest\mu\in\es^M$,
and $\Ult_0(M,\bar{E}))||\mu^{+\Ult_0(M,\bar{E})}=U||\mu^{+U}$.
So $E^\Tt_0=\bar{E}$. Moreover,
because $\bar{E}=E^*\rest\mu$,
therefore $1<^\Tt\infty$. Since $F^M\rest\mu$ is whole (by assumption of the claim), we have $\lambda(F^M\rest\mu)=\mu$,
so $\mu=\crit(E^\Tt_{\alpha_1})$
where $\alpha_1+1=\min(1,\infty]^\Tt$
(we have $M^\Tt_1\neq M^\Tt_\infty$,
since $E^*$ is not generated by $\mu'$,
so $\alpha_1$ exists).
And $E^\Tt_{\alpha_1}$ is short
by Claim \ref{clm:chain_of_exts_in_branches}.

 Part
 \ref{item:analyse_E^T_alpha_1}: As mentioned above, $\nu(E^\Tt_{0})=\lambda(E^\Tt_0)=\mu$,
 so $M^\Tt_{1}\neq M^\Tt_\infty$,
 so we can let $\alpha_1+1=\min((1,\infty]^\Tt)$.
 Note that with $k:\bar{U}\to M^\Uu_\infty$ and $F^*$
 as in Claim \ref{clm:Uu_gens_movement} part \ref{item:when_E_rest_mu_notin_U}, we have $\bar{U}=M^\Tt_1$ and $k=i^\Tt_{1\infty}$.
 Since $F^*\rest i^\Uu(\mu)\in M^\Uu_\infty$,
 and 
 $E^\Tt_{\alpha_1}$ is short,
 we have $i^\Uu(\mu)<\lambda(E^\Tt_{\alpha_1})$.
 So $F^*\rest i^\Uu(\mu)^{+M^\Uu_\infty}=E^\Tt_{\alpha_1}\rest i^\Uu(\mu)^{+M^\Uu_\infty}$,
 and note that $F^*\rest i^\Uu(\mu)^{+M^\Uu_\infty}$
 is just the extender of $i^\Uu_{0\gamma}\com k$
 where $\gamma\in b^\Uu$ is least such that $i^\Uu_{0\gamma}(\mu)<\crit(i^\Uu_{\gamma\infty})$,
 and so $M^\Uu_\gamma=\Ult_0(U,F^*\rest i^\Uu(\mu)^{+M^\Uu_\infty})$.
 
  Now we claim $E^\Tt_{\alpha_1}=F(M^\Tt_{\alpha_1})$.
 For suppose not. Then the full MS-ISC
 holds for $E^\Tt_{\alpha_1}$.
 So since $\chi=\lh(F^*\rest i^\Uu(\mu))<i^\Uu(\mu)^{+M^\Uu_\infty}$
 and $\chi$ is a generator of $F^*$, and hence also of $E^\Tt_{\alpha_1}$,
 we have $E^\Tt_{\alpha_1}\rest\chi\in M^\Tt_\infty$.
 But  by Claim \ref{clm:Uu_gens_movement},
 $F^*\rest (i^\Uu(\mu)+1)\notin M^\Uu_\infty$, so $E^\Tt_{\alpha_1}\rest(i^\Uu(\mu)+1)\notin M^\Tt_\infty$,
 although $i^\Uu(\mu)+1<\chi$,
 contradiction.
 
 So $E^\Tt_{\alpha_1}=F(M^\Tt_{\alpha_1})$.
 Just like in the proof of Claim \ref{clm:if_F^M_rest_mu_in_U} part \ref{item:0=1},
 it follows that $[0,\alpha_1]^\Tt\cap\dropset^\Tt\neq\emptyset$.
 So by Claim \ref{clm:drop_change_ISC} part \ref{item:when_F_rest_nu'_type_Z}
 and since $F^*\rest(i^\Uu(\mu)+1)$ is type Z,
 and also since $\mu^{+U}=\mu^{+M^\Tt_1}=\lh(E^\Tt_0)$, part \ref{item:analyse_E^T_alpha_1} follows.
 (We have $M^{*\Tt}_{\beta_1+1}\pins M^\Tt_1$, as opposed to $M^{*\Tt}_{\beta_1+1}\pins M^\Tt_\alpha$ with some $\alpha>1$,
 since $F(M^{*\Tt}_{\beta_1+1})$ has to be type 1
 with $\crit(F(M^{*\Tt}_{\beta_1+1}))=\mu$,
 and hence $\rho_1(M^{*\Tt}_{\beta_1+1})=\mu^{+U}$.)
 
Part \ref{item:0=1_in_proof}: Note that the normal measure derived from $F(M^{*\Tt}_{\beta_1+1})$ is just the measure
derived from $E^\Tt_{\alpha_1}$ with seed $i^\Uu(\mu)$,
which is just the normal measure $D$ derived from $k$.
But then \[\Ult_0(M^\Tt_1,F(M^{*\Tt}_{\beta_1+1}))=\Ult_0(\bar{U},D)=U,\]
and it follows that in fact $E^\Tt_1=F(M^{*\Tt}_{\beta_1+1})$
and $M^\Tt_2=U$, so $\Uu$ is trivial, a contradiction.
\end{proof}

\begin{clmfive}\label{clm:if_F^M_rest_mu_not_in_U_and_is_non-whole}
 Suppose $F^M\rest\mu'\notin U$ and $F^M\rest\mu'$ is non-whole (so $\mu'=\mu<\nu$). Then $0=1$.
 \end{clmfive}
 \begin{proof}
 Claim \ref{clm:drop_change_ISC} part \ref{item:when_E_rest_mu_notin_U} applies again.
  And like before, no non-dropping image of $F^M$ is used along $b^\Tt$. Note that $F^M\rest\mu\ins E^\Tt_{\alpha_0}$. If $E^\Tt_{\alpha_0}=F^M\rest\mu$,
then since this is non-whole, we get $\mu<\crit(i^\Tt_{\alpha_0+1,\infty})$, but then $\mu$ is not a generator of the derived extender, a contradiction. So $\mu<\nu(E^\Tt_{\alpha_0})$.
But since $F^M\rest\mu\notin M^\Tt_\infty$,
$E^\Tt_{\alpha_0}$ fails the MS-ISC at $\mu$.
So $E^\Tt_{\alpha_0}=F(M^\Tt_{\alpha_0})$
and $[0,\alpha_0]^\Tt\cap\dropset^\Tt\neq\emptyset$,
and $\rho_1(M^\Tt_{\alpha_0})=\mu$.
Let $\beta_0+1=\min((0,\alpha_0]^\Tt)$.
Then $\pred^\Tt(\beta_0+1)=0$,
since $\mu<\lh(E^\Tt_0)$ and $\rho_1(M^{*\Tt}_{\beta_0+1})=\mu$. If $\crit(E^\Tt_{\beta_0})<\mu$
then the MS-ISC would hold for $E^\Tt_{\alpha_0}$ at $\mu$. So $\crit(E^\Tt_{\beta_0})=\mu$.
The rest of the argument is basically like in the proof of Claim \ref{clm:if_F^M_rest_mu_not_in_U_and_is_whole}. 
 \end{proof}

So under the assumption that $\crit(i^\Uu)<\nu$, we have a reached a contradiction. This completes the proof of the theorem.
\end{proof}

\bibliographystyle{plain}
\bibliography{../bibliography/bibliography}

\end{document}